\documentclass[reqno,final]{amsart} 

\usepackage[margin=1in]{geometry}
\usepackage{kotex}

\usepackage[normalem]{ulem}
\usepackage{soul,footmisc}
\usepackage{booktabs,showlabels} 
\usepackage{float}
\usepackage{tikz}
\usetikzlibrary{calc} 
\usetikzlibrary{arrows.meta, decorations.pathmorphing, positioning}

 \usepackage{showkeys}

\usepackage{amsthm}

\usepackage{tikz}
\usetikzlibrary{arrows.meta}

\usetikzlibrary{decorations.markings}

 \usepackage{soul}
\usepackage{cancel}
\usepackage{ulem}
\usepackage[utf8]{inputenc} 
\usepackage{amssymb}
\usepackage{verbatim} 
\usepackage{amssymb,enumerate}
\usepackage{amsthm}
\usepackage{amsmath}
\usepackage{amsfonts}
\usepackage{latexsym}
\usepackage{mathtools}
\usepackage{enumitem} 
\usepackage[english]{babel}
\usepackage{esint}
\usepackage{cite}
\usepackage{calrsfs}
\usepackage[backref=none]{hyperref} 
\usepackage{color}

\newcommand{\Black}{\textcolor{black}}
\usepackage{float}
\theoremstyle{plain}

\newtheorem{theorem}{Theorem}
\newtheorem{lemma}[theorem]{Lemma}
\newtheorem{definition}[theorem]{Definition}
\newtheorem{proposition}[theorem]{Proposition}
\newtheorem{remark}[theorem]{Remark}

\newtheorem{corollary}[theorem]{Corollary}

\usepackage{CJKutf8}
\let\hide\iffalse

\numberwithin{equation}{section}
\numberwithin{theorem}{section}
\usepackage{enumerate}
\newcommand{\eqdef }{\overset{\mbox{\tiny{def}}}{=}}
\newcommand{\rth}{{\mathbb{R}^3}}

\usepackage{ulem, bm}

\newcommand{\R}{\mathbb{R}}

\newcommand{\ph}{\bm{\varphi}}
 
\newcommand{\A}{\mathbf{A}}
\newcommand{\E}{\mathbf{E}}
\newcommand{\Ei}{\mathbf{E}_{i}}

\newcommand{\Eone}{\mathbf{E}_{1}}
\newcommand{\Etwo}{\mathbf{E}_{2}}

\newcommand{\Eoned}{\E^{(1)}}
\newcommand{\Etwod}{\E^{(2)}}

\newcommand{\Est}{\E_{\textup{st}}}

\newcommand{\Bst}{\B_{\textup{st}}}
\newcommand{\Eonedst}{\E_{\textup{st}}^{(1)}}
\newcommand{\Etwodst}{\E_{\textup{st}}^{(2)}}

\newcommand{\p}{\partial}
\newcommand{\Boned}{\B^{(1)}}
\newcommand{\Btwod}{\B^{(2)}}

\newcommand{\Bonedst}{\B_{\textup{st}}^{(1)}}
\newcommand{\Btwodst}{\B_{\textup{st}}^{(2)}}
\newcommand{\Eth}{\mathbf{E}_{3}}
\newcommand{\Eione}{\mathbf{E}^{(1)}_{\pm,\textup{par},i}}

\newcommand{\Eitwo}{\mathbf{E}^{(2)}_{\pm,\textup{par},i}}
\newcommand{\Ethone}{\mathbf{E}^{(1)}_{\pm,\textup{par},3}}

\newcommand{\Ethtwo}{\mathbf{E}^{(2)}_{\pm,\textup{par},3}}

\renewcommand{\paragraph}[1]{\medskip\noindent\textbf{#1.} }
\newcommand{\Eipar}{\mathbf{E}_{\textup{par},i}}
\newcommand{\Ethpar}{\mathbf{E}_{\textup{par},3}}

\newcommand{\Eihomo}{\mathbf{E}_{\textup{hom},i}}
\newcommand{\Ethhomo}{\mathbf{E}_{\textup{hom},3}}
\newcommand{\Bihomo}{\mathbf{B}_{\textup{hom},i}}
\newcommand{\Bthhomo}{\mathbf{B}_{\textup{hom},3}}

\newcommand{\B}{\mathbf{B}}
\newcommand{\Bi}{\mathbf{B}_{i}}

\newcommand{\Bone}{\mathbf{B}_{1}}
\newcommand{\Btwo}{\mathbf{B}_{2}}
\newcommand{\Bth}{\mathbf{B}_{3}}

\newcommand{\vZ}{v^0_\pm}
\newcommand{\vZio}{v^0_\iota}

\newcommand{\energy}{\mathbb{E}_{\pm}}

\newcommand{\tb}{t_{\pm,\mathbf{b}}}
\newcommand{\tf}{t_{\pm,\mathbf{f}}}

\newcommand{\tblo}{t^{l+1}_{\pm,\mathbf{b}}}
\newcommand{\tbk}{t^{k}_{\pm,\mathbf{b}}}
\newcommand{\tflo}{t^{l+1}_{\pm,\mathbf{f}}}

\newcommand{\tbst}{t_{\pm,\textup{st},\mathbf{b}}}
\newcommand{\tfst}{t_{\pm,\textup{st},\mathbf{f}}}

\newcommand{\tbstlo}{t^{l+1}_{\pm,\textup{st},\mathbf{b}}}

\newcommand{\tbstk}{t^{k}_{\pm,\textup{st},\mathbf{b}}}
\newcommand{\tfstlo}{t^{l+1}_{\pm,\textup{st},\mathbf{f}}}

\newcommand{\xb}{x_{\pm,\mathbf{b}}}

\newcommand{\xblo}{x_{\pm,\mathbf{b}}^{l+1}}

\newcommand{\xbst}{x_{\pm,\textup{st},\mathbf{b}}}
\newcommand{\xbstlo}{x_{\pm,\textup{st},\mathbf{b}}^{l+1}}

\newcommand{\vb}{v_{\pm,\mathbf{b}}}

\newcommand{\vbst}{v_{\pm,\textup{st},\mathbf{b}}}
\newcommand{\vblo}{v_{\pm,\mathbf{b}}^{l+1}}
\newcommand{\vbstlo}{v_{\pm,\textup{st},\mathbf{b}}^{l+1}}

\newcommand{\fl}{F^{l}_\pm}

\newcommand{\fone}{F_{\pm}^{(1)}}

\newcommand{\ftwo}{F_{\pm}^{(2)}}

\newcommand{\foneio}{F_{\iota}^{(1)}}

\newcommand{\ftwoio}{F_{\iota}^{(2)}}

\newcommand{\fonest}{F_{\pm,\textup{st}}^{(1)}}
\newcommand{\fonestio}{F_{\iota,\textup{st}}^{(1)}}

\newcommand{\ftwost}{F_{\pm,\textup{st}}^{(2)}}
\newcommand{\ftwostio}{F_{\iota,\textup{st}}^{(2)}}
\newcommand{\fm}{F^{k}_\pm}

\newcommand{\fmst}{F^{k}_{\pm,\textup{st}}}

\newcommand{\fmmstio}{F^{k-1}_{\iota,\textup{st}}}

\newcommand{\fzerost}{F^{0}_{\pm,\textup{st}}}

\newcommand{\fmmnstio}{F^{k-n}_{\iota,\textup{st}}}

\newcommand{\fzerostio}{F^{0}_{\iota,\textup{st}}}
\newcommand{\fmm}{F^{k-1}_\pm}

\newcommand{\fn}{F^{n}_\pm}
\newcommand{\fnst}{F^{n}_{\pm,\textup{st}}}

\newcommand{\fnmstio}{F^{n-1}_{\iota,\textup{st}}}

\newcommand{\fnm}{F^{n-1}_\pm}

\newcommand{\flo}{F^{l+1}_\pm}
\newcommand{\fpl}{f^{l}_\pm}

\newcommand{\fplo}{f^{l+1}_\pm}

\newcommand{\fploio}{f^{l+1}_\iota}

\newcommand{\fstl}{F_{\pm,\textup{st}}^{l}}
\newcommand{\fstlo}{F_{\pm,\textup{st}}^{l+1}}

\newcommand{\Elo}{\mathbf{E}^{l+1}}

\newcommand{\Epl}{\mathcal{E}^{l}}

\newcommand{\Eplo}{\mathcal{E}^{l+1}}

\newcommand{\Blo}{\mathbf{B}^{l+1}}

\newcommand{\Bpl}{\mathcal{B}^{l}}

\newcommand{\Bplo}{\mathcal{B}^{l+1}}

\newcommand{\Blobf}{\mathbf{B}^{l+1}}

\newcommand{\Elbf}{\mathbf{E}^{l}}

\newcommand{\Emmbf}{\mathbf{E}^{k-1}}

\newcommand{\Enmbf}{\mathbf{E}^{n-1}}

\newcommand{\Bmmbf}{\mathbf{B}^{k-1}}

\newcommand{\Bnmbf}{\mathbf{B}^{n-1}}

\newcommand{\Embfst}{\mathbf{E}^{k}_{\textup{st}}}
\newcommand{\Bmbfst}{\mathbf{B}^{k}_{\textup{st}}}

\newcommand{\Emmbfst}{\mathbf{E}^{k-1}_{\textup{st}}}

\newcommand{\Enmbfst}{\mathbf{E}^{n-1}_{\textup{st}}}

\newcommand{\Bmmbfst}{\mathbf{B}^{k-1}_{\textup{st}}}

\newcommand{\Bnmbfst}{\mathbf{B}^{n-1}_{\textup{st}}}

\newcommand{\Emtbf}{\mathbf{E}^{k-2}}

\newcommand{\Entbf}{\mathbf{E}^{n-2}}

\newcommand{\Bmtbf}{\mathbf{B}^{k-2}}

\newcommand{\Bntbf}{\mathbf{B}^{n-2}}

\newcommand{\Emm}{\mathbf{E}^{k-1}}

\newcommand{\Enm}{\mathbf{E}^{n-1}}

\newcommand{\Bmm}{\mathbf{B}^{k-1}}

\newcommand{\Bnm}{\mathbf{B}^{n-1}}

\newcommand{\Elobf}{\mathbf{E}^{l+1}}

\newcommand{\Blbf}{\mathbf{B}^{l}}

\newcommand{\Bstlobf}{\mathbf{B}_{\textup{st}}^{l+1}}

\newcommand{\Estlbf}{\mathbf{E}_{\textup{st}}^{l}}

\newcommand{\Estlobf}{\mathbf{E}_{\textup{st}}^{l+1}}

\newcommand{\Bstlbf}{\mathbf{B}_{\textup{st}}^{l}}

\let\hide\iffalse

\newcommand{\FS}{\mathcal{F}^l_{\pm}}
\newcommand{\ZS}{\mathcal{Z}_\pm}
\newcommand{\XS}{\mathcal{X}_\pm}
\newcommand{\VS}{\mathcal{V}_\pm}

\newcommand{\ZSl}{\mathcal{Z}^{l}_\pm}
\newcommand{\XSl}{\mathcal{X}^{l}_\pm}
\newcommand{\VSl}{\mathcal{V}^{l}_\pm}
\newcommand{\ZSlo}{\mathcal{Z}^{l+1}_\pm}
\newcommand{\XSlo}{\mathcal{X}^{l+1}_\pm}
\newcommand{\VSlo}{\mathcal{V}^{l+1}_\pm}

\newcommand{\XSk}{\mathcal{X}^{k}_\pm}
\newcommand{\VSk}{\mathcal{V}^{k}_\pm}
\newcommand{\hVSlo}{\hat{\mathcal{V}}^{l+1}_\pm}

\newcommand{\Zlo}{Z_\pm^{l+1}}
\newcommand{\Xlo}{X_\pm^{l+1}}
\newcommand{\Vlo}{V_\pm^{l+1}}

\newcommand{\Xk}{X_\pm^{k}}
\newcommand{\Vk}{V_\pm^{k}}
\newcommand{\Zst}{Z_{\pm,\textup{st}}}
\newcommand{\Xst}{X_{\pm,\textup{st}}}
\newcommand{\Vst}{V_{\pm,\textup{st}}}

\newcommand{\Vhatst}{\hat{V}_{\pm,\textup{st}}}
\newcommand{\Vhatlo}{\hat{V}_\pm^{l+1}}

\allowdisplaybreaks

\newcommand{\Fp}{f_\pm}
\newcommand{\Ep}{\mathcal{E}}
\newcommand{\Bp}{\mathcal{B}}

\title[Global solution of 3D Vlasov--Maxwell System]{
  Long-Time Dynamics of the 3D Vlasov--Maxwell System with Boundaries
}

\author[J.W. Jang]{Jin Woo Jang$^\dagger$}
\address{$^\dagger$ Department of Mathematics, POSTECH (Pohang University of Science and Technology), Pohang 37673, Republic of Korea. \href{mailto:jangjw@postech.ac.kr}{jangjw@postech.ac.kr} }

\author[C. Kim]{Chanwoo Kim$^\ddagger$}
\address{$^\ddagger$Department of Mathematics, University of Wisconsin-Madison, Madison, WI, 53706,
USA,  \href{mailto:chanwookim.math@gmail.com}{chanwoo.kim@wisc.edu} } 

\AtBeginDocument{%
   \def\MR#1{}
}

\begin{document}

\makeatletter\def\Hy@Warning#1{}\makeatother

\date{\today}


\begin{abstract}

\Black{We construct global-in-time classical solutions to the nonlinear Vlasov--Maxwell system in a three-dimensional half-space beyond the vacuum scattering regime in a physical setting of the solar wind model. Our approach combines the construction of stationary solutions to the associated boundary-value problem with a proof of their asymptotic dynamical stability in $L^\infty$ under small perturbations, providing a new framework for understanding long-time wave-particle interactions in the presence of boundaries and interacting magnetic fields. A key difficulty is that the transport dynamics remains coupled on arbitrarily long time scales to a wave component whose sharp $t^{-1}$ pointwise decay is neither integrable in time nor improved by the boundary. We resolve this obstruction by establishing a new decay mechanism that converts the spatial localization of particle--field interactions into temporal decay through the light-cone geometry of wave propagation. This enables us to close the nonlinear field--particle feedback and establish sharp $t^{-1}$ decay for both the particle distribution and the electromagnetic fields. To the best of our knowledge, this work presents the first construction of asymptotically stable non-vacuum steady states under general perturbations in the full three-dimensional nonlinear Vlasov--Maxwell system.}

\end{abstract}

\setcounter{tocdepth}{2}

\maketitle

\thispagestyle{empty}
 
\tableofcontents

\section{Introduction}\label{sec.intro}

Understanding the long-time behavior of solutions to the Vlasov equations is a central problem in collisionless plasma physics\Black{\cite{MR2863910,MR4436197}}. In particular, the construction of \textit{\textbf{space-inhomogeneous equilibria}} and the proof of their \textit{\textbf{stability}}, especially \textit{\textbf{in the presence of magnetic fields}}, remain largely open. Considerable progress has been made in stability analysis by Guo–Strauss~\cite{GuoStrauss95CPAM, GuoStrauss95}, Guo~\cite{YG97, YG99}, Lin~\cite{MR1849267, MR2119867}, Lin–Strauss~\cite{LinStrauss07, LinStrauss}, Guo–Lin~\cite{GuoLinBGK}, \Black{and Han-Kwan–Nguyen–Rousset~\cite{MR4963539},} yet the three-dimensional nonlinear Vlasov–Maxwell system exhibits substantial additional difficulties. 
Two fundamental obstacles remain. First, the global-in-time existence theory near nontrivial stable equilibria is still unresolved, due to the delicate coupling between fields and particle distributions; see, however, \cite{MR1463042, WangRVM1
} for global existence results obtained under certain symmetry assumptions. Second, classical stability criteria, such as the Penrose criterion~\cite{Penrose, GuoStrauss95}, do not extend to or directly address the nonlinear problem involving magnetic fields and spatially inhomogeneous equilibria~\cite{GuoStrauss95CPAM}.

Motivated by solar wind models~\cite{ELLS}, we study the long-time behavior of the three-dimensional \textbf{Vlasov-Maxwell system} under an ambient gravitational field for two-species distributions \(F_\pm : [0,\infty)\times \Omega \times \mathbb{R}^3\), where \(\Omega\) is \(\mathbb{R}^3_+\). 
The system reads
\begin{equation}\label{2speciesVM}
    \begin{split}
         \partial_t F_\pm +\hat{v}_\pm \cdot \nabla_x F_\pm \pm   \left(\textup{e}\E+\textup{e}\frac{\hat{v}_\pm}{c}\times \B\mp m_\pm g\hat{e}_3\right)\cdot \nabla_v F_\pm = 0,\\
          \frac{1}{c}\partial _t \E-\nabla_x\times \B= - \frac{4\pi}{c}  J  ,\ \ 
        \nabla_x \cdot \E=4\pi  \rho    , \\ \frac{1}{c}\partial_t \B+\nabla_x\times \E=0 ,
\ \ 
        \nabla_x \cdot \B=0,
    \end{split}
\end{equation}
where $v$ is the relativistic momentum and the relativistic velocity $\hat{v}_\pm $ is defined as $\hat{v}_\pm \eqdef \frac{c v}{v^0_\pm}$ with the relativistic particle energy $cv^0_\pm = \sqrt{m_\pm^2c^4+c^2|v|^2}$. Here, $m_+$ and $m_-$ stand for the mass of a proton (ion) (with the charge +1e) and an electron (with the charge -1e), respectively, and $g>0$ denotes the gravitational constant, acting in the downward direction $\hat e_3 \eqdef (0,0,-1)^\top$. The electric charge density and current flux are defined, respectively, by
\begin{equation}\label{rhoJ_dyn}
 \rho  \eqdef\int_\rth \textup{e} (  F_+- F_-)dv, \  \ \   J \eqdef  \int_\rth \textup{e} (\hat{v}_+ F_+- \hat{v}_-F_-)dv . 
\end{equation}
They solve the continuity equation \begin{equation}
    \label{continuity eq}\partial_t  {\rho}  + \nabla_x \cdot  {J}  = 0.
\end{equation}


We consider the physical situation that plasma particles \textit{evaporate} from the surface of the star (exobase). Under this interface, we have a plasma sea, which is a \textit{perfect conductor} that has zero resistance. Hence the natural macroscopic boundary condition of the electromagnetic fields is the following perfect conductor boundary condition: 
\begin{align}
\mathbf{E}_1|_{\partial\Omega}=0=\mathbf{E}_2|_{\partial\Omega},\textup{ and }
\mathbf{B}_3|_{\partial\Omega} =0.\label{perfect cond. boundary}
\end{align}
For the conditions on the particle velocity distribution $F_\pm$ at the boundary $x_3=0$, we further split the momentum domain $\mathbb{R}^3$ into incoming, outgoing, and grazing momenta, respectively. On the incoming boundary, we impose the inflow boundary condition with prescribed profiles $G_\pm$. From a physical perspective, a more realistic condensate–evaporate boundary condition could be adopted without causing significant analytical difficulties; however, for the sake of clarity and simplicity we impose the inflow formulation,
\begin{equation}\label{inflow boundary} 
    F_\pm(t,x,v) = G_\pm(x_\parallel, v), \quad (x,v) \in \gamma_-,
\end{equation}
where the incoming set is defined by $  
    \gamma_- \eqdef \{\, x_3 = 0 \ \text{and} \ v_3 > 0 \,\}$. We assume that the inflow boundary data $G_\pm$ (which may be interpreted physically as measurements taken at a low altitude) and their first-order derivatives in $x_\parallel$ and $v$ vanish rapidly as $|x_\parallel|$ and $|v|$ tend to infinity; in particular, they may be taken to be exponentially localized in both variables. 
  Since the phase space allows particles with arbitrarily large momentum, their exit times are not uniformly bounded, and some particles may remain in the domain for arbitrarily long times. At the same time, the particle dynamics remains coupled to the Maxwell field, whose wave component exhibits only the slow $t^{-1}$ decay and receives no improvement from the boundary. Consequently, the problem cannot be reduced to a finite-time transport analysis and instead requires a genuinely global-in-time treatment of the coupled Vlasov--Maxwell dynamics, much as in the problem without boundaries.

A key feature of our setting is the choice of boundary conditions: perfectly conducting walls for the electromagnetic fields and inflow-type conditions for the particle distribution. Under these assumptions the Vlasov--Maxwell system remains formally \textit{non-dissipative} in total energy, mass, and $L^p$-norms.  
Indeed, there is \textit{no strict macroscopic signature of dissipation} due to boundary fluxes, so even simpler questions, such as uniform-in-time boundedness, become nontrivial. The central contribution of this work is to resolve these difficulties through a new microscopic analysis.

\medskip
\noindent\textbf{Overview of Main Results and Key Insights.} One of the main implications of this paper is that the presence of ambient gravity
\Black{is expected to delineate} the linearly stable regime of the Vlasov--Maxwell
system\Black{: what we rigorously establish is nonlinear asymptotic stability under
sufficiently strong gravity, while the delineation of the linearly stable regime
itself remains heuristic}. In the absence
of gravity, trapped trajectories and unbounded velocity gradients are expected to
produce linearly unstable modes with arbitrarily large growth rates. By contrast, a
sufficiently strong $g>0$, as is most commonly observed in stellar atmospheres, suppresses these instabilities
and places the system in a linearly stable setting. Our results operate precisely in
this stable regime, where we establish the nonlinear stability analysis for the first
time. This novel analysis is consistent with the micro--macro interaction framework
and \Black{the cancellation of the boundary contributions in the potential
representation of the magnetic field} described below, and it underpins both the
uniqueness of steady states and the pointwise $t^{-1}$ decay.

In this paper, we construct a class of steady solutions to the boundary-value problem \eqref{2speciesVM}--\eqref{inflow boundary} and establish their asymptotic dynamic stability under general perturbations of the initial data.\footnote{Convergence in a simpler setting was first numerically observed by Jack Schaeffer in 2005 \cite{MR2187923}.} To the best of our knowledge, this constitutes the first rigorous construction of global-in-time solutions of the nonlinear Vlasov-Maxwell system beyond the vacuum scattering regime in 3D. For this purpose, \textit{the uniqueness theory of the steady problem} is one of the key questions. In Vlasov theory, however, trapped particle trajectories generally preclude uniqueness, especially in the absence of gravity. We establish uniqueness by controlling particle acceleration via moment estimates and weighted regularity techniques, where a subtle exploitation of the ambient gravity plays a crucial role. Our analysis ensures that the weak solution we construct is indeed the Lagrangian solution along the characteristics of the Lorentz force---a force we establish to be Lipschitz through these estimates.  \Black{In this regime, we prove that gravity eliminates trapped trajectories by forcing each particle characteristic to reach the boundary eventually, although the exit times are not uniform over phase space. This particle-wise property, however, does not reduce the evolution to a finite-time transport problem: the inflow continuously replenishes the particle distribution, while the Maxwell field propagates independently of particle exit times. Moreover, the field--particle feedback between the nontrivial steady state and its dynamical perturbation persists for arbitrarily long times. Consequently, neither uniform global-in-time closure nor the sharp $t^{-1}$ decay follows directly from the characteristic geometry.}

The analysis of \textit{asymptotic stability} for the Vlasov--Maxwell system faces
intrinsic challenges: the system is fully hyperbolic, and the electromagnetic fields
are coupled to the particle distribution in a long-range manner. Consequently, the
decay of the fields is not automatically tied to that of the particles and proceeds
only slowly; hence the classical vacuum stability argument of Glassey--Strauss does
not apply, and closing the asymptotic stability loop appears impossible at first
sight. A crucial structural feature we found is a \emph{Coulomb-gauge
\Black{representation of the magnetic field}}, which recasts \Black{the} magnetic
coupling into a benign transport \Black{term carrying the degeneracy factor
$1-|\hat{v}_\pm|^2$, together with a nonlinear term of Glassey--Strauss type, and in
which the boundary contributions cancel under the curl relation $\B=\nabla\times\A$};
this \Black{structure} reconciles field decay with particle dynamics even without
dissipation. Building on this mechanism, our approach further exploits weighted
regularity estimates together with the Lagrangian structure of the dynamics, enabling
precise control of particle trajectories and momentum derivatives in the mean of the
mechanical energy density, while simultaneously tracking the decay of the
electromagnetic fields. The mechanism we identify is inherently
microscopic---visible only from the Lagrangian perspective---and it remains effective
even in the absence of macroscopic dissipation.\footnote{By contrast, the nonlinear
stability of some BGK solutions in the whole space remains a distinct challenging
problem; see~\cite{GuoLinBGK}.}  For the reader's convenience, we present a brief
informal statement of the main results.

\begin{table}[h]
\centering
\begin{tabular}{lccc}
\toprule
 & \textbf{Full Problem} & \textbf{Stationary Problem} & \textbf{Dynamic Perturbation} \\
\midrule
Solution
& $(F_{\pm},\mathbf{E}, \mathbf{B})$ 
& $ (F_{\pm, \mathrm{st}}, \mathbf{E}_{\mathrm{st}}, \mathbf{B}_{\mathrm{st}})$ 
& $(f_{\pm},\mathcal{E}, \mathcal{B})$ \\[0.4em]
 Density, Flux 
 & $\rho,  J$ 
 & $\rho _{\mathrm{st}}, J_{\mathrm{st}}$ 
 & $\varrho, \mathcal{J}$ \\[0.4em]
\bottomrule
\end{tabular}
\end{table}  
 
\noindent\textbf{Informal Statement of Steady Uniqueness. }\itshape
For sufficiently large $g>0$ and $\beta>0$, the stationary two-species Vlasov--Maxwell system in the half-space $\mathbb{R}^3_+$, subject to an exponentially localized $C^1$ inflow boundary \eqref{inflow boundary} and the perfect conductor boundary condition \eqref{perfect cond. boundary}, admits a unique steady solution, where $F_{\pm,\mathrm{st}}$ is locally Lipschitz and $(\mathbf{E}_{\mathrm{st}},\, \mathbf{B}_{\mathrm{st}})$ is Lipschitz. Moreover, the steady states satisfy
\begin{align}\notag
| \nabla_v F_{\pm,\mathrm{st}}(x,v)|     \lesssim  
e^{-\beta \{ c v^0_\pm + m_\pm g  x_3
\}
} 
,  \ \
|\mathbf{E}_{\mathrm{st}}(x)| + |\mathbf{B}_{\mathrm{st}}(x)|   \lesssim 1 
.\notag
\end{align} 
Here, $cv^0_\pm + m_\pm g  x_3$ is the mechanical energy of particle. 

\normalfont
\medskip

\noindent\textbf{Informal Statement of Asymptotic Stability. }\itshape 
Under the same conditions as in the informal statement of steady uniqueness above, suppose the initial perturbations $(f_\pm^{\textup{in}}, \mathcal{E}^{\textup{in}}, \mathcal{B}^{\textup{in}})$ are small in an appropriate weighted $L^\infty$ space. Then there exists a unique global-in-time unsteady solution. Moreover, the perturbative solution decays linearly-in-time \textbf{pointwisely} as
\begin{align*}
     f_\pm(t,x,v)   \lesssim
(1+t)^{-1} 
, \ \
|\mathcal{E} (t,x) | + | \mathcal{B} (t,x) |  \lesssim  
(1+t)^{-1}.
\end{align*}

\normalfont
\smallskip
\Black{The rate $t^{-1}$ is the sharp pointwise decay rate for the homogeneous three-dimensional wave equation, and the half-space boundary provides no additional temporal decay. Thus, the nonlinear field--particle coupling must be closed at this borderline, non-integrable rate.} We emphasize once again that neither modeling ingredient diminishes the strength of
our results. Gravity is \Black{expected to be} \emph{structurally necessary} to
preclude generic linear instabilities in the non-vacuum, spatially inhomogeneous
Vlasov--Maxwell regime. The boundary, on the other hand, removes no difficulty---in
fact, it introduces additional challenges, since energy and charge fluxes may pass
through the boundary, potentially generating arbitrarily large unstable modes. These
features are quantified in Theorems~\ref{thm.ex.st} and~\ref{thm.asymp.rth} and are
developed further in Sections~\ref{sec.exist.steady} and~\ref{sec.boot.decay}. Taken
together, the above informal statements suggest\Black{, at a heuristic level,} a
quantitative threshold between instability and
stability.\footnote{\Black{\label{fn.threshold}Heuristically, for a fixed boundary temperature
$\beta^{-1}>0$, one may expect a threshold $\Gamma(\beta)>0$, decreasing in $\beta$,
such that generic steady states are linearly unstable for $0\le g<\Gamma(\beta)$. We
emphasize that we do not carry out a linearized spectral analysis in this paper and
make no claim about unstable eigenvalues. In particular, we do not specify a precise
rate for $\Gamma(\beta)$: the inverse dependence on $\beta$ is only heuristically
suggested by the characteristic travel-time estimates, the control of the velocity
gradients, and the structure of the retarded representations developed in
Sections~\ref{sec.exist.steady} and~\ref{sec.boot.decay} and
Theorems~\ref{thm.ex.st} and~\ref{thm.asymp.rth}.}}

Our model captures the generic near-stellar regime---the principal source of cosmic plasma---since the overwhelming majority (\(\gtrsim 99\%\)) of baryonic matter is ionized and expelled from stars, where hydrogenic composition and strong gravity govern vertical confinement; see Section~\ref{sec.astro}. In physical contexts such as the solar wind, the existence and asymptotic stability of steady states are of central importance, and our mathematical results underscore the critical role of gravity in sustaining such behavior. Beyond their intrinsic mathematical significance, these results provide a rigorous framework for analyzing a range of physical phenomena, including coronal heating and associated energy cascades, nonlinear instabilities in three dimensions, the long-time dynamics of interstellar plasmas, and the emergence of time- or space-periodic structures within the Vlasov--Maxwell system.

\section{Main Theorems, Difficulties and Our Strategies}\label{sec.main thms diff and strat}

 \subsection{Heuristic Explanation of Main Results, Difficulties, and Ideas}

We now discuss the major challenges of the problem, present our new main idea to overcome them. 
 
A generic difficulty in the Vlasov--Maxwell system stems from its intrinsic instability. Even in the simpler one-dimensional Vlasov--Poisson case, the maximum growth rate of unstable modes can become unbounded if $\nabla_v F_{\mathrm{st}}$ is unbounded, as observed for certain singular profiles \cite{MR4093619, MR3509003}. In boundary value problems, derivatives of the solution may become singular in finite time \cite{MR1354697}, suggesting that, in our setting, the maximum growth rate could potentially become arbitrarily large---a situation further amplified by long-time transversal acceleration in the presence of a magnetic field. To construct steady solutions with bounded $\nabla_v F_{\mathrm{st}}$, we carefully analyze the characteristic flow---a task complicated by the magnetic field---and, in order to accommodate general boundary data, construct the solution using a Lagrangian approach rather than a classical method to find invariant solutions \cite{MR1170533,MR1264300,MR3318798}. By exploiting the regularity we established, we are able to prove uniqueness, which allows us to hope for the construction of dynamic solutions that converge to the steady state. Of course, controlling the maximum growth mode alone is insufficient to fully suppress instabilities; this must be combined with control of particle travel times, as will be addressed in the next step.

Even assuming that the instability has been well controlled, as discussed above, demonstrating decay of perturbations via dispersion remains challenging. This difficulty arises because the interaction between the steady solution (cf. the classical approach of Glassey--Strauss) and the particle/wave fields is nonlinear, and the magnetic field can extend the interaction time. To illustrate this more concretely, consider the perturbation problem:
\begin{align}
    \partial_t \Fp + \hat{v}_\pm \cdot \nabla_x \Fp \pm  
    \left(\textup{e}\E+\textup{e}\frac{\hat{v}_\pm}{c}\times \B \mp m_\pm g \hat{e}_3\right)
    \cdot \nabla_v \Fp 
    &= \mp  \textup{e} \left(\Ep+ \frac{\hat{v}_\pm}{c}\times \Bp\right)
    \cdot \nabla_v F_{\pm,\textup{st}},\label{VM perturbations1}\\
\square (\Ep, \Bp) := \Big(\tfrac{1}{c^2} \p_t^2 - \Delta \Big)(\Ep, \Bp) 
    &= 4 \pi \Big( - \nabla \varrho - \tfrac{1}{c^2} \p_t \mathcal{J}, \tfrac{1}{c} \nabla \times \mathcal{J} \Big),\label{VM perturbations2}
\end{align}
where
\begin{equation}\label{rhoJ_per}
\varrho := \int_{\mathbb{R}^3} \textup{e} (f_+ - f_-) \, dv , \qquad 
\mathcal{J} := \int_{\mathbb{R}^3} \textup{e} (\hat{v}_+ f_+ - \hat{v}_- f_-) \, dv.
\end{equation}

We emphasize again that, in our setting, there is no a priori guarantee that the energy or mass dissipates. A key challenge in the asymptotic stability analysis of the VM system is the presence of additional \textbf{\textit{wave--wave interactions}} at both microscopic and macroscopic levels, along with their feedback mechanism---a phenomenon absent in the Vlasov--Poisson dynamics. Because the decay of the wave field is unfavorably decoupled from that of the particle distribution and therefore proceeds much more slowly, this interaction renders the asymptotic stability problem substantially more difficult than in the Vlasov--Poisson case  \cite{MR4843478}.

\begin{itemize}[leftmargin=0.7em, labelsep=0.3em]
  
    \item\textbf{Macroscopic wave–wave interaction:} 
   In the propagation of the fields, the crucial wave-wave interaction appears in the particle transport contribution (the ``$S$ operator'' in the Glassey--Strauss theory) of the source term in \eqref{VM perturbations2}:
  \begin{align}\label{macroSDint} 
         \textup{c} 
       \textup{e}  \int  
       \frac{dy }{|y -x|} 
       \int_{\mathbb{R}^3} dv\ 
       \frac{\omega + \frac{\hat v_\pm}{c}}{1+ \frac{\hat v_\pm}{c} \cdot \omega}\,\left(\Ep + \frac{\hat v_\pm}{c} \times \Bp\right)\big(t-\frac{|x-y |}{c},y \big)\cdot \nabla_v F_{\pm,\textup{st}}(y ,v).\tag{MaWW} 
    \end{align}Here, $\omega$ denotes the light-cone direction, and the spatial integration is restricted to the half-space part of $B(x;t)$. 
     
    \vspace{5pt}
    
    \item\textbf{Microscopic wave–wave interaction:} 
   In the dynamics of the particle distribution, a key wave-wave interaction appears through the inhomogeneous term of \eqref{VM perturbations1}, naturally expressed in the Lagrangian formulation along the characteristics $\ZS(s;t,x,v)$ 
    (see \eqref{leading char}, for the definition):
    \begin{equation}\label{SD}
    \int 
    \textup{e} \left(\Ep+ \frac{\hat{v}_\pm}{c}\times \Bp\right)  (s,\XS(s) )\cdot \nabla_v F_{\pm,\textup{st}}  (
    \ZS(s)
     )   ds,\tag{MiWW}
    \end{equation}
    where the integration extends until the particle trajectory $\ZS(s)$ exits the boundary. Since the trajectory encodes the effect of waves, the term \eqref{SD} represents a genuine \textit{wave--wave interaction mediated through the particle trajectory in the presence of a steady background distribution}.  

\end{itemize}

Without the collision mechanism\footnote{See global well-posedness and the stability results with collision operator \cite{MR3961294,Guo-VMB,Guo-Strain3,2401.00554}.}, the particle–acceleration feedback loop (see Figure \ref{fig1}) inherent to the Vlasov–Maxwell system can, in principle, trigger uncontrolled acceleration, rendering stability analysis highly delicate or even impossible. A well-known criterion provides a sufficient condition for the loop to close stably \cite{MR816621,MR1877669,MR3248056 
,MR3721415}. In the vacuum perturbation setting, only microscopic wave–particle interactions are present, while wave–wave interactions are entirely absent. This structural simplification allows one to establish existence results, study long-time behavior \cite{MR0919231,schaeffer2004small}, and carry out delicate analysis \cite{bigorgne2020sharp,wang2022propagation}.
By contrast, for perturbations around a nontrivial steady state, one must \textit{rigorously close} the full particle–acceleration feedback loop by controlling both wave–wave interactions and their back-reaction, underscoring the substantially greater analytical challenges. These difficulties are further compounded by the slow decay of the electromagnetic fields, which do not align naturally with the decay of the particle distribution. At first sight, such slow \Black{$t^{-1}$} field decay \Black{of the electromagnetic fields, which is non-integrable in time,} appears insufficient---and potentially destabilizing---within the feedback loop, making its closure far more delicate than in the Vlasov–Poisson or near-vacuum regime. \Black{These challenges are amplified by interactions between boundaries and waves/particles in the presence of an active magnetic field. }

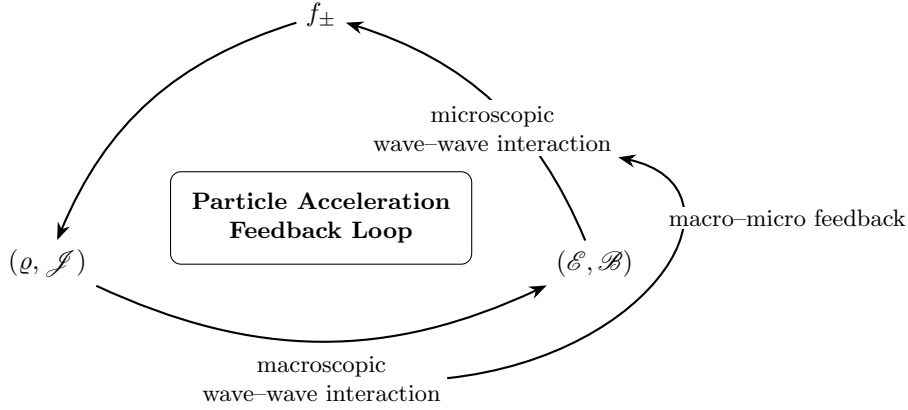
\begin{figure}[H]
\begin{center}
\begin{tikzpicture}[
    every node/.style={font=\normalsize},
    arrow/.style={-{Stealth}, thick},
    centerbox/.style={draw, rectangle, rounded corners, inner sep=8pt, font=\small\bfseries, align=center},
    labelstyle/.style={font=\small, fill=white, align=center, inner sep=2pt}, 
    feedbacklabel/.style={font=\small, fill=white, inner sep=1pt}
]
 
\def\xradius{4.2cm}
\def\yradius{2.2cm}
 
\node (f) at (90:\yradius) {$f_\pm$};
\node (rhoJ) at ({\xradius*cos(210)},{\yradius*sin(210)}) {$(\varrho, \mathcal J)$};
\node (EB)   at ({\xradius*cos(330)},{\yradius*sin(330)}) {$(\mathcal{E}, \mathcal{B})$};
 
\node[centerbox] at (0,-0.5) {Particle Acceleration\\Feedback Loop};
 
\draw[arrow, bend right=25] (f) to (rhoJ);

\draw[arrow, bend right=25] (rhoJ) to 
    node[labelstyle, midway, yshift=-14pt] (labelMacro)
    {\mbox{macroscopic} \\ \mbox{wave--wave interaction}}
    (EB);

\draw[arrow, bend right=25] (EB) to 
    node[labelstyle, midway, yshift=-10pt] (labelMicro)
    {\mbox{microscopic} \\ \mbox{wave--wave interaction}}
    (f);
 
\draw[arrow] 
    (labelMacro.east) 
    .. controls +(2.8,0.3) and +(1.8,-0.4) .. 
    (labelMicro.south east);

\node[feedbacklabel] at (6.2,  -0.5) {macro--micro feedback};

\end{tikzpicture}
\end{center}
\label{fig1}
\end{figure}


 A key new observation in our analysis is that \Black{the boundary contributions to the
magnetic field vanish identically in the half space}. We demonstrate this
\Black{cancellation} by representing the magnetic field using the vector potential in
the Coulomb gauge\Black{, combined with the method of images}. Moreover, since the magnetic part of the Lorentz force $\hat{v}_\pm\times\B$ does no
work \Black{on the particles ($\hat{v}_\pm\cdot(\hat{v}_\pm\times\B)=0$), the magnetic
field does not alter the mechanical energy along the characteristics:
$\frac{d}{ds}\big(cv^0_\pm + m_\pm g x_3\big) = \pm\,\hat{v}_\pm\cdot\E$}; it affects
only the total particle trajectory time and thereby influences the overall
acceleration due to the prolonged interaction with the steady state. \Black{This observation, combined with the
control of the full nonlinear electromagnetic fields along the characteristics,} is
particularly useful when controlling particle travel times, allowing us to conclude
that the travel time is linearly proportional to the particle's mechanical energy:
\[
t_{\pm,\mathbf{b}}(t,x,v)  \lesssim cv^0_\pm + m_\pm g x_3.
\]
Ultimately, the balance between particle travel times and mechanical energy, together
with the boundedness of $\nabla_v F_{\mathrm{st}}$ discussed above, ensures complete
control of the instability. At the same time, by employing the characteristic method
and exploiting the travel times that we have controlled, we can establish exponential
decay of $\nabla_v F_{\mathrm{st}}$ in both velocity and space $$ \exp\{-\beta(c v^0_\pm + m_\pm g x_3
+  |x_\parallel|/2) \}
. $$

Now that the instability has been fully controlled and the asymptotics of the steady profile interacting with the particle–field system have been established, we turn to proving the temporal decay in \eqref{macroSDint} and \eqref{SD}. The key idea is to exploit the highly local nature of transport propagation in order to capture the wave propagation localized around the light cone within the interaction terms \eqref{macroSDint} and \eqref{SD}, which exhibit a specific structural form. In \eqref{macroSDint}, the spatial decay of \(\nabla_v F_{\mathrm{st}}\) is crucial: it allows the \(y\)-integral to be uniformly bounded over the light cone \(|x-y| < ct\); without this decay, the growing volume would prevent temporal decay. We then exploit the a priori linear-in-time bounds on the particle distribution to obtain decay in the retarded time \(t - |x-y|/c\), and combine this with the 3D wave dispersion factor \(1/|y-x|\) to show that \eqref{macroSDint} decays linearly in time. For \eqref{SD}, we combine the linear decay of the fields with the rapid decay of the steady interaction along particle travel times to similarly establish linear decay in time of \eqref{SD}.

\subsection{Main Theorem 1: Unique Solvability of the Steady Problem}

We now state our main theorems. 
The primary part of the paper is on the stability of steady states with J\"uttner-Maxwell upper bound in the three-dimensional half-space $\mathbb{R}^3_+$. To this end, we first prove the existence and uniqueness of steady states with J\"uttner-Maxwell upper bound $(F_{\pm,\textup{st}},\E_{\textup{st}},\B_{\textup{st}})$ to the stationary system \eqref{2speciesVM-steady} for two species.  We consider a stationary problem of 2-species Vlasov--Maxwell system: 
\begin{equation}\label{2speciesVM-steady}
   \begin{split}
        \hat{v}_\pm \cdot \nabla_x F_{\pm,\textup{st}} &\pm   \left(\textup{e}\E_{\textup{st}}+\textup{e}\frac{\hat{v}_\pm}{c}\times \B_{\textup{st}}\mp m_\pm g\hat{e}_3\right)\cdot \nabla_v F_{\pm,\textup{st}} = 0,\\
        \nabla_x\times \B_{\textup{st}}&= \frac{4\pi}{c} J_{\textup{st}} =\frac{4\pi}{c} \int_\rth (\textup{e}\hat{v}_+ F_{+,\textup{st}}-\textup{e}\hat{v}_-F_{-,\textup{st}})dv,\\
        \nabla_x\times \E_{\textup{st}}&=0,\\
        \nabla_x \cdot \E_{\textup{st}}&=4\pi \rho_{\textup{st}}=4\pi \int_\rth ( \textup{e}F_{+,\textup{st}}-\textup{e}F_{-,\textup{st}})dv,\\
        \nabla_x \cdot \B_{\textup{st}}&=0,
    \end{split}
\end{equation}together with the inflow boundary conditions as
\begin{equation}\label{2species-perturbabsorbing.st}
    F_{\pm,\textup{st}}(x,v)=G_\pm(x_\parallel,v),\ \textup{ for } (x,v)\in \gamma_-,
\end{equation}and the perfect conductor boundary condition 
\begin{equation}\label{perfect cond. boundary-perturb.st}
    \mathbf{E}_{\textup{st},1}(x_\parallel,0)=\mathbf{E}_{\textup{st},2}(x_\parallel,0)=0,\ \mathbf{B}_{\textup{st},3}(x_\parallel,0)=0.
\end{equation}

Define the weight 
\begin{equation}\label{weight}
\mathrm{w}_{\pm,\beta}(x,v)=\exp\{\beta(cv^0_\pm+m_\pm g x_3)+\beta |x_\parallel|/2\}, \ \  v^0_\pm=\sqrt{m_\pm^2c^2+|v|^2}, \ \  \beta>1.
\end{equation}
We also define
\begin{equation}
        \label{norm X}|||f|||  \eqdef \| (c\vZ) ^\ell  \nabla_{x_\parallel}f\|_{L^\infty}+\left\|(c\vZ)^\ell \alpha_{\pm}\partial_{x_3}f\right\|_{L^\infty}+\| (c\vZ) ^\ell \nabla_{v}f\|_{L^\infty}, \ \ \text{$\ell >4$,}
    \end{equation} 
where 
\begin{equation}\label{def.alpha intro}
 \alpha_\pm (x,v) : =  \sqrt{\frac{ | \bar\alpha_\pm (x,v)
    |^2}{1+ | \bar\alpha_\pm (x,v)
    |^2}}, \  \  \  \text{with} \  \  \ 
\bar\alpha_\pm (x,v)\eqdef \sqrt{x_3^2+\left|\frac{(\hat{v}_\pm)_3}{c}\right|^2+ \frac{x_3}{ 2c\vZ }.
}\end{equation} 

Now our first main theorem follows: 
    \begin{theorem}[Unique Solvability of the Steady Problem]  \label{thm.ex.st} Fix $g>0$ with 
    $\textcolor{black}{\min\{m_-,m_+\}g\gg 1}$ 
 and choose $\beta>1$ such that $\min\{m_-,m_+\}g \beta^3\gg 1. $ 
    Suppose that the inflow boundary data $G_\pm$ is a $C^1$ exponentially localized:  
\begin{equation}
   \label{inicon5}
    \|\mathrm{w}_{\pm,\beta}(\cdot,0,\cdot)G_\pm(\cdot,\cdot)\|_{L^\infty_{x_\parallel,v}}\text{ and  } \  \|\mathrm{w}^2_{\pm,\beta}(\cdot,0,\cdot)\nabla_{x_\parallel,v}G_\pm(\cdot,\cdot)\|_{L^\infty_{x_\parallel,v}}\le C, \text{ for some $C>0$. }
\end{equation}  Then we construct a unique classical solution to the stationary Vlasov--Maxwell system \eqref{2speciesVM-steady} with the incoming boundary condition \eqref{2species-perturbabsorbing.st} and the perfect conductor boundary condition \eqref{perfect cond. boundary-perturb.st}. This solution solves the continuity equation $\nabla_x\cdot J_{\textup{st}}=0$ and satisfies 
  \begin{equation}
  \begin{split}\label{steady state L infty} \| e^{\frac{\beta}{2} |x_\parallel|}e^{\frac{\beta}{2} c\vZ } e^{\frac{\beta}{2}m_\pm g x_{3}}F_{\pm,\textup{st}}(x,v)\|_{L^\infty}\le C  , \ \
  |\E_{\textup{st}}(x)|+ |\B_{\textup{st}}(x)|\le \min\{m_+,m_-\}\frac{g}{16}\frac{1}{\langle x\rangle^2},\\
  |||F_{\pm,\textup{st}}||| +    \|(\E_{\textup{st}},\B_{\textup{st}})\|_{W^{1,\infty}_{x}( \mathbb{R}^3_+)}
       \lesssim 1  .\end{split}
 \end{equation}
   Furthermore, we obtain the crucial weighted estimate for the momentum derivative: \begin{equation}\label{steady state decay mom.deri}\|\mathrm{w}_{\pm,\beta}\nabla_v F_{\pm,\textup{st}}\|_{L^\infty_{x,v}}\lesssim \|\mathrm{w}^2_{\pm,\beta}(\cdot,0,\cdot)\nabla_{x_\parallel,v}G_\pm(\cdot,\cdot)\|_{L^\infty_{x_\parallel,v}}. \end{equation}  
    \end{theorem}

\begin{remark}
The parameter $\beta$ corresponds to the reciprocal of the boundary temperature.
Thus, the inverse relation between the gravitational constant $g$ and $\beta$ is
natural. \Black{The quantitative condition $g \gtrsim 1/\beta^{3}$ enters our analysis
as a sufficient condition: it is the scale at which the characteristic travel-time
estimates and the control of the velocity gradients close in the proofs of
Theorems~\ref{thm.ex.st} and~\ref{thm.asymp.rth} (cf.\
footnote~\ref{fn.threshold}). 
}
\end{remark}

\subsection{Main Theorem 2: Dynamical Asymptotic Stability
}

 For the dynamical problem \eqref{2speciesVM}, we consider the initial conditions \begin{equation}\label{2speciesIni}
    F_\pm(0,x,v)=F_\pm^{\textup{in}}(x,v),\ \E(0,x)=\E^{\textup{in}}(x),\ \B(0,x)=\B^{\textup{in}}(x),
\end{equation}with the compatibility conditions $\nabla_x \cdot\E^{\textup{in}}(x)=\rho(0,x),$ and $\nabla_x \cdot\B^{\textup{in}}(x)=0.$ 
At the boundary $x_3=0$, we consider the incoming boundary condition \eqref{inflow boundary} where the incoming profile $G_\pm$ is now given by the stationary solution $F_{\pm,\textup{st}}$ obtained in Theorem \ref{thm.ex.st}.

Denote the initial perturbed fields and their $i$-th order temporal derivatives (understood through the system of equations \eqref{2speciesVM}) as $$ \mathcal{E}^{\textup{in}}=\E^{\textup{in}}-\E^{\textup{in}}_{\textup{st}}=(\mathcal{E}_{01},\mathcal{E}_{02},\mathcal{E}_{03})^\top, \ \ \mathcal{B}^{\textup{in}}=\B^{\textup{in}}-\B^{\textup{in}}_{\textup{st}}=(\mathcal{B}_{01},\mathcal{B}_{02},\mathcal{B}_{03})^\top,$$ $$\mathcal{E}^i_0=(\mathcal{E}^i_{01},\mathcal{E}^i_{02},\mathcal{E}^i_{03})^\top, \ \ \mathcal{B}^i_0=(\mathcal{B}^i_{01},\mathcal{B}^i_{02},\mathcal{B}^i_{03})^\top, \ \ \text{for $i=1,2,$}$$ respectively. We suppose that the initial perturbations $\mathcal{E}^{\textup{in}}$ and $\mathcal{B}^{\textup{in}}$ are compactly supported in a ball $B_{R_0}(0)$ for a fixed $R_0>0.$ Furthermore,
 we assume that, for a sufficiently small $c_0>0$ and $\bar \beta> \beta >0$,  
 \begin{equation}
    \label{initial E0i}\left\| e^{m_\pm g \bar \beta x_3}
    (\mathcal{E}_{0}, \mathcal{B}_{0},\mathcal{E}^i_{0} , \mathcal{B}^i_{0}, \nabla_x \mathcal{E}_{0},\nabla_x \mathcal{B}_{0},\nabla_x \mathcal{E}^1_{0},\nabla_x \mathcal{B}^1_{0})
    \right\|_{L^\infty(\mathbb{R}^2\times (0,\infty))}\le  c_0 \min\{m_-,m_+\}g, \ \ i=1,2.
\end{equation} 

We assume that the initial perturbed particle distribution $f^{\textup{in}}_\pm=F^{\textup{in}}_\pm-F_{\textup{st}}$ is compactly supported in $x$ in a ball $B_{R_0}(0)$ for a fixed $R_0>0$. Moreover, we assume that the initial perturbation and its temporal derivative (understood through the Vlasov equation \eqref{2speciesVM}), satisfies 
\begin{equation}
    \label{f initial condition} \|\mathrm{w}_{\pm, \bar \beta} (f^{\textup{in}}_\pm,\partial_tf^{\textup{in}}_\pm)\|_{L^\infty_{x,v}} + |||f_\pm^{\textup{in}}||| < M< + \infty. 
\end{equation}  
%

We now state our main theorem on the asymptotic stability of the steady states.
\begin{theorem}[Asymptotic Stability]\label{thm.asymp.rth}
Let $(F_{\pm,\textup{st}}, \mathcal{E}_{\textup{st}}, \mathcal{B}_{\textup{st}})$ be the steady solution constructed in Theorem \ref{thm.ex.st}. 
Suppose positive parameters $(g,m_\pm, \bar \beta)$ satisfy $\bar \beta>0$, $\textcolor{black}{\min\{m_-,m_+\}g\gg 1}$ and $\min\{m_-,m_+\}\times\min\{g \bar\beta^3,\bar\beta^2\}\gg 1. $  
Let the initial perturbations $(f_\pm^{\textup{in}}, \mathcal{E}^{\textup{in}}, \mathcal{B}^{\textup{in}})$ satisfy the conditions of \eqref{initial E0i} and \eqref{f initial condition}. 

Then we construct a unique classical solution to the dynamical problem \eqref{VM perturbations1}--\eqref{VM perturbations2} with the inflow boundary condition \eqref{inflow boundary} and the perfect conductor condition \eqref{perfect cond. boundary}, such that  
%
$$
|||f_\pm(t)|||<\infty,\quad(\mathcal{E}, \mathcal{B}) \in W^{1,\infty}([0,\infty)\times\mathbb{R}^3_+), \ \ \text{for all $t>0$.}
$$
Moreover, the solution decay linearly in time 
\begin{align}
    &\sup_{t \ge 0}\ (1+t) \left\| e^{\frac{\bar\beta}{2}|x_\parallel| + \frac{\bar\beta}{4}c\vZ + \frac{\bar\beta}{4}m_\pm g x_3} f_\pm(t) \right\|_{L^\infty_{x,v}} \le C_M,\notag\\
    &\sup_{t \ge 0}\ (1+t) \|(\mathcal{E}, \mathcal{B})(t)\|_{L^\infty_{x}} \le \min\{m_+, m_-\} \frac{g}{16}.\notag
\end{align}
Furthermore, the derivatives are controlled as 
\begin{align}\notag
 \|(c v_\pm^0)^\ell \partial_t f_\pm(t)\|_{L^\infty} + |||f_\pm(t)||| 
  +  \|(\mathcal{E}, \mathcal{B})\|_{W^{1,\infty}_{t,x}([0,t]\times \mathbb{R}^3_+)} 
  &\lesssim_t 1 .
\end{align}
\end{theorem}



\begin{remark}
 Our framework admits natural extensions to astrophysical environments where gravity coexists with large-scale background electromagnetic fields. In such settings, weak external electric and magnetic components may alter particle confinement and transport, yet the stability theory developed here continues to hold under suitable smallness conditions. For a more detailed discussion of these astrophysical applications, see Section~\ref{sec.astro}.
\end{remark}

\noindent\textbf{Notation}: For simplicity, we normalize the physical constants $\textup{e}$ and $\textup{c}$ to $1$ throughout the rest of the paper, while retaining the distinct quantities $m_{+}$ and $m_{-}$, denoting the ion and electron masses, which differ significantly.


 
\section{
Characteristics for Wave and Transport Dynamics
}\label{sec.main.comparison.weight}

\subsection{
New Magnetic Field Representation via the Potentials}\label{sec.4}
We find that a new magnetic field representation 
does not contain \Black{the boundary contributions}, in contrast to the classical Glassey--Strauss representation \Black{in a half space (cf.\cite{MR4414612})}. \Black{It does contain a nonlinear term, which we show coincides with the corresponding nonlinear $S$ term of the classical Glassey--Strauss representation; see Remark \ref{rmk.BS.GS}.} This simplification \Black{of the boundary contributions} results from cancellations occurring under the curl relation $\B = \nabla \times \A$, as demonstrated below.
In the whole space, we adopt the electromagnetic four-potential in the Coulomb gauge \cite{jackson2002gauges}:
\begin{equation}\label{A}
      \B = \nabla \times \A , \  \ \E = - \nabla \ph -  \frac{\partial \A}{\partial t}; \ \ \nabla \cdot \A =0.
\end{equation}
From Gauss's law for electricity \eqref{2speciesVM}, we have that
\begin{equation}\label{ph}
    -\Delta \ph   = 4 \pi \rho. 
\end{equation}
 \begin{lemma}\label{lem.AmpMax}
We rewrite the Amp\`ere-Maxwell law \eqref{2speciesVM}$_3$ as
\begin{equation}\label{waveA}
  \Box \A \eqdef    \partial_t^2 \A - \Delta \A =  4\pi J  -    \nabla \partial_t \ph  = 4\pi \mathbf{P}J,
\end{equation}
where $\mathbf{P}$ is the divergence-free projection:
    $\mathbf{P} J \eqdef J + \nabla (-\Delta)^{-1} \nabla \cdot J.$
\end{lemma}
\begin{proof}
Inserting the potential representation \eqref{A} into the Amp\`ere--Maxwell law \eqref{2speciesVM}, we obtain
$$
  -   \partial_t \nabla \phi -   \partial_t^2 \mathbf{A} + \Delta \mathbf{A} = - 4\pi J,
$$
where we have used the Coulomb gauge condition $\nabla \cdot \mathbf{A} = 0$, along with the vector identity
$\nabla \times (\nabla \times \mathbf{A}) = -\Delta \mathbf{A} + \nabla (\nabla \cdot \mathbf{A}) = -\Delta \mathbf{A}$. Next, using the scalar potential formula \eqref{ph} together with the continuity equation \eqref{continuity eq}, we compute
$
  -   \partial_t \nabla \phi
  = - 4\pi \nabla (-\Delta)^{-1} \partial_t \rho
  = 4\pi \nabla (-\Delta)^{-1} \nabla \cdot J.
$
Combining these identities, we obtain the desired equation, completing the proof.\end{proof}
\paragraph{Retarded Solutions}
In the whole space $\mathbb{R}^3$, the inhomogeneous solution (with zero initial data) is given by the Green function.
Suppose $W$ solves $\Box W = U$, $W|_{t=0}=0= \partial_t W|_{t=0}$. Then
\begin{equation}\label{retardU}
W_U(t,x)=\int_{|y-x|\le t}\frac{U(t-  |x-y|, y )}{4 \pi |x-y|} dy.
\end{equation}
We now introduce the potential representation for the magnetic field $\B(t,x)$ as follows.
\begin{proposition}[Representation of Magnetic field]
\label{Prop 4.4}\Black{Suppose that $F_\pm(t,x,\cdot)$ decays sufficiently fast as $|v|\to\infty$.} The magnetic field   $\B(t,x)$ can be written as
\begin{equation}\notag
\B(t,x)   = \B_{J}(t,x) \Black{+ \B_{S}(t,x)} + \B_{\textup{in}}(t,x),
\end{equation}
where
\begin{align}\label{BJ}
\B_{J}(t,x)  &=\sum_{\iota=\pm}  \iota
\int_{|Y|\le t}
\int_{\mathbb{R}^3}      \frac{Y\times \hat{v}_\iota}{ |Y|^3\big(1+ \hat{v}_\iota \cdot  \frac{Y}{|Y|} \big)^2} \left(1-\left|\hat{v}_\iota\right|^2\right)
   F_\iota \left(t- |Y| , Y+x, v\right)
dv
dY   ,\text{ and }\\
\label{Bin}
\B_{\textup{in}} (t,x) &=
 {\B}_{\textup{hom}} (t,x)
 +\sum_{\iota=\pm}\iota   \int_{|Y|=t} \int_{\mathbb{R}^3}
       \frac{Y}{|Y|^2} \times
        \frac{\hat{v}_\iota}{1+ \hat{v}_\iota  \cdot \frac{Y}{|Y|}} \,
      F_\iota (0, Y+x, v ) \, dv \, dS_Y.
        \end{align}
Here,
\begin{align}\notag
    \Box  {\B}_{\textup{hom}} = 0,&  \  \ \  {\B}_{\textup{hom}} |_{t=0} = \B^{\textup{in}}, \   \partial_t  {\B}_{\textup{hom}} |_{t=0} =  - \nabla \times \E^{\textup{in}}  .
\end{align}
{\color{black}
Moreover, the nonlinear term $\B_{S}$ is given by
\begin{equation}\label{BS}
\begin{split}
\B_{S}(t,x)  &=\sum_{\iota=\pm}
\int_{|Y|\le t}\int_{\mathbb{R}^3}
\big(\mathbf{K}_\iota\cdot\nabla_v\big)\left(\frac{Y}{|Y|^2}\times\frac{\hat{v}_\iota}{1+\hat{v}_\iota\cdot\frac{Y}{|Y|}}\right)
F_\iota\left(t-|Y|,Y+x,v\right)dv\,dY\\
&=\sum_{\iota=\pm}\int_{|Y|\le t}\int_{\mathbb{R}^3}
\frac{1}{v^0_\iota\big(1+\hat{v}_\iota\cdot\frac{Y}{|Y|}\big)}\,\frac{Y}{|Y|^2}\times\left[\mathbf{K}_\iota-\frac{\big(\hat{v}_\iota+\frac{Y}{|Y|}\big)\cdot\mathbf{K}_\iota}{1+\hat{v}_\iota\cdot\frac{Y}{|Y|}}\,\hat{v}_\iota\right]
F_\iota\left(t-|Y|,Y+x,v\right)dv\,dY,
\end{split}
\end{equation}
where $\mathbf{K}_\iota\eqdef\big(\E+\hat{v}_\iota\times\B\big)\left(t-|Y|,Y+x\right)$, $v^0_\iota\eqdef\sqrt{m_\iota^2+|v|^2}$, and $\big(\mathbf{K}\cdot\nabla_v\big)\mathbf{G}\eqdef\sum_j\mathbf{K}_j\,\partial_{v_j}\mathbf{G}$.
}
\end{proposition}
\begin{proof}We consider $\nabla \times W_{ 4\pi\mathbf{P} J}$. Note that $\nabla \times  4\pi\mathbf{P} J = 4\pi  \nabla \times  J$. Thus, we derive the form of $\nabla \times W_{ 4\pi\mathbf{P} J}$ as
\begin{align}
  \nabla \times W_{ 4\pi\mathbf{P} J}
 = &  \int_{|Y|\le t}   \frac{Y}{ |Y|^3}  \times  J  ( t- |Y|, Y+x)  dY\label{retardB1}
 \\
&+ \int_{|Y|= t}  \frac{Y}{|Y|^2} \times  J( 0, Y+x ) dS_Y\label{retardB2}
\\
&   +  \int_{|Y|\le t}  \frac{Y}{   |Y|^2} \times  \partial_t J (t- |Y|, Y+x ) dY.\label{retardB_t}
\end{align}
Regarding the temporal derivative integral \eqref{retardB_t}, we use
\begin{align}
     \partial_t F_\pm(t-|Y|, Y+ x, v)
     = &\frac{1}{ 1+ \hat{v}_\pm \cdot \frac{Y}{|Y|}}\Big(
     - \hat{v}_\pm \cdot \nabla_Y [F_\pm(t-|Y|, Y+ x, v)] \Big)\label{F_t1} \\
     &+\frac{\mp 1}{ 1+ \hat{v}_\pm \cdot \frac{Y}{|Y|}}  \nabla_v \cdot  \big[(\E + \hat v_\pm \times \B  )  F_\pm \big](t-|Y|, Y+ x, v)
   .\label{F_t2}
     \end{align}
Applying the integration by parts, we express the contribution of the term \eqref{F_t1} in \eqref{retardB_t} as
\begin{equation}\label{CONT_F_t1=0}
    \begin{split}
       &
      \sum_{\iota=\pm}\iota  \int_{|Y|\le t} \int_{\mathbb{R}^3}
       \hat{v}_\iota \cdot \nabla_Y
       \left( \frac{Y}{|Y|^2} \times
        \frac{\hat{v}_\iota}{1+ \hat{v}_\iota  \cdot \frac{Y}{|Y|}} \right)
        F_\iota \Big(t-|Y|, Y+x, v\Big) \, dv \, dY\\
          & - \sum_{\iota=\pm}\iota \int_{|Y|=t} \int_{\mathbb{R}^3}
       \frac{Y}{|Y|^2} \times
        \frac{\hat{v}_\iota}{1+ \hat{v}_\iota  \cdot \frac{Y}{|Y|}} \,
        \hat{v}_\iota \cdot  \frac{Y}{|Y|} F_\iota (0, Y+x, v ) \, dv \, dS_Y.
    \end{split}
\end{equation}
The last integral of \eqref{CONT_F_t1=0} with initial data $F_\pm(0,\cdot,\cdot)$ result in the initial data terms of \eqref{Bin} after being cancelled by the integral \eqref{retardB2}.   Regarding the first integral in \eqref{CONT_F_t1=0}, we further calculate the derivative of the kernel and obtain that
\begin{equation}
\begin{split}\label{kernelB1}
   & \hat{v}_\pm  \cdot \nabla_Y   \left( \frac{Y}{|Y|^2} \times \frac{\hat{v}_\pm}{1+ \hat{v}_\pm \cdot \frac{Y}{|Y|}}
    \right)
     =    -
    \frac{  \frac{Y}{|Y|} \times \hat{v}_\pm }{|Y|^{2}  \big(1+ \hat{v}_\pm \cdot  \frac{Y}{|Y|} \big)^2}   \Big(
   2\hat{v}_\pm \cdot \frac{Y}{|Y|} +  \left|\hat v_\pm\right|^2
    + \left| \hat v_\pm \cdot \frac{Y}{|Y|}\right|^2
    \Big).
    \end{split}
\end{equation}This together with \eqref{retardB1} results in the representation \eqref{BJ} in the final representation, since
\begin{equation*}
\int_{|Y|\le t}   \frac{Y}{ |Y|^3}  \times  J  ( t- |Y|, Y+x)  dY
    =\sum_{\iota=\pm}\iota \int_{|Y|\le t}\int_\rth   \frac{Y\times \hat{v}_\iota}{ |Y|^3}  F_\iota  ( t- |Y|, Y+x,v)  dvdY,\text{ and }
\end{equation*}
\begin{equation*}
    \frac{Y\times \hat{v}_\pm}{ |Y|^3} -
    \frac{ \frac{Y}{|Y|} \times \hat{v}_\pm }{|Y|^{2}  \big(1+ \hat{v}_\pm \cdot  \frac{Y}{|Y|} \big)^2}   \Big(
   2\hat{v}_\pm \cdot \frac{Y}{|Y|} +  \left|\hat v_\pm\right|^2
    + \left| \hat v_\pm \cdot \frac{Y}{|Y|}\right|^2
    \Big)
    = \frac{Y\times \hat{v}_\pm}{ |Y|^3\big(1+ \hat{v}_\pm \cdot  \frac{Y}{|Y|} \big)^2} \left(1-\left|\hat v_\pm\right|^2\right).
\end{equation*}
On the other hand, applying the integration by parts in $v$, we express the contribution of the term \eqref{F_t2} in \eqref{retardB_t} as
\begin{equation}\begin{split}\label{CONT_F_t2=0}
& -    \int_{|Y|\le t} \frac{Y}{|Y|^2} \times       \int_{\mathbb{R}^3}\sum_{\iota=\pm}
 \frac{ \hat{v}_\iota  }{ 1+ \hat{v}_\iota  \cdot \frac{Y}{|Y|}}  \nabla_v \cdot  \big[ \mathbf{K}_\iota F_\iota   (t- |Y|, Y+ x, v)  \big] dv  dY \\
  & \Black{=       \sum_{\iota=\pm}  \int_{|Y|\le t}  \int_{\mathbb{R}^3}  \big(\mathbf{K}_\iota\cdot\nabla_v\big) \left( \frac{Y}{|Y|^2} \times  \frac{ \hat{v}_\iota   }{ 1+ \hat{v}_\iota  \cdot \frac{Y}{|Y|}}   \right) F_\iota   (t- |Y|, Y+ x, v)\, dv\, dY \,=\, \B_{S}(t,x),}
\end{split}\end{equation}
where we have abbreviated $\mathbf{K}_\pm: = \E +  \hat v_\pm \times \B  $\Black{, evaluated at $(t-|Y|,Y+x)$}.
\Black{Here we have used the componentwise integration by parts: for each fixed $i$, with $\mathbf{G}_\iota\eqdef\frac{Y}{|Y|^2}\times\frac{\hat{v}_\iota}{1+\hat{v}_\iota\cdot\frac{Y}{|Y|}}$ and $\mathbf{W}_\iota\eqdef\mathbf{K}_\iota F_\iota$ decaying as $|v|\to\infty$,
$$
\int_{\mathbb{R}^3} (\mathbf{G}_\iota)_i \, \big( \nabla_v \cdot \mathbf{W}_\iota \big) \, dv
= - \int_{\mathbb{R}^3} \sum_j \partial_{v_j} (\mathbf{G}_\iota)_i \, (\mathbf{W}_\iota)_j \, dv
= - \int_{\mathbb{R}^3} \big[ \big( \mathbf{K}_\iota \cdot \nabla_v \big) \mathbf{G}_\iota \big]_i \, F_\iota \, dv .
$$
We emphasize that the integration by parts produces the directional derivative $(\mathbf{K}_\iota \cdot \nabla_v) \mathbf{G}_\iota$, and not $(\nabla_v \cdot \mathbf{G}_\iota) \, \mathbf{K}_\iota F_\iota$; although the identity $\nabla_v \cdot \mathbf{G}_\iota = 0$ holds, it is not the quantity generated by the integration by parts, and hence the contribution \eqref{CONT_F_t2=0} does not vanish. Computing the kernel derivative directly, using $\partial_{v_j} \hat{v}_{\iota,i} = \frac{1}{v^0_\iota} \left( \delta_{ij} - \hat{v}_{\iota,i} \hat{v}_{\iota,j} \right)$ with $v^0_\iota=\sqrt{m_\iota^2+|v|^2}$, we obtain
\begin{equation}\label{kernelB2}
\big(\mathbf{K}_\iota\cdot\nabla_v\big)\left(\frac{Y}{|Y|^2}\times\frac{\hat{v}_\iota}{1+\hat{v}_\iota\cdot\frac{Y}{|Y|}}\right)
=\frac{1}{v^0_\iota\big(1+\hat{v}_\iota\cdot\frac{Y}{|Y|}\big)}\,\frac{Y}{|Y|^2}\times\left[\mathbf{K}_\iota-\frac{\big(\hat{v}_\iota+\frac{Y}{|Y|}\big)\cdot\mathbf{K}_\iota}{1+\hat{v}_\iota\cdot\frac{Y}{|Y|}}\,\hat{v}_\iota\right],
\end{equation}
which yields the expanded form in \eqref{BS}.} This completes the derivation of the magnetic field representation.
\end{proof}
{\color{black}
\begin{remark}[The nonlinear term $\B_S$ and the Glassey--Strauss $S$ term]\label{rmk.BS.GS}
The nonlinear term \eqref{BS} coincides with the nonlinear ($S$-type) term of the classical Glassey--Strauss representation \cite{MR816621}. Indeed, writing $\omega=\frac{Y}{|Y|}$, the kernel of \eqref{BS} reads
$$
\frac{1}{|Y|}\,\big(\mathbf{K}_\iota\cdot\nabla_v\big)\left(\frac{\omega\times\hat{v}_\iota}{1+\hat{v}_\iota\cdot\omega}\right),
$$
which is precisely the Glassey--Strauss $S$ kernel, generated by the very same operation in both derivations: substituting the Vlasov equation for $\partial_t F_\pm$ inside the retarded integral and integrating the $\nabla_v$-part by parts in $v$, so that the derivative falls on the kernel $\frac{\omega\times\hat{v}_\iota}{1+\hat{v}_\iota\cdot\omega}$ as the directional derivative $(\mathbf{K}_\iota\cdot\nabla_v)$. The agreement is no coincidence: under the Coulomb gauge, $\nabla\times\mathbf{P}J=\nabla\times J$, hence $\Box\B=4\pi\nabla\times J$, which is the same wave equation from which the Glassey--Strauss representation is derived, so the two representations must agree. Note that both species enter \eqref{BS} with the same overall sign, and that the kernel is $O(|Y|^{-2})$, hence locally integrable, with $1+\hat{v}_\iota\cdot\omega\ge 1-|\hat{v}_\iota|>0$.
\end{remark}
\begin{remark}[Comparison of $\B_J$ with the Glassey--Strauss $T$ term]\label{rmk.BJ.BT}
The linear term $\B_J$ in \eqref{BJ} also agrees with the $T$ term of the Glassey--Strauss representation: it carries the same kernel $\frac{Y\times\hat{v}_\iota}{|Y|^3(1+\hat{v}_\iota\cdot\omega)^2}\big(1-|\hat{v}_\iota|^2\big)$, and in particular retains the crucial degeneracy factor $1-|\hat{v}_\iota|^2 = m_\iota^2/(v^0_\iota)^2$, which vanishes as $|\hat{v}_\iota|\to 1$ and provides the decay for large momenta used in the Glassey--Strauss high-velocity criterion. Hence nothing is lost in the $T$ and $S$ kernels. The advantage of the present potential-based derivation lies in its structure: (i) the representation follows from a single integration by parts in $Y$ after taking the curl of the retarded potential, without decomposing the derivatives along the backward light cone (the division lemma of \cite{MR816621}); and (ii) more importantly, the derivation extends naturally to the half space via the method of images at the level of the potential $\A$, in which the boundary contributions cancel under the curl relation $\B=\nabla\times\A$ (see Section \ref{sec.B.halfspace}); in contrast, the direct Green-function representation of the electric field produces the additional boundary ($b2$) terms in \eqref{Ei5} and \eqref{E3}.
\end{remark}
}
\subsection{Potential Representation in the Half Space $\mathbb{R}^3_+$}\label{sec.B.halfspace}
We now consider the half space $\Omega=\mathbb{R}^3_+$. To extend the magnetic field representation $\mathbf{B}(t,x)$ for the Vlasov--Maxwell system from the whole space $\mathbb{R}^3$ to the half space $\Omega = \mathbb{R}^3_+$ under the perfect conductor boundary condition
$$
\mathbf{E}\times n =(\mathbf{E}_2,-\mathbf{E}_1,0)^\top=0,\ \mathbf{B} \cdot n = \mathbf{B}_3 = 0 \quad \textup{on} \quad x_3 = 0,
$$
we will follow the classical \textit{method of images}, combined with the Green's function for the wave equation in the half space with Neumann-type boundary condition on the tangential components of $\mathbf{A}$. The unique determination of the magnetic vector potential $\mathbf{A}$ is guaranteed as follows:
\begin{lemma}[Lemma 1.6 of \cite{MR1142472}]\label{lem.1.6.2}
Define $$H_{\textup{tan}}(\textup{curl}; \Omega) = \{f \in L^2: \nabla \times f \in L^2,\ f\times n|_{\partial\Omega}=0\}.$$
Assume that $ \Omega $ is simply connected. Then a function $ \B \in L^2(\Omega) $ satisfies
\begin{equation*}
    \nabla \cdot \B = 0 \quad \textup{in } \Omega, \
    \B \cdot n = 0 \quad \textup{on } \partial \Omega,
\end{equation*}
if and only if there exists a function $ \mathbf{A}  \in H_{\textup{tan}}(\textup{curl}; \Omega) $ such that
\begin{equation}\label{AB}
       \B
       = \nabla \times \mathbf{A}
       .
\end{equation}
Moreover, the function $ \mathbf{A} $ is uniquely determined (the Coulomb Gauge) if we assume in addition that
\begin{equation}\label{AB_W}
    \nabla \cdot \mathbf{A} = 0, \quad \int_{\partial\Omega} \mathbf{A} \cdot n  \, dS = 0, \quad \mathbf{A} \times n |_{\partial\Omega }
    =0.
\end{equation}
or equivalently $ \mathbf{A} \in H_{\textup{div}}(\textup{curl}; \Omega)$, where $$H_{\textup{div}}(\textup{curl}; \Omega) \eqdef \left\{ v \in H_{\textup{tan}}(\textup{curl}; \Omega): \nabla \cdot v = 0, \quad \int_{\partial\Omega} v \cdot n \, dS = 0 \right\}.$$
\end{lemma}This lemma implies the existence of a unique vector potential $\mathbf{A}$ satisfying both \eqref{AB} and \eqref{AB_W}. It, along with its proof, will be restated and used as Lemma \ref{lem.1.6} in Section \ref{sec.exist.steady} for the construction of steady states.
Through the rest of this section, we derive a \textit{potential representation} of the self-consistent magnetic field $\mathbf{B} $ in the half-space $\mathbb{R}^3_+$ via deriving the representation of the vector potential $\mathbf{A} $ which further satisfies the assumption \eqref{AB_W}. To this end, we first note that the Faraday equation \eqref{2speciesVM}$_3$ implies that
$$
\nabla \times \left(   \partial_t \mathbf{A} + \mathbf{E} \right) = 0.
$$
Therefore, the vector field $   \partial_t \mathbf{A} + \mathbf{E} $ is curl-free. Assuming the spatial domain is simply connected, the Poincar\'e lemma
implies that any curl-free vector field can be written as the gradient of a scalar function. Hence, there exists a scalar potential $ \ph $ such that
$$
  \partial_t \mathbf{A} + \mathbf{E} = \nabla \ph.
$$
Rearranging terms yields the decomposition
\begin{equation}\label{dec_E}
    \mathbf{E} = \nabla \ph -   \partial_t \mathbf{A},
\end{equation}
where $ \ph $ is unique up to an additive constant, since $ \nabla (\ph_1 - \ph_2) = 0 $ holds for any two scalar potentials $ \ph_1 $ and $ \ph_2 $.
Then, from the last condition of \eqref{AB_W} and the perfect boundary condition $\mathbf{E}_1=\mathbf{E}_2=0$ on the boundary $x_3=0$, we have
$$
\begin{bmatrix}
    0 \\
    0 \\
    \E_3
\end{bmatrix}
= \begin{bmatrix}
    \partial_{x_1} \ph\\
   \partial_{x_2} \ph \\
     \partial_{x_3} \ph
\end{bmatrix}
 + \begin{bmatrix}
   0\\
 0 \\
    -   \partial_t A_{ 3}
\end{bmatrix}  \ \ \textup{at} \ \ x_3=0.
$$
Therefore, we conclude that
\begin{equation}\label{ph_Dbdry}
    \ph|_{x_3=0 } = C,
\end{equation}and
\begin{equation}\label{E3_bdry}
(\E_3 -  \partial_{x_3} \ph +   \partial_t A_{ 3} )|_{x_3=0} =0.
\end{equation}
In addition, from the Gauss law for the electricity $\eqref{2speciesVM}_4$ and the boundary condition \eqref{ph_Dbdry}, we derive that
\begin{equation}\notag
    \Delta \ph = 4 \pi \rho , \ \ \ph |_{x_3 =0} = C.
\end{equation}
Therefore, we obtain that $\partial_{x_3} \E_3$ satisfies at the boundary
\begin{align}\notag
    \partial_{x_3} \E_3|_{x_3=0}
   &=  4 \pi \rho .
\end{align}
In addition, inserting \eqref{dec_E} into the Ampere-Maxwell law \eqref{2speciesVM}$_3$, we obtain the following wave equation for the magnetic potential $\mathbf{A}$:
\begin{equation}\label{eq.wave.mag.pot}
  \Box \A =   4\pi J  -    \nabla \partial_t \ph  = 4\pi \mathbf{P}J,
\end{equation}
where $\mathbf{P}$ is the divergence-free projection:
\begin{equation}\notag
    \mathbf{P} J = J + \nabla (-\Delta)^{-1} \nabla \cdot J,
\end{equation}by Lemma \ref{lem.AmpMax}. Also note that
\begin{equation}\notag
  \partial_{x_3}^2 A_i
  = - 4\pi J_i  \  \ \textup{ at $x_3=0$  for  $ i=1,2$, }
\end{equation}which implies
\begin{equation}\notag
    \partial_{x_3} \B_2|_{x_3=0} = - 4\pi J_1, \ \   \partial_{x_3} \B_1|_{x_3=0}   =  4\pi J_2
\end{equation}
Recall that in the whole space, the equation \eqref{eq.wave.mag.pot} is solved using the retarded Green’s function. In the half-space setting, however, we modify the Green’s function by introducing image charges to enforce the boundary conditions.   The last condition in \eqref{AB_W} requires that $A_1 = A_2 = 0$ on the boundary $x_3 = 0$. Moreover, the perfect-conductor boundary condition $\mathbf{B}_3 = 0$ at $x_3 = 0$ corresponds to
\[
(\nabla \times \mathbf{A})_3
= 0,
\]
which is indeed satisfied. Consequently, we represent $A_1$ and $A_2$ by taking the odd extension of the Green’s function across the boundary $x_3=0$.  In addition, under the Coulomb gauge condition $\nabla \cdot \mathbf{A} = 0$, the component $A_3$ formally satisfies a homogeneous Neumann boundary condition at $x_3 = 0$. Therefore, we represent $A_3$ using the even extension of the Green’s function.
Therefore, we extend $\mathbb{R}_+$ to $\mathbb{R}$ and derive the representation by performing the time-variable reduction in the Green function for the wave equation.
Note that we have the Green function of the wave equation for $x\in \rth$ as
\begin{equation}
    G(t,\tau,x,y)=\frac{1}{2\pi}\delta((t-\tau)^2- |x-y|^2 )1_{|x-y|^2\le (t-\tau)^2}.
\end{equation}Then for the Green function in the half space $\mathbb{R}^3_+,$ we consider both odd and even extensions. For the odd extensions, we have the Dirichlet-Green function $\bar{G}$ for $x_3\ge 0$ as
\begin{multline}\label{Green in the half space.odd}
    \bar{G}(t,\tau,x,y)=G(t,\tau,x,y)-G(t,\tau,x,\bar{y})\\
    =\frac{1}{2\pi}\bigg\{\delta((t-\tau)^2- |x-y|^2 )1_{|x-y|^2\le (t-\tau)^2}-\delta((t-\tau)^2- |x-\bar{y}|^2 )1_{|x-\bar{y}|^2\le (t-\tau)^2}\bigg\},
\end{multline}where we define $\bar{y}=(y_1,y_2,-y_3)^\top.$ This odd extension will be used to derive the representation of $A_1$ and $A_2$.
On the other hand, we similarly write the even extension of $G$ and obtain the Neumann-Green function $\tilde{G}$ as
\begin{multline}\label{Green in the half space.even}
    \tilde{G}(t,\tau,x,y)=G(t,\tau,x,y)+G(t,\tau,x,\bar{y})\\
    =\frac{1}{2\pi}\bigg\{\delta((t-\tau)^2- |x-y|^2 )1_{|x-y|^2\le (t-\tau)^2}+\delta((t-\tau)^2- |x-\bar{y}|^2 )1_{|x-\bar{y}|^2\le (t-\tau)^2}\bigg\}.
\end{multline}This even extension will be used to derive the representation of $A_3$.
The particular solutions to the wave equation \eqref{eq.wave.mag.pot} can be represented as follows. Using the extended Green functions \eqref{Green in the half space.odd} and \eqref{Green in the half space.even}, we obtain that for $i=1,2$,
the particular solutions to \eqref{eq.wave.mag.pot} are given by
\begin{align}
    \label{Ai.rep.pre}A_i(t,x)&=4\pi\int_0^t d\tau\int_{\mathbb{R}^3_+}dy\ \bar{G}(t,\tau,x,y)(\mathbf{P}J)_i(\tau,y),\text { and }\\  \label{A3.rep.pre}A_3(t,x)&=4\pi\int_0^t d\tau\int_{\mathbb{R}^3_+}dy\ \tilde{G}(t,\tau,x,y)(\mathbf{P}J)_3(\tau,y).
\end{align}
Computing the delta functions in the Green functions \eqref{Green in the half space.odd} and \eqref{Green in the half space.even} in the integrals we obtain
\begin{align}
    \label{Ai.rep}A_i(t,x)&= \int_{\mathbb{R}^3_+}dy\ \left(\frac{(\mathbf{P}J)_i(t-|x-y| ,y)}{|x-y|}1_{|x-y|\le t}-\frac{(\mathbf{P}J)_i(t-|x-\bar{y}| ,y)}{|x-\bar{y}|}1_{|x-\bar{y}|\le t}\right),\text{ and }\\
    \label{A3.rep}A_3(t,x)&= \int_{\mathbb{R}^3_+}dy\ \left(\frac{(\mathbf{P}J)_3(t-|x-y| ,y)}{|x-y|}1_{|x-y|\le t}+\frac{(\mathbf{P}J)_3(t-|x-\bar{y}| ,y)}{|x-\bar{y}|}1_{|x-\bar{y}|\le t}\right).
\end{align}
This leads to the following image rule for extending $J(t,x)$ from $\Omega$ to all of $\mathbb{R}^3$:
$$
J^{\textup{ext}}(t,x) =
\begin{cases}
J(t,x), & x_3 \geq 0, \\
\mathcal{R} J(t,\bar{x}), & x_3 < 0,
\end{cases}
\quad \textup{where } \bar{x} \eqdef (x_1, x_2, -x_3),
$$
and the reflection operator $\mathcal{R}$ acts on a vector $V = (V_1, V_2, V_3)^\top$ as
$$
\mathcal{R} V \eqdef (-V_1, -V_2, V_3)^\top.
$$         Note that $\mathcal{R}^2=\textup{Id}.$
Then the extended current $J^{\textup{ext}}$ is divergence-free and ensures that the solution $\mathbf{A}^{\textup{ext}}$ to the wave equation in $\mathbb{R}^3$ satisfies the correct boundary condition for $\mathbf{B} = \nabla \times \mathbf{A}$ on $x_3 = 0$. Thus we define the particular solution  $\mathbf{B}_{\textup{par}}(t,x)$ to the wave equation with zero initial data in the half space $\mathbb{R}^3_+$ via the whole space representation as
$$
\mathbf{B}_{\textup{par}}(t,x) = \nabla \times \int_{
|x-y|\le t}\frac{\mathbf{P} J^{\textup{ext}}(t - |x-y| , y)}{ |x - y|}   \,  dy.
$$
This leads to the modified representation with the retarded term and the image term:
$
\mathbf{B}_{\textup{par}}(t,x) = \mathbf{B}_{\textup{ret}}(t,x) + \mathbf{B}_{\textup{img}}(t,x),
$
where
$$
\mathbf{B}_{\textup{img}}(t,x) = \nabla \times \left( \int_{|x - y| \leq t} \frac{1}{|x -y|} \mathbf{P} \mathcal{R} J\left(t - |x-y|, \bar{y}\right)1_{y_3<0}\  dy \right),
\quad \textup{with } \bar{y} = (y_1, y_2, -y_3).
$$To obtain the representation of this reflected term $\B_{\textup{img}}$, we also derive the following lemma:
\begin{lemma}\label{curl_retard_img}
Suppose
\begin{equation}\notag
\mathbf{A}_{\textup{img}}(t,x)=\int_{|x - y| \leq t} \frac{1}{|x -y|} \mathbf{P} \mathcal{R} J\left(t - |x-y|, \bar{y}\right)1_{y_3<0}\  dy,
\end{equation}
and let $\B_{\textup{img}}=\nabla_x\times \mathbf{A}_{\textup{img}}$. Then we have
\begin{align}
&\B_{\textup{img}}(t,x)\notag \\*
 &=  \int_{\substack{|Y|\le t\\Y_3<-x_3}} \frac{Y}{|Y|^3}\times  \left( \mathcal{R} J\left(t - |Y|, \bar{Y}+\bar{x}\right)\right)\  dY\notag +  \int_{\substack{|Y|=t\\Y_3<-x_3}} \frac{Y}{|Y|^2} \times \left( \mathcal{R} J\left(t - |Y|, \bar{Y}+\bar{x}\right)\right)\  dS_Y\notag
\\*
 & -\int_{\sqrt{|Y_\parallel|^2+x_3^2} \leq t} \frac{(J_2\left(t - \sqrt{|Y_\parallel|^2+x_3^2} , Y_\parallel+x_\parallel,0\right),-J_1\left(t - \sqrt{|Y_\parallel|^2+x_3^2} , Y_\parallel+x_\parallel,0\right),0)^\top }{\sqrt{|Y_\parallel|^2+x_3^2} }  \  dY_\parallel\label{retardU3_img}\\*
&     + \int_{\substack{|Y|\le t\\Y_3<-x_3}} \frac{Y}{|Y|^2}\Black{\times}\,\partial_t\left( \mathcal{R} J\left(t - |Y|, \bar{Y}+\bar{x}\right)\right)\  dY,\label{retardP_t_img}
\end{align}
where  $\bar{Y}=(Y_1,Y_2,-Y_3)^\top$ and $\bar{x}=(x_1,x_2,-x_3)^\top.$
\end{lemma}
\begin{proof}
We begin with recalling that $\mathbf{P}$ is the Leray projection operator onto divergence free vector fields, and therefore $\nabla \times \mathbf{P}J=\nabla \times J.$ We start with taking the change of variables $y\mapsto Y\eqdef y-x.$ Then we observe that
$$\mathbf{A}_{\textup{img}}(t,x)=\int_{|Y|\le t} \frac{1}{|Y|} \mathbf{P} \mathcal{R} J\left(t - |Y|, \bar{Y}+\bar{x}\right)1_{Y_3<-x_3}\  dY.$$
By taking the curl, we obtain
\begin{equation}\notag
   \mathbf{B}_{\textup{img}}(t,x)  = \nabla_x\times\mathbf{A}_{\textup{img}}(t,x)=\int_{|Y|\le t} \frac{1}{|Y|} \nabla_x\times\left( \mathcal{R} J\left(t - |Y|, \bar{Y}+\bar{x}\right)1_{Y_3<-x_3}\right)\  dY.
\end{equation}Now we recall that
\begin{multline*}
  \partial_{Y_j}\left( \mathcal{R} J\left(t - |Y|, \bar{Y}+\bar{x}\right)1_{Y_3<-x_3}\right)=\partial_{x_j}\left( \mathcal{R} J\left(t - |Y|, \bar{Y}+\bar{x}\right)1_{Y_3<-x_3}\right)\\- \frac{\partial|Y|}{\partial Y_j}\partial_t\left( \mathcal{R} J\left(t - |Y|, \bar{Y}+\bar{x}\right)1_{Y_3<-x_3}\right).
\end{multline*}
Therefore, we have
\begin{multline}\notag
   \mathbf{B}_{\textup{img}}(t,x)  =\int_{|Y|\le t} \frac{1}{|Y|} \nabla_Y\times\left( \mathcal{R} J\left(t - |Y|, \bar{Y}+\bar{x}\right)1_{Y_3<-x_3}\right)\  dY\\
   +  \int_{|Y|\le t} \frac{\nabla_Y|Y|}{|Y|} \Black{\times}\,\partial_t\left( \mathcal{R} J\left(t - |Y|, \bar{Y}+\bar{x}\right)1_{Y_3<-x_3}\right)\  dY.
\end{multline}Taking the integration by parts on the first integral, we further obtain that
\begin{align*}\notag
  \mathbf{B}_{\textup{img}}(t,x)
&=\int_{\substack{|Y|\le t\\Y_3<-x_3}} \frac{Y}{|Y|^3}\times  \left( \mathcal{R} J\left(t - |Y|, \bar{Y}+\bar{x}\right)\right)\  dY
 +  \int_{\substack{|Y|=t\\Y_3<-x_3}} \frac{Y}{|Y|^2} \times \left( \mathcal{R} J\left(t - |Y|, \bar{Y}+\bar{x}\right)\right)\  dS_Y\\
    &-\int_{\sqrt{|Y_\parallel|^2+x_3^2} \leq t} \frac{(J_2\left(t - \sqrt{|Y_\parallel|^2+x_3^2} , Y_\parallel+x_\parallel,0\right),-J_1\left(t - \sqrt{|Y_\parallel|^2+x_3^2} , Y_\parallel+x_\parallel,0\right),0)^\top }{\sqrt{|Y_\parallel|^2+x_3^2} }  \  dY_\parallel\\
   &+ \int_{\substack{|Y|\le t\\Y_3<-x_3}} \frac{Y}{|Y|^2}\Black{\times}\,\partial_t\left( \mathcal{R} J\left(t - |Y|, \bar{Y}+\bar{x}\right)\right)\  dY,
\end{align*}
since $\mathcal{R}J=(-J_1,-J_2,J_3)^\top.$ This completes the proof.
\end{proof}Similarly, by considering the integrand $\frac{1}{|x -y|} \mathbf{P} J\left(t - |x-y|, y\right)1_{y_3\ge 0}$ instead of the reflected one \newline$\frac{1}{|x -y|} \mathbf{P} \mathcal{R} J\left(t - |x-y|, \bar{y}\right)1_{y_3<0},$ we also obtain the retarded field term $\B_{\textup{ret}}$ for the other half space as \begin{align}
&\B_{\textup{ret}}(t,x)\notag \\*
&= \int_{\substack{|Y|\le t\\Y_3\ge -x_3}} \frac{Y}{|Y|^3}\times  \left(  J\left(t - |Y|, Y+x\right)\right)\  dY\notag
 +  \int_{\substack{|Y|=t\\Y_3\ge -x_3}} \frac{Y}{|Y|^2} \times \left(  J\left(t - |Y|, Y+x\right)\right)\  dS_Y
\\*
 & -\int_{\sqrt{|Y_\parallel|^2+x_3^2} \leq t} \frac{(-J_2\left(t - \sqrt{|Y_\parallel|^2+x_3^2} , Y_\parallel+x_\parallel,0\right),J_1\left(t - \sqrt{|Y_\parallel|^2+x_3^2} , Y_\parallel+x_\parallel,0\right),0)^\top }{\sqrt{|Y_\parallel|^2+x_3^2} }  \  dY_\parallel\label{retardU3_ret}\\*
&     + \int_{\substack{|Y|\le t\\Y_3\ge -x_3}} \frac{Y}{|Y|^2}\Black{\times}\,\partial_t\left(  J\left(t - |Y|, Y+x\right)\right)\  dY.\notag
\end{align}
\begin{remark}
    Note that the two boundary terms \eqref{retardU3_img} and \eqref{retardU3_ret} exactly cancel each other and disappear in the final representation $\mathbf{B}(t,x) = \mathbf{B}_{\textup{hom}}(t,x) + \mathbf{B}_{\textup{ret}}(t,x) + \mathbf{B}_{\textup{img}}(t,x)$. This is by the fact that
    $$\hat{e}_3 \times \mathcal{R}J(t-|Y|,\bar{Y}+\bar{x})\bigg|_{Y_3=-x_3}+\hat{e}_3 \times J(t-|Y|,Y+x)\bigg|_{Y_3=-x_3}=0,$$ since $\mathcal{R}J=(-J_1,-J_2,J_3)^\top $ and $\bar{Y}+\bar{x}=Y+x $ if $Y_3=-x_3.$
\end{remark}
Further computing the temporal integral $\partial_t F_\pm$ via the Vlasov equation\Black{, and applying the integration by parts in $v$ as in the whole-space case \eqref{CONT_F_t2=0}}, we obtain the reflected term $\B_{\textup{img}}$ as follows:
\begin{equation}\notag
\begin{split}
\B_{\textup{img}}(t,x)  =&\sum_{\iota=\pm}\iota
\int_{\substack{|Y|\le t \\ Y_3 < -x_3}}
\int_{\mathbb{R}^3}      \frac{Y\times \mathcal{M}\hat{v}_\iota}{ |Y|^3\big(1+ \mathcal{M}\hat{v}_\iota  \cdot  \frac{Y}{|Y|} \big)^2} \left(1-\left|\mathcal{M}\hat{v}_\iota\right|^2\right)
   F_\iota \left(t- |Y| , \bar{Y}+\bar{x}, v \right)
dv
dY \\
        & \sum_{\iota=\pm}\iota  \int_{\substack{|Y|=t \\ Y_3 < -x_3}} \int_{\mathbb{R}^3}
       \frac{Y}{|Y|^2} \times
        \frac{\mathcal{M}\hat{v}_\iota}{1+ \mathcal{M}\hat{v}_\iota  \cdot \frac{Y}{|Y|}} \,
         F_\iota(0, \bar{Y}+\bar{x}, v ) \, dv \, dS_Y\\
         &\Black{-\sum_{\iota=\pm}\int_{\substack{|Y|\le t \\ Y_3 < -x_3}}\int_{\mathbb{R}^3}\big(\mathbf{K}_\iota\cdot\nabla_v\big)\left(\frac{Y}{|Y|^2}\times\frac{\mathcal{M}\hat{v}_\iota}{1+\mathcal{M}\hat{v}_\iota\cdot\frac{Y}{|Y|}}\right)F_\iota\left(t-|Y|,\bar{Y}+\bar{x},v\right)dv\,dY}\\
          & -\int_{\sqrt{|Y_\parallel|^2+x_3^2} \leq t} \frac{(J_2\left(t - \sqrt{|Y_\parallel|^2+x_3^2} , Y_\parallel+x_\parallel,0\right),-J_1\left(t - \sqrt{|Y_\parallel|^2+x_3^2} , Y_\parallel+x_\parallel,0\right),0)^\top }{\sqrt{|Y_\parallel|^2+x_3^2} }  \  dY_\parallel,
\end{split}
\end{equation}where $\bar{Y} = (Y_1, Y_2, -Y_3)^\top$ and $\mathcal{M}\hat{v}_\iota \eqdef (\hat{v}_{\iota,1}, \hat{v}_{\iota,2}, -\hat{v}_{\iota,3})^\top$\Black{, and $\mathbf{K}_\iota\eqdef\big(\E+\hat{v}_\iota\times\B\big)\left(t-|Y|,\bar{Y}+\bar{x}\right)$ in the third integral}.
 Therefore, we obtain the final representation of the magnetic field $\B$ in the half space $\mathbb{R}^3_+$:
\begin{proposition}[Representation of magnetic field in the half space $\mathbb{R}^3_+$ in terms of the distribution $F_\pm$]
Let $\Omega = \{ x \in \mathbb{R}^3 : x_3 > 0 \}$, and suppose the initial data satisfies the perfect conductor boundary condition $\mathbf{B} \cdot n = \mathbf{B}_3 = 0$ on $x_3 = 0$. Then the magnetic field $\mathbf{B}(t,x)$ for $x \in \Omega$ is represented by
\begin{equation}\label{B_half_final}
\mathbf{B}(t,x) = \mathbf{B}_{\textup{hom}}(t,x) + \mathbf{B}_{\textup{par}}(t,x),
\end{equation}
where each component is given below.

\paragraph{Homogeneous solution}:
 The normal component $\Bthhomo$ is given by\begin{multline}\label{B3 homo solution}
\Bthhomo(t, x) =\frac{1}{4\pi t^2} \int_{\partial B(x; t)\cap \{y_3>0\}}  \left(t \mathbf{B}^1_{03}(y) +\mathbf{B}_{03}(y) + \nabla \mathbf{B}_{03}(y) \cdot (y - x)\right) dS_y\\
-\frac{1}{4\pi t^2} \int_{\partial B(x; t)\cap \{y_3<0\}}  \left(t \mathbf{B}^1_{03}(\bar{y}) + \mathbf{B}_{03}(\bar{y}) + \nabla \mathbf{B}_{03}(\bar{y}) \cdot (\bar{y} - \bar{x})\right) dS_y,
\end{multline} by the Kirchhoff formula.  The tangential components $\Bihomo$ for $i=1,2$, which satisfy the Neumann boundary condition, are further decomposed as $\Bihomo=\mathbf{B}_{\textup{neu},i}+\mathbf{B}_{\textup{cau},i}$ and are written as
\begin{align}\label{additional term B1}
    \mathbf{B}_{\textup{neu},i}(t,x)
    &= 2(-1)^{j}\sum_{\iota=\pm}\iota\int_{B(x;t)\cap \{y_3=0\}} \int_\rth \frac{\hat{v}_jF_\iota(t-|y-x|,y_\parallel,0,v)}{|y-x|}dvdy_\parallel,\text{ for }i,j=1,2\text{, }j\ne i, \\
    \label{Bi homo solution}\mathbf{B}_{\textup{cau},i}(t, x) &= \frac{1}{4\pi t^2} \int_{\partial B(x; t)\cap \{y_3>0\}}  \left(t \mathbf{B}^1_{0i}(y) +\mathbf{B}_{0i}(y) + \nabla \mathbf{B}_{0i}(y) \cdot (y - x)\right) dS_y\\
&\notag+\frac{1}{4\pi t^2} \int_{\partial B(x; t)\cap \{y_3<0\}}  \left(t \mathbf{B}^1_{0i}(\bar{y}) + \mathbf{B}_{0i}(\bar{y}) + \nabla \mathbf{B}_{0i}(\bar{y}) \cdot (\bar{y} - \bar{x})\right) dS_y.
\end{align}
\paragraph{Particular solution} $\B_{\textup{par}} $ is written by
$\B_{\textup{par}} =\sum_{\iota=\pm}\left(\B^{(1)}_{\iota,\textup{par}} -\B^{(2)}_{\iota,\textup{par}} \right),$ where for $j=1,2$, we decompose further into the $T$ part\Black{,} the initial-value part\Black{, and the nonlinear $S$ part} as
$$\B^{(j)}_{\pm,\textup{par}} =\B^{(j)}_{\pm,\textup{par},T} -\B^{(j)}_{\pm,\textup{par},b1}\Black{+\B^{(j)}_{\pm,\textup{par},S}}, $$ such that
\begin{equation}
\label{Bpar_half_final}
    \begin{split}
        \B^{(1)}_{\pm,\textup{par},T}(t,x)&\eqdef \pm\int_{\substack{|Y|\le t \\ Y_3 \ge -x_3}}
\int_{\mathbb{R}^3}      \frac{Y\times \hat{v}_\pm}{ |Y|^3\big(1+ \hat{v}_\pm \cdot  \frac{Y}{|Y|} \big)^2} \left(1-\left|\hat v_\pm\right|^2\right)
   F_\pm \left(t- |Y| , Y+x, v\right)
dv
dY  ,\\
\B^{(2)}_{\pm,\textup{par},T}(t,x)&\eqdef \pm
\int_{\substack{|Y|\le t \\ Y_3 < -x_3}}
\int_{\mathbb{R}^3}      \frac{Y\times \mathcal{M}\hat{v}_\pm}{ |Y|^3\big(1+ \mathcal{M}\hat{v}_\pm \cdot  \frac{Y}{|Y|} \big)^2} \left(1-\left|\mathcal{M}\hat{v}_\pm\right|^2\right)
   F_\pm \left(t- |Y| , \bar{Y}+\bar{x}, v \right)
dv
dY,\\
\B^{(1)}_{\pm,\textup{par},b1}(t,x) & \eqdef \pm  \int_{\substack{|Y|=t \\ Y_3 \ge -x_3}} \int_{\mathbb{R}^3}
       \frac{Y}{|Y|^2} \times
        \frac{\hat{v}_\pm}{1+ \hat{v}_\pm \cdot \frac{Y}{|Y|}} \,
         F_\pm (0, Y+x, v ) \, dv \, dS_Y,\\
       \B^{(2)}_{\pm,\textup{par},b1}(t,x) & \eqdef\pm  \int_{\substack{|Y|=t \\ Y_3 < -x_3}} \int_{\mathbb{R}^3}
       \frac{Y}{|Y|^2} \times
        \frac{\mathcal{M}\hat{v}_\pm}{1+ \mathcal{M}\hat{v}_\pm \cdot \frac{Y}{|Y|}} \,
         F_\pm(0, \bar{Y}+\bar{x}, v ) \, dv \, dS_Y,
    \end{split}
    \end{equation}where $\bar{Y} = (Y_1, Y_2, -Y_3)^\top$ and $\mathcal{M}\hat{v}_\pm \eqdef (\hat{v}_{\pm,1}, \hat{v}_{\pm,2}, -\hat{v}_{\pm,3})^\top$.
{\color{black}
The nonlinear $S$ parts are given by
\begin{equation}\label{BparS_half_final}
    \begin{split}
\B^{(1)}_{\pm,\textup{par},S}(t,x)&\eqdef \int_{\substack{|Y|\le t \\ Y_3 \ge -x_3}}\int_{\mathbb{R}^3}\big(\mathbf{K}_\pm\cdot\nabla_v\big)\left(\frac{Y}{|Y|^2}\times\frac{\hat{v}_\pm}{1+\hat{v}_\pm\cdot\frac{Y}{|Y|}}\right)F_\pm\left(t-|Y|,Y+x,v\right)dv\,dY,\\
\B^{(2)}_{\pm,\textup{par},S}(t,x)&\eqdef \int_{\substack{|Y|\le t \\ Y_3 < -x_3}}\int_{\mathbb{R}^3}\big(\mathbf{K}_\pm\cdot\nabla_v\big)\left(\frac{Y}{|Y|^2}\times\frac{\mathcal{M}\hat{v}_\pm}{1+\mathcal{M}\hat{v}_\pm\cdot\frac{Y}{|Y|}}\right)F_\pm\left(t-|Y|,\bar{Y}+\bar{x},v\right)dv\,dY,
    \end{split}
\end{equation}
where $\mathbf{K}_\pm=\E+\hat{v}_\pm\times\B$ is evaluated at the same retarded point as $F_\pm$ in each integral, and $(\mathbf{K}\cdot\nabla_v)\mathbf{G}\eqdef\sum_j\mathbf{K}_j\partial_{v_j}\mathbf{G}$; cf.~\eqref{BS} and \eqref{kernelB2}. Note that, in contrast to the $T$ and $b1$ parts, no sign $\pm$ appears in front of the $S$ parts: both species enter with the same sign, cf.~Remark \ref{rmk.BS.GS}.
}
We will also write for $j=1,2$,
    \begin{multline*}
\B^{(j)}_{\pm,\textup{par},T} = (\mathbf{B}^{(j)}_{\pm,\textup{par},1T} ,\mathbf{B}^{(j)}_{\pm,\textup{par},2T} ,\mathbf{B}^{(j)}_{\pm,\textup{par},3T} )^\top\ \text{and }\  \B^{(j)}_{\pm,\textup{par},b1} = (\mathbf{B}^{(j)}_{\pm,\textup{par},1b1} ,\mathbf{B}^{(j)}_{\pm,\textup{par},2b1} ,\mathbf{B}^{(j)}_{\pm,\textup{par},3b1} )^\top,\\ \Black{\text{and analogously } \B^{(j)}_{\pm,\textup{par},S} = (\mathbf{B}^{(j)}_{\pm,\textup{par},1S} ,\mathbf{B}^{(j)}_{\pm,\textup{par},2S} ,\mathbf{B}^{(j)}_{\pm,\textup{par},3S} )^\top}.
    \end{multline*}
    \end{proposition}
\begin{remark}[Remark on the absence of \Black{the} boundary terms in $\B$]
 Compared to the representations of the electric field $\E_{\textup{par}}$ in \eqref{Ei5},  and \eqref{E3}, which will be derived in the next section using the Green function for the wave equation satisfied by $\E$, we observe that the magnetic field representation \eqref{Bpar_half_final}, obtained via the magnetic vector potential, does not involve the \Black{boundary value $b2$ terms. The nonlinear $S$ terms \eqref{BparS_half_final} do appear, and they coincide with the corresponding nonlinear terms of the Glassey--Strauss type representation; see Remark \ref{rmk.BS.GS}}. This \Black{absence of the boundary terms} is due to cancellations that occur through the curl operator in the relation $\B = \nabla \times \A$, as proved in this section.
\end{remark}
\begin{remark}
    Note that the electric field representation in \eqref{Ei5} and \eqref{E3} is written under the following change of variables, compared to the representation \eqref{Bpar_half_final}:
    $$\omega=\frac{Y}{|Y|}=\frac{y-x}{|y-x|},\textup{ with } Y=y-x\textup{ and }Y+x=y.$$
\end{remark}
This completes the introduction to the potential representation of the magnetic field $\B(t,x)$ in the half space.

\subsection{Relativistic Trajectory}\label{sec.trajectory}
We first introduce the dynamical characteristic trajectory $\ZS(s)=(\XS(s),\allowbreak\VS(s))$ which solves the following characteristic ODEs:
 \begin{equation}\label{leading char}
 \begin{split}
   \frac{d\XS(s)}{ds}&=\hat{\VS}(s)=\frac{
   \VS(s)}{\sqrt{m_\pm^2
   +|\VS(s)|^2}},\\
        \frac{d\VS(s)}{ds}&=\pm
\E(s,\XS(s))\pm
\hat{\VS}(s)
\times \B(s,\XS(s))-m_\pm g\hat{e}_3\eqdef\mathcal{F}_\pm(s,\XS(s),\VS(s)),
 \end{split} 
\end{equation}   where $\XS(s)=\XS(s;t,x,v)$, $\VS(s)=\VS(s;t,x,v)$, and $\hat{\VS}\eqdef \frac{
\VS}{\sqrt{m_\pm^2
+|\VS|^2}}$. The solution $\allowbreak (\XS(s), \VS(s))$ to \eqref{leading char} is well-defined under the a priori assumption that $\E$ and $\B$ are in $W^{1,\infty}$ and hence are locally Lipshitz continuous in the spatial variables uniformly in the temporal variable.

Similarly, we also introduce the stationary counterpart of the characteristic trajectory as
$\Zst(s)=(\Xst(s),\Vst(s))$ satisfying $\Zst(0 ; x, v)=(\Xst(0;  x, v), \Vst(0 ;  x, v))=(x, v)=z$, generated by the fields $\Est$ and $ \Bst$, which solves
\begin{equation}\label{def.steady characteristics}
 \begin{split}
   \frac{d\Xst(s)}{ds}&=\Vhatst(s)=\frac{
   \Vst(s)}{\sqrt{m_\pm^2
   +|\Vst(s)|^2}},\\
        \frac{d\Vst(s)}{ds}&=\pm  
        \Est(s,\Xst(s))\pm 
        \Vhatst(s)
        \times \Bst(s,\Xst(s))- m_\pm g\hat{e}_3,
 \end{split} 
\end{equation}   where $\hat{e}_3\eqdef (0,0,1)^\top$ and $\hat{v}_\pm\eqdef \frac{
v}{\sqrt{m_\pm^2
+|v|^2}}$.  

\paragraph{Boundary Exit Time}
Using the characteristic trajectory under the presence of the external gravity term $-m_\pm g\hat{e}_3,$ we will define the following forward and backward exit times at which the particle collides the boundary and vanishes:
\begin{definition}
\label{def.exit.times}Define the forward and backward exit times as follows:
    \begin{equation}\label{Back Forw exit time}\begin{split}
   & \tf (t,x,v)= \sup\{s\in [0,\infty): (\XS)_3(t+\tau;t,x,v)>0\ \textup{ for all } \tau\in(0,s)\}\ge 0,\\
   & \tb (t,x,v)= \sup\{s\in [0,\infty): (\XS)_3(t-\tau;t,x,v)>0\ \textup{ for all } \tau\in(0,s)\}\ge 0.
   \end{split}
\end{equation}
\end{definition}

If $t-\tb \ge 0,$ the definition of $\tb$ guarantees that $$(\XS(t-\tb(t,x,v);t,x,v),\VS(t-\tb(t,x,v);t,x,v))\in \gamma_-\cup \gamma_0,$$ with $(\XS)_3(t-\tb)=0.$ Then we observe that the solution $F_\pm$ to \eqref{2speciesVM} at $(t,x,v)$ is given either by the initial profile or by the incoming boundary profile along the characteristic trajectory; i.e., if $t-\tb >0,$ then we have 
\begin{equation}
    \label{sol as boundary}F_\pm(t,x,v)= F_\pm(t-\tb,\XS(t-\tb(t,x,v);t,x,v),\VS(t-\tb(t,x,v);t,x,v))|_{(\XS,\VS)\in \gamma_-}.
\end{equation} On the other hand, if  $t-\tb \le 0,$ then we have 
\begin{equation}\label{sol as initial}
F_\pm(t,x,v)=F^{\textup{in}}_\pm(\XS(0;t,x,v),\VS(0;t,x,v)),
\end{equation} where the initial condition is defined as $F^{\textup{in}}_\pm(x,v)\eqdef F_\pm(0,x,v).$  Thus we write
\begin{equation}\begin{split}\label{solution f}
    F_\pm(t,x,v)&= 1_{t\le \tb(t,x,v)}F^{\textup{in}}_\pm(\XS(0;t,x,v),\VS(0;t,x,v))\\&+1_{t>\tb(t,x,v)}F_\pm(t-\tb,\XS(t-\tb;t,x,v),\VS(t-\tb;t,x,v)))|_{(\XS,\VS)\in \gamma_-}\\
&=1_{t\le \tb(t,x,v)}F^{\textup{in}}_\pm(\XS(0;t,x,v),\VS(0;t,x,v))+1_{t>\tb(t,x,v)}G_\pm(t-\tb, \xb,\vb),
\end{split}\end{equation} using the definition of $\xb$ and $\vb$ from \eqref{xbvb} and the incoming boundary profile $G_\pm.$ Given that our solution $F_\pm$ is  locally Lipshitz continuous, the mild formulation \eqref{solution f} is well-defined.

We also denote the characteristics 
 as $\ZS(s ; t, x, v)=(\XS(s ; t, x, v), \VS(s ; t, x, v))$ for the dynamical problem satisfying $\ZS(t ; t, x, v)=(\XS(t ; t, x, v), \VS(t ; t, x, v))=(x, v)=z$. Suppose $\E(t, \cdot),\B(t, \cdot) \in C^{1}(\Omega)$. Then $\ZS(s ; t, x, v)$ is well-defined as long as $\XS(s ; t, x, v) \in \Omega$.
 We also define the backward exit position and momentum and the forward and backward exit times:
 \begin{definition}
     Define the backward exit position and momentum as
\begin{equation}\label{xbvb}
    \begin{split}
        &\xb(t, x, v)=\XS\left(t-\tb(t, x, v) ; t, x, v\right) \in \partial \Omega,\\  &\vb(t, x, v)=\VS\left(t-\tb(t, x, v) ; t, x, v\right).
    \end{split}
\end{equation}
Then $\ZS(s ; t, x, v)$ is continuously extended in a closed interval of $s \in\left[t-\tb(t, x, v), t\right]$.
\end{definition}
Similarly, using the stationary counterpart of the characteristic trajectory $Z_{\pm,\textup{st}}$ solving \eqref{def.steady characteristics}, we can define the analogous exit terms $\tfst$, $\tbst$, $\xbst$, and $\vbst$ for the steady characteristic trajectory as follows:
\begin{definition}Define the backward/forward exit times and the backward exit position and momentum as
\begin{equation}\label{Back Forw exit time.st}\begin{split}
   & \tfst (x,v)= \sup\{s\in [0,\infty): (\Xst)_3(\tau;x,v)>0\ \textup{ for all } \tau\in(0,s)\}\ge 0,\\
   & \tbst (x,v)= \sup\{s\in [0,\infty): (\Xst)_3(-\tau;x,v)>0\ \textup{ for all } \tau\in(0,s)\}\ge 0,\\
        &\xbst(x, v)=\Xst\left(-\tbst(x, v) ; x, v\right) \in \partial \Omega,\\  &\vbst( x, v)=\Vst\left(-\tbst(x, v) ; x, v\right).
    \end{split}
\end{equation}

 \end{definition}
\subsection{Weight Comparison}

For the stability analysis, it is important to compare weight functions along the characteristics.  For any given $\beta>1$, we define a weight function for a 2-species problem in the half space $\mathbb{R}^2\times \mathbb{R}_+$
\begin{equation}
\label{weights.wholehalf}
 \mathrm{w}_\pm(x, v)  =\mathrm{w}_{\pm,\beta}(x, v)=e^{\beta\left(\sqrt{m_\pm^2
 +|v|^2}+m_\pm g x_{3}\right)}e^{\frac{\beta}{2}|x_\parallel|}.
\end{equation}
Physically, $\beta $ and $g$ correspond to the inverse temperature $\frac{1}{T}$ and the gravity, respectively, under the assumption that the Boltzmann constant $k_B = \frac{1}{2}$. 
Note that this weight is not invariant along the characteristics.

\subsubsection{Weight Comparison in the Stationary Case}\label{sec.weight.comparison}
 We first note that the stationary trajectories satisfy
\begin{equation}\label{ODE for particle energy.st}
    \frac{d}{ds}\left(\sqrt{m_\pm^2+|\Vst(s)|^2}+m_\pm g(\Xst)_3(s)\right)=\hat{v}_\pm(s)\cdot \frac{dV_\pm}{ds}+m_\pm g\hat{v}_{\pm,3}(s)
    =\pm \hat{v}_\pm(s)\cdot \E_{\textup{st}}(\Xst(s)),
\end{equation}because $$\frac{dV_\pm}{ds}=\pm (\E_{\textup{st}}+\hat{v}_\pm\times \B_{\textup{st}}\mp m_\pm g\hat{e}_3).$$
Also, note that
\begin{equation}
    \label{D.7}
    \frac{d}{ds}\left(\frac{1}{2}(\Xst)_\parallel(s)\right)=\frac{1}{2}\hat{v}_{\pm,\parallel}(s).
\end{equation}
By assuming that \begin{equation}
    \label{EBst bound ass} \|(\E_{\textup{st}},\B_{\textup{st}})\|_{L^\infty }\le \min\{m_+,m_-\}\frac{g}{16},
\end{equation}  
we observe that$$\left|\left(\frac{dV_{\pm,1}}{ds}(s),\frac{dV_{\pm,2}}{ds}(s)\right)\right|\le \left|\E_{\textup{st}}+\hat{v}_\pm\times \B_{\textup{st}}\right| \le \min\{m_+,m_-\}\frac{g}{8},$$ and
   $$\frac{d(\Vst)_3}{ds}(s)=-(\E_{\textup{st}}+\hat{v}_\pm\times \B_{\textup{st}})_3-m_\pm g \le -\frac{7}{8}m_\pm g,$$ since $|\hat{v}_\pm|\le 1.$ Now if we define a trajectory variable $s^*=s^*(x,v)\in[-\tbst,\tfst]$ such that $(\Vst)_3(s^*;x,v)=0$, then we have
   \begin{align*}(\Vst)_3(\tfst)-(\Vst)_3(s^*)&=\int^{\tfst}_{s^*}\frac{d(\Vst)_3}{ds}(\tau)d\tau\le -\frac{7}{8}m_\pm g(\tfst-s^*),\text{ and }\\
   (\Vst)_3(s^*)-(\Vst)_3(-\tbst)&=\int_{-\tbst}^{s^*}\frac{d(\Vst)_3}{ds}(\tau)d\tau\le -\frac{7}{8}m_\pm g(s^*+\tbst).\end{align*}Therefore, we have
   \begin{equation}
       \label{tbfst bound}
       \tbst+\tfst\le -\frac{8}{7m_\pm g}((\Vst)_3(\tfst)-(\Vst)_3(-\tbst)).
   \end{equation}On the other hand, using \eqref{ODE for particle energy.st} and \eqref{EBst bound ass}, we have
  \begin{align*}
\sqrt{m_\pm^2+|\Vst(-\tbst)|^2}=\left(\sqrt{m_\pm^2+|v_\pm|^2}+m_\pm gx_3\right)
    &\pm \int_0^{-\tbst}  \hat{v}_\pm(s)\cdot \E_{\textup{st}}(\Xst(s))ds\\
    &\le \left(\sqrt{m_\pm^2+|v_\pm|^2}+m_\pm gx_3\right)+\frac{m_\pm g}{16}\tbst,\text{ and }\\
\sqrt{m_\pm^2+|\Vst(\tfst)|^2}=\left(\sqrt{m_\pm^2+|v_\pm|^2}+m_\pm gx_3\right)
   & \pm \int_0^{\tfst}  \hat{v}_\pm(s)\cdot \E_{\textup{st}}(\Xst(s))ds\\
    &\le \left(\sqrt{m_\pm^2+|v_\pm|^2}+m_\pm gx_3\right)+\frac{m_\pm g}{16}\tfst.
  \end{align*} Thus, together with \eqref{tbfst bound}, we have$$
       \tbst+\tfst\le \frac{8}{7m_\pm g}\left(2(\sqrt{m_\pm^2+|v_\pm|^2}+m_\pm gx_3)+\frac{m_\pm g}{16} (\tbst+\tfst)\right).
 $$  Therefore, we have
\begin{equation}\label{exit time bound.st}\tbst+\tfst\le \frac{14}{13} \frac{16}{7m_\pm g} (\sqrt{m_\pm^2+|v_\pm|^2}+m_\pm gx_3)=\frac{32}{13m_\pm g} (\sqrt{m_\pm^2+|v_\pm|^2}+m_\pm gx_3).\end{equation}
Therefore, for $s,s'\in[-\tbst,\tfst],
    $
    we have   \begin{align*}
    \frac{\mathrm{w}_{\pm,\beta}\left( \Zst\left(s^{\prime} ; x, v\right)\right)}{\mathrm{w}_{\pm,\beta}( \Zst(s ; x, v))} &=e^{m_\pm g\beta ((\Xst)_3(s')-(\Xst)_3(s)) + \beta (\vZ(s')-\vZ(s))+\frac{\beta}{2}  ((\Xst)_\parallel(s')-(\Xst)_\parallel(s))}\\    &=e^{\beta\left(\int_s^{s'}\frac{d}{d\tau}(\vZ(\tau)+m_\pm g(\Xst)_3(\tau)+\frac{1}{2}(\Xst)_\parallel)d\tau\right)}
        \le e^{\beta |s'-s|\sup_\tau|\hat{v}_\pm(\tau)|(|\E_{\textup{st}}(\Xst(\tau))|+1)}\\
        &\le e^{\beta(\tbst+\tfst)\sup_\tau|\hat{v}_\pm(\tau)|(|\E_{\textup{st}}(\Xst(\tau))|+1)}
        \leq e^{\left(\left\|\E_{\textup{st}}\right\|_{L_{t, x}^{\infty}} +1\right)\frac{32\beta}{13m_\pm g} (\sqrt{m_\pm^2+|v_\pm|^2}+m_\pm gx_3)},
    \end{align*}
    by \eqref{weights.wholehalf}, \eqref{ODE for particle energy.st}, \eqref{D.7} and \eqref{exit time bound.st}. 
Altogether, we obtain the following stationary counterparts:
for $s, s^{\prime} \in\left[-\tbst(x, v), \tfst(x, v)\right]$, we have
\begin{equation}\label{w comparison 1.st}
\frac{\mathrm{w}_{\pm,\beta}\left( \Zst\left(s^{\prime} ;  x, v\right)\right)}{\mathrm{w}_{\pm,\beta}( \Zst(s ;  x, v))} \leq e^{\left(\left\|\E_{\textup{st}}\right\|_{L_{ x}^{\infty}} +1\right)\frac{32\beta}{13m_\pm g} (\sqrt{m_\pm^2+|v_\pm|^2}+m_\pm gx_3)}.
\end{equation}
In addition, by considering $s'=0$ in \eqref{w comparison 1.st}, we have $$\mathrm{w}_{\pm,\beta}(\Zst(0;x,v))=\mathrm{w}_{\pm,\beta}(x,v)=e^{\beta v_\pm^0+m_\pm g\beta x_3+\frac{\beta}{2}|x_{\parallel}|},$$ and obtain that if we further assume $1\le \frac{1}{8}\min\{m_-,m_+\}g$, then by \eqref{EBst bound ass} we have 
\begin{equation}\begin{split}\label{w comparison 3.st}
    \frac{1}{\mathrm{w}_{\pm,\beta}( \Zst(s ; x, v))}&\leq e^{(\min\{m_-,m_+\}g)\left(\frac{1}{16}+\frac{1}{8}\right)\frac{32\beta}{13m_\pm g} (\sqrt{m_\pm^2+|v_\pm|^2}+m_\pm gx_3)}e^{-\beta v_\pm^0-m_\pm g\beta x_3-\frac{\beta}{2}|x_{\parallel}|}\\
   &\le  e^{\frac{6\beta}{13} (\sqrt{m_\pm^2+|v_\pm|^2}+m_\pm gx_3)}e^{-\beta v_\pm^0-m_\pm g\beta x_3-\frac{\beta}{2}|x_{\parallel}|}
   \le e^{-\frac{1}{2}\beta v_\pm^0-\frac{1}{2}m_\pm g\beta x_3-\frac{\beta}{2}|x_{\parallel}|}.
\end{split}\end{equation}

\subsubsection{Weight Comparison in Dynamical Case}
One can also check easily that the same discussion of Section \ref{sec.weight.comparison} can also be extended to the dynamical case if the stationary trajectory $\Zst$ is now replaced by the dynamical trajectory $\ZS$ which satisfies \eqref{leading char}. Namely, we obtain that 
 \begin{equation}
       \label{tbf bound.whole}
       \tb+\tf\le -\frac{8}{7m_\pm g}((\VS)_3(t+\tf)-(\VS)_3(t-\tb)),
   \end{equation}for $\ZS(s)=\ZS(s;t,x,v)$ with $s\in [t-\tb,t+\tb].$ Assume further that the self-consistent electromagnetic fields $(\E,\B)$ satisfy the following bound:
\begin{equation}\label{EB apriori bound.whole}
       \sup_t \|(\E,\B)\|_{L^\infty }\le \min\{m_+,m_-\}\frac{g}{8}, 
    \end{equation}similarly to the stationary assumption \eqref{EBst bound ass}. Then further using \eqref{ODE for particle energy.gen} and \eqref{EB apriori bound.whole}, one can obtain that 
 \begin{align*}
\sqrt{m_\pm^2+|\VS(t-\tb)|^2}
   & \le \left(\sqrt{m_\pm^2+|v_\pm|^2}+m_\pm gx_3\right)+\frac{m_\pm g}{8}\tb,\textup{ and }\\
\sqrt{m_\pm^2+|\VS(t+\tf)|^2}
  &  \le \left(\sqrt{m_\pm^2+|v_\pm|^2}+m_\pm gx_3\right)+\frac{m_\pm g}{8}\tf.
  \end{align*}Therefore, we have by \eqref{tbf bound.whole}
\begin{equation}
    \label{exit time bound.whole}
    \tb+\tf\le\frac{16}{5m_\pm g} (\sqrt{m_\pm^2+|v_\pm|^2}+m_\pm gx_3),
\end{equation}which gives the same bound to the stationary case \eqref{exit time bound.st}. Therefore, for $s,s'\in[t-\tb,t+\tf],
    $
    we have \begin{equation}\label{w comparison 1.whole}
\frac{\mathrm{w}_{\pm,\beta}\left( \ZS\left(s^{\prime} ;  t,x, v\right)\right)}{\mathrm{w}_{\pm,\beta}( \ZS(s ;t,  x, v))} \leq e^{\left(\left\|\E\right\|_{L_{t, x}^{\infty}} +1\right)\frac{16\beta}{5m_\pm g} (\sqrt{m_\pm^2+|v_\pm|^2}+m_\pm gx_3)}.
\end{equation}
Here, we observe that when $s'=t,$ $$\mathrm{w}_{\pm,\beta}(\ZS(t;t,x,v))=\mathrm{w}_{\pm,\beta}(x,v)=e^{\beta v_\pm^0+m_\pm g\beta x_3+\frac{\beta}{2}|x_{\parallel}|}.$$ Therefore, by \eqref{EB apriori bound.whole} with $\textcolor{black}{\min\{m_-,m_+\}g\gg 1}$, we have 
\begin{equation}\label{w comparison 3.whole}
    \frac{1}{\mathrm{w}_{\pm,\beta}( \ZS(s ;t, x, v))}
   \le  e^{\frac{\beta}{2} (\sqrt{m_\pm^2+|v_\pm|^2}+m_\pm gx_3)}e^{-\beta v_\pm^0-m_\pm g\beta x_3-\frac{\beta}{2}|x_{\parallel}|}
   \le e^{-\frac{1}{2}\beta v_\pm^0-\frac{1}{2}m_\pm g\beta x_3-\frac{\beta}{2}|x_{\parallel}|}.
\end{equation}

This completes the weight comparison argument, which will be used crucially in the stability analysis in the rest of the paper.

\section{Construction of the Steady States 
}\label{sec.exist.steady} In this section, we prove the existence and uniqueness of steady states with J\"uttner-Maxwell upper bound for two species (ions and electrons) that solve the stationary Vlasov--Maxwell system \eqref{2speciesVM-steady}. For the stationary system, we consider the following incoming boundary condition \eqref{2species-perturbabsorbing.st} and the perfect conductor boundary conditions \eqref{perfect cond. boundary-perturb.st}. We further assume that the incoming profiles $G_\pm$ satisfy the decay-in-$(x_\parallel,v)$ assumption \eqref{inicon5}. 
By compatibility, we also have the following Neumann type boundary conditions for the rest directions of the fields in the almost everywhere sense: 
\begin{equation}\label{st.boundary.con3}\partial_{x_3}\mathbf{E}_{\textup{st},3} = 4\pi \rho_{\textup{st}},\ \partial_{x_3}\mathbf{B}_{\textup{st},2} = -4\pi J_{1,\textup{st}},\ \textup{and }\ \partial_{x_3}\mathbf{B}_{\textup{st},1} = 4\pi J_{2,\textup{st}}, \textup{ if }x_3=0.\end{equation}

\subsection{Representations of the Stationary Fields}
In order to obtain an optimal decay rate of the stationary magnetic field $\B_{\textup{st}}$, we consider its vector potential $\mathbf{A}_{\textup{st}}$. Since $\B_{\textup{st}}$ solves the stationary Maxwell equations \eqref{2speciesVM-steady} under the perfect conductor boundary condition \ref{perfect cond. boundary-perturb.st}, we have
\begin{equation}\label{curl Bst}
    \begin{split}
          &\nabla_x\times \B_{\textup{st}}= 4\pi J_{\textup{st}}, \quad 
        \nabla_x \cdot \B_{\textup{st}}=0,\quad
        \mathbf{B}_{\textup{st},3} |_{x_3 = 0 } =0.
    \end{split}
\end{equation}
Taking the curl on \eqref{curl Bst} and using the identity $ \nabla \times (\nabla \times D) = - \Delta D + \nabla (\nabla \cdot D)$, we derive that $\B_{\textup{st}}$ satisfies
\begin{align}
- \Delta \B_{\textup{st}} &= 4 \pi \nabla \times J_{\textup{st}},\quad
\nabla \cdot \B_{\textup{st}} =0,\quad
\mathbf{B}_{\textup{st},3}|_{x_3=0} =0,\quad
(\nabla \times  \B_{\textup{st}} ) \times n |_{x_3=0} =  4 \pi J_{\textup{st}} \times n|_{x_3=0}.\notag
\end{align} We introduce a standard well-posedness theorem on its unique solvability of the system above. To this end, we first introduce the following lemma on the equivalence of the divergence-free condition on the field and the existence of its unique vector potential. To begin with, we define
\begin{align}
H_0(\textup{curl}; \Omega) &\eqdef \{ v \in H(\textup{curl}; \Omega): \nabla \cdot v=0, v\cdot n|_{\partial\Omega} =0\} 
  = \left\{ v \in H(\textup{curl}; \Omega):
\int_{\Omega} v \cdot \nabla q dx, \ \forall q \in H^1 (\Omega)
\right\},\notag\\
H_{\textup{tan}}(\textup{curl}; \Omega) &\eqdef \{f \in L^2: \nabla \times f \in L^2, f \times n |_{\partial\Omega} =0\},\notag\\
 H_{\textup{div}}(\textup{curl}; \Omega)&\eqdef \left\{ v \in H_{\textup{tan}}(\textup{curl}; \Omega):  \nabla \cdot v =0 , \int_{\partial\Omega} v \cdot n =0\right\},\notag
\end{align}
where $H(\textup{curl}; \Omega) \eqdef \{f \in L^2: \nabla \times f \in L^2\}$. Indeed $\| \nabla \times v \|_{L^2}$ is a norm of $H_{0}(\textup{curl}; \Omega)$. Now we have the following lemma:
\begin{lemma}[Lemma 1.6 of \cite{MR1142472}]\label{lem.1.6}
Assume that $ \Omega $ is simply connected. Then a function $ \B \in L^2(\Omega) $ satisfies
\begin{align}\notag
    \nabla \cdot \B &= 0 \quad \textup{in } \Omega, \qquad
    \B \cdot n = 0 \quad \textup{on } \partial \Omega,
\end{align}
if and only if there exists a function $ \A \in H_{\textup{tan}}(\textup{curl}; \Omega) $ such that
$
    \B = \nabla \times \A.
$
Moreover, the function $ \A $ is uniquely determined if we assume in addition that $ \A \in H_{\textup{div}}(\textup{curl}; \Omega) $, where
$$
W = \{ v \in H_{\textup{tan}}(\textup{curl}; \Omega): \nabla \cdot v = 0, \quad \int_{\partial\Omega} v \cdot n \, dS = 0 \}.
$$
\end{lemma}

\begin{proof}
If $ \B = \nabla \times \A $ for some $ \A \in H_{\textup{tan}}(\textup{curl}; \Omega) $, then clearly $ \nabla \cdot \B = 0 $ since the divergence of a curl is always zero. Moreover, the boundary condition $ \A \times n |_{\partial \Omega} = 0 $ implies $ \B \cdot n |_{\partial \Omega} = (\nabla \times \A) \cdot n = 0 $, so $ \B $ satisfies the given conditions.

Conversely, suppose $ \B \in L^2(\Omega) $ satisfies $ \nabla \cdot \B = 0 $ and $ \B \cdot n |_{\partial\Omega} = 0 $; i.e., $ \B \in H_{0}(\textup{curl}; \Omega) $, where
$$
H_{0}(\textup{curl}; \Omega) = \{ v \in H(\textup{curl}; \Omega): \nabla \cdot v = 0, \quad v \cdot n |_{\partial\Omega} = 0 \}.
$$
We seek $ \A \in H_{\textup{tan}}(\textup{curl}; \Omega) $ such that $ \B = \nabla \times \A $. By Lemma 1.4 in \cite{MR1142472}, the existence of such $ \A $ follows from the variational formulation:
$$
\int_{\Omega} \nabla \times \A \cdot \nabla \times v \, dx = \int_{\Omega} \B \cdot \nabla \times v \, dx, \quad \forall v \in H_{\textup{tan}}(\textup{curl}; \Omega).
$$
The bilinear form $ (\A, v) \mapsto \int_{\Omega} \nabla \times \A \cdot \nabla \times v \, dx $ is coercive on $ H_{\textup{div}}(\textup{curl}; \Omega) $, ensuring the existence of a unique solution $ \A $ in $ H_{\textup{div}}(\textup{curl}; \Omega) $. Since $ H_{\textup{div}}(\textup{curl}; \Omega) $ is a subspace of $ H_{\textup{tan}}(\textup{curl}; \Omega) $, this establishes the desired existence result.

Thus, (i) and (ii) are equivalent.
\end{proof}
Then using this lemma above, we can state the existence of a unique field $\B_{\textup{st}}$ solving \eqref{curl Bst}:
\begin{theorem}[Existence of $\mathbf{B}_{\textup{st}}$, Theorem 2.2 of \cite{MR1142472}]\label{thm.2.2}
Let $ \Omega $ be a simply connected domain, and let $ J_{\textup{st}} $ be a given steady-state current density. Then, there exists a unique $\mathbf{B}_{\textup{st}} \in H(\textup{curl}; \Omega) $ satisfying \eqref{curl Bst}.
\end{theorem}

\begin{proof}[Sketch of Proof]
We establish the existence and uniqueness of $\mathbf{B}_{\textup{st}}$ in $H(\textup{curl}; \Omega)$. Since $ \nabla \times \mathbf{B}_{\textup{st}} = 4\pi J_{\textup{st}} $, we seek $ \mathbf{B}_{\textup{st}} \in H(\textup{curl}; \Omega) $ as a weak solution of the variational problem:
$$
\int_{\Omega} (\nabla \times \mathbf{B}_{\textup{st}}) \cdot v \, dx = 4\pi \int_{\Omega} J_{\textup{st}} \cdot v \, dx, \quad \forall v \in H_{\textup{tan}}(\textup{curl}; \Omega).
$$
The Lax-Milgram theorem ensures existence since the bilinear form is coercive. Taking the divergence of $ \nabla \times \mathbf{B}_{\textup{st}} = 4\pi J_{\textup{st}} $, we obtain $ \nabla \cdot \mathbf{B}_{\textup{st}} = 0 $ automatically. Since we seek $ \mathbf{B}_{\textup{st}} \in H(\textup{curl}; \Omega) $, and the test functions $ v $ satisfy $ v \times n = 0 $ on $ \partial \Omega $, it follows that $ \mathbf{B}_{\textup{st}} \cdot n = 0 $. 

If two solutions $ \mathbf{B}_1, \mathbf{B}_2 $ satisfy the same equation and boundary conditions, their difference $ \mathbf{B} = \mathbf{B}_1 - \mathbf{B}_2 $ satisfies:
$$
\nabla \times \mathbf{B} = 0, \quad \nabla \cdot \mathbf{B} = 0, \quad \mathbf{B} \cdot n = 0 \textup{ on } \partial\Omega.
$$
By Lemma \ref{lem.1.6}, $ \mathbf{B} \equiv 0 $, proving uniqueness.
\end{proof}
Therefore, Lemma \ref{lem.1.6} and Theorem \ref{thm.2.2} together implies that there exist a unique vector potential $\mathbf{A}_{\textup{st}}$ as follows:
\begin{corollary}[Existence of $ \mathbf{A}_{\textup{st}} $]
Once $ \mathbf{B}_{\textup{st}} $ is obtained from Theorem \ref{thm.2.2}, Lemma \ref{lem.1.6} guarantees the existence of a unique vector potential $ \mathbf{A}_{\textup{st}} $ such that:
\begin{align}\label{eq.poiss.A}
    - \Delta \mathbf{A}_{\textup{st}} = 4 \pi J_{\textup{st}}, \quad 
    \nabla \cdot \mathbf{A}_{\textup{st}} =0, \quad 
    \textup{A}_{\textup{st},1} |_{x_3=0} = 0,\quad 
    \textup{A}_{\textup{st},2} |_{x_3=0} = 0,\quad 
    \int_{\partial\Omega} \textup{A}_{\textup{st},3} |_{x_3=0} \, dx = 0.
\end{align}
\end{corollary}


Note that $\textup{A}_{\textup{st},1}$ and $\textup{A}_{\textup{st},2}$ solve uniquely the 0-Dirichlet boundary conditions and the Poisson equation \eqref{eq.poiss.A}. We will have solution-representations of $\textup{A}_{\textup{st},1}$ and $\textup{A}_{\textup{st},2}$ in the subsequent section below via Green function approaches. Now $\nabla \cdot \mathbf{A}_{\textup{st}} =0$ implies that at the boundary $\textup{A}_{\textup{st},3}$ satisfies a 0-Neumann boundary condition formally. We will write the solution formula of $\textup{A}_{\textup{st},3}$ as well. The last condition in \eqref{eq.poiss.A} also holds as we have $\nabla \cdot \mathbf{A}_{\textup{st}} =0 $ already. In 
 the following subsections, we will show that $\mathbf{A}_{\textup{st}}$ decays as $x_3 \to \infty$, as does its curl $\mathbf{B}_{\textup{st}}= \nabla \times \mathbf{A}_{\textup{st}}$. 
We note that the stationary Maxwell equations \eqref{2speciesVM-steady}$_2$-\eqref{2speciesVM-steady}$_5$ generate Poisson equations for $\E_{\textup{st}}$ and $\B_{\textup{st}}$. We derive the solution representations for them using the Green function for Poisson equations in a half space. 
 \subsubsection{Solution Representations of the Vector Potential $\mathbf{A}_{\textup{st}}$ and $\mathbf{B}_{\textup{st}}$ } We consider each coordinate-component of the vector potential $\mathbf{A}_{\textup{st}}$. First of all, for $i=1,2$ note that the first two components $\textup{A}_{\textup{st},i}$ of the vector potential $\mathbf{A}_{\textup{st}}$ solve \eqref{eq.poiss.A} under the 0-Dirichlet boundary conditions \eqref{eq.poiss.A}. Then by taking the odd extension of the Green function $G(x,y)=\frac{1}{|x-y|}$ for the Poisson equation along $x_3=0,$ we can define $\mathfrak{G}_{\textup{odd}}(x,y) = \frac{1}{|x - y|} - \frac{1}{|x - \bar{y}|}$ and have 
$$
\textup{A}_{\textup{st},i}(x) = \int_{\mathbb{R}^3_+} \mathfrak{G}_{\textup{odd}}(x,y)J_{\textup{st},i}(y) \, dy,
$$
with $\bar{y}=(y_1,y_2,-y_3)^\top.$ On the other hand, since the third component $\textup{A}_{\textup{st},3}$ satisfies the 0-Neumann boundary condition  $\partial_{x_3}\textup{A}_{\textup{st},i}|_{x_3=0}=0,$ on the boundary $x_3=0,$ we take the even extension of the Green function and can define $\mathfrak{G}_{\textup{even}}(x,y) = \frac{1}{|x - y|} + \frac{1}{|x - \bar{y}|}$ to obtain
$$
\textup{A}_{\textup{st},3}(x) = \int_{\mathbb{R}^3_+}\mathfrak{G}_{\textup{even}}(x,y)J_{\textup{st},3}(y) \, dy. 
$$Since $\B_{\textup{st}}=\nabla\times \mathbf{A}_{\textup{st}},$ we obtain that
\begin{equation}\label{Field representation Bst}
    \begin{split}
        \mathbf{B}_{\textup{st},i}(x)  &= (-1)^{i}(\partial_{x_3}\textup{A}_{\textup{st},j}-\partial_{x_j}\textup{A}_{\textup{st},3})(x),\text{ for }i,j=1,2\text{ with }j\neq i,\\
          &=(-1)^{i}\int_{\mathbb{R}^3_+} \partial_{x_3}\mathfrak{G}_{\textup{odd}}(x,y) \int_\rth (\hat{v}_{+,j}F_{+,\textup{st}}(y,v)-\hat{v}_{-,j}F_{-,\textup{st}}(y,v))dv \, dy\\&\qquad - (-1)^{i}  \int_{\mathbb{R}^3_+} \partial_{x_j}\mathfrak{G}_{\textup{even}}(x,y) \int_\rth (\hat{v}_{+,3}F_{+,\textup{st}}(y,v)-\hat{v}_{-,3}F_{-,\textup{st}}(y,v))dv \, dy\\
  \mathbf{B}_{\textup{st},3}(x) &= (\partial_{x_1}\textup{A}_{\textup{st},2}-\partial_{x_2}\textup{A}_{\textup{st},1})(x)\\&= \int_{\mathbb{R}^3_+} \partial_{x_1}\mathfrak{G}_{\textup{odd}}(x,y) \int_\rth (\hat{v}_{+,2}F_{+,\textup{st}}(y,v)-\hat{v}_{-,2}F_{-,\textup{st}}(y,v))dv \, dy\\&
  \qquad-\int_{\mathbb{R}^3_+} \partial_{x_2}\mathfrak{G}_{\textup{odd}}(x,y) \int_\rth (\hat{v}_{+,1}F_{+,\textup{st}}(y,v)-\hat{v}_{-,1}F_{-,\textup{st}}(y,v))dv \, dy.
    \end{split}
\end{equation}
\begin{remark}Note that \eqref{Field representation Bst} satisfies $-\Delta \B_{\textup{st}}=\nabla\times J_{\textup{st}}$, $\partial_{x_3}\mathbf{B}_{\textup{st},i}(x_\parallel,0)=(-1)^{j}4\pi J_{j}$ for $i,j=1,2$ with $j\ne i$, and $\mathbf{B}_{\textup{st},3}(x_\parallel,0)=0$ in the distributional sense. 
\end{remark}
\subsubsection{Solution Representations of $\E_{\textup{st}}$ and its Potential $\phi_{\textup{st}}$}
Since $\E_{\textup{st}}$ solves \eqref{iterated Maxwell.st}$_2$, there is a potential $\phi_{\textup{st}}$ such that $\E_{\textup{st}}= -\nabla_x \phi_{\textup{st}}.$ By \eqref{iterated Maxwell.st}$_3$, we obtain that 
$$
-\Delta \phi_{\textup{st}} = 4\pi \rho_{\textup{st}},
$$for $x_3 \ge 0$. We consider the perfect conductor boundary condition and assume that $\phi=0$ on $x_3=0.$ Then taking the odd extension of the Green function, we have
$$
\phi_{\textup{st}}(x) = \int_{\mathbb{R}^3_+} \mathfrak{G}_{\textup{odd}}(x,y) \rho_{\textup{st}}(y) \, dy,
$$
with $\bar{y}=(y_1,y_2,-y_3)^\top.$ Then by taking the derivative in $x$, we obtain that
\begin{equation}
    \label{Field representation Est}\E_{\textup{st}} = -\nabla_x \phi_{\textup{st}}=-\int_{\mathbb{R}^3_+}\nabla_x  \mathfrak{G}_{\textup{odd}}(x,y)\rho_{\textup{st}}(y) \, dy. \end{equation}

\begin{remark}
Note that \eqref{Field representation Est} gives  $\mathbf{E}_{\textup{st},i}(x_\parallel,0)=0,$ for $i=1,2,$ and $\partial_{x_3}\mathbf{E}_{\textup{st},3}(x_\parallel,0)=4\pi \rho_{\textup{st}}$ in the distributional sense. 
\end{remark}

\subsection{Bootstrap Argument and Uniform $L^\infty$ Estimates}\label{sec.bootstrap.st}
For the nonlinear problem \eqref{2speciesVM-steady}, we consider the sequence of iterated solutions $(F_{\pm,\textup{st}}^l,\E_{\textup{st}}^l,\B_{\textup{st}}^l)$ for any $l\in \mathbb{N}\cup \{0\}.$ Construct the sequence $(F_{\pm,\textup{st}}^l,\E_{\textup{st}}^l,\B_{\textup{st}}^l)$ via the solutions to the following stationary system
\begin{equation}\label{iterated Vlasov.st}
    \begin{split}
        &\hat{v}_\pm \cdot \nabla_x F_{\pm,\textup{st}}^{l+1}\pm \left(\E_{\textup{st}}^l+(\hat{v}_\pm)\times \B_{\textup{st}}^l\mp m_\pm g\hat{e}_3\right)\cdot \nabla_v F_{\pm,\textup{st}}^{l+1}= 0,\\
        &F_{\pm,\textup{st}}^{l+1}(x_\parallel,0,v)|_{v_3>0}=G_\pm(x_\parallel,v),
    \end{split}
\end{equation}and the stationary Maxwell system 
\begin{equation}\label{iterated Maxwell.st}
    \begin{split}
         \nabla_x\times \Bstlbf= 4\pi J^l_{\textup{st}},\ \nabla_x\times \Estlbf=0,\ \
         \nabla_x \cdot \Estlbf=4\pi \rho^l_{\textup{st}},\ \nabla_x \cdot \Bstlbf=0,
        \end{split}
        \end{equation}
where we define  $$\rho^l_{\textup{st}}(x)\eqdef \int_\rth (F^l_{+,\textup{st}}(x,v)-F^l_{+,\textup{st}}(x,v))dv\textup{ and }J^l_{\textup{st}}(x)\eqdef \int_\rth (\hat{v}_+F^l_{+,\textup{st}}(x,v)-\hat{v}_-F^l_{+,\textup{st}}(x,v))dv,$$ and we assume that  $F_{\pm,\textup{st}}^0,\E_{\textup{st}}^0,\B_{\textup{st}}^0\eqdef 0$.  Recall that the boundary profiles $G_\pm$ satisfy the assumption \eqref{inicon5}. 

We consider the iterated stationary  characteristic trajectory variables
$\Zlo(s;x,v)=(\Xlo(s;x,v),\allowbreak \Vlo(s;x,v))$ which solve
\begin{equation}\label{iterated char.st}
 \begin{split}
   \frac{d\Xlo(s)}{ds}&=\Vhatlo(s)=\frac{\Vlo(s)}{\sqrt{m_\pm^2+|\Vlo(s)|^2}},\\
        \frac{d\Vlo(s)}{ds}&=\pm  \Estlbf(s,\Xlo(s))\pm \Vhatlo(s)\times \Bstlbf(s,\Xlo(s))- m_\pm g\hat{e}_3,
 \end{split} 
\end{equation}   where $\hat{e}_3\eqdef (0,0,1)^\top$ and $\hat{v}_\pm\eqdef \frac{v}{\vZ}=\frac{v}{\sqrt{m_\pm^2+|v|^2}}$. 
Iterating the stationary characteristic trajectory \eqref{Back Forw exit time.st}, we define
\begin{equation}\label{iterated Back Forw exit time.st}\begin{split}
   & \tfstlo (x,v)= \sup\{s\in [0,\infty): (\Xst)_3^{l+1}(\tau;x,v)>0\ \textup{ for all } \tau\in(0,s)\}\ge 0,\\
   & \tbstlo (x,v)= \sup\{s\in [0,\infty): (\Xst)_3^{l+1}(-\tau;x,v)>0\ \textup{ for all } \tau\in(0,s)\}\ge 0\\
        &\xbstlo(x, v)=\Xst^{l+1}\left(-\tbstlo(x, v) ; x, v\right) \in \partial \Omega,\\  &\vbstlo( x, v)=\Vst^{l+1}\left(-\tbstlo(x, v) ; x, v\right).
   \end{split}
\end{equation}
As in the solution in the mild form \eqref{solution f} for the dynamical case, we can also write our solution $\fstlo$ in the steady case as
\begin{equation}\label{solution flo.st}
    \fstlo(x,v)= G_\pm((\Xst)_\parallel^{l+1}(-\tbstlo;x,v),V_{\pm}^{l+1}(-\tbstlo;x,v)).
\end{equation}

Now we obtain the following uniform $L^\infty$ estimates for the iterated sequence $(F_{\pm,\textup{st}}^k,\E_{\textup{st}}^k,\B_{\textup{st}}^k)$ with $k\in \mathbb{N}$ via bootstrap argument:
\begin{proposition}\label{prop.st.boot}
    For any $k\in \mathbb{N}$, we have
\begin{equation}
    \label{bootstrap.assump.st}
    \| e^{\frac{\beta}{2} |x_{\parallel}|}e^{\frac{\beta}{2} \vZ } e^{\frac{1}{2}m_\pm g\beta x_3}F_{\pm,\textup{st}}^k(\cdot,\cdot)\|_{L^\infty}\le C ,\textup{ and }
|\E_{\textup{st}}^{k}(x)|,\ |\B_{\textup{st}}^{k}(x)|\le \min\{m_+,m_-\}\frac{g}{16}\frac{1}{\langle x\rangle^2},
\end{equation}for some $C>0$ with $\min\{m_+,m_-\}g\gg 1$ and $\beta>1$.\end{proposition} It is trivial that the solutions are zero and satisfy \eqref{bootstrap.assump.st} when $k=0.$ Assume \eqref{bootstrap.assump.st} holds for $k=l.$ Then we prove that the next sequence element $(F_{\pm,\textup{st}}^{l+1},\E_{\textup{st}}^{l+1},\B_{\textup{st}}^{l+1})$ will satisfy the same upper-bounds \eqref{bootstrap.assump.st}.

\subsubsection{Weighted $L^\infty$ Estimate for the Velocity Distribution $\fstlo$}
Using \eqref{solution flo.st}, we observe that
\begin{multline}\notag
    |\fstlo(x,v)|=|G_\pm((\Xst)_\parallel^{l+1}(-\tbstlo;x,v),\Vlo(-\tbstlo;x,v))|\\
  = \frac{1}{\mathrm{w}_{\pm,\beta}(\Zlo(-\tbstlo ;  x, v))}\|(\mathrm{w}_{\pm,\beta} G_\pm)((\Xst)_\parallel^{l+1}(-\tbstlo;x,v),\Vlo(-\tbstlo;x,v))\|_{L^\infty_{x,v}(\gamma_-)}.
\end{multline}
Using the boundary condition \eqref{inicon5} and the weight comparison \eqref{w comparison 3.st}, we have
\begin{equation}\label{fstlo final est}
    |\fstlo(x,v)|
   \le Ce^{-\frac{1}{2}\beta \vZ} e^{-\frac{1}{2}m_\pm g\beta x_3}e^{-\frac{\beta}{2}|x_{\parallel}|},
\end{equation}where the weight function $\mathrm{w}_{\pm,\beta}$ is defined in \eqref{weights.wholehalf}. This proves the bootstrap assumption \eqref{bootstrap.assump.st} for $\fstlo$.

\subsubsection{$L^\infty $ Estimates for the Steady Fields $\Estlobf$ and $\Bstlobf$}Now, given the estimates \eqref{fstlo final est} for the steady distribution $\fstlo$, we will prove the bootstrap estimates \eqref{bootstrap.assump.st} for the fields $\Estlobf$ and $\Bstlobf$ using the field representations \eqref{Field representation Bst} and \eqref{Field representation Est}.

For $i=1,2,3$, the field components $\mathbf{E}^{l+1}_{\textup{st},i}$ of $\Estlobf$ in \eqref{Field representation Est} solving \eqref{iterated Maxwell.st} satisfy that
$$|\mathbf{E}^{l+1}_{\textup{st},i}(x)|\le   \int_{\mathbb{R}^3_+} \left|\partial_{x_i}  \mathfrak{G}_{\textup{odd}}(x,y)\right|\left| \int_\rth F^{l+1}_{+,\textup{st}}(y,v)dv- \int_\rth F^{l+1}_{-,\textup{st}}(y,v)dv\right| \ dy.$$ Using the estimate \eqref{fstlo final est}, we observe that
\begin{multline}\label{est.Elost mid}
    |\mathbf{E}^{l+1}_{\textup{st},i}(x)| \le \sum_{\pm}2C\int_{\mathbb{R}^3_+} dy\ \left|\partial_{x_i}  \mathfrak{G}_{\textup{odd}}(x,y)\right| e^{-\frac{\beta }{2}|y_\parallel|} e^{-\frac{1}{2}m_\pm g\beta y_3} \int_\rth dv\ e^{-\frac{1}{2}\beta \vZ}\\*
     \lesssim  \sum_{\pm}2\frac{C}{\beta^3}\int_{\mathbb{R}^3_+} dy\ \left|\partial_{x_i}  \mathfrak{G}_{\textup{odd}}(x,y)\right| e^{-\frac{\beta }{2}|y_\parallel|} e^{-\frac{1}{2}m_\pm g\beta y_3} ,
\end{multline}
 where we further used that
\begin{equation}\begin{split}\label{additional beta decay.st}
  \int_{\rth} dv\   e^{-\frac{\beta}{2} \vZ}
     &= \int_{\rth} dv\   e^{-\frac{\beta}{2} \sqrt{m_\pm^2+|v|^2}}
       = 4\pi \int_0^\infty d|v|\  |v|^2e^{-\frac{\beta}{2} \sqrt{m_\pm^2+|v|^2}}
   \\
       &=4\pi \int_{m_\pm}^\infty dz\ z \sqrt{z^2-m^2_\pm} e^{-\frac{\beta}{2} z}\le 4\pi \int_0^\infty dz\  z^2 e^{-\frac{\beta}{2} z}
       =\frac{32\pi}{\beta^3} \int_0^\infty dz'\  z'^2 e^{-z'}\approx \frac{1}{\beta^3}, 
 \end{split}\end{equation}where we made the changes of variables $|v|\mapsto z\eqdef \sqrt{m_\pm^2+|v|^2}$ and $z\mapsto z'\eqdef \frac{\beta}{2}z.$ 
 %
Then since $\left|\partial_{x_i}  \mathfrak{G}_{\textup{odd}}(x,y)\right|\le \frac{1}{|x-y|^2}+\frac{1}{|\bar{x}-y|^2}$ and the upper bound is even in $y_3$, note that 
\begin{multline*}\int_{\mathbb{R}^3_+} dy\ \left|\partial_{x_i}  \mathfrak{G}_{\textup{odd}}(x,y)\right| e^{-\frac{\beta }{2}|y_\parallel|} e^{-\frac{1}{2}m_\pm g\beta y_3}
    \le \int_{\mathbb{R}^3} dy\ \left(\frac{1}{|x-y|^2}+\frac{1}{|\bar{x}-y|^2}\right) e^{-\frac{\beta }{2}|y_\parallel|} e^{-\frac{1}{2}m_\pm g\beta |y_3|}\\
    \lesssim  \frac{1}{m_\pm g\beta^3}\frac{1}{\langle x\rangle^2},
\end{multline*}using the elementary inequality
\begin{equation}
     \label{asymptotics}\int_{\mathbb{R}^3}d z\ \frac{e^{-a|z_\parallel|}e^{-b|z_3|}}{|x_\parallel-z_\parallel|^k+|x_3-z_3|^k}\lesssim \frac{1}{a^2b}\frac{1}{ \langle x\rangle^k},
 \end{equation} for $k<3.$ 
Therefore, in \eqref{est.Elost mid},
choosing $\min\{m_+,m_-\}g\beta^3 \gg 1,$ we have  \begin{equation}
     \label{Est diri final}|\mathbf{E}^{l+1}_{\textup{st},i}(x)|\lesssim \frac{1}{\min\{m_+,m_-\} g\beta^6}\frac{1}{\langle x\rangle^2} \ll \min\{m_+,m_-\}g\frac{1}{\langle x\rangle^2}.
 \end{equation} 
 

 Moreover, since $ \int_\rth |\hat{v}_{+,i}||F_{+,\textup{st}}(y,v)|dv\le \int_\rth |F_{+,\textup{st}}(y,v)|dv,$ $\B^{l+1}_{\textup{st}}(x)$ in \eqref{Field representation Bst} also has the same upper-bound (up to constant) as that of $\E^{l+1}_{\textup{st}}(x)$ and hence 
 $$|\B^{l+1}_{\textup{st}}(x)| \ll \min\{m_+,m_-\}g\frac{1}{\langle x\rangle^2}.$$
Altogether, we have 
\begin{equation}\label{final estimate for flo.st}
|\E_{\textup{st}}^{l+1}(x)|,\  |\B_{\textup{st}}^{l+1}(x)|\le  \min\{m_+,m_-\}\frac{g}{16}\frac{1}{\langle x\rangle^2},
 \end{equation}which closes the bootstrap argument by proving the upper-bounds in  \eqref{bootstrap.assump.st} at the sequential level of $(l+1).$ 
\begin{proof}[Proof of Proposition \ref{prop.st.boot}]
   Proposition \ref{prop.st.boot} now follows by \eqref{fstlo final est} and \eqref{final estimate for flo.st}.
\end{proof}

\subsection{Derivative Estimates}
We can further show that the stationary solution satisfies the following regularity estimates at the sequential level. We first define the following kinetic weight functions:
\begin{definition}
    \begin{equation}
\label{alpha.s.st.def}\tilde{\alpha}_{\pm,\textup{st}}(x,v)\eqdef\sqrt{\frac{\alpha_{\pm,\textup{st}}^2(x,v)}{1+\alpha_{\pm,\textup{st}}^2(x,v)}},
  \end{equation}where $\alpha_{\pm,\textup{st}}$ is defined as
\begin{equation}\label{alpha.st}
    \alpha_{\pm,\textup{st}}(x,v)= \sqrt{x_3^2+\left|(\hat{v}_\pm)_3\right|^2-2 \left((\FS)_3(x_\parallel,0,v)\right)\frac{x_3}{(\vZ)}
}.
\end{equation}
with $(\FS)_{\textup{st}}\eqdef \pm\Estlbf\pm \hat{v}_\pm\times \Bstlbf-m_\pm g \hat{e}_3 .$ 
\end{definition}

Then we have the following derivative estimates associated to the kinetic weight $\tilde{\alpha}_{\pm,\textup{st}}$:
\begin{proposition}\label{prop.deri.st}
    Fix $m>4$ and $R>0.$ 
    Suppose that the boundary data  $G_\pm$ satisfy \begin{equation}\label{stationary G condition}
        \| (\vZ)^m \nabla_{x_\parallel} G_\pm\|_{L^\infty_{x_\parallel,v}}
      +\| (\vZ)^m\nabla_v G_\pm\|_{L^\infty_{x_\parallel,v}}<\infty.
    \end{equation} Consider the corresponding solution sequence $(\fstl,\Estlbf,\Bstlbf)_{l\in\mathbb{N}}$ of \eqref{iterated Vlasov.st}--\eqref{iterated char.st} associated to the boundary data $G_\pm$. Fix any arbitrary $l\in \mathbb{N}$. Define \begin{equation}
        \label{eq.forcing.st}(\FS)_{\textup{st}}\eqdef \pm\Estlbf\pm \hat{v}_\pm\times \Bstlbf-m_\pm g \hat{e}_3 .
    \end{equation}
Suppose that \begin{equation}
    \label{Fl C2 bound.st}\|\nabla_x (\Estlbf,\Bstlbf)\|_{L^\infty}<C_1\text{ and } 
   \|(\FS)_{\textup{st}}\|_{L^\infty}<C_2,
\end{equation} for some $C_1>0$ and $C_2>0.$ 
    Define $\Omega_R = \rth\times [0,R].$ Then
    \begin{equation}\label{fstlo bound}
        \| (\vZ) ^m  \nabla_{x_\parallel}\fstlo\|_{L^\infty(\Omega_R\times \rth)}+\left\|(\vZ)^m \tilde{\alpha}_{\pm,\textup{st}} \partial_{x_3}\fstlo\right\|_{L^\infty(\Omega_R\times \rth)}+\| (\vZ) ^m  \nabla_{v}\fstlo\|_{L^\infty(\Omega_R\times \rth)}\le C_R,
    \end{equation} for some constant $C_R>0$ which depends only on $R$, $C_1, C_2$ and $G_\pm$. 
   Suppose that $-(\FS)_{\textup{st},3}(x_\parallel,0,v)>c_0,$ for some $c_0>0.$ Moreover, the following estimates hold:
   \begin{multline}\label{EBst bound final}
       \|(\E^{l+1}_{\textup{st}},\B^{l+1}_{\textup{st}})\|_{W^{1,\infty}_{x}( \Omega)}\lesssim \|(\vZ)^m F^{l+1}_{\pm,\textup{st}}\|_{L^\infty_{x,v}(\Omega\times \rth)}\\+\|(\vZ)^m\nabla_{x_\parallel}F^{l+1}_{\pm,\textup{st}}\|_{L^\infty_{x,v}( \Omega\times \rth)}
       +\|(\vZ)^m\tilde{\alpha}_{\pm,\textup{st}}(x,v)\partial_{x_3}F^{l+1}_{\pm,\textup{st}}\|_{L^\infty_{x,v}(\Omega\times \rth)},
   \end{multline} where the weight $\tilde{\alpha}_{\pm,\textup{st}}$ is defined as in \eqref{alpha.s.st.def}.
\end{proposition}

\begin{remark}
The constant $C_R$ remains finite on each finite slab $x_3 \in [0,R]$.  
Once the estimates are established on $[0,R]$, they can be extended to $[R,2R]$ by redefining the inflow boundary data at $x_3=R$ using the solution values there. Note that this new inflow data also satisfy \eqref{stationary G condition} by \eqref{fstlo bound}. Iterating this continuation procedure covers all intervals $[kR,(k+1)R]$, $k\in\mathbb{N}$, and thus yields the desired global regularity estimates on $x_3 \in [0,\infty)$.

We note that the derivative estimate for the distribution \eqref{fstlo bound} is uniform in $l$ and hence the derivative estimate for the fields \eqref{EBst bound final} is also uniform in $l$ by \eqref{fstlo final est}. Hence those bounds are even preserved when we pass to the limit $l\to \infty.$
\end{remark}

For the proof of the proposition, we collect several lemmas on the kinetic weight $\tilde{\alpha}_{\pm,\textup{st}}$ including the velocity lemma (Lemma \ref{lem.velo.st}) originally established by Guo \cite{MR1354697}.
\begin{lemma}[Velocity Lemma]
    \label{lem.velo.st}Let $\alpha_{\pm,\textup{st}}$ and $\tilde{\alpha}_{\pm,\textup{st}}$ be defined as in \eqref{alpha.st} and \eqref{alpha.s.st.def}, respectively. Define $(\FS)_{\textup{st}}$ as \eqref{eq.forcing.st}. Suppose $$
     \|\Elbf_{\textup{st}}\|_{L^\infty}+ \|\Blbf_{\textup{st}}\|_{L^\infty} + \|\nabla_x(\FS)_{\textup{st}}\|_{L^\infty} <C.
   $$ Suppose that for all $x_\parallel\in 
   \mathbb{R}^2
   ,$ $-(\FS)_{3,\textup{st}}(x_\parallel,0)>c_0,$  for some $c_0>0.$ Then for any $(x,v)\in \Omega\times \rth,$ with the trajectory $\Xlo(s;x,v)$ and $\Vlo(s;x,v)$ satisfying \eqref{iterated char.st}, 
   \begin{equation}
       \label{alpha.s.lem.st}
       e^{-10\frac{C}{c_0}|s|}\tilde{\alpha}_{\pm,\textup{st}}(x,v)\le \tilde{\alpha}_{\pm,\textup{st}}(s,\Xlo(s;x,v),\Vlo(s;x,v))\le e^{10\frac{C}{c_0}|s|}\tilde{\alpha}_{\pm,\textup{st}}(x,v)
   \end{equation}In addition, regarding the stationary material derivative
    $\frac{D}{Ds}\eqdef (\hat{V}_\pm(s))\cdot \nabla_x +(\FS)_{\textup{st}}(X^l_\pm(s))\cdot \nabla_v,$ we have 
    \begin{equation}
        \label{bound of DalphaDt.st}
    \left|\frac{D}{Ds}\alpha_\pm^2(s)
    \right|\le 20\frac{C}{c_0}\alpha_\pm^2(s).\end{equation}
\end{lemma}\begin{proof}
    We first observe that
    $$\frac{D}{Ds}\tilde{\alpha}_{\pm,\textup{st}}^2=\frac{1}{1+\alpha_{\pm,\textup{st}}^2}\frac{D}{Ds}\alpha_{\pm,\textup{st}}^2-\frac{\alpha_{\pm,\textup{st}}^2}{(1+\alpha_{\pm,\textup{st}}^2)^2}\frac{D}{Ds}\alpha_{\pm,\textup{st}}^2=\frac{1}{(1+\alpha_{\pm,\textup{st}}^2)^2}\frac{D}{Ds}\alpha_{\pm,\textup{st}}^2.$$Then using the bound \eqref{bound of DalphaDt.st} of the material derivative $\frac{D}{Ds}\alpha_{\pm,\textup{st}}^2$ we further obtain
    $$\frac{D}{Ds}\tilde{\alpha}_{\pm,\textup{st}}^2\le 20\frac{C}{c_0(1+\alpha_{\pm,\textup{st}}^2)}\frac{\alpha_{\pm,\textup{st}}^2}{1+\alpha_{\pm,\textup{st}}^2}\le 20\frac{C}{c_0}\tilde{\alpha}_{\pm,\textup{st}}^2.$$ By the Gr\"onwall lemma, we finally obtain $$\tilde{\alpha}_{\pm,\textup{st}}^2(s,\Xlo(s),\Vlo(s))\le e^{\frac{20C}{c_0}|s|}\tilde{\alpha}_{\pm,\textup{st}}^2(x,v).$$
    This completes the proof of Lemma \ref{lem.velo.st}.  Lastly, the proof of \eqref{bound of DalphaDt.st} follows by \cite[Eq. (4.10)]{MR4645724}.
\end{proof}
We also record the following upper bound on the singularity $\frac{1}{|\hat{V}^{l+1}_{\pm,3}|}$:
    \begin{lemma}[Lemma 10 of \cite{MR4645724}]
\label{lemma.tb over v3.st}
     For $(x,v)\in\Omega\times \rth,$ let the trajectory $\Xlo(s;x,v)$ and $\Vlo(s;x,v)$ satisfy \eqref{iterated char.st}.
    Suppose for all $x,v$, $-(\FS)_{3,\textup{st}}(x_\parallel,0,v)>c_0,$  then there exists a constant $C$ depending on $g,$ $\|\E_{\textup{st}}^l\|_{W^{1,\infty}( \Omega)},$ and $\|\B_{\textup{st}}^l\|_{W^{1,\infty}(\Omega)},$ such that 
    \begin{equation}\label{V3V0 bound max.st}
        \frac{\tbstlo(x,v)}{(\hat{V}_\pm^{l+1})_3(-\tbstlo)}\le \frac{C}{c_0}\max_{s\in\{-\tbstlo,0\}}\sqrt{m_\pm^2+|\Vlo(s)|^2}.
    \end{equation} 
 \end{lemma}
\begin{proof}[Proof of Proposition \ref{prop.deri.st}]
      Fix $m>4$. By differentiating the stationary Vlasov equation \eqref{iterated Vlasov.st} with respect to $x_\parallel$, we observe that $ (\vZ) ^m|\nabla_{x_\parallel} \fstlo|$ is bounded from above by  
    \begin{align*}
          &(\vZ)^m|\nabla_{x_\parallel} \fstlo(x,v)|\\
          &\le  (\vZ)^m\bigg|  (\nabla_{x_\parallel} G_\pm)( (\xbstlo)_\parallel, \vbstlo) \cdot \nabla_{x_\parallel} (\xbstlo)_\parallel+ (\nabla_v G_\pm)( (\xbstlo)_\parallel, \vbstlo) \cdot \nabla_{x_\parallel} \vbstlo\bigg| \\
&\lesssim    (\vZ)^m |(\nabla_{x_\parallel} G_\pm)( (\xbstlo)_\parallel , \vbstlo)||\nabla_{x_\parallel} (\xbstlo)_\parallel|+  (\vZ)^m |(\nabla_v G_\pm)( (\xbstlo)_\parallel, \vbstlo)||\nabla_{x_\parallel} \vbstlo|.
    \end{align*} 
In general, note that for $-\tbstlo\le s\le \tfstlo,$
\begin{equation}\label{additional v bound.st}
        (\vZ) \lesssim  \langle\Vlo(s)\rangle+\bigg|\int_s^0d\tau\  (\FS)_{\textup{st}}(\XSlo(\tau),\Vlo(\tau))\bigg|\lesssim\langle \Vlo(s)\rangle +C_2|s| ,
    \end{equation} by \eqref{Fl C2 bound.st}. Also recall \eqref{exit time bound.st} that we have for $x_3\in[0,R],$
\begin{equation}\label{tbst bound for R} \tbst+\tfst\lesssim C_2(\sqrt{m_\pm^2+|v_\pm|^2}+m_\pm gx_3)\lesssim \vZ+R,\end{equation} under \eqref{Fl C2 bound.st}.
   On the other hand, given \eqref{Fl C2 bound.st}, note that the derivatives of $\xbstlo$ and $\vbstlo$ satisfy the same upper-bounds estimates \eqref{derivatives of tbxbvb} with the dynamical trajectory variables $\ZS^{l+1}=(\XSlo,\VSlo)$ and the variables $(t,x,v)$ now replaced by the stationary variables $Z^{l+1}=(X^{l+1}_\pm,V^{l+1}_\pm)$ and $(0,x,v)$, respectively.  Thus, using the stationary counterparts of \eqref{derivatives of tbxbvb}--\eqref{derivatives of tbxbvb better}, we obtain
    \begin{multline*}
 (\vZ)^m|\nabla_{x_\parallel} \fstlo(x,v)|
\lesssim C\bigg(  (\vZ)^m |(\nabla_{x_\parallel} G_\pm)((\xbstlo)_\parallel, \vbstlo)|\bigg|\frac{\tbstlo}{|\hat{V}_\pm^{l+1}(-\tbstlo)|(\vZ) }+1\bigg|\\
       +  (\vZ)^m |(\nabla_v G_\pm)( (\xbstlo)_\parallel, \vbstlo)|\bigg|\frac{\tbstlo}{|\hat{V}_\pm^{l+1}(-\tbstlo)|(\vZ) }+1\bigg|\bigg).
    \end{multline*}
   By Lemma \ref{lemma.tb over v3.st} and \eqref{tbst bound for R}, we further observe that
    \begin{multline}\label{tb over v3 bound.st}
        \bigg|\frac{\tbstlo}{|(\hat{V}_\pm^{l+1})_3(-\tbstlo)|(\vZ) }\bigg|\le \frac{C}{c_0}\frac{\max_{s\in\{-\tbstlo,0\}}\sqrt{m_\pm^2+|\Vlo(s)|^2}}{(\vZ)}\\\lesssim \frac{C}{c_0}\sup_{-\tbstlo<s<0}\left(1+ \frac{1}{(\vZ)}\bigg|\int_s^0 (\FS)_{\textup{st}}(\XSlo(\tau),\Vlo(\tau))d\tau\bigg|\right)\lesssim \frac{C}{c_0}\left(1+\frac{C_2\tbstlo}{(\vZ)}\right)\lesssim C_R.
    \end{multline}Thus we conclude that
    \begin{equation*}
\| (\vZ)^m\nabla_{x_\parallel} \fstlo\|_{L^\infty_{x,v}(\Omega_R\times \rth)}
\lesssim C_R \bigg(\| (\vZ)^m \nabla_{x_\parallel} G_\pm\|_{L^\infty_{x_\parallel,v}}
      +\| (\vZ)^m \nabla_v G_\pm\|_{L^\infty_{x_\parallel,v}}\bigg).
    \end{equation*}

Regarding the derivative $\partial_{x_3}\fstlo,$ we differentiate the Vlasov equation \eqref{iterated Vlasov.st} with respect to $x_3$ and obtain
\begin{align*}
        &(\vZ)^m\tilde{\alpha}_{\pm,\textup{st}}(x,v)  |\partial_{x_3}\fstlo(x,v)|\\ & \le (\vZ)^m\tilde{\alpha}_{\pm,\textup{st}}(x,v)\bigg| (\nabla_{x_\parallel} G_\pm)( (\xbstlo)_\parallel, \vbstlo) \cdot \partial_{x_3} (\xbstlo)_\parallel+ (\nabla_v G_\pm)( (\xbstlo)_\parallel, \vbstlo) \cdot \partial_{x_3} \vbstlo\bigg| \\
&\lesssim   (\vZ)^m\tilde{\alpha}_{\pm,\textup{st}}(x,v)|(\nabla_{x_\parallel} G_\pm)( (\xbstlo)_\parallel, \vbstlo)||\partial_{x_3} (\xbstlo)_\parallel|\\
&+   (\vZ)^m\tilde{\alpha}_{\pm,\textup{st}}(x,v)|(\nabla_v G_\pm)( (\xbstlo)_\parallel, \vbstlo)||\partial_{x_3}\vbstlo|.
    \end{align*}
  Again, by the stationary counterpart of \eqref{derivatives of tbxbvb}, we have
\begin{align*}
   &(\vZ)^m\tilde{\alpha}_{\pm,\textup{st}}(x,v)  |(\nabla_{x_\parallel} G_\pm)( (\xbstlo)_\parallel, \vbstlo)||\partial_{x_3} (\xbstlo)_\parallel|\\
       &+(\vZ)^m\tilde{\alpha}_{\pm,\textup{st}}(x,v) |(\nabla_v G_\pm)( (\xbstlo)_\parallel, \vbstlo)||\partial_{x_3}\vbstlo|\\
       &\lesssim C\bigg( (\vZ)^m |(\nabla_{x_\parallel} G_\pm)((\xbstlo)_\parallel, \vbstlo)|\tilde{\alpha}_{\pm,\textup{st}}(x,v)\left|\frac{1}{|(\hat{V}_\pm^{l+1})_3(-\tbstlo)|} +\frac{1}{\langle v \rangle}\right|\\
       &+ (\vZ)^m |(\nabla_v G_\pm)( (\xbstlo)_\parallel, \vbstlo)|\tilde{\alpha}_{\pm,\textup{st}}(x,v)\left|\frac{1}{|(\hat{V}_\pm^{l+1})_3(-\tbstlo)|} +1\right|\bigg).
    \end{align*}     By using Lemma \ref{lem.velo.st}, \eqref{additional v bound.st}, and \eqref{tbst bound for R}  with $s=-\tbstlo$, we conclude that
   $$ \|(\vZ)^m\tilde{\alpha}_{\pm} \partial_{x_3}  \fstlo\|_{L^\infty_{x,v}(\Omega_R\times \rth)}
\lesssim  C_R\bigg( \| (\vZ)^m \nabla_{x_\parallel} G_\pm\|_{L^\infty_{x_\parallel,v}}
      +\| (\vZ)^m\nabla_v G_\pm\|_{L^\infty_{x_\parallel,v}}\bigg).$$

      Regarding the momentum derivative $|\nabla_v\fstlo|,$ we differentiate \eqref{iterated Vlasov.st}  with respect to $v$ and obtain
    \begin{align*}
         & (\vZ)^m|\nabla_v \fstlo(x,v)|\\
         & \le (\vZ)^m\bigg|(\nabla_{x_\parallel} G_\pm)( (\xbstlo)_\parallel, \vbstlo) \cdot \nabla_v  (\xbstlo)_\parallel+ (\nabla_v G_\pm)( (\xbstlo)_\parallel, \vbstlo) \cdot \nabla_v  \vbstlo\bigg| \\
&\lesssim    (\vZ)^m |(\nabla_{x_\parallel} G_\pm)( (\xbstlo)_\parallel, \vbstlo)||\nabla_v  (\xbstlo)_\parallel|+  (\vZ)^m |(\nabla_v G_\pm)( (\xbstlo)_\parallel, \vbstlo)||\nabla_v  \vbstlo|.
    \end{align*} 
Using the stationary counterpart of \eqref{derivatives of tbxbvb}, we obtain
    \begin{multline*}
 (\vZ)^m|\nabla_v  \fstlo(x,v)|
\lesssim  C_T(\vZ)^{m-1}\bigg(   |(\nabla_{x_\parallel} G_\pm)((\xbstlo)_\parallel, \vbstlo)|\frac{\tbstlo}{|\hat{\mathcal{V}}_\pm^{l+1}(t-\tbstlo)| }\\
       +   |(\nabla_v G_\pm)( (\xbstlo)_\parallel, \vbstlo)|\bigg|\frac{\tbstlo}{|\hat{\mathcal{V}}_\pm^{l+1}(t-\tbstlo)|(\vZ) }+1\bigg|\bigg).
    \end{multline*}
    By using \eqref{tb over v3 bound.st} and \eqref{additional v bound.st} with $s=-\tbstlo$, we conclude that
   $$ \| (\vZ)^m\nabla_v  \fstlo\|_{L^\infty_{x,v}(\Omega_R\times \rth)}\lesssim C_R\bigg(  \| (\vZ)^m \nabla_{x_\parallel} G_\pm\|_{L^\infty_{x_\parallel,v}}
      +\| (\vZ)^{m-1}\nabla_v G_\pm\|_{L^\infty_{x_\parallel,v}}\bigg).$$

Lastly, concerning the derivatives of the stationary fields $\E_{\textup{st}}, \B_{\textup{st}}$, the arguments of Lemma~\ref{lem.normal deri} and Lemma~\ref{lem.tangential deri}, stated in the dynamical case, extend to the stationary case with only minor modifications. For brevity, we omit the proof.

\end{proof}

\subsubsection{Enhanced Decay Estimates for $|\nabla_v F_{\pm,\textup{st}}|$}\label{sec.deri decay.st}
In this subsection, we further obtain enhanced decay estimates for $|\nabla_v F_{\pm,\textup{st}}|$ given that the incoming boundary profile $G_\pm$ further satisfies the following fast-decay condition on the first-order derivative in the velocity variable.
\begin{proposition}[Momentum Derivatives]\label{prop.enhanced decay}
  Suppose \eqref{Fl C2 bound.st} holds.  Suppose that $G_\pm$ satisfy
\begin{equation}
    \label{decay condition for g- deri}
    \|\mathrm{w}^2_{\pm,\beta}(\cdot,0,\cdot)\nabla_{x_\parallel,v}G_\pm(\cdot,\cdot)\|_{L^\infty_{x_\parallel,v}}<\infty.
\end{equation}
Then for each $l\in\mathbb{N},$ we have  \begin{equation}\label{decay of Fst}\|\mathrm{w}_{\pm,\beta}\nabla_v F^l_{\textup{st}}\|_{L^\infty_{x,v}}\le C\|\mathrm{w}^2_{\pm,\beta}(\cdot,0,\cdot)\nabla_{x_\parallel,v}G_\pm(\cdot,\cdot)\|_{L^\infty_{x_\parallel,v}}, \end{equation}  for some $C>0$. 
\end{proposition}
Note that \eqref{decay of Fst} is uniform in $l$ and is preserved when we pass to the limit $l\to\infty.$

\begin{proof}[Proof for Proposition \ref{prop.enhanced decay}]Fix $l\in\mathbb{N}.$ By Proposition \ref{prop.deri.st}, we have for some $C_1>0$ and $C_2>0$, $$ \|\nabla_x (\Estlbf,\Bstlbf)\|_{L^\infty}<C_1,\text{ and }
   \|(\FS)_{\textup{st}}\|_{L^\infty}<C_2.
$$  
By taking the momentum derivative on \eqref{solution flo.st}, we obtain
$$ \nabla_v \fstlo(x, v) = (\nabla_{x_\parallel} G_\pm)( (\xbstlo)_\parallel, \vbstlo) \cdot \nabla_v (\xbstlo)_\parallel+ (\nabla_v G_\pm)( (\xbstlo)_\parallel, \vbstlo) \cdot \nabla_v \vbstlo,$$
   where the stationary backward exit position and velocity $\xbstlo$ and $\vbstlo$ are defined as 
\begin{equation}\label{xbvb.st}
    \begin{split}
        \xbstlo(x, v)=X^{l+1}\left(-\tbstlo(x, v) ; x, v\right) \in \partial \Omega,\ \   \vbstlo(x, v)=V^{l+1}\left(-\tbstlo(x, v) ;x, v\right).
    \end{split}
\end{equation}Then, given \eqref{Fl C2 bound.st}, we note that the derivatives of $\xbstlo$ and $\vbstlo$ satisfy the same upper-bounds estimates \eqref{derivatives of tbxbvb} with the dynamical trajectory variables $\ZS^{l+1}=(\XSlo,\VSlo)$ the variables $(t,x,v)$ now replaced by the stationary variables $Z^{l+1}=(X^{l+1}_\pm,V^{l+1}_\pm)$ and $(0,x,v)$, respectively. Therefore, using \eqref{derivatives of tbxbvb}, we observe that
  \begin{align*}
    |\mathrm{w}_{\pm,\beta}\nabla_v \fstlo(x, v)| &\le
      \mathrm{w}_{\pm,\beta}(x,v) |(\nabla_{x_\parallel} G_\pm)((\xbstlo)_\parallel, \vbstlo)||\nabla_v  (\xbstlo)_\parallel|\\*
      &+\mathrm{w}_{\pm,\beta}(x,v) |(\nabla_v G_\pm)( (\xbstlo)_\parallel, \vbstlo)||\nabla_v  \vbstlo|\\
&\lesssim
       \mathrm{w}_{\pm,\beta}(x,v)\bigg( |(\nabla_{x_\parallel} G_\pm)((\xbstlo)_\parallel, \vbstlo)\left(\frac{\tbstlo}{|(\hat{V}^{l+1}_\pm)_3(-\tbstlo)|(\vZ) }+\frac{\tbstlo}{\vZ}\right)\\*
      & +  |(\nabla_v G_\pm)((\xbstlo)_\parallel, \vbstlo)|\bigg|\frac{\tbstlo}{|(\hat{V}^{l+1}_\pm)_3(-\tbstlo)|(\vZ) }+(\vZ)^{-1}\bigg|\bigg)\\
     &  \lesssim   \frac{1}{\mathrm{w}_{\pm,\beta}(x,v)}\left(\frac{\mathrm{w}_{\pm,\beta}(x,v)}{\mathrm{w}_{\pm,\beta}(\xbstlo(x, v),\vbstlo(x, v))}\right)^2\bigg( |(\mathrm{w}^2_{\pm,\beta}\nabla_{x_\parallel} G_\pm)((\xbstlo)_\parallel, \vb)|\\*
      & +  |(\mathrm{w}^2_{\pm,\beta}\nabla_v G_\pm)((\xbstlo)_\parallel, \vbstlo)|\bigg),    \end{align*} by Lemma \ref{lemma.tb over v3}. Then we further use the weight comparison \eqref{w comparison 1.st} and observe that 
\begin{align*}
    \frac{1}{\mathrm{w}_{\pm,\beta}(x,v)}&\left(\frac{\mathrm{w}_{\pm,\beta}\left( Z^{l+1}_\pm\left(0 ;  x, v\right)\right)}{\mathrm{w}_{\pm,\beta}( Z^{l+1}_\pm(-\tbstlo(x, v);  x, v))} \right)^2
    \leq  \frac{1}{\mathrm{w}_{\pm,\beta}(x,v)}e^{\left(\left\|\Elbf_{\textup{st}}\right\|_{L_{ x}^{\infty}} +1\right)\frac{64\beta}{13m_\pm g} (\sqrt{m_\pm^2+|v_\pm|^2}+m_\pm gx_3)}\\
  &\le   e^{(\min\{m_-,m_+\}g)\left(\frac{1}{16}+\frac{1}{8}\right)\frac{64\beta}{13m_\pm g} (\sqrt{m_\pm^2+|v_\pm|^2}+m_\pm gx_3)}e^{-\beta v_\pm^0-m_\pm g\beta x_3-\frac{\beta}{2}|x_{\parallel}|}\\
   &\le  e^{\frac{12\beta}{13} (\sqrt{m_\pm^2+|v_\pm|^2}+m_\pm gx_3)}e^{-\beta v_\pm^0-m_\pm g\beta x_3-\frac{\beta}{2}|x_{\parallel}|}
   \le e^{-\frac{1}{13}\beta v_\pm^0-\frac{1}{13}m_\pm g\beta x_3-\frac{\beta}{2}|x_{\parallel}|}\le 1,
\end{align*}given that $\Elbf_{\textup{st}}$ satisfies the upper-bound \eqref{bootstrap.assump.st} and that 
$\min\{m_+,m_-\}g\gg 1.$ This completes the proof.\end{proof}

\subsection{Stability and Construction of Solutions}

Given the uniform estimates for the iterated sequence elements of steady states $(F_{\pm,\textup{st}}^l,\E_{\textup{st}}^l,\B_{\textup{st}}^l)$ and the enhanced decay estimates on the momentum derivatives $\nabla_v F_{\pm,\textup{st}}^l$, we can now prove the stability of the sequence which yields Cauchy property of the sequences.  Then we will obtain the strong convergence to the limit $(F_{\pm,\textup{st}},\E_{\textup{st}},\B_{\textup{st}})$. This is necessary to pass to the limit on the nonlinear terms.  Fix $N_0\in \mathbb{N}.$ Then for any $k,n\ge N_0$ integers with $k\ge n$, we have \begin{equation}\label{zero boundary cauchy.st}\begin{split}
(\fmst-\fnst) (x_\parallel,0,v)|_{\gamma_-}=0,
        \end{split}
\end{equation}
and 
   \begin{multline*}
  (\hat{v}_\pm) \cdot \nabla_x (\fmst-\fnst) +\left(\pm\Emmbfst\pm(\hat{v}_\pm)\times \Bmmbfst-m_\pm g\hat{e}_3\right)\cdot \nabla_v (\fmst-\fnst)\\
    = -\left(\pm(\Emmbfst-\Enmbfst)\pm(\hat{v}_\pm)\times (\Bmmbfst-\Bnmbfst)\right)\cdot \nabla_v \fnst,
   \end{multline*} 
by \eqref{iterated Vlasov.st}. By \eqref{zero boundary cauchy.st}, we have
\begin{multline}\notag 
    (\fmst-\fnst)(x,v)
    = \mp \int^0_{-\tbstk} \left((\Emmbf-\Enmbf)(\Xk(s))+\hat{V}^{k}_\pm(s)\times (\Bmmbf-\Bnmbf)(\Xk(s))\right)\\\cdot \nabla_v \fnst(\Xk(s),\Vk(s))ds,
\end{multline}using the iterated stationary characteristic trajectories $(\Xk,\Vk)$ in \eqref{iterated char.st}. 
Therefore, we have \begin{multline}\label{eq.F.diff.st.cauchy}
   | (\fmst-\fnst)(x,v)|\\
    \le \tbstk \sup_{s\in [-\tbstk,0]}|( \nabla_v \fnst)(\Xst(s),\Vst(s))| \left(\| (\Emmbfst-\Enmbfst)(\cdot)\|_{L^\infty_x}+\| (\Bmmbfst-\Bnmbfst)(\cdot)\|_{L^\infty_x}\right).
\end{multline}Indeed, given that $\|(\Emmbfst,\Bmmbfst)\|_{L^\infty}\le \min\{m_+,m_-\}\frac{g}{16}$ holds by the previous uniform estimates,  we have $$\tbstk \le \frac{3}{m_\pm g} (v^0_\pm +m_\pm g x_3)$$ by \eqref{exit time bound.st}.  By using the uniform estimate \eqref{decay of Fst} on the momentum derivative $\nabla_v \fnst$ and the weight comparison estimate \eqref{w comparison 3.st}, we have \begin{multline}\label{eq.F.diff.st.cauchy2}
   e^{\frac{1}{2}\beta|x_\parallel|}e^{\frac{1}{4}\beta\left(\vZ+m_\pm g x_{3}\right)}| (\fmst-\fnst)(x,v)|\\
    \lesssim \frac{1}{\beta m_\pm g} \|\mathrm{w}^2_{\pm,\beta}(\cdot,0,\cdot)\nabla_{x_\parallel,v}G_\pm(\cdot,\cdot)\|_{L^\infty_{x_\parallel,v}}  \left(\| (\Emmbfst-\Enmbfst)(\cdot)\|_{L^\infty_x}+\| (\Bmmbfst-\Bnmbfst)(\cdot)\|_{L^\infty_x}\right).
\end{multline}
Since the representations \eqref{Field representation Bst} and \eqref{Field representation Est} for $\Emmbfst$, $\Enmbfst$, $\Bmmbfst$, and $\Bnmbfst$ are linear in $F_{\pm,\textup{st}}$, the differences $\Emmbfst - \Enmbfst$ and $\Bmmbfst - \Bnmbfst$ can be expressed in the same form, with $F_{\pm,\textup{st}}$ replaced by $\fmst - \fnst$.
Therefore, we have
$$|\Emmbfst(x) - \Enmbfst(x)|\lesssim   \int_{\mathbb{R}^3_+} \left|\nabla  \mathfrak{G}_{\textup{odd}}(x,y)\right|\left| \sum_{\iota=\pm}\iota\int_\rth (\fmmstio - \fnmstio)(y,v)dv\right| \ dy.$$ Then we further observe that 
\begin{multline}\notag
|\Emmbfst(x) - \Enmbfst(x)|\\
     \lesssim\sum_{\iota=\pm} \frac{1}{\beta^3}\left(\sup_{x,v}
e^{\frac{1}{2}\beta|x_\parallel|} e^{\frac{1}{4}\beta\left(\vZio + m_\iota g x_3\right)} |(\fmmstio - \fnmstio)(x, v)|\right)\int_{\mathbb{R}^3_+} dy\  \left|\nabla  \mathfrak{G}_{\textup{odd}}(x,y)\right| e^{-\frac{\beta }{2}|y_\parallel|} e^{-\frac{1}{2}m_\iota g\beta y_3} ,
\end{multline}by \eqref{additional beta decay.st}. Then since $\left|\nabla  \mathfrak{G}_{\textup{odd}}(x,y)\right|\le \frac{1}{|x-y|^2}+\frac{1}{|\bar{x}-y|^2}$ and the upper bound is even in $y_3$, note that 
\begin{multline*}
    \int_{\mathbb{R}^3_+} dy\ \left|\nabla  \mathfrak{G}_{\textup{odd}}(x,y)\right| e^{-\frac{\beta }{2}|y_\parallel|} e^{-\frac{1}{2}m_\pm g\beta y_3}
    \le \int_{\mathbb{R}^3} dy\ \left(\frac{1}{|x-y|^2}+\frac{1}{|\bar{x}-y|^2}\right) e^{-\frac{\beta }{2}|y_\parallel|} e^{-\frac{1}{2}m_\pm g\beta |y_3|}\\
    \lesssim  \frac{1}{m_\pm g\beta^3}\frac{1}{\langle x\rangle^2},
\end{multline*}by \eqref{asymptotics}. Since $ \int_\rth |\hat{v}_{\pm,i}||F_{\pm,\textup{st}}(y,v)|dv\le \int_\rth |F_{\pm,\textup{st}}(y,v)|dv,$ we also expect the same upper bound for the difference $|\Bmmbfst(x) - \Bnmbfst(x)|$.  Therefore, we conclude that   \begin{multline}
     \label{Est diff final.cauchy}|\Emmbfst(x) - \Enmbfst(x)|,|\Bmmbfst(x) - \Bnmbfst(x)|\\\lesssim \frac{1}{\min\{m_+,m_-\} g\beta^6}\max_{\iota=\pm}\left(\sup_{x,v}
e^{\frac{1}{2}\beta|x_\parallel|} e^{\frac{1}{4}\beta\left(\vZio + m_\iota g x_3\right)} |(\fmmstio - \fnmstio)(x, v)|\right),
 \end{multline} 
and hence by \eqref{eq.F.diff.st.cauchy2},
\begin{multline}\label{eq.F.diff.st.cauchy3}
   e^{\frac{1}{2}\beta|x_\parallel|}e^{\frac{1}{4}\beta\left(\vZ+m_\pm g x_{3}\right)}| (\fmst-\fnst)(x,v)|
    \lesssim  \frac{1}{\min\{m^2_+,m^2_-\} g^2\beta^7}\|\mathrm{w}^2_{\pm,\beta}(\cdot,0,\cdot)\nabla_{x_\parallel,v}G_\pm(\cdot,\cdot)\|_{L^\infty_{x_\parallel,v}} \\\times \max_{\iota=\pm}\left(\sup_{x,v}
e^{\frac{1}{2}\beta|x_\parallel|} e^{\frac{1}{4}\beta\left(\vZio + m_\iota g x_3\right)} |(\fmmstio - \fnmstio)(x, v)|\right).
\end{multline}Note that for a sufficiently large $\beta\gg1,$ we have 
$$\kappa\eqdef \frac{1}{\min\{m^2_+,m^2_-\} g^2\beta^7}\|\mathrm{w}^2_{\pm,\beta}(\cdot,0,\cdot)\nabla_{x_\parallel,v}G_\pm(\cdot,\cdot)\|_{L^\infty_{x_\parallel,v}}   \ll 1. $$ Then by repeating the argument, we have
\begin{multline}\label{eq.F.diff.st.cauchy4}
   e^{\frac{1}{2}\beta|x_\parallel|}e^{\frac{1}{4}\beta\left(\vZ+m_\pm g x_{3}\right)}| (\fmst-\fnst)(x,v)|\\*
    \lesssim \kappa^{n}\max_{\iota=\pm}\left(\sup_{x,v}
e^{\frac{1}{2}\beta|x_\parallel|} e^{\frac{1}{4}\beta\left(\vZio + m_\iota g x_3\right)} |(\fmmnstio - \fzerostio)(x, v)|\right)\lesssim C\kappa^n, 
\end{multline}by the uniform estimate \eqref{bootstrap.assump.st} and that $ \fzerost=0$. Therefore, we conclude that $\{\fmst\}_{k\in \mathbb{N}}$ is Cauchy, and hence $\{(\Embfst,\Bmbfst)\}_{k\in \mathbb{N}}$ are also Cauchy by \eqref{Est diff final.cauchy}. We record this fact in the following lemma:
\begin{lemma}
Both sequences $\{\fmst\}_{k\in \mathbb{N}}$ and $\{(\Embfst,\Bmbfst)\}_{k\in \mathbb{N}}$ are Cauchy.
\end{lemma}

Once the Cauchy property is verified as above, the same argument as in \eqref{nonlinear.passingtothelimit.M} applies to pass to the limit in the nonlinear terms via the strong convergence of Cauchy sequences.  
   Also, these solutions $(F_{\pm,\textup{st}}^\infty,\E_{\textup{st}}^\infty,\B_{\textup{st}}^\infty)$ satisfy the same (weighted-) $L^\infty$ bounds \eqref{bootstrap.assump.st} as well as the uniform derivative estimate \eqref{decay of Fst}. This completes the proof of the existence of steady-states with J\"uttner-Maxwell upper bounds (Theorem \ref{thm.ex.st}).

\subsection{Uniqueness and Non-Negativity}
We now establish the uniqueness of solutions to the stationary Vlasov--Maxwell system \eqref{2speciesVM-steady}.

Suppose that there are two stationary solutions $(\fonest, \Eonedst, \Bonedst)$ and $(\ftwost, \Etwodst, \Btwodst)$ for the system \eqref{2speciesVM-steady} under \eqref{2species-perturbabsorbing.st} and \eqref{perfect cond. boundary-perturb.st}. 
Then note that we have \begin{equation}\label{zero initial boundary difference.st}\begin{split}
        &(\fonest-\ftwost) (x_\parallel,0,v)|_{\gamma_-}=0,
        \end{split}
\end{equation}
and the difference $\fonest-\ftwost$ solves the following Vlasov equation:
   \begin{multline}\label{eq.uni1st}
   \hat{v}_\pm \cdot \nabla_x (\fonest-\ftwost) +\left(\pm\Eonedst\pm(\hat{v}_\pm)\times \Bonedst-m_\pm g\hat{e}_3\right)\cdot \nabla_v (\fonest-\ftwost)\\
    = -\left(\pm(\Eonedst-\Etwodst)\pm(\hat{v}_\pm)\times (\Bonedst-\Btwodst)\right)\cdot \nabla_v \ftwost.
   \end{multline} Similarly to \eqref{def.steady characteristics}, we define the stationary characteristic trajectory variables 
$\Zst(s)=(\Xst(s),\Vst(s))$ satisfying $\Zst(0 ; x, v)=(\Xst(0;  x, v), \Vst(0 ;  x, v))=(x, v)=z$, generated by the fields $\Eonedst$ and $ \Bonedst$, which solves
\begin{equation}\notag
 \begin{split}
   \frac{d\Xst(s)}{ds}&=\Vhatst(s)=\frac{\Vst(s)}{\sqrt{m_\pm^2+|\Vst(s)|^2}},\\
        \frac{d\Vst(s)}{ds}&=\pm  \Eonedst(\Xst(s))\pm \Vhatst(s)\times \Bonedst(\Xst(s))- m_\pm g\hat{e}_3,
 \end{split} 
\end{equation}   where $\hat{e}_3\eqdef (0,0,1)^\top$ and $\hat{v}_\pm\eqdef \frac{v}{\vZ}=\frac{v}{\sqrt{m_\pm^2+|v|^2}}$. 
In addition, similarly to \eqref{Back Forw exit time.st}, denote the corresponding forward and backward and exit times $\tfst$ and $\tbst$ for the steady characteristic trajectory as
\begin{equation}\notag \begin{split}
   & \tfst (x,v)= \sup\{s\in [0,\infty): (\Xst)_3(\tau;x,v)>0\ \textup{ for all } \tau\in(0,s)\}\ge 0,\\
   & \tbst (x,v)= \sup\{s\in [0,\infty): (\Xst)_3(-\tau;x,v)>0\ \textup{ for all } \tau\in(0,s)\}\ge 0.
   \end{split}
\end{equation}Then, by integrating \eqref{eq.uni1st} along the characteristics $\Zst(s)=(\Xst(s),\Vst(s))$ (associated with $\Eonedst$ and $ \Bonedst$) for $s\in [-\tbst,0]$, we obtain\begin{multline}\notag 
    (\fonest-\ftwost)(x,v)
    = \mp \int^0_{-\tbst} \left((\Eonedst-\Etwodst)(\Xst(s))+\Vhatst(s)\times (\Bonedst-\Btwodst)(\Xst(s))\right)\\\cdot \nabla_v \ftwost(\Xst(s),\Vst(s))ds.
\end{multline}Therefore, we obtain 
\begin{multline}\label{eq.F.diff.st}
   | (\fonest-\ftwost)(x,v)|\\
    \le \tbst \sup_{s\in [-\tbst,0]}|( \nabla_v \ftwost)(\Xst(s),\Vst(s))| \left(\| (\Eonedst-\Etwodst)(\cdot)\|_{L^\infty_x}+\| (\Bonedst-\Btwodst)(\cdot)\|_{L^\infty_x}\right).
\end{multline}  Regarding the momentum derivative $\nabla_v \ftwost$,
we use the uniform estimate \eqref{decay of Fst} and obtain that
\begin{align*}
    &\sup_{s\in [-\tbst,0]}|( \nabla_v \ftwost)(\Xst(s),\Vst(s))| \\
    &\le \sup_{s\in [-\tbst,0]}\frac{1}{\mathrm{w}_{\pm,\beta}(\Xst(s),\Vst(s))}|( \mathrm{w}_{\pm,\beta}\nabla_v \ftwost)(\Xst(s),\Vst(s))| \\
    &\le C e^{-\frac{1}{2}\beta\left(\vZ+m_\pm g x_{3}+|x_\parallel|\right)}\|\mathrm{w}^2_{\pm,\beta}(\cdot,0,\cdot)\nabla_{x_\parallel,v}G_\pm(\cdot,\cdot)\|_{L^\infty_{x_\parallel,v}},
\end{align*} by the weight comparison estimate \eqref{w comparison 3.st} along the steady characteristic trajectory $(\Xst(s),\Vst(s)).$
In addition, note that by \eqref{exit time bound.st} we have
\begin{equation}\notag
\tbst\le \frac{32}{13m_\pm g} (\sqrt{m_\pm^2+|v_\pm|^2}+m_\pm gx_3)\le \frac{3}{m_\pm g} \left(v^0_\pm+ m_\pm g x_3\right) .\end{equation} Regarding the upper bound of $\tbst$, we further observe that 
$$(v^0_\pm +m_\pm g x_3) e^{-\frac{1}{2}\beta\left(\vZ+m_\pm g x_{3}+|x_\parallel|\right)}\lesssim \frac{1}{\beta}e^{-\frac{1}{2}\beta|x_\parallel|}e^{-\frac{1}{4}\beta\left(\vZ+m_\pm g x_{3}\right)}.$$Therefore, by \eqref{eq.F.diff.st}, we have
\begin{multline}\label{eq.F.diff.st2}
   e^{\frac{1}{2}\beta|x_\parallel|}e^{\frac{1}{4}\beta\left(\vZ+m_\pm g x_{3}\right)}| (\fonest-\ftwost)(x,v)|\\
    \lesssim \frac{1}{\beta m_\pm g} \|\mathrm{w}^2_{\pm,\beta}(\cdot,0,\cdot)\nabla_{x_\parallel,v}G_\pm(\cdot,\cdot)\|_{L^\infty_{x_\parallel,v}}  \left(\| (\Eonedst-\Etwodst)(\cdot)\|_{L^\infty_x}+\| (\Bonedst-\Btwodst)(\cdot)\|_{L^\infty_x}\right).
\end{multline}
We now derive upper bounds for the differences $\Eonedst - \Etwodst$ and $\Bonedst - \Btwodst$. Our objective is to obtain uniform estimates for these quantities in terms of
\begin{equation}
    \label{objective term}
e^{\frac{1}{2}\beta|x_\parallel|} e^{\frac{1}{4}\beta\left(\vZ + m_\pm g x_3\right)} |(\fonest - \ftwost)(x, v)|.
\end{equation}
Recall that we use the representations given in \eqref{Field representation Bst} and \eqref{Field representation Est} for $\Eonedst$, $\Etwodst$, $\Bonedst$, and $\Btwodst$. Since these representations are linear in $F_{\pm,\textup{st}}$, the differences $\Eonedst - \Etwodst$ and $\Bonedst - \Btwodst$ can be expressed in the same form, with $F_{\pm,\textup{st}}$ replaced by $\fonest - \ftwost$.
Therefore, we have
$$|\Eonedst(x) - \Etwodst(x)|\lesssim   \int_{\mathbb{R}^3_+} |\nabla\mathfrak{G}_{\textup{odd}}(x,y)|\left| \sum_{\iota=\pm}\iota\int_\rth (\fonestio - \ftwostio)(y,v)dv\right| \ dy.$$ By factoring out the term \eqref{objective term}, we further  observe that 
\begin{multline}\notag
   |\Eonedst(x) - \Etwodst(x)| \\
     \lesssim\sum_{\iota=\pm} \frac{1}{\beta^3}\left(\sup_{x,v}
e^{\frac{1}{2}\beta|x_\parallel|} e^{\frac{1}{4}\beta\left(\vZio + m_\iota g x_3\right)} |(\fonestio - \ftwostio)(x, v)|\right)\int_{\mathbb{R}^3_+} dy\ |\nabla\mathfrak{G}_{\textup{odd}}(x,y)| e^{-\frac{\beta }{2}|y_\parallel|} e^{-\frac{1}{2}m_\iota g\beta y_3} ,
\end{multline}by \eqref{additional beta decay.st}. Then since $|\nabla\mathfrak{G}_{\textup{odd}}(x,y)|\le \frac{1}{|x-y|^2}+\frac{1}{|\bar{x}-y|^2}$ and the upper bound is even in $y_3$, note that 
\begin{multline*}
    \int_{\mathbb{R}^3_+} dy\ |\nabla\mathfrak{G}_{\textup{odd}}(x,y)| e^{-\frac{\beta }{2}|y_\parallel|} e^{-\frac{1}{2}m_\pm g\beta y_3}
    \le \int_{\mathbb{R}^3} dy\ \left(\frac{1}{|x-y|^2}+\frac{1}{|\bar{x}-y|^2}\right) e^{-\frac{\beta }{2}|y_\parallel|} e^{-\frac{1}{2}m_\pm g\beta |y_3|}\\
    \lesssim  \frac{1}{m_\pm g\beta^3}\frac{1}{\langle x\rangle^2},
\end{multline*}by \eqref{asymptotics}. Therefore, we conclude that   \begin{equation}
     \label{Est diff final}|\Eonedst(x) - \Etwodst(x)|\lesssim \frac{1}{\min\{m_+,m_-\} g\beta^6}\max_{\iota=\pm}\left(\sup_{x,v}
e^{\frac{1}{2}\beta|x_\parallel|} e^{\frac{1}{4}\beta\left(\vZio + m_\iota g x_3\right)} |(\fonestio - \ftwostio)(x, v)|\right).
 \end{equation} 
Moreover, since $ \int_\rth |\hat{v}_{\pm,i}||F_{\pm,\textup{st}}(y,v)|dv\le \int_\rth |F_{\pm,\textup{st}}(y,v)|dv,$ $\B_{\textup{st}}(x)$ in \eqref{Field representation Bst} also has the same upper-bound (up to constant) as that of $\E_{\textup{st}}(x)$ and hence we have the same upper bound on the difference 
 $$|\Bonedst(x) - \Btwodst(x)|\lesssim \frac{1}{\min\{m_+,m_-\} g\beta^6}\max_{\iota=\pm}\left(\sup_{x,v}
e^{\frac{1}{2}\beta|x_\parallel|} e^{\frac{1}{4}\beta\left(\vZio + m_\iota g x_3\right)} |(\fonestio - \ftwostio)(x, v)|\right).$$
Consequently, by \eqref{eq.F.diff.st2}, we obtain that
\begin{multline}\label{eq.F.diff.st3}
 \sup_{x,v}  e^{\frac{1}{2}\beta|x_\parallel|}e^{\frac{1}{4}\beta\left(\vZ+m_\pm g x_{3}\right)}| (\fonest-\ftwost)(x,v)|\\
    \lesssim \frac{1}{\min\{m^2_+,m^2_-\} g^2\beta^7} \|\mathrm{w}^2_{\pm,\beta}(\cdot,0,\cdot)\nabla_{x_\parallel,v}G_\pm(\cdot,\cdot)\|_{L^\infty_{x_\parallel,v}}\max_{\iota=\pm}\left(\sup_{x,v}
e^{\frac{1}{2}\beta|x_\parallel|} e^{\frac{1}{4}\beta\left(\vZio + m_\iota g x_3\right)} |(\fonestio - \ftwostio)(x, v)|\right).
\end{multline} By choosing $\beta\gg1$ sufficiently large, we conclude that
$$\max_{\iota=\pm}\left(\sup_{x,v}
e^{\frac{1}{2}\beta|x_\parallel|} e^{\frac{1}{4}\beta\left(\vZio + m_\iota g x_3\right)} |(\fonestio - \ftwostio)(x, v)|\right)=0,$$ and hence $$|\Eonedst(x) - \Etwodst(x)|=|\Bonedst(x) - \Btwodst(x)|=0,\text{ for any }x\in\mathbb{R}^3_+.$$

This completes the proof of uniqueness for the stationary solution.

Next we address the non-negativity of the solution we have contructed. Assume that the inflow boundary profile $G_\pm$ is non-negative. Since $ F_{\pm,\textup{st}} $ remains constant along the stationary characteristics described by \eqref{def.steady characteristics}, it follows that $ F_{\pm,\textup{st}}  $ is also non-negative.

This concludes our analysis of the existence and uniform estimates for steady states under J\"uttner-Maxwell upper bounds. In the next section, we explore perturbative solutions around these steady states and establish their asymptotic stability using a bootstrap argument.

\section{Dynamical Asymptotic Stability}\label{sec.boot.decay}
In this section, we establish the asymptotic stability of the steady states $(F_{\pm,\textup{st}},\E_{\textup{st}},\B_{\textup{st}})$, whose unique existence is guaranteed by Theorem \ref{thm.ex.st}. 
We show that the perturbation $(\Fp,\Ep,\Bp)$ from the steady state decays linearly in time, thereby concluding that the stationary states are asymptotically stable.  

We assume that the inflow boundary data $G_\pm$ at $x_3=0$ coincide with the stationary states $F_{\pm,\textup{st}}$ for incoming particles with $v\in\rth$ such that $n_x\cdot v<0$. 
Recall that these profiles are bounded above by J\"uttner equilibrium distributions \eqref{steady state L infty}. 
As before, we suppose that $\E$, $\B$, $\E_{\textup{st}}$, and $\B_{\textup{st}}$ (and thus also $\mathcal{E}$ and $\mathcal{B}$) satisfy the perfect conductor boundary condition \eqref{perfect cond. boundary} on $x_3=0$.

\subsection{Perturbations from the Steady States}We first define the perturbation $(\Fp,\Ep,\Bp)$ from the steady-state $(F_{\pm,\textup{st}},\E_{\textup{st}},\B_{\textup{st}}) $:
\begin{definition}
    Define the perturbations $(\Fp,\Ep,\Bp)$ from the steady-state $(F_{\pm,\textup{st}},\E_{\textup{st}},\B_{\textup{st}}) $ as
\begin{equation}\label{def.perturb}
    \begin{split}
        \Fp(t,x,v)&\eqdef F_{\pm}(t,x,v)-F_{\pm,\textup{st}}(x,v),\ 
        \Ep(t,x)\eqdef \E(t,x)-\E_{\textup{st}}(x),\textup{ and }
        \Bp(t,x)\eqdef \B(t,x)-\B_{\textup{st}}(x),
    \end{split}
\end{equation}for $t\in[0,\infty)$, $x\in \mathbb{R}^3_+,$ and $v\in\rth$ where the full solution $(F_\pm,\E,\B)$ and the steady-state $(F_{\pm,\textup{st}},\E_{\textup{st}},\B_{\textup{st}}) $ solve the dynamical and the stationary systems of the Vlasov--Maxwell equations \eqref{2speciesVM} and \eqref{2speciesVM-steady}, respectively, in the sense of distributions.  
\end{definition}
Then by \eqref{2speciesVM} and \eqref{2speciesVM-steady}, we observe that the perturbations $(\Fp,\Ep,\Bp)$ solve the perturbative system of Vlasov--Maxwell equations \eqref{VM perturbations1} (with $c=1$ and $\textup{e}=1$ normalized)
where $\varrho$ and $\mathcal{J}$ are defined as \eqref{rhoJ_per} and  satisfy the continuity equation \begin{equation}\label{continuity eq.perturb}
    \partial_t \varrho +\nabla_x\cdot \mathcal{J} =0.
\end{equation}

Under the assumptions above, we consider an iterated sequence of solutions $(\fpl,\Epl,\Bpl)$ that solves the following Vlasov--Maxwell system under \eqref{inflow boundary}. Note that we can consider the same characteristic trajectory $\ZSl=(\XSl,\VSl)$ solving \eqref{iterated char} but now in the whole half space $\mathbb{R}^3_+$. The iterated sequence of solutions $(\fpl,\Epl,\Bpl)$ solve
\begin{equation}\label{iterated Vlasov.M}
    \begin{split}
        \partial_t \fplo+\hat{v}_\pm \cdot \nabla_x \fplo +\left(\pm\Elbf\pm\hat{v}_\pm\times \Blbf-m_\pm g\hat{e}_3\right)\cdot \nabla_v \fplo &= \mp \left(\Epl+\hat{v}_\pm\times \Bpl\right)\cdot \nabla_v F_{\pm,\textup{st}},\\
         \fplo(0,x,v)= f^{\textup{in}}_\pm(x,v),\
        \fplo (t,x_\parallel,0,v)|_{\gamma_-}&=0,\text{ and }
    \end{split}
\end{equation}
\begin{equation}\label{iterated Maxwell.M}
    \begin{split}
        \partial _t \Epl-\nabla_x\times \Bpl&= - 4\pi \mathcal{J}^l ,\ 
        \partial_t \Bpl+\nabla_x\times \Epl=0,\\
        \nabla_x \cdot \Epl&=4\pi \varrho^l,\  
        \nabla_x \cdot \Bpl=0.
        \end{split}
        \end{equation}

\noindent By \eqref{leading char}, we can also define the characteristic trajectory $\ZSlo=(\XSlo,\VSlo)$ which solves
\begin{equation}\label{iterated char}
 \begin{split}
   \frac{d\XSlo(s)}{ds}&=\hVSlo(s)=\frac{\VSlo(s)}{\sqrt{m_\pm^2+|\VSlo(s)|^2}},\\
        \frac{d\VSlo(s)}{ds}&=\pm\Elbf(s,\XSlo(s))\pm\hVSlo(s)\times \Blbf(s,\XSlo(s))-m_\pm g\hat{e}_3,
 \end{split} 
\end{equation}   where $\XSlo(s)=\XSlo(s;t,x,v)$, $\VSlo(s)=\VSlo(s;t,x,v)$, $\hat{e}_3\eqdef (0,0,1)^\top$, and $(\hat{v}_\pm)\eqdef \frac{v}{\sqrt{m_\pm^2+|v|^2}}$.

Our main goal is to prove that the perturbations $(f_\pm, \mathcal{E}, \mathcal{B})$ decay linearly in time. 
In particular, the argument controls nonlinear terms while simultaneously extracting decay from the linearized dynamics, thereby establishing full nonlinear asymptotic stability.
 In the subsequent sections, we establish decay-in-time estimates for these iterates and close the nonlinear argument to obtain asymptotic stability.

In the rest of the section, we prove the following main proposition on the linear-in-time decay of the perturbations:
\begin{proposition}\label{prop.dyna.boot}
For any $l\in\mathbb{N}$, we have
\begin{align}
     \sup_{t\ge 0}\ \langle t\rangle&\left\|e^{\frac{\beta}{2} |x_{\parallel}|}e^{\frac{\beta}{4}(\vZ+m_\pm gx_3)}\fpl(t,\cdot,\cdot)\right\|_{L^\infty} \label{apriori_f}\\
     &\le  \frac{4}{\beta } \bigg(\|\mathrm{w}_{\pm,\beta} f^{\textup{in}}_\pm\|_{L^\infty_{x,v}}
  +C \notag\min\{m_-,m_+\}\frac{g}{8}\|\mathrm{w}^2_{\pm,\beta}(\cdot,0,\cdot)\nabla_{x_\parallel,v}G_\pm(\cdot,\cdot)\|_{L^\infty_{x_\parallel,v}}\bigg),  \\
   &\sup_{t\ge 0}\  \langle t \rangle\|(\Epl,\Bpl)\|_{L^\infty} \le \min\{m_-,m_+\}\frac{g}{16}  ,\label{apriori_EB} 
    \end{align}for a sufficiently large $\beta>1$ where the weight $\mathrm{w}=\mathrm{w}_{\pm,\beta}$ is defined as \eqref{weights.wholehalf}.
\end{proposition} In the following sections, we fix $ l \in \mathbb{N} $ and assume that \eqref{apriori_f}--\eqref{apriori_EB} hold at the iteration level $ (l) $. We then show that these same estimates remain valid at the next level $ (l+1) $, thereby closing the bootstrap argument.

\begin{proof}[Proof of Proposition \ref{prop.dyna.boot}]
  Proposition \ref{prop.dyna.boot} follows from Lemma \ref{lem.dyna.fplo} and Lemma \ref{lem.dyna.EB}, which will be established in the subsequent sections.
\end{proof}

\begin{remark}[Compatibility Conditions]\label{remark.comp.cond}
For the limiting weak solution we impose only the perfect–conductor Dirichlet data on
\(
\mathbf{E}_1,\ \mathbf{E}_2,\ \mathbf{B}_3,\) 
while the Neumann-type relations for
\(
\mathbf{E}_3,\ \mathbf{B}_1,\ \mathbf{B}_2\) 
are used only at the approximate level and are encoded in the weak formulation; no additional boundary conditions are imposed on the limit, and with
\(
W^{1,\infty}\)
regularity this suffices to define the trace at
\(
x_3=0\)
and to close all boundary terms consistently with the continuity equation and the wave system.

The Neumann boundary conditions for the iterated sequence $\Elo_3$, $\Blo_1$, and $\Blo_2$ can be formally derived and be justified. We first impose the Dirichlet-type perfect conductor boundary conditions \eqref{perfect cond. boundary} to the iterated fields $\Elo_1,$ $\Elo_2$ and $\Blo_3$. 
Then using the Gauss's law, we obtain that 
\begin{equation}\notag
\partial_{x_3} \Elo_3 = 4 \pi \rho^{l+1} - \partial_{x_1} \Elo_1 - \partial_{x_2} \Elo_2.
\end{equation}
Formally (needs some justification that $\partial_{x_1}\Elo_1, \partial_{x_2}\Elo_2, 4 \pi \rho^{l+1} $ have their traces in a proper space such as $C^0 (\bar{\Omega})$ at the sequential level of construction of solutions), we have $\partial_{x_1}\Elo_1 =0 = \partial_{x_2}\Elo_2$ from \eqref{perfect cond. boundary}. Hence $\Elo_3$ formally satisfies the Neumann boundary condition: 
\begin{equation}\notag
(\partial_{x_3} \Elo_3- 4 \pi \rho^{l+1})|_{\partial\Omega}  =0.
\end{equation}Also, using the Amp\`{e}re-Maxwell equation, we derive that 
\begin{align}\notag
n \times (\nabla \times \Blobf) - 4\pi  n\times J^{l+1} =n\times  \partial_t \Elobf
\end{align} for any $n\in\rth.$ 
 In addition, from \eqref{perfect cond. boundary}, we formally (needs some justification that $\partial_t \mathbf{E}^{l+1}_\parallel, \nabla_{x_\parallel}\Blo_3, $ and $J_\parallel$ have their traces in a proper space such as $C^0(\bar{\Omega})$ at the sequential level of construction of solutions) have $\partial_t \Elo_1=0= \partial_t \Elo_2$ at $\partial\Omega$, and $\partial_{x_1}\Blo_3 = 0 = \partial_{x_2}\Blo_3$ at $\partial\Omega$. Then by choosing $n$ to be the outward normal vector at the boundary $x_3=0$ as $n=(0,0,-1)^\top,$ formally we derive that 
 \begin{equation}\notag
 \begin{bmatrix}
0 \\
0 \\
-1
 \end{bmatrix}
 \times  \begin{bmatrix}
 \partial_{x_2}\Blo_3  - \partial_{x_3}\Blo_2   \\
 -(\partial_{x_1}\Blo_3 - \partial_{x_3}\Blo_1) \\
 \partial_{x_1}\Blo_2 - \partial_{x_2}\Blo_1 
 \end{bmatrix}
 - 4 \pi 
 \begin{bmatrix}
 J^{l+1}_2 \\
 -J^{l+1}_1 \\
 0
  \end{bmatrix}
  = \begin{bmatrix} \partial_t \Elo_2 \\
  - \partial_t \Elo_1 \\ 
  0 
  \end{bmatrix}
  = \begin{bmatrix} 0 \\
 0 \\ 
  0 
  \end{bmatrix} \ \ \textup{at $\partial\Omega$, }
 \end{equation}
and hence
\begin{align}\notag
(\partial_{x_3}\Blo_1  - 4 \pi J^{l+1}_2)|_{\partial\Omega}=0,\ \
(\partial_{x_3}\Blo_2 + 4 \pi J^{l+1}_1)|_{\partial\Omega}=0.
\end{align}Therefore, we have 
\begin{equation}
\label{perfect.conductor.neumann}
\partial_{x_3}\Eth = 4\pi \rho, \quad 
\partial_{x_3}\Btwo = -4\pi J_1, \quad 
\textup{and} \quad 
\partial_{x_3}\Bone = 4\pi J_2.
\end{equation}
Note that the Neumann conditions above are not really \textit{boundary conditions}. They are the identities that the smooth solution should satisfy at the boundary as long as all the quantities have a proper sense of trace at the boundary. 

\end{remark}

\subsection{Enhanced Decay-in-$t$ for the Distributions and Fields}
\label{sec.fplo.decay}In this section, we prove the estimate \eqref{apriori_f} at the iteration level $(l+1)$.
\begin{lemma}\label{lem.dyna.fplo}
   Fix $l\in\mathbb{N}$ and suppose \eqref{apriori_EB} hold for $(\Epl,\Bpl).$ Then $\fplo$ satisfies  \begin{align}
     \sup_{t\ge 0}\ \langle t\rangle&\left\|e^{\frac{\beta}{2} |x_{\parallel}|}e^{\frac{\beta}{4}(\vZ+m_\pm gx_3)}\fplo(t,\cdot,\cdot)\right\|_{L^\infty}\notag\\
     &\le  \frac{4}{\beta } \bigg(\|\mathrm{w}_{\pm,\beta} f^{\textup{in}}_\pm\|_{L^\infty_{x,v}}
  +C \notag\min\{m_-,m_+\}\frac{g}{8}\|\mathrm{w}^2_{\pm,\beta}(\cdot,0,\cdot)\nabla_{x_\parallel,v}G_\pm(\cdot,\cdot)\|_{L^\infty_{x_\parallel,v}}\bigg).\end{align}
\end{lemma}

\begin{proof}
    By writing the solution $\fplo$ in the mild form \begin{multline}\notag
    \fplo(t,x,v)= 1_{t\le \tblo(t,x,v)}f^{\textup{in}}_\pm(\XSlo(0;t,x,v),\VSlo(0;t,x,v))\\\mp \int^t_{\max\{0,t-\tblo\}} \left(\Epl(s,\XSlo(s))+\hat{\mathcal{V}}^{l+1}_\pm(s)\times \Bpl(s,\XSlo(s))\right)\cdot \nabla_v F_{\pm,\textup{st}}(\XSlo(s),\VSlo(s))ds,
\end{multline} and using \eqref{apriori_EB}, we obtain 
\begin{align*}\notag
   \langle t\rangle |\fplo(t,x,v)|
   &\le  \langle t\rangle\frac{1_{t\le \tblo(t,x,v)}}{\mathrm{w}_{\pm,\beta}( \ZSlo(0 ; t, x, v))}\|\mathrm{w}_{\pm,\beta} f^{\textup{in}}_\pm\|_{L^\infty_{x,v}}\\
   &+1_{t\le \tblo(t,x,v)}  \langle t \rangle\|(\Epl,\Bpl)\|_{L^\infty} \|\mathrm{w}_{\pm,\beta}\nabla_v F_{\pm,\textup{st}}\|_{L^\infty_{x,v}}\int^t_0 \frac{1}{\mathrm{w}_{\pm,\beta}( \ZSlo(s ; t, x, v))}ds\\
   &+1_{t> \tblo(t,x,v)}  \langle t \rangle\|(\Epl,\Bpl)\|_{L^\infty}\|\mathrm{w}_{\pm,\beta}\nabla_v F_{\pm,\textup{st}}\|_{L^\infty_{x,v}}\int^t_{t-\tblo} \frac{1}{\mathrm{w}_{\pm,\beta}( \ZSlo(s ; t, x, v))}ds\\
     &\le  \langle \tblo\rangle e^{-\frac{1}{2}\beta v_\pm^0-\frac{1}{2}m_\pm g\beta x_3-\frac{\beta}{2}|x_{\parallel}|}\|\mathrm{w}_{\pm,\beta} f^{\textup{in}}_\pm\|_{L^\infty_{x,v}}\\
  &+\langle \tblo\rangle 1_{t\le \tblo(t,x,v)}\min\{m_-,m_+\}\frac{g}{8} \|\mathrm{w}_{\pm,\beta}\nabla_v F_{\pm,\textup{st}}\|_{L^\infty_{x,v}}e^{-\frac{1}{2}\beta v_\pm^0-\frac{1}{2}m_\pm g\beta x_3-\frac{\beta}{2}|x_{\parallel}|} \\
   &+ \tblo 1_{t> \tblo(t,x,v)}\min\{m_-,m_+\}\frac{g}{8} \|\mathrm{w}_{\pm,\beta}\nabla_v F_{\pm,\textup{st}}\|_{L^\infty_{x,v}}e^{-\frac{1}{2}\beta v_\pm^0-\frac{1}{2}m_\pm g\beta x_3-\frac{\beta}{2}|x_{\parallel}|},
\end{align*}
by\eqref{exit time bound.whole} and \eqref{w comparison 3.whole}. Then using \eqref{exit time bound.whole} and \eqref{decay of Fst} with $G_\pm$ replaced by $\delta G_\pm$, we further have
\begin{equation}\begin{split}\label{fplo decay est}
&   \langle t\rangle |\fplo(t,x,v)|\\
  & \le \left(1+\frac{16}{5m_\pm g} (v_\pm^0+m_\pm gx_3)\right)e^{-\frac{1}{2}\beta v_\pm^0-\frac{1}{2}m_\pm g\beta x_3-\frac{\beta}{2}|x_{\parallel}|}\\&\qquad\times\bigg(\|\mathrm{w}_{\pm,\beta} f^{\textup{in}}_\pm\|_{L^\infty_{x,v}}+\min\{m_-,m_+\}\frac{g}{8} \|\mathrm{w}_{\pm,\beta}\nabla_v F_{\pm,\textup{st}}\|_{L^\infty_{x,v}}\bigg)\\
   &\le \frac{11}{10\beta}(\beta v_\pm^0+m_\pm g\beta x_3)e^{-\frac{1}{2}\beta v_\pm^0-\frac{1}{2}m_\pm g\beta x_3-\frac{\beta}{2}|x_{\parallel}|}\bigg(\|\mathrm{w}_{\pm,\beta} f^{\textup{in}}_\pm\|_{L^\infty_{x,v}}+\min\{m_-,m_+\}\frac{g}{8} \|\mathrm{w}_{\pm,\beta}\nabla_v F_{\pm,\textup{st}}\|_{L^\infty_{x,v}}\bigg)\\
    & \le \frac{44}{5\beta e} e^{-\frac{1}{4}\beta v_\pm^0-\frac{1}{4}m_\pm g\beta x_3-\frac{\beta}{2}|x_{\parallel}|} \bigg(\|\mathrm{w}_{\pm,\beta} f^{\textup{in}}_\pm\|_{L^\infty_{x,v}}
  +C \min\{m_-,m_+\}\frac{g}{8}\|\mathrm{w}^2_{\pm,\beta}(\cdot,0,\cdot)\nabla_{x_\parallel,v}G_\pm(\cdot,\cdot)\|_{L^\infty_{x_\parallel,v}}\bigg),
\end{split}\end{equation} for $g>0$ such that 
$\textcolor{black}{\min\{m_-,m_+\}g\gg 1}.$ Here, we also used the inequality that for $z\ge 0,$ $\beta ze^{-\frac{\beta}{2} z}\le \frac{4}{ e}e^{-\frac{\beta}{4}z}.$ 
This completes the proof of Lemma \ref{lem.dyna.fplo} for $\fplo.$\end{proof}

Now we prove the estimate \eqref{apriori_EB} at the iteration level $( l+1) $. This estimate ensures an additional linear decay in time for the perturbations $ \Eplo $ and $ \Bplo $, thereby implying the asymptotic stability of the steady states $ \E_{\textup{st}} $ and $ \B_{\textup{st}} $. Recall that the total fields are given by
$$
\Elobf = \E_{\textup{st}} + \Eplo \quad \text{and} \quad \Blobf = \B_{\textup{st}} + \Bplo,
$$
where $ \Eplo $ and $ \Bplo $ represent perturbations around the steady states $ \E_{\textup{st}} $ and $ \B_{\textup{st}} $, respectively.

\paragraph{Field representations for the perturbative fields $\Eplo$ and $\Bplo$}
For the estimates of the perturbative components $\Eplo$ and $\Bplo$, we employ the field representations of the electromagnetic fields given in \eqref{Ei5}, \eqref{E3}, \eqref{B_half_final}, \eqref{Bpar_half_final}, \Black{and \eqref{BparS_half_final}}, which were derived from the corresponding wave equations. We note that the Maxwell system \eqref{iterated Maxwell.M} governing the perturbations $\Eplo$ and $\Bplo$ has the same structure as the full Maxwell system \eqref{2speciesVM} for $\E$ and $\B$, provided that $F$ is replaced by $\fplo$ (and consequently $\rho$ and $J$ are replaced by $\varrho$ and $\mathcal{J}$, respectively). Under this replacement, the only difference that affects the final representation lies in the nonlinear $S$-term. Specifically, when constructing the field representations for $\Eplo$ and $\Bplo$, we use the inhomogeneous Vlasov equation \eqref{iterated Vlasov.M} for $\fplo$, which contains an additional inhomogeneity
\[
- \nabla_v \cdot \Big( (\pm \Epl \pm \hat{v}_\pm \times \Bpl) F_{\pm,\textup{st}} \Big).
\]
This term introduces a new nonlinearity that appears \Black{in both the electric and the magnetic field representations: recall from Section~\ref{sec.4} that the magnetic field representation contains the nonlinear $S$-terms \eqref{BparS_half_final}, which coincide with the corresponding nonlinear terms of the Glassey--Strauss type representation (Remark \ref{rmk.BS.GS})}. Therefore, the perturbative fields $\Eplo$ and $\Bplo$ can be expressed as follows:
for each $i = 1,2,3$,
$$
\Eplo_i = \Eplo_{\textup{hom},i} + \Eplo_{ib1} + \Eplo_{ib2} + \Eplo_{iT} + \Eplo_{iS} + \delta_{i3} \Eplo_{\textup{add},3}, \quad
\Bplo_i = \Bplo_{\textup{hom},i} + \Bplo_{ib1} + \Bplo_{iT} \Black{+ \Bplo_{iS}},
$$
where the terms $\Eplo_{\textup{hom},i},\ \Eplo_{ib1},\ \Eplo_{ib2},\ \Eplo_{iT},\ \Bplo_{\textup{hom},i},\ \Bplo_{ib1},$ and $\Bplo_{iT}$ are given by the same representations as in \eqref{Ei5}, \eqref{E3}, \eqref{B_half_final}, and \eqref{Bpar_half_final}, respectively, with $F_\pm$ replaced by $\fplo$. Note that the normal electric field contains an additional term $\Eplo_{\textup{add},3}$ defined as
$$
\Eplo_{\textup{add},3}(t,x) = \sum_{\iota=\pm}(-\iota) 2 \int_{B(x;t) \cap \{y_3=0\}} \int_{\mathbb{R}^3} \frac{f^{l+1}_\iota(t - |y - x|, y_\parallel, 0, v)}{|y - x|} \, dv \, dy_\parallel.
$$
As noted above, the nonlinear $S$-term\Black{s} $\Eplo_{iS}$ \Black{and $\Bplo_{iS}$} for $i=1,2,3$ arise not only from the nonlinear source
\[
- \nabla_v \cdot \Big( (\pm \Elbf \pm \hat{v}_\pm \times \Blbf - m_\pm g \hat{e}_3)\,\fplo \Big),
\]
but also from the additional \textit{inhomogeneous} stationary source
\[
- \nabla_v \cdot \Big( (\pm \Epl \pm \hat{v}_\pm \times \Bpl)\, F_{\pm,\textup{st}} \Big),
\]
which appears in the equation for $f^{l+1}_\pm$ in \eqref{iterated Vlasov.M}.
 Therefore, for each $i=1,2,3,$, we further write $\Eplo_{iS}$ as a sum of two parts:
\begin{align*}
    \Eplo_{iS} &= \sum_\pm((\Eplo)^{(1)}_{\pm,iS} -(\Eplo)^{(2)}_{\pm,iS}), \ \text{ for $i=1,2$, and }\\
    \Eplo_{iS} &= \sum_\pm((\Eplo)^{(1)}_{\pm,iS} + (\Eplo)^{(2)}_{\pm,iS}), \ \text{ for $i=3$, }
\end{align*}
where, with $a^{\mathbf{E}}_{\pm,i}$ defined as \eqref{aEi},
\begin{equation}\notag
\begin{split}
(\Eplo)^{(1)}_{\pm,iS}(t,x) &= \pm\int_{B^+(x;t)} dy \int_{\mathbb{R}^3} dv\, a^{\mathbf{E}}_{\pm,i}(v,\omega) \cdot (\pm \Elbf \pm \hat{v}_\pm \times \Blbf - m_\pm g \hat{e}_3) \frac{\fplo(t - |x - y|, y, v)}{|x - y|} \\
&\quad\mp \int_{B^+(x;t)}  
       \frac{dy }{|y -x|} 
       \int_{\mathbb{R}^3} dv\ 
       \frac{\omega + \hat v_\pm}{1+ \hat v_\pm \cdot \omega}\,
    \left(\pm\Ep  \pm\hat v_\pm\times \Bp\right)\left(t-|x-y |,y\right)\cdot \nabla_v F_{\pm,\textup{st}}(y ,v)\\
&\eqdef (\Eplo)^{(1),\textup{acc}}_{\pm,iS}(t,x)+(\Eplo)^{(1),\textup{st}}_{\pm,iS}(t,x),\text{ and }\\
(\Eplo)^{(2)}_{\pm,iS}(t,x) &=\pm \int_{B^-(x;t)} dy \int_{\mathbb{R}^3} dv\, a^{\mathbf{E}}_{\pm,i}(v,\bar{\omega}) \cdot (\pm \Elbf \pm \hat{v}_\pm \times \Blbf - m_\pm g \hat{e}_3) \frac{\fplo(t - |x - y|, \bar{y}, v)}{|x - y|} \\
&\quad\mp \int_{B^-(x;t)}  
       \frac{dy }{|y -x|} 
       \int_{\mathbb{R}^3} dv\ 
       \frac{\bar{\omega} + \hat v_\pm}{1+ \hat v_\pm \cdot \bar{\omega}}\,
    \left(\pm\Ep  \pm\hat v_\pm\times \Bp\right)\left(t-|x-y |,\bar{y}\right)\cdot \nabla_v F_{\pm,\textup{st}}(\bar{y} ,v)\\
&\eqdef (\Eplo)^{(2),\textup{acc}}_{\pm,iS}(t,x)+(\Eplo)^{(2),\textup{st}}_{\pm,iS}(t,x),\text{ with }\bar{z}\eqdef (z_1,z_2,-z_3)^\top.
\end{split}
\end{equation}
{\color{black}
Analogously, for each $i=1,2,3$, we decompose the magnetic nonlinear term as
$$
\Bplo_{iS} = \sum_\pm\big((\Bplo)^{(1)}_{\pm,iS} - (\Bplo)^{(2)}_{\pm,iS}\big),
$$
following the sign convention of $\B_{\textup{par}}=\sum_{\iota=\pm}(\B^{(1)}_{\iota,\textup{par}}-\B^{(2)}_{\iota,\textup{par}})$ in \eqref{Bpar_half_final} and \eqref{BparS_half_final} (both species enter the $S$-terms with the same sign, cf.~Remark \ref{rmk.BS.GS}; in any case, only absolute upper bounds are used below, so the relative signs play no role). Here, with the magnetic $S$-kernel
\begin{equation}\label{aBi}
a^{\B}_{\pm,i}(v,\omega) \eqdef \nabla_v\left(\frac{(\omega\times \hat{v}_\pm)_i}{1+\hat{v}_\pm\cdot \omega}\right)
=\frac{\nabla_v\big[(\omega\times \hat{v}_\pm)_i\big]}{1+\hat{v}_\pm\cdot \omega}
-\frac{(\omega\times \hat{v}_\pm)_i\,\big(\omega-(\omega\cdot \hat{v}_\pm)\hat{v}_\pm\big)}{\vZ\,(1+\hat{v}_\pm\cdot \omega)^2},
\end{equation}
we define
\begin{equation}\notag
\begin{split}
(\Bplo)^{(1)}_{\pm,iS}(t,x) &= \int_{B^+(x;t)} dy \int_{\mathbb{R}^3} dv\, a^{\B}_{\pm,i}(v,\omega) \cdot (\pm \Elbf \pm \hat{v}_\pm \times \Blbf - m_\pm g \hat{e}_3)\, \frac{\fplo(t - |x - y|, y, v)}{|x - y|} \\
&\quad+ \int_{B^+(x;t)}
       \frac{dy }{|y -x|}
       \int_{\mathbb{R}^3} dv\
       \frac{(\omega \times \hat v_\pm)_i}{1+ \hat v_\pm \cdot \omega}\,
    \left(\pm\Epl  \pm\hat v_\pm\times \Bpl\right)\left(t-|x-y |,y\right)\cdot \nabla_v F_{\pm,\textup{st}}(y ,v)\\
&\eqdef (\Bplo)^{(1),\textup{acc}}_{\pm,iS}(t,x)+(\Bplo)^{(1),\textup{st}}_{\pm,iS}(t,x),\text{ and }\\
(\Bplo)^{(2)}_{\pm,iS}(t,x) &= \int_{B^-(x;t)} dy \int_{\mathbb{R}^3} dv\, a^{\B}_{\pm,i}(v,\bar{\omega}) \cdot (\pm \Elbf \pm \hat{v}_\pm \times \Blbf - m_\pm g \hat{e}_3)\, \frac{\fplo(t - |x - y|, \bar{y}, v)}{|x - y|} \\
&\quad+ \int_{B^-(x;t)}
       \frac{dy }{|y -x|}
       \int_{\mathbb{R}^3} dv\
       \frac{(\bar{\omega} \times \hat v_\pm)_i}{1+ \hat v_\pm \cdot \bar{\omega}}\,
    \left(\pm\Epl  \pm\hat v_\pm\times \Bpl\right)\left(t-|x-y |,\bar{y}\right)\cdot \nabla_v F_{\pm,\textup{st}}(\bar{y} ,v)\\
&\eqdef (\Bplo)^{(2),\textup{acc}}_{\pm,iS}(t,x)+(\Bplo)^{(2),\textup{st}}_{\pm,iS}(t,x).
\end{split}
\end{equation}
As in the electric case, the ``acc''-terms and ``st''-terms refer to the nonlinear contributions arising from the dynamical source and the stationary source, respectively, and in the ``st''-terms no integration by parts in $v$ is performed, since the decay of $\nabla_v F_{\pm,\textup{st}}$ is already available from \eqref{steady state decay mom.deri}.
}

In the following subsections, we will derive decay-in-time estimates for each of the above decomposed components of $\Eplo$ and $\Bplo$.
\subsubsection{Decay Estimates for $\Eplo_{\textup{hom},i}$ and $\Bplo_{\textup{hom},i}$}\label{sec.homo.decay}
Recall that, by \eqref{Eparallel homo solution}, \eqref{E3}, \eqref{B3 homo solution}, \eqref{additional term B1}, and \eqref{Bi homo solution}, the homogeneous components $\Eplo_{\textup{hom},i}$ and $\Bplo_{\textup{hom},i}$ have the following representation: 
\begin{equation}\begin{split}\notag
\Eplo_{\textup{hom},i}(t, x) &=\frac{1}{4\pi t^2} \int_{\partial B(x; t)\cap \{y_3>0\}}  \left(t \mathcal{E}^1_{0i}(y) + \mathcal{E}_{0i}(y) + \nabla \mathcal{E}_{0i}(y) \cdot (y - x)\right) dS_y\\
&\qquad-\frac{1}{4\pi t^2} \int_{\partial B(x; t)\cap \{y_3<0\}}  \left(t \mathcal{E}^1_{0i}(\bar{y}) + \mathcal{E}_{0i}(\bar{y}) + \nabla \mathcal{E}_{0i}(\bar{y}) \cdot (\bar{y} - \bar{x})\right) dS_y,\text{ for }i=1,2,\\
  \Eplo_{\textup{hom},3}(t,x)&=  \frac{1}{4\pi t^2} \int_{\partial B(x; t)\cap \{y_3>0\}}  \left(t \mathcal{E}^1_{03}(y) + \mathcal{E}_{03}(y) + \nabla \mathcal{E}_{03}(y) \cdot (y - x)\right) dS_y\\
   &\qquad +\frac{1}{4\pi t^2} \int_{\partial B(x; t)\cap \{y_3<0\}}  \left(t \mathcal{E}^1_{03}(\bar{y}) + \mathcal{E}_{03}(\bar{y}) + \nabla \mathcal{E}_{03}(\bar{y}) \cdot (\bar{y} - \bar{x})\right) dS_y\\
    &\qquad-2 \sum_{\iota=\pm}\iota\int_{B(x;t)\cap \{y_3=0\}} \int_\rth \frac{\fploio(t-|y-x|,y_\parallel,0,v)}{|y-x|}dvdy_\parallel,\\
   \Bplo_{\textup{hom},i}(t,x)&=\frac{1}{4\pi t^2} \int_{\partial B(x; t)\cap \{y_3>0\}}  \left(t \mathcal{B}^1_{0i}(y) +\mathcal{B}_{0i}(y) + \nabla \mathcal{B}_{0i}(y) \cdot (y - x)\right) dS_y\\
&\qquad+\frac{1}{4\pi t^2} \int_{\partial B(x; t)\cap \{y_3<0\}}  \left(t \mathcal{B}^1_{0i}(\bar{y}) + \mathcal{B}_{0i}(\bar{y}) + \nabla \mathcal{B}_{0i}(\bar{y}) \cdot (\bar{y} - \bar{x})\right) dS_y\\&\qquad+2\sum_{\pm}\pm\int_{B(x;t)\cap \{y_3=0\}} \int_\rth \frac{\hat{v}_j\fplo(t-|y-x|,y_\parallel,0,v)}{|y-x|}dvdy_\parallel,\text{ for }i,j=1,2\text{ with }j\ne i\\
\Bplo_{\textup{hom},3}(t,x) &=\frac{1}{4\pi t^2} \int_{\partial B(x; t)\cap \{y_3>0\}}  \left(t \mathcal{B}^1_{03}(y) +\mathcal{B}_{03}(y) + \nabla \mathcal{B}_{03}(y) \cdot (y - x)\right) dS_y\\
&\qquad-\frac{1}{4\pi t^2} \int_{\partial B(x; t)\cap \{y_3<0\}}  \left(t \mathcal{B}^1_{03}(\bar{y}) + \mathcal{B}_{03}(\bar{y}) + \nabla \mathcal{B}_{03}(\bar{y}) \cdot (\bar{y} - \bar{x})\right) dS_y.
\end{split}\end{equation}
Without loss of generality, we make the decay estimates for the following integrals:
\begin{equation}\label{decay est. b01}
    \frac{1}{4\pi t^2} \int_{\partial B(x; t)\cap \{y_3>0\}}  \left(t \mathcal{B}^1_{01}(y) +\mathcal{B}_{01}(y) + \nabla \mathcal{B}_{01}(y) \cdot (y - x)\right) dS_y,
\end{equation} and 
\begin{equation}\label{neu term}
    \int_{B(x;t)\cap \{y_3=0\}} \int_\rth \frac{\fplo(t-|y-x|,y_\parallel,0,v)}{|y-x|}dvdy_\parallel,
\end{equation} since $|\hat{v}|\le 1.$
Indeed, the integral \eqref{neu term} has the same upper bound as that of $ib2$ terms, whose decay estimate will be given in Section \ref{sec.ib2 decay}. We omit it here. 

Now, we establish a linear-in-time decay estimate for the integral \eqref{decay est. b01}. To this end, we assume that the initial data $\mathcal{B}^1_{01}(y)$ and $\mathcal{B}_{01}(y)$ (as well as the other components  $\mathcal{B}^1_{0i}(y)$, $\mathcal{B}_{0i}(y)$, $\mathcal{E}^1_{0i}(y)$ and $\mathcal{E}_{0i}(y)$ for $i=1,2,3$) are compactly supported in the region $|y| \le R_0$, for some $R_0 > 0$. 
We perform the standard change of variables $y = x + t\omega$, where $\omega \in \mathbb{S}^2$. Then $y - x = t\omega$ and $dS_y = t^2\, d\omega$, so the integral becomes
\begin{multline}\label{utx}
u(t,x)\eqdef \frac{1}{4\pi t^2} \int_{\partial B(x; t)\cap \{y_3 > 0\}} \left( t\, \mathcal{B}^1_{01}(y) + \mathcal{B}_{01}(y) + \nabla \mathcal{B}_{01}(y) \cdot (y - x) \right) dS_y \\
= \frac{1}{4\pi} \int_{\mathbb{S}^2} \chi_{R_0}(x + t\omega)  1_{\{(x + t\omega)_3 > 0\}} \left( t\, \mathcal{B}^1_{01}(x + t\omega) + \mathcal{B}_{01}(x + t\omega) + t\, \nabla \mathcal{B}_{01}(x + t\omega) \cdot \omega \right) d\omega,
\end{multline}
where $\chi_{R_0}$ denotes the characteristic function of the ball $B(0; R_0)$.
Define
$$
\Omega_t(x) \eqdef \left\{ \omega \in \mathbb{S}^2 : x + t\omega \in B(0; R_0),\ (x + t\omega)_3 > 0 \right\}.
$$
Then the above expression reduces to
$$
\frac{1}{4\pi} \int_{\Omega_t(x)} \left( t\, \mathcal{B}^1_{01}(x + t\omega) + \mathcal{B}_{01}(x + t\omega) + t\, \nabla \mathcal{B}_{01}(x + t\omega) \cdot \omega \right) d\omega.
$$
Next, let
$$
M \eqdef \sup_{|z| \le R_0} \left( |\mathcal{B}^1_{01}(z)| + |\mathcal{B}_{01}(z)| + |\nabla \mathcal{B}_{01}(z)| \right) < \infty,
$$
so that the integrand is pointwise bounded by $M(1+t)$. Therefore, we first obtain a simple upper bound $M(1+t)$ of $|u(t,x)| $ in \eqref{utx} for each $t>0$ and $x\in\mathbb{R}^3_+$, since $|\mathbb{S}^2|=4\pi$ and $\chi_{R_0} 1_{(x+t\omega)_3>}\le 1.$ 

It remains to estimate the surface measure of the integration domain $\Omega_t(x)$.
To this end, we observe that the condition $|x + t\omega| \le R_0$ defines a spherical cap on $\mathbb{S}^2$. 
Fixing $x \in \mathbb{R}^3$, the inequality
$$
|x + t\omega|^2 = |x|^2 + 2t\, x \cdot \omega + t^2 \le R_0^2
$$
implies
$$
-\frac{x}{|x|} \cdot \omega \ge \frac{t^2 + |x|^2 - R_0^2}{2t|x|} \eqdef R(t,x).
$$Note that if $R(t,x)>1$, then there is no such $\omega\in\mathbb{S}^2$ exists and  hence $\Omega_t(x)$ becomes empty. Therefore, we only consider the case that $R(t,x)\le 1,$ which provides another restriction that
$$(t-|x|)^2\le R_0^2.$$ Namely, if $(t,x)$ satisfies $(t-|x|)^2> R_0^2,$ then $|\Omega_t(x)|=0$ and hence the integral \eqref{utx} is zero. Now define the opening angle $\theta_{t,x}$ of the spherical cap $\Omega_t(x)$. Note that the radius of the spherical cap has been normalized to $1.$  Thus the surface area of the spherical cap is defined as 
$$|\Omega_t(x)|= 2\pi(1-\cos\theta_{t,x}),$$ and the opening angle $\theta_{t,x}$ is defined through $\omega_0\in\mathbb{S}^2$ which satisfies
$$\cos\theta_{t,x}=-\frac{x}{|x|} \cdot \omega_0 =\min\{R(t,x),1\}.$$ Therefore, we have
\begin{multline*}
    |\Omega_t(x)|= 2\pi (1-\cos\theta_{t,x})=2\pi \left(1+\frac{x}{|x|} \cdot \omega_0\right)
    =2\pi \left(1+\max\left\{-\frac{t^2 + |x|^2 - R_0^2}{2t|x|},-1\right\}\right)\\
    =\max\left\{\pi \frac{R_0^2-(t-|x|)^2}{t|x|},0\right\}.
\end{multline*}
Combining the pointwise bound on the integrand and the measure of $\Omega_t(x)$, we obtain that the integral $u(t,x)$ in \eqref{utx} is bounded from above as $$ |u(t,x)| \le \frac{1}{4\pi}M(1+t)\pi \max\left\{\frac{R_0^2-(t-|x|)^2}{t|x|},0\right\}\begin{cases}
     \lesssim \frac{M(1+t)}{t|x|}, \text{ if }(t-|x|)^2\le R_0^2. \\
  = 0, \text{ otherwise}. 
 \end{cases}$$ Therefore, it suffices to consider the case  $(t-|x|)^2\le R_0^2$ from now on, since the integral becomes trivial, otherwise.
 
We now split the case into two: $t>\frac{3}{2}R_0$ and $t\le \frac{3}{2}R_0.$ If $t> \frac{3}{2}R_0$, then since
$t-|x|\le R_0$, we have
$$\frac{1}{|x|}\le \frac{1}{t-R_0}.$$Suppose $t=sR_0$ for some $s>\frac{3}{2}. $ Then 
$$\frac{1}{t-R_0} = \frac{1}{(s-1)R_0}=\frac{s}{s-1}\frac{1}{t}\le 3\frac{1}{t},$$ since $\frac{s}{s-1}<3 $ uniformly for any $s>\frac{3}{2}. $ Furthermore, since $t>\frac{3}{2}R_0,$ we have 
$$3\frac{1}{t}\le \left(3+\frac{2}{R_0}\right)\frac{1}{1+t}.$$
 Therefore, in this region, we have
  $$ |u(t,x)| \lesssim \frac{M(1+t)}{t(t-R_0)}\lesssim \frac{3M(1+t)}{t^2}\lesssim \left(3+\frac{2}{R_0}\right)^2\frac{M}{3(1+t)},\text{ for }t>\frac{3}{2}R_0. $$ 
 On the other hand, if $t\le \frac{3}{2}R_0$, then note that $|u(t,x)|$ in \eqref{utx} is simply bounded from above by $M(1+t)$, and hence by $M(1+\frac{3}{2}R_0)$.
 Altogether, we obtain \begin{equation}
     \label{utx upper bound}|u(t,x)|\lesssim \frac{M}{1+t},\text{ for any }t>0\text{ and }x\in\mathbb{R}^3_+.
 \end{equation} By choosing $M$ sufficiently small such that $M\ll \min\{m_-,m_+\}g$, we obtain 
$$(1+t)|u(t,x)|\ll\min\{m_-,m_+\}g .$$

 This completes the proof of the linear-in-time decay estimate for the integral \eqref{decay est. b01}, and hence establishes the corresponding linear decay for the homogeneous solutions $\Eplo_{\textup{hom},i}$ and $\Bplo_{\textup{hom},i}$.

\subsubsection{Decay Estimates for $(\Eplo)_{\pm,ib1}$  and $(\Bplo)_{\pm,ib1}$ } 
Now we consider the relativistic radiation contribution $(\Eplo)_{\pm,ib1}$  and $(\Bplo)_{\pm,ib1}$ from the initial data $\fplo(0,\cdot,\cdot).$
Recall that 
\begin{equation*}
  (\Eplo)^{(1)}_{\pm,ib1}(t,x)  =\pm\int_{\partial B(x;t)\cap \{y_3>0\}} \frac{dS_{y}}{|y-x|}\int_\rth dv\ \left((\delta_{ij})^\top_{i=1,2,3}-\frac{(\omega+\hat{v}_\pm)(\hat{v}_\pm)_j}{1+\hat{v}_\pm\cdot \omega}\right)\omega^j\fplo(0,y,v),
\end{equation*} with the standard Einstein summation convention. We aim to prove linear-in-$t$ decay of the boundary integral expression
$
\sum_\pm (\Eplo)^{(1)}_{\pm,ib1}(t,x)$
assuming $\fplo(0,y,v)$ decays fast in both $y$ and $v$, and where
 $\omega = \frac{y-x}{|y-x|} \in \mathbb{S}^2$ is the outward unit normal at the sphere of radius $t$, and 
$\hat{v}_\pm = \frac{v}{\sqrt{m_\pm^2+|v|^2}}$. 
As in \eqref{rrcn condition}, we follow the notation
$$
K_{ij}^{(\pm)}(w,v) \eqdef \left(\delta_{ij} - \frac{(\omega+\hat{v}_\pm)(\hat{v}_\pm)_j}{1+\hat{v}_\pm \cdot \omega}\right) \omega^j. 
$$By the estimate \eqref{eq.ib1 kernel}, we obtain that 
$$|K_{ij}^{(\pm)}(w, \hat{v}_\pm)|\lesssim \frac{v^0_\pm}{m_\pm}.$$
Note that $|y-x| = t$, since $y \in \partial B(x;t)$, and hence
$$
|(\Eplo)^{(1)}_{\pm,ib1}(t,x)|\lesssim \frac{1}{t} \int_{\partial B(x;t)\cap\{y_3>0\}} dS_{y} \int_{\mathbb{R}^3} dv\, \frac{v^0_\pm}{m_\pm}|\fplo(0,y,v)|.
$$
Recall that the initial perturbation $f^{\textup{in}}_\pm$ satisfies the decay assumption \eqref{f initial condition} and hence
$$
|\fplo(0,y,v)| \lesssim e^{-\frac{\beta}{2} |y_\parallel|}e^{-\frac{\beta}{2} \vZ } e^{-\frac{1}{2}m_\pm g\beta y_{3}}.
$$
Then, we obtain
\begin{equation}
    \label{Eploib1 first bound}
|(\Eplo)^{(1)}_{\pm,ib1}(t,x)| \lesssim \frac{1}{t}\frac{\langle m_\pm\rangle}{\beta^4 m_\pm} \int_{\partial B(x;t)\cap\{y_3>0\}} dS_{y}\, e^{-\frac{\beta}{2} |y_\parallel|}e^{-\frac{1}{2}m_\pm g\beta y_{3}} ,
\end{equation}
since 
\begin{equation}\begin{split}\label{additional beta decay}
    &\int_{\rth} dv\  \vZ e^{-\frac{\beta}{4} \vZ}
     = \int_{\rth} dv\  \sqrt{m_\pm^2+|v|^2} e^{-\frac{\beta}{4} \sqrt{m_\pm^2+|v|^2}}\\
      & = 4\pi \int_0^\infty d|v|\  |v|^2\sqrt{m_\pm^2+|v|^2} e^{-\frac{\beta}{4} \sqrt{m_\pm^2+|v|^2}}
       = 4\pi \int_0^\infty \frac{|v|d|v|}{\sqrt{m_\pm^2+|v|^2}}\  |v|(m_\pm^2+|v|^2) e^{-\frac{\beta}{4} \sqrt{m_\pm^2+|v|^2}}\\
       &=4\pi \int_{m_\pm}^\infty dz\  \sqrt{z^2-m_\pm^2}z^2 e^{-\frac{\beta}{4} z}\le 4\pi \int_0^\infty dz\  z^3 e^{-\frac{\beta}{4} z}
       =\frac{1024\pi}{\beta^4} \int_0^\infty dz'\  z'^3 e^{-z'}\approx \frac{1}{\beta^4}, 
 \end{split}\end{equation}where we made the change of variables $|v|\mapsto z\eqdef \sqrt{m_\pm^2+|v|^2}$ and then made another change of variables $z\mapsto z'\eqdef \frac{\beta}{4}z$.    
 Note that on $\partial B(x;t)$ we have $|x-y|=t.$ We consider the surface integral
\begin{equation}\label{9fullint}
I\eqdef \int_{\partial B(x;t)\cap\{y_3>0\}} dS_{y}\, e^{-\frac{\beta}{2} |y_\parallel|} e^{-\frac{1}{2} m_\pm g \beta y_3},
\end{equation}
where \( y = x + t\omega \), and \( \omega = (\sin\theta\cos\phi, \sin\theta\sin\phi, \cos\theta) \). Then we have
$$
y_\parallel = x_\parallel + t\omega_\parallel = (x_1 + t\sin\theta\cos\phi,\; x_2 + t\sin\theta\sin\phi),
\quad
y_3 = x_3 + t\cos\theta.
$$
We compute the pointwise bound for the integrand:
\begin{align*}
|y_\parallel|^2 &= |x_\parallel + t\omega_\parallel|^2 
= |x_\parallel|^2 + 2t\, x_\parallel \cdot \omega_\parallel + t^2 \sin^2\theta \\
&= |x_\parallel|^2 + 2t \sin\theta (x_1\cos\phi + x_2\sin\phi) + t^2 \sin^2\theta.
\end{align*}
Since \( x_1\cos\phi + x_2\sin\phi \ge -|x_\parallel| \), we have
$$
|y_\parallel|^2 \ge (t\sin\theta - |x_\parallel|)^2,
 \text{ and hence } 
|y_\parallel| \ge \left| t\sin\theta - |x_\parallel| \right|.
$$
Therefore,
\begin{equation}
e^{ -\frac{\beta}{2} |y_\parallel| } \le \exp\left( -\frac{\beta}{2} \left| t\sin\theta - |x_\parallel| \right| \right),
\end{equation}
and
$$
e^{ -\frac{1}{2} m_\pm g \beta y_3 } = e^{ -\frac{1}{2} m_\pm g \beta (x_3 + t\cos\theta) }.
$$
Therefore, the full integrand in spherical coordinates is bounded as
\begin{align*}
t^2 \sin\theta\, e^{ -\frac{\beta}{2} |y_\parallel| } e^{ -\frac{1}{2} m_\pm g \beta y_3 }
&\le
t^2 \sin\theta\, \exp\left( -\frac{\beta}{2} \left| t\sin\theta - |x_\parallel| \right| \right) \cdot e^{ -\frac{1}{2} m_\pm g \beta (x_3 + t\cos\theta) },
\end{align*}
and the full integral \eqref{9fullint} is bounded by
\begin{multline*}
    I\le 
\int_0^{2\pi} d\phi \int_0^{\frac{\pi}{2}} d\theta\,
t^2 \sin\theta\, \exp\left( -\frac{\beta}{2} \left| t\sin\theta - |x_\parallel| \right| \right) e^{ -\frac{1}{2} m_\pm g \beta (x_3 + t\cos\theta) }\\
\le 2\pi \int_0^1 dk\ t^2 e^{-\frac{1}{2} m_\pm g \beta (x_3 + tk) }\le \frac{4\pi t}{ m_\pm g \beta }  e^{-\frac{1}{2} m_\pm g \beta x_3 } \left(1-e^{-\frac{1}{2} m_\pm g \beta t }\right),
\end{multline*}where we made a change of variables $\theta\mapsto k\eqdef \cos\theta$.

Putting it all together, under the decay condition \eqref{f initial condition} of the initial data, we obtain for any $t\ge 0$ and $x\in\mathbb{R}^3_+,$
\begin{equation}\label{Eploib1 to be improved}
|(\Eplo)^{(1)}_{\pm,ib1}(t,x)| \lesssim  \frac{\langle m_\pm\rangle }{ m_\pm^2 g \beta^5 } e^{-\frac{1}{2} m_\pm g \beta x_3 } \left(1-e^{-\frac{1}{2} m_\pm g \beta t }\right).
\end{equation}
One can further improve this bound \eqref{Eploib1 to be improved} to a linearly decaying in time upper bound estimate by using the compact-support-in-$x$ and decay-in-$v$ assumption.
In this case, by the estimate \eqref{Eploib1 first bound}, we have
\begin{equation}
    \label{Eploib1 second bound}|(\Eplo)^{(1)}_{\pm,ib1}(t,x)| \lesssim \frac{\langle m_\pm\rangle}{\beta^4 m_\pm} \frac{1}{t}\int_{\partial B(x;t)\cap\{y_3>0\}} dS_{y}\, \mathcal{M}(y) .
\end{equation} Now, define the following integral:
$$\bar{u}(t,x)\eqdef\frac{1}{4\pi t}\int_{\partial B(x;t)\cap\{y_3>0\}} dS_{y}\, \mathcal{M}(y). $$ Note that this integral $\bar{u}(t,x)$ is the same as $u(t,x)$ of \eqref{utx} if the integrand $$\left( t\, \mathcal{B}^1_{01}(y) + \mathcal{B}_{01}(y) + \nabla \mathcal{B}_{01}(y) \cdot (y - x) \right)$$ in \eqref{utx} is now replaced by $t\mathcal{M}(y).$ Then the same estimate can be made for $\bar{u}$ as that of $u(t,x)$ in Section \ref{sec.homo.decay}, since $t\mathcal{M}(y) $ is also assumed to be compactly supported. Thus, by \eqref{utx upper bound}, we have 
$$|\bar{u}(t,x)|\lesssim \frac{\sup_{y\in\mathbb{R}^3_+}|\mathcal{M}(y)|}{1+t}.$$ Therefore, by \eqref{Eploib1 second bound}, we obtain
\begin{equation}\notag|(\Eplo)^{(1)}_{\pm,ib1}(t,x)| \lesssim \frac{\langle m_\pm\rangle}{\beta^4 m_\pm} \frac{\sup_{y\in\mathbb{R}^3_+}|\mathcal{M}(y)|}{1+t}.
\end{equation}By choosing $\beta\gg1 $ sufficiently large such that 
$$\frac{\langle m_\pm\rangle}{\beta^4 m_\pm}\sup_{y\in\mathbb{R}^3_+}|\mathcal{M}(y)| \ll \min\{m_-,m_+\}g,$$ we obtain
\begin{equation}\notag |(1+t)(\Eplo)^{(1)}_{\pm,ib1}(t,x)|\ll \min\{m_-,m_+\}g.
\end{equation} Here note that we do not need the smallness on $\sup_{y\in\mathbb{R}^3_+}|\mathcal{M}(y)|$, as we can choose $\beta$ sufficiently large in this case. 
\begin{remark}[Under the Initial Radiation Charge Neutrality Condition]\label{RRCN} The compact-support-in-$x$ assumption can be replaced by the following weaker condition, called the \textit{initial relativistic radiation charge neutrality condition}:\begin{equation}
    \label{rrcn condition}
\sup_{t\ge 1,\ x\in\mathbb{R}^3_+} \left| \int_{|x - y| = t} dS_y \int_{\mathbb{R}^3} dv\, \left( K_{ij}^{(+)} f_+(0, y, v) - K_{ij}^{(-)} f_-(0, y, v) \right) \right| \le \min\{m_-,m_+\}\frac{g}{128},
\end{equation}
where $K_{ij}^{(\pm)}$ denotes the relativistic projection tensor associated with each species, defined by $$
K_{ij}^{(\pm)}(w,v) \eqdef \delta_{ij} - \frac{(\omega_i + (\hat{v}_\pm)_i)(\hat{v}_\pm)_j}{1 + \omega \cdot \hat{v}_\pm},
$$
where $ \hat{v}_\pm = \frac{v}{\sqrt{m_\pm^2 + |v|^2}} $ denotes the normalized relativistic velocity, and $ \omega=\frac{y-x}{|y-x|} \in \mathbb{R}^3 $ is a fixed reference direction.  This condition prevents the emergence of unbounded transverse field components arising from the initial charge imbalance and is essential to closing the nonlinear decay estimates. 

In this scenario, for $t\in[0,1],$ we further obtain from \eqref{Eploib1 to be improved} that 
$$
\sup_{x\in\mathbb{R}^3_+}(1+t)|(\Eplo)^{(1)}_{\pm,ib1}(t,x)|\bigg|_{t\in[0,1]} \lesssim  \frac{\langle m_\pm\rangle }{ m_\pm^2 g \beta^5 }  .
$$Choosing sufficiently large $\beta>0$ such that $\frac{\langle m_\pm\rangle}{\beta^5 m_\pm^2 g}\ll \min\{m_-,m_+\}g,$  we have $$
(1+t)|(\Eplo)^{(1)}_{\pm,ib1}(t,x)|\bigg|_{t\in[0,1]} \ll \min\{m_-,m_+\}g  .
$$ On the other hand, if $t\ge 1,$ we use the \textit{initial relativistic radiation charge neutrality condition} \eqref{rrcn condition} to obtain 
$$
|\sum_\pm(\Eplo)^{(1)}_{\pm,ib1}(t,x)|\bigg|_{t\ge 1}  \le \frac{1}{t}\min\{m_-,m_+\} \frac{g}{128}\le \frac{1}{(1+t)}\min\{m_-,m_+\} \frac{g}{64}.
$$Therefore, we observe that $|\sum_\pm(\Eplo)^{(1)}_{\pm,ib1}(t,x)|$ is decaying linearly in time under the additional neutrality assumption \eqref{rrcn condition}.
\end{remark}

The same estimates also hold for $(\Bplo)^{(1)}_{\pm,ib1}$ and $(\Bplo)^{(2)}_{\pm,ib1}$ as well as $(\Eplo)^{(2)}_{\pm,ib1}$, as long as we have similar kernel estimates with the same upper-bound (up to constant). In the rest of the proof, we make the kernel estimates.

In general, we will first have an upper bound of the kernel $\left|\frac{(|(\hat{v}_\pm)|^2-1)(\hat{v}_\pm+\omega)}{(1+\hat{v}_\pm\cdot \omega)^2}\right|$ in terms of $v.$ 
Note that if $\hat{v}_\pm\cdot \omega \ge -\delta $ for some constant $\delta \in [-1,1),$ then we have $$\left|\frac{(|(\hat{v}_\pm)|^2-1)(\hat{v}_\pm+\omega)}{(1+\hat{v}_\pm\cdot \omega)^2}\right| \le 2 (1-\delta)^{-2}. $$
Indeed the term $1+\hat{v}_\pm\cdot \omega $ is singular at $\omega=-\frac{(\hat{v}_\pm)}{|(\hat{v}_\pm)|}$ and this is the worst-case scenario in terms of the upper-bound estimates. At $\omega=-\frac{(\hat{v}_\pm)}{|(\hat{v}_\pm)|},$ we observe that the singularity cancels out as $$
   \left|\frac{\hat{v}_\pm+\omega}{1+\hat{v}_\pm\cdot \omega} \bigg|_{\omega=-\frac{(\hat{v}_\pm)}{|(\hat{v}_\pm)|}}\right|=\left|\frac{(\hat{v}_\pm)-\frac{(\hat{v}_\pm)}{|(\hat{v}_\pm)|}}{1-|(\hat{v}_\pm)|}\right|=1.$$
On the other hand, observe that we have another cancellation
\begin{equation}
    \label{kernel estimate 1}\left|\frac{|(\hat{v}_\pm)|^2-1}{1+\hat{v}_\pm\cdot \omega}\right|\le\left|\frac{|(\hat{v}_\pm)|^2-1}{1+\hat{v}_\pm\cdot \omega}\bigg|_{\omega=-\frac{(\hat{v}_\pm)}{|(\hat{v}_\pm)|}}\right|=\left|\frac{|(\hat{v}_\pm)|^2-1}{1-|(\hat{v}_\pm)|}\right| = \left|1+|(\hat{v}_\pm)|\right|\le 2.
\end{equation}
In order to see that $\frac{\hat{v}_\pm+\omega}{1+\hat{v}_\pm\cdot \omega} $ is not singular for any $\omega\in\mathbb{S}^2$, we decompose the sphere $\mathbb{S}^2 $ around the vector $z\eqdef -\frac{(\hat{v}_\pm)}{|(\hat{v}_\pm)|}$ and consider the decomposition of polar angle $\phi\in [0,\pi]$ into $[0,\epsilon)$ and $[\epsilon,\pi]$ such that $$\omega\cdot z=-\frac{\omega\cdot (\hat{v}_\pm)}{|(\hat{v}_\pm)|}=\cos\phi. $$ 
Then we observe that a further orthogonal decomposition gives \begin{equation}\label{ortho decomp}
    \frac{\hat{v}_\pm+\omega}{1+\hat{v}_\pm\cdot \omega} = \frac{(\hat{v}_\pm)+(\omega\cdot \frac{(\hat{v}_\pm)}{|(\hat{v}_\pm)|})\frac{(\hat{v}_\pm)}{|(\hat{v}_\pm)|}+\omega_\perp}{1-|(\hat{v}_\pm)|\cos\phi} =\frac{(\hat{v}_\pm)-\cos\phi\frac{(\hat{v}_\pm)}{|(\hat{v}_\pm)|}+\omega_\perp}{1-|(\hat{v}_\pm)|\cos\phi} ,
\end{equation}where $\omega_\perp\cdot (\hat{v}_\pm)=0.$ Then if $\phi \in [0,\epsilon),$ we have$$\left|\frac{\omega_\perp}{1-|(\hat{v}_\pm)|\cos\phi}\right|\le \frac{|\sin\phi|}{|1-|(\hat{v}_\pm)|\cos\phi|}\le \frac{\epsilon}{1-|(\hat{v}_\pm)|}.$$On the other hand, if $\phi \in [\epsilon,\pi],$ we have
$$\left|\frac{\omega_\perp}{1-|(\hat{v}_\pm)|\cos\phi}\right|\le \frac{1}{|1-|(\hat{v}_\pm)|\cos\epsilon|}=\frac{1}{|1-|(\hat{v}_\pm)|+2|(\hat{v}_\pm)|\sin^2(\frac{\epsilon}{2})|}.$$
Indeed, we let $\sin\phi = x $ and find the maximal value of 
$f(\sin\phi)=\frac{|\sin\phi|}{|1-|(\hat{v}_\pm)|\cos\phi|}$ at the critical point $x$ for $\phi\in\ [0,\pi/2].$ Note that
\begin{multline*}
    f'(x)= \frac{1-|(\hat{v}_\pm)|\sqrt{1-x^2} - x^2\frac{|(\hat{v}_\pm)|}{\sqrt{1-x^2}}}{(1-|(\hat{v}_\pm)|\sqrt{1-x^2})^2}= \frac{\sqrt{1-x^2}-|(\hat{v}_\pm)|(1-x^2) - x^2|(\hat{v}_\pm)|}{\sqrt{1-x^2}(1-|(\hat{v}_\pm)|\sqrt{1-x^2})^2}\\=\frac{\sqrt{1-x^2}-|(\hat{v}_\pm)|}{\sqrt{1-x^2}(1-|(\hat{v}_\pm)|\sqrt{1-x^2})^2}.
\end{multline*}
It becomes zero when $x=\sqrt{1-|(\hat{v}_\pm)|^2}.$ Then the maximal value for $F_\pm$ is 
\begin{equation}
    \label{bound for v+w term}f(x)\le f(\sqrt{1-|(\hat{v}_\pm)|^2}) = \frac{\sqrt{1-|(\hat{v}_\pm)|^2}}{1-|(\hat{v}_\pm)|^2}=\frac{1}{\sqrt{1-|(\hat{v}_\pm)|^2}}=\frac{1}{\sqrt{1-\left|\frac{v}{\sqrt{m_\pm^2+|v|^2}}\right|^2}}=\frac{\sqrt{m_\pm^2+|v|^2}}{m_\pm}.
\end{equation}
We also have
\begin{multline}
    \label{bound for v+w term 2}\left|\frac{(\hat{v}_\pm)-\cos\phi\frac{(\hat{v}_\pm)}{|(\hat{v}_\pm)|}}{1-|(\hat{v}_\pm)|\cos\phi} \right|\le |(\hat{v}_\pm)|\left|\frac{1-\cos\phi\frac{1}{|(\hat{v}_\pm)|}}{1-|(\hat{v}_\pm)|\cos\phi} \right|\le |(\hat{v}_\pm)|+|\cos\phi|\left|\frac{1-|(\hat{v}_\pm)|^2}{1-|(\hat{v}_\pm)|\cos\phi} \right|\\
    \le |(\hat{v}_\pm)|+|1+|(\hat{v}_\pm)||\le 3 .
\end{multline}
Altogether we conclude that for any $\omega \in\mathbb{S}^2$
\begin{equation}\label{ET kernel estimate}
    \left|\frac{(|(\hat{v}_\pm)|^2-1)(\hat{v}_\pm+\omega)}{(1+\hat{v}_\pm\cdot \omega)^2}\right|\lesssim \frac{\sqrt{m_\pm^2+|v|^2}}{m_\pm}=\frac{\vZ}{m_\pm}.
\end{equation}

Then, for the magnetic field, by using \eqref{ortho decomp}--\eqref{bound for v+w term 2} again, we have \begin{equation}
    \label{eq.ib1 kernel}\left|\left((\delta_{ij})^\top_{i=1,2,3}-\frac{(\omega+\hat{v}_\pm)(\hat{v}_\pm)_j}{1+\hat{v}_\pm\cdot \omega}\right)\omega^j\right|\le 1+\frac{\sqrt{m_\pm^2+|v|^2}}{m_\pm}\le 2\frac{\sqrt{m_\pm^2+|v|^2}}{m_\pm}.
\end{equation} 

On the other hand, regarding the electric field,  we define $$a^{\E}_{\pm,i}(v,\omega)=\frac{(\partial_{v_i}v-(\hat{v}_\pm)_i(\hat{v}_\pm))}{(\vZ) (1+\hat{v}_\pm\cdot \omega)}-\frac{(\omega_i+(\hat{v}_\pm)_i)(\omega-(\omega\cdot (\hat{v}_\pm))(\hat{v}_\pm))}{(\vZ) (1+\hat{v}_\pm\cdot \omega)^2}\eqdef a^{(1)}_{\pm,i}+a^{(2)}_{\pm,i}.$$ We need to have an upper bound of the kernel $|a_i^E|.$ For $a^{(2)}_{\pm,i},$ we use \eqref{ortho decomp}--\eqref{bound for v+w term 2} and obtain first
\begin{equation}\label{aEfirst}
    \frac{\omega+(\hat{v}_\pm)}{(\vZ) (1+\hat{v}_\pm\cdot \omega)}\le \frac{1}{m_\pm}.
\end{equation} Then note that \begin{equation*}
    \omega-(\omega\cdot(\hat{v}_\pm))(\hat{v}_\pm)=\left(\omega\cdot\frac{(\hat{v}_\pm)}{|(\hat{v}_\pm)|}\right)\frac{(\hat{v}_\pm)}{|(\hat{v}_\pm)|}+\omega_\perp-(\omega\cdot(\hat{v}_\pm))(\hat{v}_\pm)=-\cos\phi\frac{(\hat{v}_\pm)}{|(\hat{v}_\pm)|}+\omega_\perp +\cos\phi |(\hat{v}_\pm)|(\hat{v}_\pm)
\end{equation*}following the orthogonal decomposition as in \eqref{ortho decomp}. Thus, 
$$|\omega-(\omega\cdot(\hat{v}_\pm))(\hat{v}_\pm)|\le  \left|\cos\phi \frac{\hat{v}_\pm}{|\hat{v}_\pm|}(|\hat{v}_\pm|^2-1)\right|+|\omega_\perp| \le |\cos\phi||(\hat{v}_\pm)|^2-1|+|\sin\phi|.$$ Thus, following the bounds \eqref{bound for v+w term}, we have
$$\frac{|\omega-(\omega\cdot(\hat{v}_\pm))(\hat{v}_\pm)|}{ (1+\hat{v}_\pm\cdot \omega)}\le \frac{|\cos\phi||(\hat{v}_\pm)|^2-1|+|\sin\phi|}{ (1-|(\hat{v}_\pm)|\cos\phi)}\le (1+|\hat{v}_\pm| )+\frac{\sqrt{m_\pm^2+|v|^2}}{m_\pm}\le 3 \frac{\sqrt{m_\pm^2+|v|^2}}{m_\pm}.$$
Together with \eqref{aEfirst}, we have
\begin{equation}
\label{a2}|a^{(2)}_{\pm,i}|\le\frac{3\sqrt{m_\pm^2+|v|^2}}{m_\pm^2}.
\end{equation}
Now, regarding $a^{(1)}_{\pm,i}$, a simple calculation gives
\begin{multline*}
    \frac{(\partial_{v_i}v-(\hat{v}_\pm)_i(\hat{v}_\pm))}{(\vZ) (1+\hat{v}_\pm\cdot \omega)}\le \frac{2}{(\vZ) (1-|(\hat{v}_\pm)|)}=\frac{2}{ \left(\sqrt{m_\pm^2+|v|^2}-|v|\right)}=\frac{2}{m_\pm^2}\left(\sqrt{m_\pm^2+|v|^2}+|v|\right)\\\le \frac{4}{m_\pm^2} \sqrt{m_\pm^2+|v|^2}.
\end{multline*}
Thus, we have 
\begin{equation}
\label{a1}|a^{(1)}_{\pm,i}|\le\frac{4\sqrt{m_\pm^2+|v|^2}}{m_\pm^2}.
\end{equation}

This completes the estimates for $(\Eplo)_{\pm,ib1}$ and $(\Bplo)_{\pm,ib1}$.


\subsubsection{Decay Estimates for $(\Eplo)_{\pm,ib2}$ }\label{sec.ib2 decay}Lastly, we consider the contribution $(\Eplo)_{\pm,ib2}$  from the boundary profile $\fplo(t,x,v)$ at $x_3=0$ for the electric field. To obtain desired estimates, we have to estimate the following term 
\begin{equation}
    \label{Eb2 kernel estimate}(0,0,1)^\top-\frac{(\omega+\hat{v}_\pm)(\hat{v}_\pm)_3}{1+\hat{v}_\pm\cdot \omega}\textup{ and }(0,0,1)^\top-\frac{(\bar{\omega}+(\hat{v}_\pm))(\hat{v}_\pm)_3}{1+(\hat{v}_\pm)\cdot \bar{\omega}},
\end{equation} where $\bar{\omega}=(\omega_1,\omega_2,-\omega_3)^\top.$ For each, we use \eqref{ortho decomp}--\eqref{bound for v+w term 2} (and the latter one with $\bar{w}$ replacing $w$) and obtain that in both cases we have
\begin{equation}\label{3.13}\left|(0,0,1)^\top-\frac{(\omega+\hat{v}_\pm)(\hat{v}_\pm)_3}{1+\hat{v}_\pm\cdot \omega}\right|,\ \left|(0,0,1)^\top-\frac{(\bar{\omega}+(\hat{v}_\pm))(\hat{v}_\pm)_3}{1+(\hat{v}_\pm)\cdot \bar{\omega}}\right| \le 1+|(\hat{v}_\pm)_3|\frac{\sqrt{m_\pm^2+|v|^2}}{m_\pm}.\end{equation}

Since the perturbation $\fplo$ from the steady-state satisfies the zero inflow boundary condition for the inflow direction $v_3 \ge 0$ at $x_3=0,$ we will obtain the following upper bound for $(\Eplo)^{(1)}_{\pm,ib2}$ and $(\Eplo)^{(2)}_{\pm,ib2}$ via the following kernel estimates for \eqref{Eb2 kernel estimate}-\eqref{3.13} and the decay estimate \eqref{fplo decay est}:
\begin{align*}\notag
   &\langle t\rangle \left( |(\Eplo)^{(1)}_{\pm,ib2}(t,x)|+|(\Eplo)^{(2)}_{\pm,ib2}(t,x)|\right)\\
   &\le  2\int_{B(x;t)\cap \{y_3=0\}} \frac{dy_\parallel}{|y-x|}\int_{v_3\le 0} dv\  \left(1+|(\hat{v}_\pm)_3|\frac{v^0_\pm}{m_\pm}\right)\langle t\rangle|\fplo(t-|x-y|,y_\parallel,0,v)|\\
   &\lesssim \int_{B(x;t)\cap \{y_3=0\}} \frac{dy_\parallel}{|y-x|}\int_{v_3\le 0} dv\  \frac{v^0_\pm}{m_\pm}\frac{\langle t\rangle}{\langle t-|x-y|\rangle}\frac{4}{\beta}C_{f_\pm^{\textup{in}},G_\pm}e^{-\frac{\beta}{2} |y_{\parallel}|}e^{-\frac{\beta}{4}\vZ} \\
   &\lesssim \frac{c_{\pm,\beta}C_{f_\pm^{\textup{in}},G_\pm}}{\beta m_\pm}\int_{B(x;t)\cap \{y_3=0\}} \frac{dy_\parallel}{|y-x|}\frac{\langle t\rangle}{\langle t-|x-y|\rangle}e^{-\frac{\beta}{2} |y_{\parallel}|},
   \end{align*}where $c_{\pm,\beta}\eqdef \int_{\rth} dv\  \vZ e^{-\frac{\beta}{4} \vZ}\lesssim \beta^{-4}$ by \eqref{additional beta decay} and $C_{f_\pm^{\textup{in}},G_\pm}$ is defined by \eqref{def.CfG}.
   We split the integral region $B(x;t)$ into two: $|x-y|< 1$ and $|x-y|\ge 1. $ If $|x-y|\ge 1,$ then the following inequality holds uniformly: 
$$\langle t \rangle \le \sqrt{2}\langle t-|x-y|\rangle|x - y|.$$ Therefore, we obtain that if $|x-y| \ge 1,$ 
$$ \int_{B(x;t)\cap \{y_3=0\}\cap |x-y|\ge 1} \frac{dy_\parallel}{|y-x|}\frac{\langle t\rangle}{\langle t-|x-y|\rangle}e^{-\frac{\beta}{2} |y_{\parallel}|}\lesssim \int_{\mathbb{R}^2} dy_\parallel\ e^{-\frac{\beta}{2} |y_{\parallel}|}\lesssim 1.$$
   On the other hand, if $|x-y|<1, $ we further note that 
   $\langle t\rangle \le \sqrt{2} \langle t-|x-y|\rangle\langle |x-y|\rangle,$ and also note that $|x-y|=\sqrt{|x_\parallel-y_\parallel|^2+x_3^2}$ if $y_3=0.$ Then we obtain
   \begin{align*}
&       \int_{B(x;t)\cap \{y_3=0\}\cap |x-y|< 1} \frac{dy_\parallel}{|y-x|}\frac{\langle t\rangle}{\langle t-|x-y|\rangle}e^{-\frac{\beta}{2} |y_{\parallel}|}
       \\
       &\lesssim \int_{\sqrt{|x_\parallel-y_\parallel|^2+x_3^2}< \min\{1,t\}} \frac{dy_\parallel}{\sqrt{|x_\parallel-y_\parallel|^2+x_3^2}}\left\langle \sqrt{|x_\parallel-y_\parallel|^2+x_3^2} \right\rangle e^{-\frac{\beta}{2} |y_{\parallel}|}\\
       &\lesssim  \int_{\sqrt{|x_\parallel-y_\parallel|^2+x_3^2}< \min\{1,t\}} \frac{dy_\parallel}{|x_\parallel-y_\parallel|}e^{-\frac{\beta}{2} |y_{\parallel}|}\lesssim \int_{\mathbb{R}^2} \frac{dy_\parallel}{|y_\parallel-x_\parallel|}e^{-\frac{\beta}{2}|y_\parallel|}\approx\frac{1}{\beta} \left\langle \frac{\beta}{2}x_\parallel\right\rangle^{-1}\lesssim \frac{1}{\beta} ,\end{align*} by \eqref{asymptotics}. Therefore, we conclude that $$
    \langle t\rangle \left( |(\Eplo)^{(1)}_{\pm,ib2}(t,x)|+|(\Eplo)^{(2)}_{\pm,ib2}(t,x)|\right)
    \lesssim \frac{C_{f_\pm^{\textup{in}},G_\pm}}{\beta^5 m_\pm}(1+\beta^{-1}),
 $$by \eqref{additional beta decay}. If $\beta>1$ is chosen sufficiently large such that $\min\{m_-^2,m_+^2\}g\beta^5 C_{f_\pm^{\textup{in}},G_\pm}\gg 1$, then we have $$\langle t\rangle \left( |(\Eplo)^{(1)}_{\pm,ib2}(t,x)|+|(\Eplo)^{(2)}_{\pm,ib2}(t,x)|\right)\ll \min\{m_-,m_+\}g. $$   This completes the estimates for $|(\Eplo)^{(1)}_{\pm,ib2}(t,x)|$ and $|(\Eplo)^{(2)}_{\pm,ib2}(t,x)|$ boundary contribution terms.

\subsubsection{Decay Estimates for $(\Eplo)_{\pm,iS}$ \Black{and $(\Bplo)_{\pm,iS}$}}

One of the main challenges in establishing temporal decay estimates for $\Eplo$ lies in handling the nonlinear term $(\Eplo)^{\textup{acc}}_{\pm,iS}$ and the inhomogeneous stationary source term $(\Eplo)^{\textup{st}}_{\pm,iS}$. Our strategy is to control the nonlinear term $(\Eplo)^{\textup{acc}}_{\pm,iS}$ by using the linear-in-time decay estimate for $\fplo$ established in Section~\ref{sec.fplo.decay}, together with the uniform boundedness of the total fields $\Elbf$ and $\Blbf$ provided by the bootstrap assumption \eqref{apriori_EB} and the steady-state estimate \eqref{steady state L infty}.  For the inhomogeneous stationary source term $(\Eplo)^{\textup{st}}_{\pm,iS}$, we employ the linear decay-in-time estimates for the perturbations $\Epl$ and $\Bpl$ from \eqref{apriori_EB}.

Namely, we observe that
\begin{multline}\label{acc.iS.integrand}
    \langle t\rangle|(\Eplo)^{(1),\textup{acc}}_{\pm,iS}(t,x)|\le \int_{B^+(x;t)} dy \int_{\mathbb{R}^3} dv\, |a^{\mathbf{E}}_{\pm,i}(v,\omega)| |\pm\Elbf \pm \hat{v}_\pm \times \Blbf - m_\pm g \hat{e}_3| \frac{\langle t\rangle\fplo(t - |x - y|, y, v)}{|x - y|} \\
    \le  \frac{63}{8}m_\pm g\int_{B^+(x;t)} dy \int_{\mathbb{R}^3} dv\, \frac{v^0_\pm}{m_\pm^2}  \frac{\langle t\rangle}{\langle t-|x-y|\rangle|x - y|}\frac{4}{\beta} C_{f_\pm^{\textup{in}},G_\pm}e^{-\frac{\beta}{2} |y_{\parallel}|}e^{-\frac{\beta}{4}(\vZ+m_\pm gy_3)},
\end{multline}by the kernel estimates \eqref{a2}--\eqref{a1}, the decay estimate for $\fplo$ \eqref{fplo decay est}, and the uniform bounds for $\E_{\textup{st}},$ $\B_{\textup{st}},$ $\mathcal{E}^l,$ and $\mathcal{B}^l$ in \eqref{steady state L infty} and \eqref{apriori_EB}  where we define 
\begin{equation}
\label{def.CfG}
C_{f_\pm^{\textup{in}},G_\pm}\eqdef \|\mathrm{w}_{\pm,\beta} f^{\textup{in}}_\pm\|_{L^\infty_{x,v}}
  +C \notag\min\{m_-,m_+\}\frac{g}{8}\|\mathrm{w}^2_{\pm,\beta}(\cdot,0,\cdot)\nabla_{x_\parallel,v}G_\pm(\cdot,\cdot)\|_{L^\infty_{x_\parallel,v}}.
\end{equation}
We split the integral region $B^+(x;t)$ into two: $|x-y|< 1$ and $|x-y|\ge 1. $ If $|x-y|\ge 1,$ then the following inequality holds uniformly: 
$$\langle t \rangle \le \sqrt{2}\langle t-|x-y|\rangle|x - y|.$$ Therefore, we obtain that
\begin{align*}
    & \frac{63}{8}m_\pm g\int_{B^+(x;t)\cap |x-y|\ge 1} dy \int_{\mathbb{R}^3} dv\, \frac{v^0_\pm}{m_\pm^2}  \frac{\langle t\rangle}{\langle t-|x-y|\rangle|x - y|}\frac{4}{\beta} C_{f_\pm^{\textup{in}},G_\pm}e^{-\frac{\beta}{2} |y_{\parallel}|}e^{-\frac{\beta}{4}(\vZ+m_\pm gy_3)}\\*
    & \le \frac{63\sqrt{2}g}{2\beta m_\pm}C_{f_\pm^{\textup{in}},G_\pm}\int_{B^+(x;t)\cap |x-y|\ge 1} dy \int_{\mathbb{R}^3} dv\, v^0_\pm e^{-\frac{\beta}{2} |y_{\parallel}|}e^{-\frac{\beta}{4}(\vZ+m_\pm gy_3)}\\
    & \le \frac{63\sqrt{2}g}{2\beta m_\pm}C_{f_\pm^{\textup{in}},G_\pm}c_{\pm,\beta}\int_{\mathbb{R}^3_+} dy e^{-\frac{\beta}{2} |y_{\parallel}|}e^{-\frac{\beta}{4}m_\pm gy_3} 
     \lesssim \frac{1}{\beta^4 m_\pm^2}C_{f_\pm^{\textup{in}},G_\pm}c_{\pm,\beta},
\end{align*}
where $c_{\pm,\beta}$ is defined as $$c_{\pm,\beta}=\int_{\mathbb{R}^3} dv\, v^0_\pm e^{-\frac{\beta}{4}\vZ},$$ and satisfies $c_{\pm,\beta}\approx \beta^{-4}$ by \eqref{additional beta decay}. On the other hand, if $|x-y|<1,$ then we further make a change of variables $y\mapsto z\eqdef y-x$ and then another change of variables to spherical coordinates $z\mapsto (r,\theta,\phi)$ such that we have
\begin{multline*}
     \frac{63}{8}m_\pm g\int_{B^+(x;t)\cap |x-y|< 1} dy \int_{\mathbb{R}^3} dv\, \frac{v^0_\pm}{m_\pm^2}  \frac{\langle t\rangle}{\langle t-|x-y|\rangle|x - y|}\frac{4}{\beta} C_{f_\pm^{\textup{in}},G_\pm}e^{-\frac{\beta}{2} |y_{\parallel}|}e^{-\frac{\beta}{4}(\vZ+m_\pm gy_3)}\\
     =\frac{63\pi}{4}m_\pm g\int_0^{\min\{1,t\}}dr \int_0^\pi d\phi\ \sin\phi\int_{\mathbb{R}^3} dv\, \frac{v^0_\pm}{m_\pm^2}  \frac{\langle t\rangle r}{\langle t-r\rangle}\frac{4}{\beta} C_{f_\pm^{\textup{in}},G_\pm}e^{-\frac{\beta}{2} r|\sin\phi|}e^{-\frac{\beta}{4}(\vZ+m_\pm g(r\cos\phi+x_3))}.
     \end{multline*}
Using the inequality that 
$$\langle t\rangle \le \sqrt{2} \langle t-r\rangle \langle r\rangle,$$ we have
\begin{multline*}
    \frac{63\pi}{4}m_\pm g\int_0^{\min\{1,t\}}dr \int_0^\pi d\phi\ \sin\phi\int_{\mathbb{R}^3} dv\, \frac{v^0_\pm}{m_\pm^2}  \frac{\langle t\rangle r}{\langle t-r\rangle}\frac{4}{\beta} C_{f_\pm^{\textup{in}},G_\pm}e^{-\frac{\beta}{2} r|\sin\phi|}e^{-\frac{\beta}{4}(\vZ+m_\pm g(r\cos\phi+x_3))}\\*
    \le \frac{63\sqrt{2}\pi g}{\beta m_\pm} C_{f_\pm^{\textup{in}},G_\pm}c_{\pm,\beta}\int_0^{\min\{1,t\}}dr \int_0^\pi d\phi\ \sin\phi \langle r\rangle^2\le \frac{504\sqrt{2}\pi g}{\beta m_\pm} C_{f_\pm^{\textup{in}},G_\pm}c_{\pm,\beta}.
     \end{multline*}
Altogether, we conclude that 
\begin{equation}
    \label{est.acc.iS}\langle t\rangle|(\Eplo)^{(1),\textup{acc}}_{\pm,iS}(t,x)|\lesssim \frac{C_{f_\pm^{\textup{in}},G_\pm}}{\beta^8 m_\pm^2}\left(1+m_\pm g \beta^3\right).
\end{equation} 
Choosing $\beta>1$ sufficiently large such that $\min\{m_-,m_+\}\times\min\{g \beta^3,\beta^2\}\gg 1 $, we obtain 
$$  \langle t\rangle|(\Eplo)^{(1),\textup{acc}}_{\pm,iS}(t,x)|\ll \min\{m_-,m_+\}g,$$ which ensures \eqref{apriori_EB} for the decomposed piece $(\Eplo)^{(1),\textup{acc}}_{\pm,iS}.$ The other term $(\Eplo)^{(2),\textup{acc}}_{\pm,iS}$ follows exactly the same estimate. 

On the other hand, regarding the inhomogeneous stationary source term $(\Eplo)^{\textup{st}}_{\pm,iS}$, we observe that
\begin{align*}
    &\langle t\rangle|(\Eplo)^{(1),\textup{st}}_{\pm,iS}(t,x)|\\
    &\le \langle t\rangle \int_{B^+(x;t)}  
       \frac{dy }{|y -x|} 
       \int_{\mathbb{R}^3} dv\ 
      \left| \frac{\omega + \hat v_\pm}{1+ \hat v_\pm \cdot \omega}\right|
    \left|\left(\Ep^l  +\hat v_\pm\times \Bp^l\right)\left(t-\frac{|x-y |}{c},y\right)\right|| \nabla_v F_{\pm,\textup{st}}(y ,v)| \\
    &\lesssim m_\pm gC_{G_\pm}\int_{B^+(x;t)} dy \int_{\mathbb{R}^3} dv\, \frac{v^0_\pm}{m_\pm}  \frac{\langle t\rangle}{\langle t-|x-y|\rangle|x - y|}  e^{-\frac{\beta}{2} |y_{\parallel}|}e^{-\beta(\vZ+m_\pm gy_3)},
 \end{align*}by the kernel estimate \eqref{aEfirst}, decay of the momentum derivative of the stationary solution \eqref{steady state decay mom.deri},  and the uniform bounds for  $\mathcal{E}^l$ and $\mathcal{B}^l$ in \eqref{apriori_EB}   where the constant $C_{G_\pm}$ is defined as
$$C_{G_\pm}\eqdef  C\|\mathrm{w}^2_{\pm,\beta}(\cdot,0,\cdot)\nabla_{x_\parallel,v}G_\pm(\cdot,\cdot)\|_{L^\infty_{x_\parallel,v}}.$$  
Again, we split the integral region $B^+(x;t)$ into two: $|x-y|< 1$ and $|x-y|\ge 1. $ If $|x-y|\ge 1,$ then the following inequality holds uniformly: 
$$\langle t \rangle \le \sqrt{2}\langle t-|x-y|\rangle|x - y|.$$ Therefore, we obtain that
\begin{align*}
     &m_\pm g  C_{G_\pm}\int_{B^+(x;t)\cap |x-y|\ge 1} dy \int_{\mathbb{R}^3} dv\, \frac{v^0_\pm}{m_\pm}  \frac{\langle t\rangle}{\langle t-|x-y|\rangle|x - y|}e^{-\frac{\beta}{2} |y_{\parallel}|}e^{-\beta(\vZ+m_\pm gy_3)}\\
    & \lesssim g  C_{G_\pm}\int_{B^+(x;t)\cap |x-y|\ge 1} dy \int_{\mathbb{R}^3} dv\, v^0_\pm e^{-\frac{\beta}{2} |y_{\parallel}|}e^{-\beta(\vZ+m_\pm gy_3)}\\
     &\lesssim gC_{G_\pm} c_{\pm,\beta}\int_{\mathbb{R}^3_+} dy\ e^{-\frac{\beta}{2} |y_{\parallel}|}e^{-\beta m_\pm gy_3} 
     \lesssim \frac{  C_{G_\pm}}{\beta^3 m_\pm}c_{\pm,\beta},
\end{align*}
where $c_{\pm,\beta}$ is defined as $c_{\pm,\beta}=\int_{\mathbb{R}^3} dv\, v^0_\pm e^{-\beta\vZ},$ and satisfies $c_{\pm,\beta}\approx \beta^{-4}$ by \eqref{additional beta decay}. On the other hand, if $|x-y|<1,$ then we further make a change of variables $y\mapsto z\eqdef y-x$ and then another change of variables to spherical coordinates $z\mapsto (r,\theta,\phi)$ such that we have
\begin{multline*}
m_\pm gC_{G_\pm}\int_{B^+(x;t)\cap |x-y|< 1} dy \int_{\mathbb{R}^3} dv\, \frac{v^0_\pm}{m_\pm}  \frac{\langle t\rangle}{\langle t-|x-y|\rangle|x - y|}  e^{-\frac{\beta}{2} |y_{\parallel}|}e^{-\beta (\vZ+m_\pm gy_3)}\\
     \approx m_\pm g  C_{G_\pm}\int_0^{\min\{1,t\}}dr \int_0^\pi d\phi\ \sin\phi\int_{\mathbb{R}^3} dv\, \frac{v^0_\pm}{m_\pm}  \frac{\langle t\rangle r}{\langle t-r\rangle}e^{-\frac{\beta}{2} r|\sin\phi|}e^{-\beta (\vZ+m_\pm g(r\cos\phi+x_3))}.
     \end{multline*}
Using the inequality that 
$$\langle t\rangle \le \sqrt{2} \langle t-r\rangle \langle r\rangle,$$ we have
\begin{multline*}
    m_\pm g  C_{G_\pm}\int_0^{\min\{1,t\}}dr \int_0^\pi d\phi\ \sin\phi\int_{\mathbb{R}^3} dv\, \frac{v^0_\pm}{m_\pm}  \frac{\langle t\rangle r}{\langle t-r\rangle}e^{-\frac{\beta}{2} r|\sin\phi|}e^{-\beta(\vZ+m_\pm g(r\cos\phi+x_3))}\\
    \lesssim g  C_{G_\pm} c_{\pm,\beta}\int_0^{\min\{1,t\}}dr \int_0^\pi d\phi\ \sin\phi \langle r\rangle^2\lesssim g  C_{G_\pm} c_{\pm,\beta}.
     \end{multline*}
Altogether, we conclude that 
\begin{equation}
    \label{est.st.iS}\langle t\rangle|(\Eplo)^{(1),\textup{st}}_{\pm,iS}(t,x)|\lesssim   \frac{C_{G_\pm}}{m_\pm \beta^7}\left(1+ m_\pm g \beta^3 \right).
\end{equation} 
Choosing $\beta>1$ sufficiently large such that $\min\{m_-,m_+\}\times\min\{g \beta^3,\beta^2\}\gg 1 $, we obtain 
$$  \langle t\rangle|(\Eplo)^{(1),\textup{st}}_{\pm,iS}(t,x)|\ll \min\{m_-,m_+\}g,$$ which ensures \eqref{apriori_EB} for the decomposed piece $(\Eplo)^{(1),\textup{st}}_{\pm,iS}.$ The other term $(\Eplo)^{(2),\textup{st}}_{\pm,iS}$ follows exactly the same estimate.  This completes the estimate for $(\Eplo)_{\pm,iS}.$ 

{\color{black}
We now turn to the magnetic nonlinear terms $(\Bplo)^{(j)}_{\pm,iS}$ for $j=1,2$ and $i=1,2,3$. Observe that the integrands of $(\Bplo)^{(j),\textup{acc}}_{\pm,iS}$ and $(\Bplo)^{(j),\textup{st}}_{\pm,iS}$ differ from those of the electric counterparts $(\Eplo)^{(j),\textup{acc}}_{\pm,iS}$ and $(\Eplo)^{(j),\textup{st}}_{\pm,iS}$ only through the kernels: $a^{\B}_{\pm,i}$ in place of $a^{\E}_{\pm,i}$ for the ``acc''-terms, and $\frac{(\omega\times\hat{v}_\pm)_i}{1+\hat{v}_\pm\cdot\omega}$ in place of $\frac{\omega+\hat{v}_\pm}{1+\hat{v}_\pm\cdot\omega}$ for the ``st''-terms. Therefore, it suffices to show that the magnetic kernels obey the same upper bounds as the electric ones; the decay estimates then follow verbatim.

We first estimate $a^{\B}_{\pm,i}$ in \eqref{aBi}. For the derivative in the first term, we note that, writing $(i,j,k)$ for the cyclic permutation of $(1,2,3)$ starting from $i$, we have $(\omega\times\hat{v}_\pm)_i=\omega_j(\hat{v}_\pm)_k-\omega_k(\hat{v}_\pm)_j$, and hence, by
$$\partial_{v_l}(\hat{v}_\pm)_m= \frac{\delta_{lm}(\vZ)^2-v_lv_m}{(\vZ)^3},$$
each entry of $\nabla_v[(\omega\times\hat{v}_\pm)_i]$ is a sum of two terms of the form $\pm\omega_j\frac{\delta_{lm}(\vZ)^2-v_lv_m}{(\vZ)^3}$, each bounded by $\frac{2}{\vZ}$ in absolute value (cf.\ the explicit computation for $i=3$: $\partial_{v_1}[\omega_1(\hat{v}_\pm)_2-\omega_2(\hat{v}_\pm)_1]=-\frac{\omega_1v_1v_2}{(\vZ)^3}-\frac{\omega_2((\vZ)^2-(v_1)^2)}{(\vZ)^3}$, and similarly for the other entries and components). Therefore,
\begin{equation}
    \label{aB.deri}
    \big|\nabla_v \big[(\omega\times\hat{v}_\pm)_i\big]\big|\le \frac{6}{\vZ},\qquad i=1,2,3.
\end{equation}
Since $\left|\frac{1}{1+\hat{v}_\pm\cdot \omega}\right|$ attains its maximum at $\omega=-\frac{\hat{v}_\pm}{|\hat{v}_\pm|}$, we obtain
\begin{equation}
    \label{aB.bound.1}
    \left|\frac{\nabla_v \big[(\omega\times\hat{v}_\pm)_i\big]}{1+\hat{v}_\pm\cdot \omega}\right|
    \le\frac{6}{\vZ\,(1-|\hat{v}_\pm|)}
    =\frac{6\,(1+|\hat{v}_\pm|)}{\vZ\,(1-|\hat{v}_\pm|^2)}
    =\frac{6\vZ\,(1+|\hat{v}_\pm|)}{(\vZ)^2-|v|^2}
    \le  \frac{12\vZ}{m_\pm^2}.
\end{equation}
For the second term of \eqref{aBi}, we use the orthogonal decomposition \eqref{ortho decomp} together with \eqref{bound for v+w term}, which yields
\begin{equation}\label{aB.st.kernel}
    \left|\frac{(\omega\times \hat{v}_\pm)_i}{1+\hat{v}_\pm\cdot \omega}\right| \le \frac{|\hat{v}_\pm|\,|\sin\phi|}{|1-|\hat{v}_\pm|\cos\phi|}\le |\hat{v}_\pm|\,\frac{\sqrt{m_\pm^2+|v|^2}}{m_\pm}\le \frac{\vZ}{m_\pm},
\end{equation}
and, as established below \eqref{aEfirst},
$$\frac{\big|\omega-(\omega\cdot \hat{v}_\pm)\hat{v}_\pm\big|}{1+\hat{v}_\pm\cdot \omega}\le 3\,\frac{\sqrt{m_\pm^2+|v|^2}}{m_\pm}.$$
Therefore,
\begin{equation}\label{aB.bound.2}
   \left| \frac{(\omega\times \hat{v}_\pm)_i\,\big(\omega-(\omega\cdot \hat{v}_\pm)\hat{v}_\pm\big)}{\vZ\,(1+\hat{v}_\pm\cdot \omega)^2}\right|\le \frac{1}{\vZ}\cdot\frac{\vZ}{m_\pm}\cdot\frac{3\vZ}{m_\pm}= \frac{3\vZ}{m_\pm^2}.
\end{equation}
Altogether, by \eqref{aB.bound.1} and \eqref{aB.bound.2}, we conclude that for any $\omega\in\mathbb{S}^2$ and $i=1,2,3$,
\begin{equation}
    \label{aB.final}
    \big|a^{\B}_{\pm,i}(v,\omega)\big|\lesssim  \frac{\vZ}{m_\pm^2},
\end{equation}
which is the same upper bound as that for the electric kernel $a^{\E}_{\pm,i}$ in \eqref{a2}--\eqref{a1}. Likewise, the ``st''-kernel bound \eqref{aB.st.kernel} coincides with the bound $\left|\frac{\omega+\hat{v}_\pm}{1+\hat{v}_\pm\cdot\omega}\right|\le\frac{\vZ}{m_\pm}$ used for $(\Eplo)^{(j),\textup{st}}_{\pm,iS}$ (cf.~\eqref{aEfirst}). The same bounds hold for the reflected kernels with $\bar{\omega}$ in place of $\omega$, since the estimates above are uniform in $\omega\in\mathbb{S}^2$.

Consequently, repeating the arguments leading to \eqref{est.acc.iS} and \eqref{est.st.iS} with the kernels replaced by their magnetic counterparts, we obtain, for $j=1,2$ and $i=1,2,3$,
\begin{equation}
    \label{est.acc.iS.B}
    \langle t\rangle\big|(\Bplo)^{(j),\textup{acc}}_{\pm,iS}(t,x)\big|\lesssim \frac{C_{f_\pm^{\textup{in}},G_\pm}}{\beta^8 m_\pm^2}\left(1+m_\pm g \beta^3\right),
    \qquad
    \langle t\rangle\big|(\Bplo)^{(j),\textup{st}}_{\pm,iS}(t,x)\big|\lesssim   \frac{C_{G_\pm}}{m_\pm \beta^7}\left(1+ m_\pm g \beta^3 \right).
\end{equation}
Choosing $\beta>1$ sufficiently large such that $\min\{m_-,m_+\}\times\min\{g \beta^3,\beta^2\}\gg 1$, as in the electric case, we obtain
$$  \langle t\rangle\big|(\Bplo)^{(j)}_{\pm,iS}(t,x)\big|\ll \min\{m_-,m_+\}g,\qquad j=1,2,\ i=1,2,3,$$
which ensures \eqref{apriori_EB} for the decomposed pieces $(\Bplo)^{(j)}_{\pm,iS}$. This completes the estimate for $(\Bplo)_{\pm,iS}$; in particular, the magnetic nonlinear contributions $\Bplo_{iS}$ are now included in the bound \eqref{decay bound for EB} of Lemma \ref{lem.dyna.EB}.
}

\subsubsection{Decay Estimates for $(\Eplo)_{\pm,iT}$ and $(\Bplo)_{\pm,iT}$ }
Recall that $\Eplo_{\pm,T}$ terms are written as 
\begin{equation*}
    (\Eplo)^{(1)}_{\pm,T}(t,x)=\mp\int_{B^+(x;t)} \frac{dy}{|y-x|^2}\int_\rth dv\ \frac{(|(\hat{v}_\pm)|^2-1)(\hat{v}_\pm+\omega)}{(1+\hat{v}_\pm\cdot \omega)^2} \fplo(t-|x-y|,y,v).
\end{equation*} In the followings, we split the cases into two: $t<1$ and $t\ge 1.$ 

Firstly, if $t<1$, we utilize the estimate \eqref{fplo decay est} and the kernel estimate \eqref{ET kernel estimate} to obtain
\begin{equation}\begin{split}\label{E5T estimate maxwell}|(\Eplo)^{(1)}_{\pm,iT}(t,x)|&\lesssim \frac{1}{m_\pm} \int_{B^+(x;t)} \frac{dy}{|y-x|^2}\int_\rth dv\ \vZ\fplo(t-|x-y|,y,v)\\
&\lesssim \frac{1}{m_\pm} \int_{B^+(x;t)} \frac{dy}{|y-x|^2}\int_\rth dv\ \vZ e^{-\frac{\beta}{2}(\vZ+m_\pm gy_3)}\\&\qquad\times\left(\|\mathrm{w}_{\pm,\beta} f^{\textup{in}}_\pm\|_{L^\infty_{x,v}}+\frac{C}{\beta}\|\mathrm{w}^2_{\pm,\beta}(\cdot,0,\cdot)\nabla_{x_\parallel,v}G_\pm(\cdot,\cdot)\|_{L^\infty_{x_\parallel,v}}\right) e^{-\frac{\beta}{2}|y_\parallel|}\\
&\approx \frac{1}{m_\pm} c_{\pm,\beta}\left(\|\mathrm{w}_{\pm,\beta} f^{\textup{in}}_\pm\|_{L^\infty_{x,v}}+\frac{C}{\beta}\|\mathrm{w}^2_{\pm,\beta}(\cdot,0,\cdot)\nabla_{x_\parallel,v}G_\pm(\cdot,\cdot)\|_{L^\infty_{x_\parallel,v}}\right) \\&\qquad\times  \int_{B^+(x;t)} \frac{dy}{|y-x|^2}e^{-\frac{m_\pm g\beta}{2}y_3}e^{-\frac{\beta}{2}|y_\parallel|},
\end{split}\end{equation}where $c_{\pm,\beta}$ is defined as 
$c_{\pm,\beta} \eqdef \int_\rth dv \ \vZ e^{-\frac{\beta}{2}\vZ}. $ Now we further split the integral domain into two: $|y-x|\le 1$ and $|y-x|>1.$

If $|y-x|> 1,$ we have
\begin{multline*}
   \frac{1}{m_\pm} c_{\pm,\beta} \int_{B^+(x;t)} \frac{dy}{|y-x|^2}e^{-\frac{m_\pm g\beta}{2}y_3}e^{-\frac{\beta}{2}|y_\parallel|}1_{\{|y-x|> 1\}}
\le  \frac{1}{m_\pm} c_{\pm,\beta} \int_{\rth} dye^{-\frac{m_\pm g\beta}{2}y_3}e^{-\frac{\beta}{2}|y_\parallel|}
   \lesssim \frac{1}{m_\pm^2g\beta^3}c_{\pm,\beta}  .
\end{multline*}
On the other hand, if $|y-x|\le 1,$ we further proceed as
\begin{align*}
&\frac{1}{m_\pm}c_{\pm,\beta} \int_{B^+(x;t)} \frac{dy}{|y-x|^2}e^{-\frac{m_\pm g\beta}{2}y_3}e^{-\frac{\beta}{2}|y_\parallel|}1_{\{|y-x|\le 1\}}\\
&\approx\frac{1}{m_\pm} c_{\pm,\beta}\int_{B(x;t)\cap \{z_3+x_3>0\}} \frac{dz}{|z|^2} e^{-\frac{m_\pm g\beta}{2} (z_3+x_3)}e^{-\frac{\beta}{2}|z_\parallel+x_\parallel|}1_{\{|z|\le 1\}}\\
&\approx\frac{1}{m_\pm} c_{\pm,\beta}\int_0^1dr\int_{\mathbb{S}^2}d\omega\  1_{\{(r\omega)_3+x_3>0\}} e^{-\frac{m_\pm g\beta}{2} ((r\omega)_3+x_3)}e^{-\frac{\beta}{2}|z_\parallel+x_\parallel|}
\lesssim \frac{1}{m_\pm}c_{\pm,\beta}.
\end{align*} 

Altogether, we conclude that for $i=1,2,3$,
\begin{equation}
    \label{E1T final maxwell}|(\Eplo)^{(1)}_{\pm,iT}(t,x)|
    \lesssim \frac{1}{m_\pm^2g\beta^7}(1+m_\pm g\beta^3)\left(\|\mathrm{w}_{\pm,\beta} f^{\textup{in}}_\pm\|_{L^\infty_{x,v}}+\frac{C}{\beta}\|\mathrm{w}^2_{\pm,\beta}(\cdot,0,\cdot)\nabla_{x_\parallel,v}G_\pm(\cdot,\cdot)\|_{L^\infty_{x_\parallel,v}}\right),
\end{equation}since we have the estimate \eqref{additional beta decay} for the coefficient $c_{\pm,\beta}$.

On the other hand, if $t\ge 1,$ by using the kernel estimate \eqref{ET kernel estimate} and the decay estimate \eqref{fplo decay est} for $\fplo$, we obtain that
\begin{align*}
  &  \langle t\rangle|(\Eplo)^{(1)}_{\pm,iT}(t,x)| \\*
  &\lesssim \int_{B^+(x;t)} dy \int_{\mathbb{R}^3} dv\, \frac{v^0_\pm}{m_\pm} \frac{\langle t\rangle\fplo(t - |x - y|, y, v)}{|x - y|^2} \\
    &\le  \int_{B^+(x;t)} dy \int_{\mathbb{R}^3} dv\, \frac{v^0_\pm}{m_\pm}  \frac{\langle t\rangle}{\langle t-|x-y|\rangle|x - y|^2}\frac{4}{\beta} C_{f_\pm^{\textup{in}},G_\pm}e^{-\frac{\beta}{2} |y_{\parallel}|}e^{-\frac{\beta}{4}(\vZ+m_\pm gy_3)}\\
   &\le  \left(\int_{|y-x|< 1}dy+\int_{\substack{|y-x|\ge 1 \\ |y-x|<t}}dy\right)  \int_{\mathbb{R}^3} dv\, \frac{v^0_\pm}{m_\pm}  \frac{\langle t\rangle}{\langle t-|x-y|\rangle|x - y|^2}\frac{4}{\beta} C_{f_\pm^{\textup{in}},G_\pm}e^{-\frac{\beta}{2} |y_{\parallel}|}e^{-\frac{\beta}{4}(\vZ+m_\pm gy_3)} \\
    &\le  \int_{|y-x|< 1}dy\int_{\mathbb{R}^3} dv\, \frac{v^0_\pm}{m_\pm}  \frac{\langle t\rangle}{\langle t-1\rangle|x - y|^2}\frac{4}{\beta} C_{f_\pm^{\textup{in}},G_\pm}e^{-\frac{\beta}{2} |y_{\parallel}|}e^{-\frac{\beta}{4}(\vZ+m_\pm gy_3)}\\
    &+\int_{\substack{|y-x|\ge 1 \\ |y-x|<t}}dy \int_{\mathbb{R}^3} dv\, \frac{v^0_\pm}{m_\pm}  \frac{\langle t\rangle}{\langle t-|x-y|\rangle|x - y|}\frac{4}{\beta} C_{f_\pm^{\textup{in}},G_\pm}e^{-\frac{\beta}{2} |y_{\parallel}|}e^{-\frac{\beta}{4}(\vZ+m_\pm gy_3)} \\
        &\le  \int_{|y-x|< 1}dy\int_{\mathbb{R}^3} dv\, \frac{v^0_\pm}{m_\pm}  \frac{(1+\sqrt{5})}{2|x - y|^2}\frac{4}{\beta} C_{f_\pm^{\textup{in}},G_\pm}e^{-\frac{\beta}{2} |y_{\parallel}|}e^{-\frac{\beta}{4}(\vZ+m_\pm gy_3)}\\
   & +\int_{\substack{|y-x|\ge 1 \\ |y-x|<t}}dy \int_{\mathbb{R}^3} dv\, \frac{v^0_\pm}{m_\pm}  \frac{\langle t\rangle}{\langle t-|x-y|\rangle|x - y|}\frac{4}{\beta} C_{f_\pm^{\textup{in}},G_\pm}e^{-\frac{\beta}{2} |y_{\parallel}|}e^{-\frac{\beta}{4}(\vZ+m_\pm gy_3)}
    \eqdef \text{I}+\text{II}, 
\end{align*}  where $C_{f_\pm^{\textup{in}},G_\pm}$ is defined as \eqref{def.CfG}. Here note that the integrand of the latter integral II is the same as that of \eqref{acc.iS.integrand} up to some constant and $g$. Therefore, the same estimate follows, and by \eqref{est.acc.iS} we obtain 
$$ \text{II}\lesssim C_{f_\pm^{\textup{in}},G_\pm}\left(\frac{1}{\beta^8 m_\pm^2 g}+ \frac{ 1}{\beta^5 m_\pm} \right).$$On the other hand, the integral I can be treated the same as the integral \eqref{E5T estimate maxwell} up to a minor correction on the coefficients, and hence the estimate \eqref{E1T final maxwell} follows as
$$\text{I}\lesssim \frac{1}{m_\pm g\beta^4}C_{f_\pm^{\textup{in}},G_\pm}.$$
Altogether, choosing $\beta>0$ sufficiently large, we obtain 
$$  \langle t\rangle|(\Eplo)^{(1)}_{\pm,iT}(t,x)|\ll \min\{m_-,m_+\}g,$$ which ensures \eqref{apriori_EB} for the decomposed piece $(\Eplo)^{(1)}_{\pm,iT}.$ The other term $(\Eplo)^{(2)}_{\pm,iT}$ follows exactly the same estimate. In addition, 
in order to conclude the same upper bound for the magnetic field $(\Bplo)^{(1)}_{\pm,iT}$ and $(\Bplo)^{(2)}_{\pm,iT}$ up to constant, we now make some kernel estimates as follows. We first note that
$$\frac{(1-|(\hat{v}_\pm)|^2)(\omega\times (\hat{v}_\pm))_i}{(1+\hat{v}_\pm\cdot \omega)^2}=\frac{(1-|(\hat{v}_\pm)|^2)(\omega_1(\hat{v}_\pm)_2-\omega_2(\hat{v}_\pm)_1)}{(1+\hat{v}_\pm\cdot \omega)^2}.$$
    By \eqref{kernel estimate 1}, we first have 
    $$\left|\frac{1-|(\hat{v}_\pm)|^2}{1+\hat{v}_\pm\cdot \omega}\right|\le 2.$$
Now, for the estimate of the remainder part $\frac{\omega_1(\hat{v}_\pm)_2-\omega_2(\hat{v}_\pm)_1}{1+\hat{v}_\pm\cdot \omega}$, 
 define $z\eqdef -\frac{(\hat{v}_\pm)}{|(\hat{v}_\pm)|}$ such that $$\omega\cdot z=-\frac{\omega\cdot (\hat{v}_\pm)}{|(\hat{v}_\pm)|}=\cos\phi. $$ 
Similarly to what we did in \eqref{ortho decomp}, we observe that \begin{equation}\notag 
    \left|\frac{(\omega\times (\hat{v}_\pm))_i}{1+\hat{v}_\pm\cdot \omega}\right| \le \frac{|(\hat{v}_\pm)||\sin\phi|}{|1-|(\hat{v}_\pm)|\cos\phi|}.
\end{equation} Define $f(\sin\phi)= \frac{|\sin\phi|}{|1-|(\hat{v}_\pm)|\cos\phi|}$. Then by \eqref{bound for v+w term}, we obtain $f(x)\le \frac{\sqrt{m_\pm^2+|v|^2}}{m_\pm}.$ 
Thus, \begin{equation}
    \label{wvcross kernel estimate}\left|\frac{(\omega\times (\hat{v}_\pm))_i}{1+\hat{v}_\pm\cdot \omega}\right|\le |(\hat{v}_\pm)|\frac{\sqrt{m_\pm^2+|v|^2}}{m_\pm}.
\end{equation}Altogether, we have
\begin{equation}\label{B35 kernel final}
    \left|\frac{(1-|(\hat{v}_\pm)|^2)(\omega\times (\hat{v}_\pm))_i}{(1+\hat{v}_\pm\cdot \omega)^2}\right|\le 2|(\hat{v}_\pm)|\frac{\sqrt{m_\pm^2+|v|^2}}{m_\pm}\le 2\frac{\sqrt{m_\pm^2+|v|^2}}{m_\pm}.
\end{equation}

\subsubsection{Final Upper-Bounds for $\Eplo$ and $\Bplo$}
Combining the previous estimates, we obtain the following lemma on the linear-in-time decay upper bound for $\Eplo$ and $\Bplo$:
\begin{lemma}\label{lem.dyna.EB}Fix $l\in\mathbb{N}$ and suppose \eqref{apriori_f}-\eqref{apriori_EB}  hold for $(\fplo,\Epl,\Bpl).$ Then $(\Eplo,\Bplo)$ satisfies 
\begin{equation}\label{decay bound for EB}
    \sup_{t \geq 0} \ \langle t \rangle \|(\Eplo, \Bplo)\|_{L^\infty} \leq \min\{m_+, m_-\} \frac{g}{16}.
\end{equation}
\end{lemma}
This bound guarantees the validity of \eqref{apriori_EB} at the $(l+1)$-th iteration level, provided that the parameter $\beta > 1$ is chosen sufficiently large. Consequently, the estimates \eqref{apriori_f}-\eqref{apriori_EB} are verified uniformly for all $l \in \mathbb{N}$, and thus remain valid in the limit as $l \to \infty$.

\subsection{An Alternative Orbital Stability Result}
In this section, we additionally introduce a weaker result, the orbital stability of the steady states with the J\"uttner-Maxwell upper bound, \textbf{without} assuming that the initial Cauchy data for field perturbations $\mathcal{B}^1_{01}(y)$ and $\mathcal{B}_{01}(y)$ (as well as the other components  $\mathcal{B}^1_{0i}(y)$, $\mathcal{B}_{0i}(y)$, $\mathcal{E}^1_{0i}(y)$ and $\mathcal{E}_{0i}(y)$ for $i=1,2,3$) are compactly supported in the region $|y| \le R_0$, for some $R_0 > 0$, nor assuming the \textit{Initial Radiation Charge Neutrality Condition} 
\eqref{rrcn condition}. We first state the alternative theorem on the stability:
\begin{theorem}[Orbital Stability]\label{thm.asymp.rth.orbi}
Let $(F_{\pm,\textup{st}}, \mathcal{E}_{\textup{st}}, \mathcal{B}_{\textup{st}})$ be the steady solution constructed in Theorem \ref{thm.ex.st}. 
Suppose positive parameters $(g,m_\pm, \bar \beta)$ satisfy $\bar \beta>0$, $\textcolor{black}{\min\{m_-,m_+\}g\gg 1}$ and $\min\{m_-,m_+\}\times\min\{g \bar\beta^3,\bar\beta^2\}\allowbreak\gg 1. $  
Let the initial perturbations $(f_\pm^{\textup{in}}, \mathcal{E}^{\textup{in}}, \mathcal{B}^{\textup{in}})$ satisfy the conditions of \eqref{initial E0i} and \eqref{f initial condition}. We do not assume that the initial field perturbations and their temporal derivatives (understood via the equations) are compactly supported in $x.$

Then we construct a unique classical solution to the dynamical problem \eqref{VM perturbations1}--\eqref{VM perturbations2} with the inflow boundary condition \eqref{inflow boundary} and the perfect conductor condition \eqref{perfect cond. boundary}, such that  
%
$$
|||f_\pm(t)|||<\infty,\quad(\mathcal{E}, \mathcal{B}) \in W^{1,\infty}([0,\infty)\times\mathbb{R}^3_+), \ \ \text{for all $t>0$.}
$$
Moreover, the solution does not grow in time 
\begin{align}
    &\sup_{t \ge 0}\  \left\| e^{\frac{\bar\beta}{2}|x_\parallel| + \frac{\bar\beta}{4}\vZ + \frac{\bar\beta}{4}m_\pm g x_3} f_\pm(t) \right\|_{L^\infty_{x,v}} \le C_M,\notag\\
    &\sup_{t \ge 0}\ \|(\mathcal{E}, \mathcal{B})(t)\|_{L^\infty_{x}} \le \min\{m_+, m_-\} \frac{g}{16}.\notag
\end{align}
Furthermore, the derivatives are controlled as 
\begin{align}\notag
 \|( v_\pm^0)^\ell \partial_t f_\pm(t)\|_{L^\infty} + |||f_\pm(t)||| 
  +  \|(\mathcal{E}, \mathcal{B})\|_{W^{1,\infty}_{t,x}([0,t]\times \mathbb{R}^3_+)} 
  &\lesssim_t 1 .
\end{align}
\end{theorem}

In the rest of this section, we introduce a modified bootstrap argument for the proof of the orbital stability, which replaces the bootstrap argument made for the asymptotic stability.

\subsubsection{Bootstrap Argument} 
For the proof of orbital stability (Theorem \ref{thm.asymp.rth.orbi}), it suffices to provide the following proposition via bootstrap argument; the rest of the proof including the regularity estimates and the proof of existence and uniqueness are the same as those for the main asymptotic stability (Theorem \ref{thm.asymp.rth}), and we omit them.  In the rest of the section, we prove the following main proposition on the boundedness of the perturbations:
\begin{proposition}\label{prop.dyna.boot.orbi}Suppose that $\beta > 1$ is sufficiently large such that  
$$
\min\{m_+^2, m_-^2\}g^2\beta \gg 1 \quad \text{and} \quad \min\{m_+^2, m_-^2\}\beta^4 \gg 1. 
$$ 
For any $l\in\mathbb{N}$, we have

\begin{align}\label{Ansatz for fl maxwell} 
        \sup_{t\ge 0}\left\|e^{\frac{\beta}{2} |x_{\parallel}|}e^{\frac{\beta}{4}(\vZ+m_\pm gx_3)}\fpl(t,\cdot,\cdot)\right\|_{L^\infty} &\le  \left(\|\mathrm{w}_{\pm,\beta} f^{\textup{in}}_\pm\|_{L^\infty_{x,v}}+\frac{C}{\beta}\|\mathrm{w}^2_{\pm,\beta}(x_\parallel,0,v)\nabla_{x_\parallel,v}G_\pm(x_\parallel,v)\|_{L^\infty_{x,v}(\gamma_-)}\right),
    \\\label{ansatz for EBl maxwell}
        \sup_{t\ge 0}\|(\Epl,\Bpl)\|_{L^\infty}&\le \min\{m_+,m_-\}\frac{g}{16},
    \end{align}
for some $\beta>1$ where the weight $\mathrm{w}=\mathrm{w}_{\pm,\beta}$ is defined as \eqref{weights.wholehalf}. 
\end{proposition} In the following sections, we fix $ l \in \mathbb{N} $ and assume that \eqref{Ansatz for fl maxwell}--\eqref{ansatz for EBl maxwell} hold at the iteration level $ (l) $. We then show that these same estimates remain valid at the next level $ (l+1) $, thereby closing the bootstrap argument.

\begin{proof}[Proof of Proposition \ref{prop.dyna.boot.orbi}]
  Proposition \ref{prop.dyna.boot.orbi} follows from Lemma \ref{lem.dyna.fplo.orbi} and Lemma \ref{lem.dyna.EB.orbi}, which will be established in the subsequent sections.
\end{proof}

Note that, as long as $\partial_{x_3}\mathbf{E}_3$, $\partial_{x_3}\mathbf{B}_1$, and $\partial_{x_3}\mathbf{B}_2$ exist as functions in $L^\infty([0,T]\times \Omega)$ at the sequential level for each $l\in\mathbb{N}$, their traces at the boundary are well-defined in the distributional sense. In this case, the fields satisfy the additional Neumann-type compatibility conditions on the boundary, as described in Remark \ref{remark.comp.cond}, in the weak sense:
\begin{align}\notag
(\partial_{x_3}\Eplo_3  - 4 \pi \rho^{l+1}_{\textup{pert}})\big|_{\partial\Omega}=0,\quad 
(\partial_{x_3}\Bplo_1  - 4 \pi J^{l+1}_{\textup{pert},2})\big|_{\partial\Omega}=0,\quad 
(\partial_{x_3}\Bplo_2 + 4 \pi J^{l+1}_{\textup{pert},1})\big|_{\partial\Omega}=0.
\end{align}

Given that the bootstrap ansatz \eqref{Ansatz for fl maxwell}-\eqref{ansatz for EBl maxwell} hold at the level of the sequential index $(l)$ for some fixed $l\in\mathbb{N},$ we will prove that the same estimates also hold at the sequential level of $(l+1)$ if $\beta$ is sufficiently large.

\subsubsection{Estimates for the Velocity Distribution Function}In this section, we prove the estimate \eqref{apriori_f} at the iteration level $(l+1)$.
\begin{lemma}\label{lem.dyna.fplo.orbi}
   Fix $l\in\mathbb{N}$ and suppose \eqref{ansatz for EBl maxwell} hold for $(\Epl,\Bpl).$ Then $\fplo$ satisfies  \begin{equation}
    \label{final estimate for flo maxwell}
    e^{\frac{\beta}{2}|x_\parallel|} e^{\frac{1}{4}\beta \vZ} e^{\frac{1}{4}m_\pm g\beta x_{3}}|\fplo(t,x,v)|\le   \left(\|\mathrm{w}_{\pm,\beta} f^{\textup{in}}_\pm\|_{L^\infty_{x,v}}+\frac{C}{\beta}\|\mathrm{w}^2_{\pm,\beta}(x_\parallel,0,v)\nabla_{x_\parallel,v}G_\pm(x_\parallel,v)\|_{L^\infty_{x,v}(\gamma_-)}\right).
\end{equation}
\end{lemma}

\begin{proof}
Since $\fplo$ satisfies \eqref{iterated Vlasov.M}, we can write $\fplo$ in the mild form as 
\begin{multline}\label{solution flo maxwell.orbi}
    \fplo(t,x,v)= 1_{t\le \tblo(t,x,v)}f^{\textup{in}}_\pm(\XSlo(0;t,x,v),\VSlo(0;t,x,v))\\\mp \int^t_{\max\{0,t-\tblo\}} \left(\Epl(s,\XSlo(s))+\hat{\mathcal{V}}^{l+1}_\pm(s)\times \Bpl(s,\XSlo(s))\right)\cdot \nabla_v F_{\pm,\textup{st}}(\XSlo(s),\VSlo(s))ds.
\end{multline}
Based on the bootstrap ansatz \eqref{Ansatz for fl maxwell}--\eqref{ansatz for EBl maxwell}, we provide a decay estimate for $\fplo.$ Given that $\|\mathrm{w}_{\pm,\beta}\nabla_v F_{\pm,\textup{st}}\|_{L^\infty}$ is bounded (see \eqref{decay of Fst}), by \eqref{ansatz for EBl maxwell}, \eqref{solution flo maxwell.orbi} and that $|\hat{\mathcal{V}}^{l+1}_\pm|\le 1$, we obtain that
\begin{equation}\begin{split}\label{fplo mid est}
    |\fplo(t,x,v)|&\le  1_{t\le \tblo(t,x,v)}|f^{\textup{in}}_\pm(\XSlo(0;t,x,v),\VSlo(0;t,x,v))|\\
   & +1_{t\le \tblo(t,x,v)}\int^t_0\ \min\{m_-,m_+\}\frac{g}{8} |\nabla_v F_{\pm,\textup{st}}(\XSlo(s),\VSlo(s))|ds\\
   & +1_{t> \tblo(t,x,v)}\int^t_{t-\tblo} \min\{m_-,m_+\}\frac{g}{8} |\nabla_v F_{\pm,\textup{st}}(\XSlo(s),\VSlo(s))|ds\\
  & \le  \frac{1_{t\le \tblo(t,x,v)}}{\mathrm{w}_{\pm,\beta}( \ZSlo(0 ; t, x, v))}\|\mathrm{w}_{\pm,\beta} f^{\textup{in}}_\pm\|_{L^\infty_{x,v}}\\
  & +1_{t\le \tblo(t,x,v)}\min\{m_-,m_+\}\frac{g}{8} \|\mathrm{w}_{\pm,\beta}\nabla_v F_{\pm,\textup{st}}\|_{L^\infty_{x,v}}\int^t_0 \frac{1}{\mathrm{w}_{\pm,\beta}( \ZSlo(s ; t, x, v))}ds\\
   &+1_{t> \tblo(t,x,v)}\min\{m_-,m_+\}\frac{g}{8} \|\mathrm{w}_{\pm,\beta}\nabla_v F_{\pm,\textup{st}}\|_{L^\infty_{x,v}}\int^t_{t-\tblo} \frac{1}{\mathrm{w}_{\pm,\beta}( \ZSlo(s ; t, x, v))}ds.
\end{split}\end{equation}
Using \eqref{exit time bound.whole} and \eqref{w comparison 3.whole}, we further have 
\begin{align*}\notag
    |\fplo(t,x,v)|
   &\le  1_{t\le \tblo(t,x,v)}e^{-\frac{1}{2}\beta v_\pm^0-\frac{1}{2}m_\pm g\beta x_3-\frac{\beta}{2}|x_{\parallel}|}\|\mathrm{w}_{\pm,\beta} f^{\textup{in}}_\pm\|_{L^\infty_{x,v}}\\
   &+1_{t\le \tblo(t,x,v)}\min\{m_-,m_+\}\frac{g}{8} \|\mathrm{w}_{\pm,\beta}\nabla_v F_{\pm,\textup{st}}\|_{L^\infty_{x,v}}e^{-\frac{1}{2}\beta v_\pm^0-\frac{1}{2}m_\pm g\beta x_3-\frac{\beta}{2}|x_{\parallel}|} t\\
   &+1_{t> \tblo(t,x,v)}\min\{m_-,m_+\}\frac{g}{8} \|\mathrm{w}_{\pm,\beta}\nabla_v F_{\pm,\textup{st}}\|_{L^\infty_{x,v}}e^{-\frac{1}{2}\beta v_\pm^0-\frac{1}{2}m_\pm g\beta x_3-\frac{\beta}{2}|x_{\parallel}|} \tblo\\
    &\le  e^{-\frac{1}{2}\beta v_\pm^0-\frac{1}{2}m_\pm g\beta x_3-\frac{\beta}{2}|x_{\parallel}|}\|\mathrm{w}_{\pm,\beta} f^{\textup{in}}_\pm\|_{L^\infty_{x,v}}\\
   &+\min\{m_-,m_+\}\frac{g}{8} \|\mathrm{w}_{\pm,\beta}\nabla_v F_{\pm,\textup{st}}\|_{L^\infty_{x,v}}e^{-\frac{1}{2}\beta v_\pm^0-\frac{1}{2}m_\pm g\beta x_3-\frac{\beta}{2}|x_{\parallel}|} \tblo\\
  & \le  e^{-\frac{1}{2}\beta v_\pm^0-\frac{1}{2}m_\pm g\beta x_3-\frac{\beta}{2}|x_{\parallel}|}\|\mathrm{w}_{\pm,\beta} f^{\textup{in}}_\pm\|_{L^\infty_{x,v}}\\
   &+\frac{2}{5} (v_\pm^0+m_\pm g\beta x_3)\|\mathrm{w}_{\pm,\beta}\nabla_v F_{\pm,\textup{st}}\|_{L^\infty_{x,v}}e^{-\frac{1}{2}\beta v_\pm^0-\frac{1}{2}m_\pm g\beta x_3-\frac{\beta}{2}|x_{\parallel}|} \\
   &\le  e^{-\frac{1}{2}\beta v_\pm^0-\frac{1}{2}m_\pm g\beta x_3-\frac{\beta}{2}|x_{\parallel}|}\|\mathrm{w}_{\pm,\beta} f^{\textup{in}}_\pm\|_{L^\infty_{x,v}}
   +\frac{8}{5\beta e}\|\mathrm{w}_{\pm,\beta}\nabla_v F_{\pm,\textup{st}}\|_{L^\infty_{x,v}}e^{-\frac{1}{4}\beta v_\pm^0-\frac{1}{4}m_\pm g\beta x_3-\frac{\beta}{2}|x_{\parallel}|} ,
\end{align*}where the last inequality is by the inequality that $xe^{-\frac{\beta}{2} x}\le \frac{4}{\beta e}e^{-\frac{\beta}{4}x}.$
Therefore, using \eqref{decay of Fst}, we conclude that
\begin{equation}
    \notag
    e^{\frac{\beta}{2}|x_\parallel|} e^{\frac{1}{4}\beta \vZ} e^{\frac{1}{4}m_\pm g\beta x_{3}}|\fplo(t,x,v)|\le   \left(\|\mathrm{w}_{\pm,\beta} f^{\textup{in}}_\pm\|_{L^\infty_{x,v}}+\frac{C}{\beta}\|\mathrm{w}^2_{\pm,\beta}(x_\parallel,0,v)\nabla_{x_\parallel,v}G_\pm(x_\parallel,v)\|_{L^\infty_{x,v}(\gamma_-)}\right).
\end{equation}This completes the proof. \end{proof}

We now prove the estimate \eqref{ansatz for EBl maxwell} in the following subsections. 
\subsubsection{Outline for the Estimates on the Perturbed Electromagnetic Fields}\label{sec.outline.perturb.change}
Using the estimate \eqref{final estimate for flo maxwell} for $\fplo$, 
we will now make upper bound estimates for the fields $\Eplo$ and $\Bplo$ so that we can close the bootstrap argument at the sequential level $(l+1)$ given the assumptions \eqref{Ansatz for fl maxwell}--\eqref{ansatz for EBl maxwell} at level $(l)$. To this end, we need to make upper-bound estimates of $\Eplo$ and $\Bplo$. Note that $\Eplo$ and $\Bplo$ satisfy the Maxwell equations \eqref{iterated Maxwell.M}, which have the same structure as that of the Maxwell equations \eqref{2speciesVM}$_2$-\eqref{2speciesVM}$_5$ for full fields $\E$ and $\B$. In addition, the perturbation $\rho_{\textup{pert}}$ and $J_{\textup{pert}}$ also satisfy the continuity equation \eqref{continuity eq.perturb}, which again has the same structure to \eqref{continuity eq}. Lastly, the perturbation $\Eplo$ and $\Bplo$ also satisfy the perfect conductor boundary condition \eqref{perfect cond. boundary}, and hence we obtain the representations for $\Eplo$ and $\Bplo$ the same as in \eqref{Eparallel homo solution}, \eqref{Ei5},  \eqref{E3}, \eqref{B_half_final}, \eqref{Bpar_half_final}, \Black{and \eqref{BparS_half_final}} with $F_\pm$ now replaced by $\fpl.$ However, one difference this time is that the Vlasov equation \eqref{iterated Vlasov.M}$_1$ is written in terms of the perturbation $\fplo$ and this contains an additional term $\mp \left(\Epl+\hat{v}_\pm\times \Bpl\right)\cdot \nabla_v F_{\pm,\textup{st}}$ in the equation. This will modify the representations of $\Eplo_S$ and $\Bplo_S$ terms since the derivation of them uses the Vlasov equation \eqref{iterated Vlasov.M}$_1$. Namely, when we substitute the transport term as
$$(\partial_t+\hat{v}_\pm \cdot \nabla_x)F_\pm = -\nabla_v\cdot\left( (\pm \E 
\pm \hat{v}_\pm \times  \B -m_\pm g\hat{e}_3)  F_\pm \right),$$ to obtain the $S$ terms (ex. \eqref{Ei5}$_1$), we instead use the following identity this time:
\begin{equation} \label{eq.perturb.f}(\partial_t +\hat{v}_\pm \cdot \nabla_x) \fplo =-\nabla_v \cdot \left((\pm\Elbf\pm\hat{v}_\pm\times \Blbf-m_\pm g\hat{e}_3)\fplo\right) - \nabla_v\cdot \left((\pm \Epl\pm\hat{v}_\pm\times \Bpl)F_{\pm,\textup{st}}\right). \end{equation}Therefore, this creates the following additional terms in the $S$-term
 representations in addition to the previous $S$ representations in \eqref{Ei5}, \eqref{E3}\Black{, and \eqref{BparS_half_final}}:
 \begin{itemize}
     \item For $\mathcal{E}_{\pm, iS}^{l+1}$ with $i=1,2,3$, we additionally have to control
     \begin{equation}
         \label{new S E1}\int_{B^+(x;t)} dy' \int_\rth dv\ a^{\mathbf{E}}_{\pm,i}(v,\omega)\cdot (\pm\Epl\pm(\hat{v}_\pm)\times \Bpl)\frac{F_{\pm,\textup{st}}(t-|x-y'|,y',v)}{|x-y'|},
     \end{equation} and
      \begin{equation}
         \label{new S E2}\int_{B^-(x;t)} dy' \int_\rth dv\ a^{\mathbf{E}}_{\pm,i}(v,\bar{\omega})\cdot (\pm\Epl\pm(\hat{v}_\pm)\times \Bpl)\frac{F_{\pm,\textup{st}}(t-|x-y'|,\bar{y}',v)}{|x-y'|}.\end{equation}
     \item \Black{For $\mathcal{B}_{\pm, iS}^{l+1}$ with $i=1,2,3$, we additionally have to control
     \begin{equation}
         \label{new S B1}\int_{B^+(x;t)} dy' \int_\rth dv\ a^{\B}_{\pm,i}(v,\omega)\cdot (\pm\Epl\pm(\hat{v}_\pm)\times \Bpl)\frac{F_{\pm,\textup{st}}(t-|x-y'|,y',v)}{|x-y'|},
     \end{equation} and
      \begin{equation}
         \label{new S B2}\int_{B^-(x;t)} dy' \int_\rth dv\ a^{\B}_{\pm,i}(v,\bar{\omega})\cdot (\pm\Epl\pm(\hat{v}_\pm)\times \Bpl)\frac{F_{\pm,\textup{st}}(t-|x-y'|,\bar{y}',v)}{|x-y'|},\end{equation}
         where $a^{\B}_{\pm,i}(v,\omega)=\nabla_v\big(\frac{(\omega\times\hat{v}_\pm)_i}{1+\hat{v}_\pm\cdot\omega}\big)$ is the magnetic $S$-kernel defined in \eqref{aBi}.}
 \end{itemize}All other terms have the same representations in \eqref{Ei5}, \eqref{E3}, \eqref{B_half_final}, and \eqref{Bpar_half_final}, with $F_\pm$ replaced by $\fplo$. The following outlines how we decompose $\Eplo$ and $\Bplo$ and estimate each component without repeating proofs that will be established later, for the sake of conciseness.
\begin{itemize}
    \item Note that the upper-bound estimates for $\Eplo_{\textup{hom},i}$ and $\Bplo_{\textup{hom},i}$ are analogous to those in Section~\ref{sec.Ei.start}, with the distinction that the Cauchy data is not periodically extended in this case. We recall that the enhanced decay-in-$t$ estimates for $\Eplo$ and $\Bplo$ were already established in Section~\ref{sec.homo.decay}, where the proof relied on the compact-in-$x$ support condition of the Cauchy data for field perturbations. In contrast, the argument to be presented later in Section~\ref{sec.Ei.start} will show that the boundedness estimates without decay-in-$t$ can be achieved without assuming such compact support. To avoid redundancy, we postpone the detailed proof to Section~\ref{sec.Ei.start}.
    \item Boundary-value parts $\Eplo_{\pm,ib2}$.
    \item Initial-value parts $\Eplo_{\pm,ib1}$ and $\Bplo_{\pm,ib1}$. Indeed, the upper-bound estimates for these terms are analogous to those provided in Section~\ref{sec.5.7}, with the distinction that the Cauchy data is not periodically extended in this case. Even in the case $\mathbb{R}^3$, as long as we have the decay-in-$x_3$ condition \eqref{ex.ke decay condition} of the initial energy density $\energy^{\textup{in}}$, with the $\sup$ taken over $x_\parallel \in \mathbb{R}^2$ instead of $x_\parallel \in \mathbb{T}^2$, the estimate \eqref{Eb1 estimate ex} still holds in the same form. Here, we define and use the kinetic energy density notation $\energy(t,x)$ for two species as 
\begin{equation}\label{energy density}
     \energy(t,x)=\int_\rth \sqrt{m_\pm^2+|v|^2} F_\pm(t,x,v)\,dv, \qquad 
     \energy^{\textup{in}}(x)=\energy(0,x).
 \end{equation}
For conciseness, we postpone the detailed proof to Section~\ref{sec.5.7}.
      \item Parts corresponding to the vector-fields $T$ and $S$: $\Eplo_T$,  $\Bplo_T$, $\Eplo_S$\Black{, and $\Bplo_S$}, with the $S$-terms including the additional new contributions of \eqref{new S E1}--\eqref{new S E2} \Black{and \eqref{new S B1}--\eqref{new S B2}}.
\end{itemize}In the following subsections, we will make an upper-bound estimate for each term above.

\subsubsection{Estimates for the Boundary-Value Components of the Tangential Electric Fields} 
For the estimates of $(\Eplo)^{(1)}_{\pm,ib2}$ and $(\Eplo)^{(2)}_{\pm,ib2}$ in \eqref{Ei5} with $F_\pm$ now replaced by $\fpl$, we follow the same kernel estimates for \eqref{Eb2 kernel estimate}-\eqref{3.13} and use the estimate \eqref{final estimate for flo maxwell} on $\fplo$, and the assumption \eqref{iterated Vlasov.M} on the zero-incoming boundary perturbation profile to obtain that
\begin{equation}\begin{split}\label{Eb2 estimate beginning maxwell}
  &  |(\Eplo)^{(1)}_{\pm,ib2}(t,x)|+|(\Eplo)^{(2)}_{\pm,ib2}(t,x)|\\&\le  2\int_{B(x;t)\cap \{y'_3=0\}} \frac{dy'_\parallel}{|y'-x|}\int_{v_3\le 0} dv\  \left(1+|(\hat{v}_\pm)_3|\frac{\sqrt{m_\pm^2+|v|^2}}{m_\pm}\right)|\fplo(t-|x-y'|,y'_\parallel,0,v)|\\
   & \le \frac{4}{m_\pm}\int_{B(x;t)\cap \{y'_3=0\}} \frac{dy'_\parallel}{|y'-x|}\int_{v_3\le 0} dv\  \vZ e^{-\frac{\beta}{2} \vZ} \\&\times \left(\|\mathrm{w}_{\pm,\beta} f^{\textup{in}}_\pm\|_{L^\infty_{x,v}}+\frac{C}{\beta}\|\mathrm{w}^2_{\pm,\beta}(x_\parallel,0,v)\nabla_{x_\parallel,v}G_\pm(x_\parallel,v)\|_{L^\infty_{x,v}(\gamma_-)}\right)e^{-\frac{\beta}{2}|y'_\parallel|}\\
    & \le \frac{4}{m_\pm}c_{\pm,\beta}   \left(\|\mathrm{w}_{\pm,\beta} f^{\textup{in}}_\pm\|_{L^\infty_{x,v}}+\frac{C}{\beta}\|\mathrm{w}^2_{\pm,\beta}(x_\parallel,0,v)\nabla_{x_\parallel,v}G_\pm(x_\parallel,v)\|_{L^\infty_{x,v}(\gamma_-)}\right) \int_{\mathbb{R}^2} \frac{dy'_\parallel}{\sqrt{|y'_\parallel-x_\parallel|^2+x_3^2}}e^{-\frac{\beta}{2}|y'_\parallel|},
\end{split}\end{equation}where $c_{\pm,\beta} \eqdef \int_{\rth} dv\  \vZ e^{-\frac{\beta}{2} \vZ} $ and by rescaling we have $c_{\pm,\beta} \approx \frac{1}{\beta^4}$ (see \eqref{additional beta decay}).
Since $$\int_{\mathbb{R}^2} \frac{dy'_\parallel}{\sqrt{|y'_\parallel-x_\parallel|^2+x_3^2}}e^{-\frac{\beta}{2}|y'_\parallel|}\le \int_{\mathbb{R}^2} \frac{dy'_\parallel}{|y'_\parallel-x_\parallel|}e^{-\frac{\beta}{2}|y'_\parallel|}\approx\frac{1}{\beta} \left\langle \frac{\beta}{2}x_\parallel\right\rangle^{-1}\lesssim \frac{1}{\beta} ,$$ by \eqref{asymptotics}, we conclude that \begin{multline}\label{Final estimate for b2 terms maxwell}
    |(\Eplo)^{(1)}_{\pm,ib2}(t,x)|+|(\Eplo)^{(2)}_{\pm,ib2}(t,x)|\\
    \lesssim \frac{1}{m_\pm \beta^5}\left(\|\mathrm{w}_{\pm,\beta} f^{\textup{in}}_\pm\|_{L^\infty_{x,v}}+\frac{C}{\beta}\|\mathrm{w}^2_{\pm,\beta}(x_\parallel,0,v)\nabla_{x_\parallel,v}G_\pm(x_\parallel,v)\|_{L^\infty_{x,v}(\gamma_-)}\right),
\end{multline} by \eqref{additional beta decay}. If $\beta>1$ is chosen sufficiently large such that $\min\{m_-^2,m_+^2\}g\beta^5\gg 1$, then we have $$|(\Eplo)^{(1)}_{\pm,ib2}(t,x)|+|(\Eplo)^{(2)}_{\pm,ib2}(t,x)|\ll \min\{m_-,m_+\}g. $$   This completes the estimates for $|(\Eplo)^{(1)}_{\pm,ib2}(t,x)|$ and $|(\Eplo)^{(2)}_{\pm,ib2}(t,x)|$ boundary terms.

\subsubsection{Estimates for the Transverse and Nonlinear Source Components of the \Black{Electromagnetic} Fields}
\label{sec.field est maxwell}
We now use the decaying bound \eqref{final estimate for flo maxwell} for $\fplo$ for a sufficiently large choice of $\beta$ and get the upper-bound estimates for $(\Eplo)^{(1)}_{\pm,iT}$ and $(\Eplo)^{(1)}_{\pm,iS}$\Black{, as well as for their magnetic counterparts $(\Bplo)^{(1)}_{\pm,iT}$ and $(\Bplo)^{(1)}_{\pm,iS}$,} for $i=1,2$. Again, denote that $\Eplo_T\eqdef (\Eplo)^{(1)}_{\pm,T}+\Eplo_{2T}$ where
$\Eplo_{jT}\eqdef (\Eplo_{1jT},\Eplo_{2jT},\Eplo_{3jT})^\top$ and $\Eplo_{jS}\eqdef (\Eplo_{1jS},\Eplo_{2jS},\Eplo_{3jS})^\top$ for $j=1,2$\Black{, and analogously for $\Bplo_{jT}$ and $\Bplo_{jS}$}.

\paragraph{Estimates for $\Eplo_T$ \Black{and $\Bplo_T$} terms}\label{ET uniform boundedness}
To begin with, we observe that $\Eplo_T$ terms are written as
\begin{equation*}
    (\Eplo)^{(1)}_{\pm,T}(t,x)=\mp\int_{B^+(x;t)} \frac{dy'}{|y'-x|^2}\int_\rth dv\ \frac{(|(\hat{v}_\pm)|^2-1)(\hat{v}_\pm+\omega)}{(1+\hat{v}_\pm\cdot \omega)^2} \fplo(t-|x-y'|,y',v)\Black{,}
\end{equation*}
\Black{while the magnetic $T$-terms $(\Bplo)^{(1)}_{\pm,iT}$ carry the same integral structure with the kernel $\frac{(1-|(\hat{v}_\pm)|^2)(\omega\times \hat{v}_\pm)_i}{(1+\hat{v}_\pm\cdot \omega)^2}$ in place of the $i$-th component of $\frac{(|(\hat{v}_\pm)|^2-1)(\hat{v}_\pm+\omega)}{(1+\hat{v}_\pm\cdot \omega)^2}$.}
For the estimates of $(\Eplo)^{(1)}_{\pm,iT}$ with $i=1,2,3$, we utilize the estimate \eqref{final estimate for flo maxwell} and the kernel estimate \eqref{ET kernel estimate} to obtain
\begin{equation}\begin{split}\label{E5T estimate maxwell.orbi}|(\Eplo)^{(1)}_{\pm,iT}(t,x)|&\lesssim \frac{1}{m_\pm} \int_{B^+(x;t)} \frac{dy'}{|y'-x|^2}\int_\rth dv\ \vZ\fplo(t-|x-y'|,y',v)\\
&\lesssim \frac{1}{m_\pm} \int_{B^+(x;t)} \frac{dy'}{|y'-x|^2}\int_\rth dv\ \vZ e^{-\frac{\beta}{2}(\vZ+m_\pm gy'_3)}\\&\times\left(\|\mathrm{w}_{\pm,\beta} f^{\textup{in}}_\pm\|_{L^\infty_{x,v}}+\frac{C}{\beta}\|\mathrm{w}^2_{\pm,\beta}(x_\parallel,0,v)\nabla_{x_\parallel,v}G_\pm(x_\parallel,v)\|_{L^\infty_{x,v}(\gamma_-)}\right) e^{-\frac{\beta}{2}|y'_\parallel|}\\
&\approx \frac{1}{m_\pm} c_{\pm,\beta}\left(\|\mathrm{w}_{\pm,\beta} f^{\textup{in}}_\pm\|_{L^\infty_{x,v}}+\frac{C}{\beta}\|\mathrm{w}^2_{\pm,\beta}(x_\parallel,0,v)\nabla_{x_\parallel,v}G_\pm(x_\parallel,v)\|_{L^\infty_{x,v}(\gamma_-)}\right) \\&\times  \int_{B^+(x;t)} \frac{dy'}{|y'-x|^2}e^{-\frac{m_\pm g\beta}{2}y'_3}e^{-\frac{\beta}{2}|y'_\parallel|},
\end{split}\end{equation}where $c_{\pm,\beta}$ is defined as
$c_{\pm,\beta} \eqdef \int_\rth dv \ \vZ e^{-\frac{\beta}{2}\vZ}. $ Now we further split the integral domain into two: $|y'-x|\le 1$ and $|y'-x|>1.$
If $|y'-x|> 1,$ we have
\begin{multline*}
   \frac{1}{m_\pm} c_{\pm,\beta} \int_{B^+(x;t)} \frac{dy'}{|y'-x|^2}e^{-\frac{m_\pm g\beta}{2}y'_3}e^{-\frac{\beta}{2}|y'_\parallel|}1_{\{|y'-x|> 1\}}
\le  \frac{1}{m_\pm} c_{\pm,\beta} \int_{\rth} dy'e^{-\frac{m_\pm g\beta}{2}y'_3}e^{-\frac{\beta}{2}|y'_\parallel|}
    \lesssim \frac{1}{m_\pm^2g\beta^3}c_{\pm,\beta}  .
\end{multline*}
On the other hand, if $|y'-x|\le 1,$ we further proceed as
\begin{multline}
\frac{1}{m_\pm}c_{\pm,\beta} \int_{B^+(x;t)} \frac{dy'}{|y'-x|^2}e^{-\frac{m_\pm g\beta}{2}y'_3}e^{-\frac{\beta}{2}|y'_\parallel|}1_{\{|y'-x|\le 1\}}\\
\approx\frac{1}{m_\pm} c_{\pm,\beta}\int_{B(x;t)\cap \{z_3+x_3>0\}} \frac{dz}{|z|^2} e^{-\frac{m_\pm g\beta}{2} (z_3+x_3)}e^{-\frac{\beta}{2}|z_\parallel+x_\parallel|}1_{\{|z|\le 1\}}\\
\approx\frac{1}{m_\pm} c_{\pm,\beta}\int_0^1dr\int_{\mathbb{S}^2}d\omega\  1_{\{(r\omega)_3+x_3>0\}} e^{-\frac{m_\pm g\beta}{2} ((r\omega)_3+x_3)}e^{-\frac{\beta}{2}|z_\parallel+x_\parallel|}
\lesssim \frac{1}{m_\pm}c_{\pm,\beta}.
\end{multline}
Altogether, we conclude that for $i=1,2,3$,
\begin{multline}
    \label{E1T final maxwell.orbi}|(\Eplo)^{(1)}_{\pm,iT}(t,x)|\\
    \lesssim \frac{1}{m_\pm^2g\beta^7}(1+m_\pm g\beta^3)\left(\|\mathrm{w}_{\pm,\beta} f^{\textup{in}}_\pm\|_{L^\infty_{x,v}}+\frac{C}{\beta}\|\mathrm{w}^2_{\pm,\beta}(x_\parallel,0,v)\nabla_{x_\parallel,v}G_\pm(x_\parallel,v)\|_{L^\infty_{x,v}(\gamma_-)}\right),
\end{multline}since we have the estimate \eqref{additional beta decay} for the coefficient $c_{\pm,\beta}$. Thus, if we choose $\beta>1$ sufficiently large such that $\min\{m_-^3,m_+^3\} g^2\beta^7\gg 1$ and $\min\{m_-^2,m_+^2\} g\beta^4\gg 1$, then we have$$|(\Eplo)^{(1)}_{\pm,iT}(t,x)|\ll \min\{m_-,m_+\} g.$$
Similarly, we obtain the same upper bound for $(\Eplo)^{(2)}_{\pm,iT}$ for $i=1,2,3$.
\Black{The magnetic $T$-terms are estimated in exactly the same way: by \eqref{B35 kernel final}, the magnetic $T$-kernel obeys the same upper bound $\frac{2\vZ}{m_\pm}$ as the electric $T$-kernel bound \eqref{ET kernel estimate}, uniformly in $\omega\in\mathbb{S}^2$ (hence also with $\bar{\omega}$). Therefore, the estimates \eqref{E5T estimate maxwell.orbi}--\eqref{E1T final maxwell.orbi} hold verbatim for $(\Bplo)^{(1)}_{\pm,iT}$ and $(\Bplo)^{(2)}_{\pm,iT}$, and under the same largeness conditions on $\beta$ we obtain, for $i=1,2,3$,
$$|(\Bplo)^{(1)}_{\pm,iT}(t,x)|+|(\Bplo)^{(2)}_{\pm,iT}(t,x)|\ll \min\{m_-,m_+\} g.$$}

\paragraph{Estimates for $\Eplo_S$ \Black{and $\Bplo_S$} terms}
On the other hand, regarding the nonlinear $S$ terms, $(\Eplo)^{(1)}_{\pm,iS}$ and $(\Eplo)^{(2)}_{\pm,iS}$ with $i=1,2,3$, we recall \eqref{Ei5} and \eqref{E3}. Here we define $$a^{\E}_{\pm,i}(v,\omega)=\frac{(\partial_{v_i}v-(\hat{v}_\pm)_i(\hat{v}_\pm))}{(\vZ) (1+\hat{v}_\pm\cdot \omega)}-\frac{(\omega_i+(\hat{v}_\pm)_i)(\omega-(\omega\cdot (\hat{v}_\pm))(\hat{v}_\pm))}{(\vZ) (1+\hat{v}_\pm\cdot \omega)^2}=:a^{(1)}_{\pm,i}+a^{(2)}_{\pm,i}.$$
In this section, let us call $(\Eplo)^{(1)}_{\pm,iS}$ and $(\Eplo)^{(2)}_{\pm,iS}$ \Black{(respectively, $(\Bplo)^{(1)}_{\pm,iS}$ and $(\Bplo)^{(2)}_{\pm,iS}$)} as the integral representations \eqref{Ei5}$_1$ \Black{(respectively, \eqref{BparS_half_final})} without the additional contributions \eqref{new S E1} and \eqref{new S E2} \Black{(respectively, \eqref{new S B1} and \eqref{new S B2})} from the steady-state profile $F_{\pm,\textup{st}}.$ We will estimate these additional terms \eqref{new S E1}\Black{--}\eqref{new S B2} below in Section \ref{sec.additional.S}.
Then using the kernel estimate \eqref{a2}--\eqref{a1} and the estimate \eqref{final estimate for flo maxwell} for $\fplo$,
we finally obtain for $i=1,2,3$,
\begin{equation}\begin{split}
\label{EloiS bound.max1}
|(\Eplo)^{(1)}_{\pm,iS}(t,x)|&\lesssim
\frac{1}{m_\pm^2}(m_\pm g+\|(\Elbf,\Blbf)\|_{L^\infty})\int_{B^+(x;t)} \frac{dy'}{|y'-x|}e^{-\frac{\beta}{2}|y'_\parallel|}\\&\times \int_\rth dv\ \vZ e^{-\frac{\beta}{2}(\vZ+m_\pm gy'_3)}\left(\|\mathrm{w}_{\pm,\beta} f^{\textup{in}}_\pm\|_{L^\infty_{x,v}}+\frac{C}{\beta}\|\mathrm{w}^2_{\pm,\beta}(x_\parallel,0,v)\nabla_{x_\parallel,v}G_\pm(x_\parallel,v)\|_{L^\infty_{x,v}(\gamma_-)}\right)\\
&\approx \frac{1}{m_\pm}c_{\pm,\beta} g\left(\|\mathrm{w}_{\pm,\beta} f^{\textup{in}}_\pm\|_{L^\infty_{x,v}}+\frac{C}{\beta}\|\mathrm{w}^2_{\pm,\beta}(x_\parallel,0,v)\nabla_{x_\parallel,v}G_\pm(x_\parallel,v)\|_{L^\infty_{x,v}(\gamma_-)}\right)\\&\times \int_{B^+(x;t)} \frac{dy'}{|y'-x|}e^{-\frac{\beta}{2}|y'_\parallel|} e^{-\frac{m_\pm g\beta}{2}y'_3},
\end{split}\end{equation}where we used the a priori ansatz \eqref{ansatz for EBl maxwell} for the perturbations $(\Epl,\Bpl)$ as well as the known estimates \eqref{steady state L infty} for the steady-state  $(\E_{\textup{st}},\B_{\textup{st}})$ to obtain the following upper-bound for the full fields $(\Elbf,\Blbf)$ as
\begin{equation}
    \label{eq.full.EBl bound}\sup_{t\ge 0}\|(\Elbf,\Blbf)\|_{L^\infty}\le \min\{m_+,m_-\}\frac{g}{8}.
\end{equation}
Now we further split the integral domain into two: $|y'-x|\le 1$ and $|y'-x|>1.$
If $|y'-x|> 1,$ we have
\begin{multline*}
   \frac{1}{m_\pm} c_{\pm,\beta} g \int_{B^+(x;t)} \frac{dy'}{|y'-x|}e^{-\frac{\beta}{2}|y'_\parallel|} e^{-\frac{m_\pm g\beta}{2}y'_3}1_{\{|y'-x|> 1\}}
   \le \frac{1}{m_\pm} c_{\pm,\beta} g \int_{\rth} dy'e^{-\frac{\beta}{2}|y'_\parallel|} e^{-\frac{m_\pm g\beta}{2}y'_3}
    \lesssim \frac{1}{m_\pm^2\beta^3} c_{\pm,\beta}   .
\end{multline*}On the other hand, if $|y'-x|\le 1,$ we have
\begin{align*}
&\frac{1}{m_\pm}  c_{\pm,\beta} g \int_{B^+(x;t)} \frac{dy'}{|y'-x|}e^{-\frac{\beta}{2}|y'_\parallel|} e^{-\frac{m_\pm g\beta}{2}y'_3}1_{\{|y'-x|\le 1\}}\\
&\approx \frac{1}{m_\pm} c_{\pm,\beta} g \int_{B(x;t)\cap \{z_3+x_3>0\}} \frac{dz}{|z|} e^{-\frac{m_\pm g\beta}{2}(z_3+x_3)}1_{\{|z|\le 1\}}\\
&\approx\frac{1}{m_\pm}  c_{\pm,\beta} g \int_0^1dr\int_{\mathbb{S}^2}d\omega\  1_{\{(r\omega)_3+x_3>0\}} re^{-\frac{m_\pm g\beta}{2}((r\omega)_3+x_3)}
\lesssim \frac{1}{m_\pm} c_{\pm,\beta} g .
\notag
\end{align*}
Altogether, we conclude
\begin{multline}
\label{E5S estimate maxwell}
|(\Eplo)^{(1)}_{\pm,iS}(t,x)|\\*
\lesssim \frac{1}{m_\pm^2 \beta^7}(1+m_\pm\beta^3g)\left(\|\mathrm{w}_{\pm,\beta} f^{\textup{in}}_\pm\|_{L^\infty_{x,v}}+\frac{C}{\beta}\|\mathrm{w}^2_{\pm,\beta}(x_\parallel,0,v)\nabla_{x_\parallel,v}G_\pm(x_\parallel,v)\|_{L^\infty_{x,v}(\gamma_-)}\right).
\end{multline}Thus, if we choose $\beta>1$ sufficiently large such that $\min\{m_-^3,m_+^3\} g\beta^7\gg 1$ and $\min\{m_-^2,m_+^2\} \beta^4\gg 1$, then we have$$|(\Eplo)^{(1)}_{\pm,iS}(t,x)|\ll \min\{m_-,m_+\} g.$$
Similarly, we obtain the same upper bounds for $(\Eplo)^{(2)}_{\pm,iS}$ terms for $i=1,2,3,$ as those of \eqref{E5S estimate maxwell}.
\Black{The magnetic $S$-terms are estimated in exactly the same way: the integrands of $(\Bplo)^{(1)}_{\pm,iS}$ and $(\Bplo)^{(2)}_{\pm,iS}$ differ from those of $(\Eplo)^{(1)}_{\pm,iS}$ and $(\Eplo)^{(2)}_{\pm,iS}$ only through the kernel $a^{\B}_{\pm,i}$ of \eqref{aBi} in place of $a^{\E}_{\pm,i}$, and by \eqref{aB.final} the magnetic kernel obeys the same upper bound $|a^{\B}_{\pm,i}(v,\omega)|\lesssim \frac{\vZ}{m_\pm^2}$ as \eqref{a2}--\eqref{a1}, uniformly in $\omega\in\mathbb{S}^2$ (hence also with $\bar{\omega}$). Therefore, the estimates \eqref{EloiS bound.max1}--\eqref{E5S estimate maxwell} hold verbatim for $(\Bplo)^{(1)}_{\pm,iS}$ and $(\Bplo)^{(2)}_{\pm,iS}$, and under the same largeness conditions on $\beta$ we obtain, for $i=1,2,3$,
$$|(\Bplo)^{(1)}_{\pm,iS}(t,x)|+|(\Bplo)^{(2)}_{\pm,iS}(t,x)|\ll \min\{m_-,m_+\} g.$$}
In conclusion, together with \eqref{E1T final maxwell.orbi} we have for $i=1,2,3,$
\begin{multline}\label{EiTEiS final i12 maxwell}|(\Eplo)^{(1)}_{\pm,iT}(t,x)|+|(\Eplo)^{(2)}_{\pm,iT}(t,x)|+|(\Eplo)^{(1)}_{\pm,iS}(t,x)|+|(\Eplo)^{(2)}_{\pm,iS}(t,x)|\\
\Black{+|(\Bplo)^{(1)}_{\pm,iT}(t,x)|+|(\Bplo)^{(2)}_{\pm,iT}(t,x)|+|(\Bplo)^{(1)}_{\pm,iS}(t,x)|+|(\Bplo)^{(2)}_{\pm,iS}(t,x)|}\ll \min\{m_-,m_+\} g.\end{multline}

\paragraph{Estimates on the additional contributions by $F_{\pm,\textup{st}}$ in \eqref{new S E1}-\eqref{new S E2} \Black{and \eqref{new S B1}-\eqref{new S B2}}}\label{sec.additional.S}
As we have outlined in Section \ref{sec.outline.perturb.change},  we additionally have to control
     \begin{equation}
          \label{new S E1.new} S_1\eqdef \int_{B^+(x;t)} dy' \int_\rth dv\ a^{\mathbf{E}}_{\pm,i}(v,\omega)\cdot (\pm\Epl\pm(\hat{v}_\pm)\times \Bpl)\frac{F_{\pm,\textup{st}}(t-|x-y'|,y',v)}{|x-y'|},
     \end{equation} and
      \begin{equation}
         \notag S_2\eqdef \int_{B^-(x;t)} dy' \int_\rth dv\ a^{\mathbf{E}}_{\pm,i}(v,\bar{\omega})\cdot (\pm\Epl\pm(\hat{v}_\pm)\times \Bpl)\frac{F_{\pm,\textup{st}}(t-|x-y'|,\bar{y}',v)}{|x-y'|},\end{equation}for $\mathcal{E}_{\pm, iS}^{l+1}$ with $i=1,2,3$, as in \eqref{new S E1}-\eqref{new S E2}\Black{, together with their magnetic counterparts
      \begin{equation}
          \label{new S B1.new} S^{\B}_1\eqdef \int_{B^+(x;t)} dy' \int_\rth dv\ a^{\B}_{\pm,i}(v,\omega)\cdot (\pm\Epl\pm(\hat{v}_\pm)\times \Bpl)\frac{F_{\pm,\textup{st}}(t-|x-y'|,y',v)}{|x-y'|},
     \end{equation}
     and $S^{\B}_2$ defined analogously on $B^-(x;t)$ with $(\bar{\omega},\bar{y}')$ in place of $(\omega,y')$, for $\mathcal{B}^{l+1}_{\pm,iS}$ with $i=1,2,3$, as in \eqref{new S B1}-\eqref{new S B2}}.
         Note that the only difference of \eqref{new S E1.new} and  $(\Eplo)^{(1)}_{\pm,iS}$ of \eqref{Ei5}$_1$ is that we have perturbations $(\Epl,\Bpl)$ and the steady-state $F_{\pm,\textup{st}}$ replacing $(\Elbf,\Blbf)$ and $F^{l+1}_\pm,$ as well as we remove the gravity from \eqref{Ei5}$_1$\Black{; the same holds for \eqref{new S B1.new} and $(\Bplo)^{(1)}_{\pm,iS}$}. Observe that  the upper-bounds \eqref{eq.full.EBl bound} of $(\Elbf,\Blbf)$ differs from the ansatz \eqref{ansatz for EBl maxwell} only by a constant, and also that $F_{\pm,\textup{st}}$ also has the same exponential decay in the variables $x_\parallel,$ $x_3$ and $|v|$ as in \eqref{steady state L infty}. \Black{Moreover, the magnetic kernel obeys the same upper bound as the electric one, $|a^{\B}_{\pm,i}(v,\omega)|\lesssim \frac{\vZ}{m_\pm^2}$, by \eqref{aB.final} and \eqref{a2}--\eqref{a1}, uniformly in $\omega\in\mathbb{S}^2$.} Therefore, we modify \eqref{EloiS bound.max1} and \eqref{EiTEiS final i12 maxwell} and obtain  that
         $$|S_1|\Black{+|S^{\B}_1|}\lesssim \frac{C}{m_\pm^2 \beta^7}(1+m_\pm\beta^3g),$$ where the constant $C$ is from \eqref{steady state L infty}. Thus, if we choose $\beta>1$ sufficiently large such that $$\min\{m_-^3,m_+^3\} g\beta^7\gg 1 \text{ and }\min\{m_-^2,m_+^2\} \beta^4\gg 1,$$ then we have$$|S_1|\Black{+|S^{\B}_1|}\ll \min\{m_-,m_+\} g.$$ The estimate\Black{s} for $S_2$ \Black{and $S^{\B}_2$} \Black{are} also the same.
\subsubsection{Final Estimates for the Perturbed Electromagnetic Fields}\label{sec.final upper bound for EplBpl}
We now establish the final upper bounds for $\Eplo$ and $\Bplo$. \Black{Recall that the transverse components $\Bplo_T$ and the nonlinear source components $\Bplo_S$, including the additional stationary contributions \eqref{new S B1}--\eqref{new S B2}, have already been estimated in Sections~\ref{sec.field est maxwell} and \ref{sec.additional.S}, in parallel with their electric counterparts. The remaining components of $\Bplo$, namely $\Bplo_{\textup{hom},i}$ and $\Bplo_{ib1}$,} differ from those for $\Eplo$ only in the kernels appearing in their integral representations. However, as seen in \Black{\eqref{eq.ib1 kernel},} \eqref{ET kernel estimate}\Black{,} and \eqref{B35 kernel final}, these kernels satisfy identical upper bounds. Consequently, the $L^\infty$ estimates for \Black{these components of} $\Bplo$ follow analogously, with only minimal modifications.

Combining these bounds, we establish the following lemma for $\Eplo$ and $\Bplo$: 
\begin{lemma}\label{lem.dyna.EB.orbi}Fix $l\in\mathbb{N}$ and suppose that $\beta > 1$ is sufficiently large such that  
$$
\min\{m_+^2, m_-^2\}g^2\beta \gg 1 \quad \text{and} \quad \min\{m_+^2, m_-^2\}\beta^4 \gg 1. 
$$ Suppose that \eqref{Ansatz for fl maxwell}-\eqref{ansatz for EBl maxwell}  hold for $(\fplo,\Epl,\Bpl).$ Then $(\Eplo,\Bplo)$ satisfies 
\begin{equation}\label{final estimate for EloBlo maxwell}
    \sup_{t \geq 0} \|(\Eplo, \Bplo)\|_{L^\infty} \leq \min\{m_+, m_-\} \frac{g}{16}
\end{equation}
\end{lemma}
This bound guarantees the validity of \eqref{ansatz for EBl maxwell} at the $(l+1)$-th iteration level, provided that the parameter $\beta > 1$ is chosen sufficiently large. Consequently, the estimates \eqref{Ansatz for fl maxwell}--\eqref{ansatz for EBl maxwell} are satisfied uniformly for all $l \in \mathbb{N}$, and thus remain valid in the limit as $l\to \infty$.

\begin{remark}[$\mathbb{T}^2\times\mathbb{R}_+$ case: quadratically growing self-consistent fields in time]
\label{Torus case remark}
The above bootstrap argument does not work in the toroidal domain $\mathbb{T}^2\times\mathbb{R}_+$. Namely, 
in $\mathbb{T}^2\times\mathbb{R}_+$, we mention that the estimate \eqref{final estimate for flo maxwell} on $\fplo$ will create an additional quadratic growth in $t$ in the upper-bounds of $\Eplo$ and $\Bplo$ in \eqref{final estimate for EloBlo maxwell}. This is due to the lack of temporal decay for $\fplo$ in the estimate \eqref{final estimate for flo maxwell} for perturbations near steady states satisfying a J\"uttner--Maxwell upper bound. By contrast, the estimate for $\flo$ in \eqref{final estimate for flo} exhibits exponential decay in time, which is used in the bootstrap argument for the asymptotic stability of vacuum. 
More precisely, if we look at both $T$ terms and $S$ terms in \eqref{E5T estimate maxwell.orbi} and \eqref{EloiS bound.max1} without $e^{-\frac{\beta}{2}|y'_\parallel|}$ factor in the  $\mathbb{T}^2\times\mathbb{R}_+$ counterpart, note that
\begin{multline}\label{torus.growth in T term}
|(\Eplo)^{(1)}_{\pm,iT}(t,x)|\lesssim  \frac{1}{m_\pm} c_{\pm,\beta}\left(\|\mathrm{w}_{\pm,\beta} f^{\textup{in}}_\pm\|_{L^\infty_{x,v}}+\frac{C}{\beta}\|\mathrm{w}^2_{\pm,\beta}(x_\parallel,0,v)\nabla_{x_\parallel,v}G_\pm(x_\parallel,v)\|_{L^\infty_{x,v}(\gamma_-)}\right) \\\times  \int_{B^+(x;t)} \frac{dy'}{|y'-x|^2}e^{-\frac{m_\pm g\beta}{2}y'_3},
\end{multline}and 
\begin{multline}\label{torus.growth in S term}
|(\Eplo)^{(1)}_{\pm,iS}(t,x)|\lesssim  \frac{1}{m_\pm}c_{\pm,\beta} g\left(\|\mathrm{w}_{\pm,\beta} f^{\textup{in}}_\pm\|_{L^\infty_{x,v}}+\frac{C}{\beta}\|\mathrm{w}^2_{\pm,\beta}(x_\parallel,0,v)\nabla_{x_\parallel,v}G_\pm(x_\parallel,v)\|_{L^\infty_{x,v}(\gamma_-)}\right)\\\times \int_{B^+(x;t)} \frac{dy'}{|y'-x|}e^{-\frac{m_\pm g\beta}{2}y'_3}.
\end{multline}Here we observe that the following integral appears in both $T$ and $S$ terms:
$$\int_{B^+(x;t)} \frac{dy'}{|y'-x|^n}e^{-\frac{m_\pm g\beta}{2}y'_3},\ \textup{ for } n=1,2.$$ For each $n=1,2$ we have 
\begin{equation}\begin{split}\label{poly growth estimate with y_3 decay}
    &\int_{B^+(x;t)} \frac{dy'}{|y'-x|^n}e^{-\frac{m_\pm g\beta}{2}y'_3}= \int_0^tdr\  r^2 \int_{\mathbb{S}^2}d\omega \ 1_{\{(r\omega)_3+x_3\ge 0\}}  \frac{1}{r^n}e^{-\frac{m_\pm g\beta}{2}((r\omega)_3+x_3)}\\
    &= \int_0^tdr\  r^2 \int_0^{2\pi}d\psi\int_0^\pi d\phi \sin\phi \ 1_{\{r\cos\phi+x_3\ge 0\}}  \frac{1}{r^n}e^{-\frac{m_\pm g\beta}{2}(r\cos\phi+x_3)}\\
    &=2\pi \int_0^tdr\  r^{2-n}  \int_{-1}^1 dk \ 1_{\{rk+x_3\ge 0\}}  \frac{1}{r^n}e^{-\frac{m_\pm g\beta}{2}(rk+x_3)}=\frac{4\pi}{m_\pm g\beta} \int_0^tdr\  r^{2-n} (1-e^{-\frac{m_\pm g\beta}{2}(r+x_3)}).
\end{split}\end{equation}This term grows quadratically in $t$ when $n=1$ and grows linearly in $t$ when $n=2,$ and hence do $\mathcal{E}^{l+1}$ terms. The same growth also occurs for the analogous $\mathcal{B}^{l+1}$ terms.

Then the additional quadratic temporal growth in $\Eplo$ and $\Bplo$ will then make \eqref{EB apriori bound.whole} hold only for a bounded time interval whose span depends on $\beta$ and $g$. Thus in this case, we only obtain that the perturbation exists and is bounded within a finite time interval $[0,T]$ where $T$ depends on $\beta$ and $g$ in the $\mathbb{T}^2\times\mathbb{R}_+$ case for steady states with J\"uttner-Maxwell upper bound.\end{remark}

\section{Regularity Estimates for the Distributions }\label{sec.deri.dist}

This section is devoted to establishing regularity estimates for the iterated sequence of solutions $(\flo,$ $\Elobf,$ $\Blobf)$ to \eqref{iterated Vlasov.st}--\eqref{iterated Maxwell.st}, \eqref{iterated Vlasov.M}--\eqref{iterated Maxwell.M} and \eqref{def.perturb}--\eqref{continuity eq.perturb}. We prove that $(\flo, \Elobf, \Blobf)$ possess sufficient time and space regularity in appropriate weak function spaces. More precisely, we show that
$$\flo \in W^{1,\infty}\left([0,T]; L^\infty(\Omega\times \mathbb{R}^3)\right) \cap L^\infty\left([0,T]; X(\Omega\times \mathbb{R}^3)\right),\text { and }$$ 
$$( \Elobf, \Blobf)\in  W^{1,\infty}_{t,x}([0,T]\times\Omega) \times W^{1,\infty}_{t,x}([0,T]\times\Omega),
$$
where $X$ is a weighted first-order derivative space for $\flo$. In addition, we prove that the temporal and momentum derivatives of $\flo$ exhibit sufficient decay in $x$ and $v$, controlled within suitable weighted Sobolev spaces.

Given $ l \in \mathbb{N} $, we interpret the given fields in the iterated equation \eqref{iterated Vlasov.M} as $ \E^l, \B^l $ at the level of the sequential index $ (l) $, while the trajectory $ \ZS^{l+1} = (\XSlo,\VSlo) $ is understood at the level of $ (l+1) $, defined via the fields $ \E^l, \B^l $ as in \eqref{iterated char}. In this section, we study the derivative estimates of the distribution $ F^{l+1}_\pm $, and in Section \ref{sec.deri.EB}, we study the derivatives of the fields $ \E^{l+1}, \B^{l+1} $ at the level of $ (l+1) $. Notice that the final upper-bound estimates for the derivatives-\eqref{deri F estimate final} for $ F^{l+1}_\pm $ and \eqref{EloBlo estimate final} for $ \E^{l+1}, \B^{l+1} $-are uniform in $ l $, ensuring that these bounds are preserved in the limit as $ l \to \infty $. For the rest of Section \ref{sec.deri.dist} and Section \ref{sec.deri.EB} we keep the same iterated sequence elements $(\flo,\Elobf,\Blobf)$ of Section \ref{sec.boot.decay}.

Denote the given forcing term in the linear Vlasov equation \eqref{iterated Vlasov.M} as
\begin{equation}\label{eq.forcing}
    \FS(t,x,v)\eqdef \pm\Elbf(t,x) \pm (\hat{v}_\pm) \times \Blbf(t,x) -m_\pm g \hat{e}_3.
\end{equation}

\subsection{Derivatives of the Distribution}
Given the system \eqref{leading char} of ordinary differential equations for the characteristic trajectories $\XSlo$ and $\VSlo$, we can now write the representations of the derivatives of the distribution function $\flo(t, x, v)$ using the solution representation \eqref{solution f}. 
\subsubsection{Temporal Derivative $ \partial_t \flo $}
The temporal derivative $\partial_t \flo$ of a distribution $\flo$ will be estimated via the Vlasov equation \eqref{2speciesVM}$_1$ using the given estimates for $(\Elbf,\Blbf)$ and the estimates for $\nabla_x\flo$ and $\nabla_v\flo$ obtained below.

\subsubsection{Position Derivative $ \nabla_x\flo $}
For the position derivative, we again need to consider the two cases:
$ t \leq \tblo(t, x, v) $ and $ t > \tblo(t, x, v) $.

If $ t \leq \tblo(t, x, v) $ observe that
\begin{multline*}
    \nabla_x \flo(t, x, v) = \nabla_x F^{\textup{in}}_\pm(\XSlo(0; t, x, v), \VSlo(0; t, x, v)) \cdot \nabla_x \XSlo(0; t, x, v) \\+\nabla_v F^{\textup{in}}_\pm(\XSlo(0; t, x, v), \VSlo(0; t, x, v))\cdot \nabla_x \VSlo(0; t, x, v) .
\end{multline*}
On the other hand, if $ t > \tblo(t, x, v) $,
we have
$$
    \nabla_x \flo(t, x, v) = 
     (\nabla_{x_\parallel} G_\pm)( (\xblo)_\parallel, \vblo) \cdot \nabla_x (\xblo)_\parallel+ (\nabla_v G_\pm)( (\xblo)_\parallel, \vblo) \cdot \nabla_x \vblo.
$$

\subsubsection{Momentum Derivative $ \nabla_v\flo $}
For the momentum derivative, if $ t \leq \tblo(t, x, v) $
we have
\begin{multline*}
    \nabla_v \flo(t, x, v) =  \nabla_x F^{\textup{in}}_\pm(\XSlo(0; t, x, v), \VSlo(0; t, x, v)) \cdot \nabla_v \XSlo(0; t, x, v) \\+\nabla_v F^{\textup{in}}_\pm(\XSlo(0; t, x, v), \VSlo(0; t, x, v))\cdot \nabla_v \VSlo(0; t, x, v).
\end{multline*}
 On the other hand, if $ t > \tblo(t, x, v) $
we have
$$
    \nabla_v \flo(t, x, v) \\= 
    (\nabla_{x_\parallel} G_\pm)( (\xblo)_\parallel, \vblo) \cdot \nabla_v (\xblo)_\parallel\\+ (\nabla_v G_\pm)( (\xblo)_\parallel, \vblo) \cdot \nabla_v \vblo.
$$

Thus, the derivatives of the distribution function $\flo(t, x, v)$ with respect to $t$, $x$, and $v$ can be collected as follows:
\begin{equation}\label{Derivatives of f}
\begin{aligned}
    &\partial_t \flo(t, x, v) \\&= -(\hat{v}_\pm)\cdot \nabla_x\flo -\FS \cdot \nabla_v\flo, \textup{ where we further represent $\nabla_x\flo$ and $\nabla_v\flo$ by}
    \\
    &\nabla_x \flo(t, x, v) \\&= 
    \begin{cases} 
       \nabla_x F^{\textup{in}}_\pm(\XSlo(0; t, x, v), \VSlo(0; t, x, v)) \cdot \nabla_x \XSlo(0; t, x, v) \\\quad+\nabla_v F^{\textup{in}}_\pm(\XSlo(0; t, x, v), \VSlo(0; t, x, v))\cdot \nabla_x \VSlo(0; t, x, v),\ \textup{if } t \leq \tblo(t, x, v), \\
        (\nabla_{x_\parallel} G_\pm)( (\xblo)_\parallel, \vblo) \cdot \nabla_x (\xblo)_\parallel+ (\nabla_v G_\pm)( (\xblo)_\parallel, \vblo) \cdot \nabla_x \vblo,\  \textup{if } t > \tblo(t, x, v),
    \end{cases}
    \\
  &  \nabla_v \flo(t, x, v)\\ &= 
    \begin{cases} 
        \nabla_x F^{\textup{in}}_\pm(\XSlo(0; t, x, v), \VSlo(0; t, x, v)) \cdot \nabla_v \XSlo(0; t, x, v) \\\quad+\nabla_v F^{\textup{in}}_\pm(\XSlo(0; t, x, v), \VSlo(0; t, x, v))\cdot \nabla_v \VSlo(0; t, x, v),\ \textup{if } t \leq \tblo(t, x, v), \\
       (\nabla_{x_\parallel} G_\pm)( (\xblo)_\parallel, \vblo) \cdot \nabla_v (\xblo)_\parallel+ (\nabla_v G_\pm)( (\xblo)_\parallel, \vblo) \cdot \nabla_v \vblo,\ \textup{if } t > \tblo(t, x, v).
    \end{cases}
\end{aligned}
\end{equation}
Observe that the representations above still contain derivatives, gradients, and Jacobian matrices of $\tblo$, $\xblo$, and $\vblo.$ 

To begin with, for the representation of $\partial_t \tblo$, we recall that 
\begin{equation}\label{X3 property}
    (\XSlo)_3(t-\tblo(t,x,v);t,x,v)=0,
\end{equation}
by definition of $\tblo(t,x,v).$
By differentiating \eqref{X3 property} with respect to $t,$ we have
$$(\hat{\mathcal{V}}_\pm^{l+1})_3(t-\tblo)(1-\partial_t \tblo)+\partial_t (\XSlo)_3(t-\tblo)=0.$$ Thus, we have
\begin{equation}\label{partial t tb}
    \partial_t \tblo=1+\frac{\partial_t (\XSlo)_3(t-\tblo)}{(\hat{\mathcal{V}}_\pm^{l+1})_3(t-\tblo)}.
\end{equation}
Similarly, by differentiating \eqref{X3 property} with respect to $x_i$ and $v_i$ for $i=1,2,3,$ we obtain
\begin{align*}&-(\hat{\mathcal{V}}_\pm^{l+1})_3(t-\tblo)\partial_{x_i} \tblo+\partial_{x_i} (\XSlo)_3(t-\tblo)=0,\text{ and }\\&-(\hat{\mathcal{V}}_\pm^{l+1})_3(t-\tblo)\partial_{v_i} \tblo+\partial_{v_i} (\XSlo)_3(t-\tblo)=0.\end{align*} Therefore, we have
\begin{equation}\label{partial xv tb}
    \partial_{x_i}\tblo=\frac{\partial_{x_i} (\XSlo)_3(t-\tblo)}{(\hat{\mathcal{V}}_\pm^{l+1})_3(t-\tblo)}\textup{ and }\partial_{v_i}\tblo=\frac{\partial_{v_i} (\XSlo)_3(t-\tblo)}{(\hat{\mathcal{V}}_\pm^{l+1})_3(t-\tblo)},
\end{equation}for $i=1,2,3.$

Regarding the Jacobian matrices of $\xblo$ and $\vblo$, we observe that
\begin{equation}\label{Jacobian for xbvb}
    \begin{split}
         \partial_{x_i} \xblo&=     \partial_{x_i} (\XSlo(t - \tblo; t, x, v)) = -\hat{\mathcal{V}}_\pm^{l+1}(t - \tblo; t, x, v)    \partial_{x_i} \tblo  +(    \partial_{x_i} \XSlo)(t - \tblo; t, x, v), \\
         \partial_{v_i} \xblo&=     \partial_{v_i} (\XSlo(t - \tblo; t, x, v)) = -\hat{\mathcal{V}}_\pm^{l+1}(t - \tblo; t, x, v)    \partial_{v_i} \tblo  +(    \partial_{v_i} \XSlo)(t - \tblo; t, x, v), \\
    \partial_{x_i} \vblo&=\partial_{x_i} (\VSlo(t - \tblo; t, x, v)) =- (\partial_{x_i}\tblo) \FS(t - \tblo, \xblo,\vblo)+(\partial_{x_i} \VSlo)(t - \tblo; t, x, v),\\
     \partial_{v_i} \vblo&=\partial_{v_i} (\VSlo(t - \tblo; t, x, v)) =- (\partial_{v_i}\tblo  )\FS(t - \tblo, \xblo,\vblo)+(\partial_{v_i} \VSlo)(t - \tblo; t, x, v),
    \end{split} 
\end{equation}for $i=1,2,3$ by \eqref{leading char}. This completes the representations of the derivatives of $\tblo$, $\xblo$, and $\vblo$ with respect to $t,x,$ and $v.$
  
We also need the derivatives of the characteristic trajectory variables $\XSlo$ and $\VSlo$. We first observe that
\begin{equation}
    \label{VS representation}
    \VSlo(s;t,x,v)= v+ \int_t^s \left( \FS(\tau,\XSlo(\tau),\VSlo(\tau))\right)d\tau, 
\end{equation}by \eqref{leading char}. Thus, we have
\begin{multline}\label{xiVS.rep}
     (\partial_{x_i}\VSlo)(s;t,x,v)
     = \pm\int_t^s \bigg((\nabla_x\Elbf)(\tau,\XSlo(\tau))\cdot (\partial_{x_i}\XSlo)(\tau)+(\partial_{x_i}\hat{\mathcal{V}}_\pm^{l+1})(\tau)\times \Blbf(\tau,\XSlo(\tau))\\+\hat{\mathcal{V}}_\pm^{l+1}(\tau)\times (\nabla_x\Blbf)(\tau,\XSlo(\tau))\cdot(\partial_{x_i}\XSlo)(\tau) \bigg)d\tau,
\end{multline}and
\begin{multline*}
     (\partial_{v_i}\VSlo)(s;t,x,v)
     = \partial_{v_i}v \pm \int_t^s \bigg((\nabla_x\Elbf)(\tau,\XSlo(\tau))\cdot (\partial_{v_i}\XSlo)(\tau)+(\partial_{v_i}\hat{\mathcal{V}}_\pm^{l+1})(\tau)\times \Blbf(\tau,\XSlo(\tau))\\+\hat{\mathcal{V}}_\pm^{l+1}(\tau)\times( (\nabla_x\Blbf)(\tau,\XSlo(\tau))\cdot(\partial_{v_i}\XSlo)(\tau)) \bigg)d\tau,
\end{multline*}for $i=1,2,3.$ Therefore, we conclude that \begin{multline}\label{xiVS}
    |(\partial_{x_i}\VSlo)(s)|\lesssim   \sup_{\tau \in (\min\{s,t\},\max\{s,t\})}\|\nabla_x (\Elbf,\Blbf)(\tau)\|_{L^\infty}\int_{\min\{s,t\}}^{\max\{s,t\}}|(\partial_{x_i}\XSlo)(\tau)|d\tau\\
    + \sup_{\tau \in (\min\{s,t\},\max\{s,t\})}\|\Blbf\|_{L^\infty}\int_{\min\{s,t\}}^{\max\{s,t\}}|(\partial_{x_i}\VSlo)(\tau)|d\tau,
    \end{multline} and
 \begin{multline}\label{viVS}
    |(\partial_{v_i}\VSlo)(s)-\partial_{v_i}v|\lesssim   \sup_{\tau \in (\min\{s,t\},\max\{s,t\})}\|\nabla_x (\Elbf,\Blbf)(\tau)\|_{L^\infty}\int_{\min\{s,t\}}^{\max\{s,t\}}|(\partial_{v_i}\XSlo)(\tau)|d\tau \\
    + \sup_{\tau \in (\min\{s,t\},\max\{s,t\})}\|\Blbf\|_{L^\infty}\int_{\min\{s,t\}}^{\max\{s,t\}}|(\partial_{v_i}\VSlo)(\tau)|d\tau,
    \end{multline}
since $|\partial_{v_i}\hat{\mathcal{V}}_\pm^{l+1}(\tau)|\le 2|(\partial_{v_i}\VSlo)(\tau)|$. 
In addition, by \eqref{leading char}, we also obtain
\begin{equation}\label{XS representation}
         \XSlo(s;t,x,v)=x +\int_t^s d\tau\  \hat{\mathcal{V}}_\pm^{l+1}(\tau)= x +\int_t^s d\tau \left((\hat{v}_\pm)+\int _t ^\tau d\tau'\ \frac{d\hat{\mathcal{V}}_\pm^{l+1}(\tau';t,x,v)}{d\tau'}\right).
         \end{equation}
         Differentiating \eqref{XS representation}, we obtain\begin{align*}
      &  (\partial_{x_i}\XSlo)(s;t,x,v)=\partial_{x_i}x  +\int_t^s d\tau \int _t ^\tau d\tau'\ \frac{d}{d\tau'}\partial_{x_i}(\hat{\mathcal{V}}_\pm^{l+1}(\tau';t,x,v))\\
    &=\partial_{x_i}x  +\int_t^s d\tau \int _t ^\tau d\tau'\ \frac{d}{d\tau'}\left(\frac{(\partial_{x_i}\VSlo)(\tau')}{\sqrt{m_\pm^2+|\VSlo(\tau')|^2}}- \VSlo(\tau')\frac{\VSlo(\tau')\cdot (\partial_{x_i} \VSlo)(\tau')}{(1+|\VSlo(\tau')|^2)^\frac{3}{2}}\right)\\
    &=\partial_{x_i}x  +\int_t^s d\tau \int _t ^\tau d\tau'\ \bigg(\frac{\partial_{x_i}\frac{d\VSlo}{d\tau'}(\tau')}{\sqrt{m_\pm^2+|\VSlo(\tau')|^2}}-(\partial_{x_i}\VSlo)(\tau')\frac{\VSlo(\tau')\cdot (\frac{d\VSlo}{d\tau'})(\tau')}{(1+|\VSlo(\tau')|^2)^\frac{3}{2}}\\
    &- \frac{d\VSlo}{d\tau'}(\tau')\frac{\VSlo(\tau')\cdot (\partial_{x_i} \VSlo)(\tau')}{(1+|\VSlo(\tau')|^2)^\frac{3}{2}}-\VSlo(\tau')\frac{\frac{d\VSlo}{d\tau'}(\tau')\cdot (\partial_{x_i} \VSlo)(\tau')}{(1+|\VSlo(\tau')|^2)^\frac{3}{2}}\\
   & -\VSlo(\tau')\frac{\VSlo(\tau')\cdot (\partial_{x_i} \frac{d\VSlo}{d\tau'})(\tau')}{(1+|\VSlo(\tau')|^2)^\frac{3}{2}}+3\VSlo(\tau')\frac{\VSlo(\tau')\cdot (\partial_{x_i} \VSlo)(\tau')}{(1+|\VSlo(\tau')|^2)^{\frac{5}{2}}}\VSlo(\tau')\cdot \frac{d\VSlo}{d\tau'}(\tau')\bigg)\\
     &=\partial_{x_i}x  +\int_t^s d\tau \int _t ^\tau d\tau'\ \bigg[\frac{\partial_{x_i}(\FS(\tau',\XSlo(\tau'),\VSlo(\tau')))}{\sqrt{m_\pm^2+|\VSlo(\tau')|^2}}
  \\ & - \FS(\tau',\XSlo(\tau'),\VSlo(\tau'))\frac{\VSlo(\tau')\cdot (\partial_{x_i} \VSlo)(\tau')}{(1+|\VSlo(\tau')|^2)^\frac{3}{2}}\\
    &-\VSlo(\tau')\frac{\FS(\tau',\XSlo(\tau'),\VSlo(\tau'))\cdot (\partial_{x_i} \VSlo)(\tau')}{(1+|\VSlo(\tau')|^2)^\frac{3}{2}}
  \\  &-\VSlo(\tau')\frac{\VSlo(\tau')\cdot \partial_{x_i} (\FS(\tau',\XSlo(\tau'),\VSlo(\tau')))}{(1+|\VSlo(\tau')|^2)^\frac{3}{2}}\\
    &+\bigg(3\VSlo(\tau')\frac{\VSlo(\tau')\cdot (\partial_{x_i} \VSlo)(\tau')}{(1+|\VSlo(\tau')|^2)^{\frac{5}{2}}}-\frac{(\partial_{x_i}\VSlo)(\tau')}{(1+|\VSlo(\tau')|^2)^{\frac{3}{2}}}\bigg)\VSlo(\tau')\cdot (\FS(\tau',\XSlo(\tau'),\VSlo(\tau')))\bigg].
    \end{align*} Note that $$\partial_{x_i}(\FS(\tau',\XSlo(\tau'),\VSlo(\tau')))=\nabla_{\XSlo} \FS\cdot \partial_{x_i}\XSlo+\nabla_{\VSlo} \FS\cdot \partial_{x_i}\VSlo,$$ and $$|\nabla_{\XSlo} \FS|\lesssim \|\nabla_{\XSlo} (\Elbf,\Blbf)\|_{L^\infty}\textup{ and }|\nabla_{\VSlo}\FS|\lesssim \frac{\|\Blbf\|_{L^\infty}}{\langle \VSlo\rangle}.$$
    Thus, we have
    \begin{equation}\begin{split}
\label{xiXS} &|(\partial_{x_i}\XSlo)(s)-\partial_{x_i}x|
\lesssim \int_{\min\{s,t\}}^{\max\{s,t\}}d\tau \int_{\min\{s,t\}}^\tau d\tau'  \frac{1}{\sqrt{m_\pm^2+|\VSlo(\tau')|^2}}\\&\times\bigg(\|D (\Elbf,\Blbf)(\tau')\|_{L^\infty}|\partial_{x_i}\XSlo(\tau')|+\|\Blbf(\tau')\|_{L^\infty}|\partial_{x_i}\VSlo(\tau')|
+\|\FS(\tau')\|_{L^\infty}|\partial_{x_i}\VSlo(\tau')|\bigg)
\\
&\lesssim \frac{1}{(\vZ)}\int_{\min\{s,t\}}^{\max\{s,t\}}d\tau \int_{\min\{s,t\}}^\tau d\tau'  \frac{1+ |\VSlo(\tau')|+|\int_{\tau'}^t d\tau'' \FS(\tau'')|}{\sqrt{m_\pm^2+|\VSlo(\tau')|^2}}\\&\times \bigg(\|D (\Elbf,\Blbf)(\tau')\|_{L^\infty}|\partial_{x_i}\XSlo(\tau')|+\|\Blbf(\tau')\|_{L^\infty}|\partial_{x_i}\VSlo(\tau')|
+\|\FS(\tau')\|_{L^\infty}|\partial_{x_i}\VSlo(\tau')|\bigg)\\
&\lesssim \frac{1}{(\vZ)}\int_{\min\{s,t\}}^{\max\{s,t\}}d\tau \int_{\min\{s,t\}}^\tau d\tau'  \bigg(\|D (\Elbf,\Blbf)(\tau')\|_{L^\infty}|\partial_{x_i}\XSlo(\tau')|+\|\Blbf(\tau')\|_{L^\infty}|\partial_{x_i}\VSlo(\tau')|\\
&+\|\FS(\tau')\|_{L^\infty}|\partial_{x_i}\VSlo(\tau')|+ |s-t|\sup_{\tau \in (\min\{s,t\},\max\{s,t\})}\|\FS(\tau)\|_{L^\infty}\bigg)
\\
&\lesssim\frac{|s-t|^3}{(\vZ)}\sup_{\tau \in (\min\{s,t\},\max\{s,t\})}\|\FS(\tau)\|_{L^\infty}\\
       &+ \frac{|s-t|}{(\vZ)}\sup_{\tau \in (\min\{s,t\},\max\{s,t\})}\|D (\Elbf,\Blbf)(\tau)\|_{L^\infty}\int_{\min\{s,t\}}^{\max\{s,t\}}|(\partial_{x_i}\XSlo)(\tau)|d\tau\\
       &+ \frac{|s-t|}{(\vZ)}\sup_{\tau \in (\min\{s,t\},\max\{s,t\})}\|\FS(\tau)\|_{L^\infty}\int_{\min\{s,t\}}^{\max\{s,t\}}|(\partial_{x_i}\VSlo)(\tau)|d\tau,
    \end{split}\end{equation}
    where for the second inequality we used
$$v=\VSlo(\tau')+\int_{\tau'}^t d\tau''\FS(\tau'').$$
    Similarly, we have
\begin{multline}\notag
        (\partial_{v_i}\XSlo)(s;t,x,v)     =\int_t^s d\tau \int _t ^\tau d\tau'\ \bigg[\frac{\partial_{v_i}(\FS(\tau',\XSlo(\tau'),\VSlo(\tau')))}{\sqrt{m_\pm^2+|\VSlo(\tau')|^2}}
 \\*   - \FS(\tau',\XSlo(\tau'),\VSlo(\tau'))\frac{\VSlo(\tau')\cdot (\partial_{v_i} \VSlo)(\tau')}{(1+|\VSlo(\tau')|^2)^\frac{3}{2}}
    \\*
    -\VSlo(\tau')\frac{\FS(\tau',\XSlo(\tau'),\VSlo(\tau'))\cdot (\partial_{v_i} \VSlo)(\tau')}{(1+|\VSlo(\tau')|^2)^\frac{3}{2}}
  \\*  -\VSlo(\tau')\frac{\VSlo(\tau')\cdot \partial_{v_i} (\FS(\tau',\XSlo(\tau'),\VSlo(\tau')))}{(1+|\VSlo(\tau')|^2)^\frac{3}{2}}\\*
+\bigg(3\VSlo(\tau')\frac{\VSlo(\tau')\cdot (\partial_{v_i} \VSlo)(\tau')}{(1+|\VSlo(\tau')|^2)^{\frac{5}{2}}}-\frac{(\partial_{v_i}\VSlo)(\tau')}{(1+|\VSlo(\tau')|^2)^{\frac{3}{2}}}\bigg)\VSlo(\tau')\cdot \FS(\tau',\XSlo(\tau'),\VSlo(\tau'))\bigg],\end{multline} and hence 
     \begin{equation}\begin{split}
    \label{viXS}|(\partial_{v_i}\XSlo)(s)|&\lesssim  \frac{|s-t|^3}{(\vZ)}\sup_{\tau \in (\min\{s,t\},\max\{s,t\})}\|\FS(\tau)\|_{L^\infty}\\
      & + \frac{|s-t|}{(\vZ)}\sup_{\tau \in (\min\{s,t\},\max\{s,t\})}\|D (\Elbf,\Blbf)(\tau)\|_{L^\infty}\int_{\min\{s,t\}}^{\max\{s,t\}}|(\partial_{v_i}\XSlo)(\tau)|d\tau\\
      & + \frac{|s-t|}{(\vZ)}\sup_{\tau \in (\min\{s,t\},\max\{s,t\})}\|\FS(\tau)\|_{L^\infty}\int_{\min\{s,t\}}^{\max\{s,t\}}|(\partial_{v_i}\VSlo)(\tau)|d\tau.
    \end{split}\end{equation}
    Collecting these, we can summarize the estimates as follows. 
        Suppose that $$\sup_{\tau \in (\min\{s,t\},\max\{s,t\})}\|\nabla_x (\Elbf,\Blbf)(\tau)\|_{L^\infty}\le C_1,$$ and $$\sup_{\tau \in (\min\{s,t\},\max\{s,t\})}\|\FS(\tau)\|_{L^\infty}\le C_2,$$ for some $C_1,C_2>0.$ Then we have
\begin{equation}\label{derivatives of char.proto}
    \begin{split}|(\partial_{x_i}\XSlo)(s)|&\lesssim \bigg(1+\frac{|s-t|^3}{(\vZ)}C_2
      \bigg) \exp\bigg( \frac{|s-t|^2}{(\vZ)}C_1\bigg(1+ C_2|s-t|e^{C_2|s-t|}\bigg)\bigg)\eqdef C_3,\\
      |(\partial_{x_i}\VSlo)(s)|&\lesssim   C_1|s-t|\bigg(1+|s-t|^3C_2
      \bigg) \\&\ \times \exp\bigg(C_2|s-t|+ |s-t|^2C_1\bigg(1+ C_2|s-t|e^{C_2|s-t|}\bigg)\bigg)\eqdef C_1|s-t|C_4,\\
      |(\partial_{v_i}\XSlo)(s)|&\lesssim \frac{|s-t|^3}{(\vZ)}C_2
       \exp\bigg( \frac{|s-t|^2}{(\vZ)}\bigg(C_1+C_2+ C_1C_2|s-t|e^{C_2|s-t|}\bigg)\bigg),\\
       |(\partial_{v_i}\VSlo)(s)|&\lesssim 1+C_1e^{C_2|s-t|} |s-t|^4C_2
      \exp\bigg( |s-t|^2\bigg(C_1+C_2+ C_1C_2|s-t|e^{C_2|s-t|}\bigg)\bigg).
    \end{split}
\end{equation}for $i=1,2,3,$ where we denote $\XSlo(s)=\XSlo(s;t,x,v)$ and $\VSlo(s)=\VSlo(s;t,x,v)$.

If $i,j=1,2,3$ and $i\ne j$, then we can further have
 \begin{multline}\label{dxixj.proto}
            |(\partial_{x_i}(\XSlo)_j)(s)|
\lesssim\bigg(\frac{|s-t|^3}{(\vZ)}C_2
       + \frac{|s-t|}{(\vZ)}C_1C_4\bigg(1+ C_2|s-t|e^{C_2|s-t|}\bigg)\bigg)\\\times \exp\bigg(\frac{|s-t|^2}{(\vZ)}C_1\bigg(1+ C_2|s-t|e^{C_2|s-t|}\bigg)\bigg)\bigg).
        \end{multline}

        The proof is mainly by \eqref{xiVS}, \eqref{viVS}, \eqref{xiXS}, \eqref{viXS}, and the Gr\"onwall lemma. By \eqref{xiVS}, we first have\begin{equation}\label{dxV middle}
             |(\partial_{x_i}\VSlo)(s)|\lesssim   C_1e^{C_2|s-t|} \int_{\min\{s,t\}}^{\max\{s,t\}}|(\partial_{x_i}\XSlo)(\tau)|d\tau,
        \end{equation}
        by the Gr\"onwall lemma. Plugging this into \eqref{xiXS}, we have
        \begin{equation}\begin{split}\label{dxX middle}
            |(\partial_{x_i}\XSlo)(s)-\partial_{x_i}x|
&\lesssim\frac{|s-t|^3}{(\vZ)}C_2
       + \frac{|s-t|}{(\vZ)}C_1\int_{\min\{s,t\}}^{\max\{s,t\}}|(\partial_{x_i}\XSlo)(\tau)|d\tau\\
      & + \frac{|s-t|}{(\vZ)}C_2\int_{\min\{s,t\}}^{\max\{s,t\}}C_1e^{C_2|\tau-t|} \int_{\min\{\tau,t\}}^{\max\{\tau,t\}}|(\partial_{x_i}\XSlo)(\tau')|d\tau'd\tau\\
&\lesssim\frac{|s-t|^3}{(\vZ)}C_2
       + \frac{|s-t|}{(\vZ)}C_1\bigg(1+ C_2|s-t|e^{C_2|s-t|}\bigg)\int_{\min\{s,t\}}^{\max\{s,t\}}|(\partial_{x_i}\XSlo)(\tau)|d\tau.
        \end{split}\end{equation}
        Finally, by the Gr\"onwall lemma, we have
        \begin{equation}\notag
             |(\partial_{x_i}\XSlo)(s)|\lesssim \bigg(1+\frac{|s-t|^3}{(\vZ)}C_2
      \bigg) \exp\bigg( \frac{|s-t|^2}{(\vZ)}C_1\bigg(1+ C_2|s-t|e^{C_2|s-t|}\bigg)\bigg).
        \end{equation} Plugging this back into \eqref{dxV middle}, we obtain the upper bound for $|\partial_{x_i}\VSlo|.$

        Similarly, by \eqref{viVS}, we first have\begin{equation}\label{dvV middle}
             |(\partial_{v_i}\VSlo)(s)|\lesssim   1+C_1e^{C_2|s-t|} \int_{\min\{s,t\}}^{\max\{s,t\}}|(\partial_{v_i}\XSlo)(\tau)|d\tau,
        \end{equation}
        by the Gr\"onwall lemma. Plugging \eqref{dvV middle} into \eqref{viXS}, we obtain \begin{align*}
            |(\partial_{v_i}\XSlo)(s)|
&\lesssim\frac{|s-t|^3}{(\vZ)}C_2
       + \frac{|s-t|}{(\vZ)}C_1\int_{\min\{s,t\}}^{\max\{s,t\}}|(\partial_{v_i}\XSlo)(\tau)|d\tau\\
     &  + \frac{|s-t|}{(\vZ)}C_2\int_{\min\{s,t\}}^{\max\{s,t\}}(1+C_1e^{C_2|\tau-t|}) \int_{\min\{\tau,t\}}^{\max\{\tau,t\}}|(\partial_{v_i}\XSlo)(\tau')|d\tau'd\tau\\
&\lesssim\frac{|s-t|^3}{(\vZ)}C_2
       + \frac{|s-t|}{(\vZ)}\bigg(C_1+C_2+ C_1C_2|s-t|e^{C_2|s-t|}\bigg)\int_{\min\{s,t\}}^{\max\{s,t\}}|(\partial_{v_i}\XSlo)(\tau)|d\tau.
        \end{align*} Then the Gr\"onwall lemma yields the bound for $|\partial_{v_i}\XSlo|$. Plugging this bound again into \eqref{dvV middle} we obtain the upper bound for $|\partial_{v_i}\VSlo|$. 
        Lastly, for $i,j=1,2,3$, if $i\ne j$, we can further get from \eqref{dxX middle} that
        \begin{multline}\label{dxixj middle}
            |(\partial_{x_i}(\XSlo)_j)(s)|
\\
\lesssim\frac{|s-t|^3}{(\vZ)}C_2
       + \frac{|s-t|}{(\vZ)}C_1\bigg(1+ C_2|s-t|e^{C_2|s-t|}\bigg)\int_{\min\{s,t\}}^{\max\{s,t\}}(C_3+|(\partial_{x_i}(\XSlo)_j)(\tau)|)d\tau,
        \end{multline}where we used $|\partial_{x_i}(\XSlo)_k|\lesssim C_3$ from \eqref{derivatives of char.proto}. Note that $$C_3\le \bigg(1+|s-t|^3C_2
      \bigg) \exp\bigg(C_2|s-t|+ |s-t|^2C_1\bigg(1+ C_2|s-t|e^{C_2|s-t|}\bigg)\bigg)\eqdef C_4.$$ Then we use the Gr\"onwall lemma on \eqref{dxixj middle} and obtain
      \begin{multline}\notag
            |(\partial_{x_i}(\XSlo)_j)(s)|
\lesssim\bigg(\frac{|s-t|^3}{(\vZ)}C_2
       + \frac{|s-t|}{(\vZ)}C_1C_4\bigg(1+ C_2|s-t|e^{C_2|s-t|}\bigg)\bigg)\\\times \exp\bigg(\frac{|s-t|^2}{(\vZ)}C_1\bigg(1+ C_2|s-t|e^{C_2|s-t|}\bigg)\bigg)\bigg).
        \end{multline} 
As a corollary, we obtain the following lemma.
\begin{lemma}
    \label{cor.tbxbvb deri}Define $\FS$ as \eqref{eq.forcing}. 
                Suppose that\begin{equation}\sup_{0\le \tau\le T}\|\nabla_x (\Elbf,\Blbf)(\tau)\|_{L^\infty}\le C_1,\text { and }
                    \label{F bound C2}\sup_{0\le \tau\le T}\|\FS(\tau)\|_{L^\infty}\le C_2,
            \end{equation}for some $C_1,C_2>0.$ Then for any $s,t\in (0,T)$ we have
\begin{equation}\label{derivatives of char}
|\nabla_x \XSlo(s)|,\ |\nabla_x \VSlo(s)|, \ (\vZ)|\nabla_v \XSlo(s)|,\ |\nabla_v \VSlo(s)|\lesssim_T 1,
\end{equation}where we denote $\XSlo(s)=\XSlo(s;t,x,v)$ and $\VSlo(s)=\VSlo(s;t,x,v)$. If $i,j=1,2,3$ and $i\ne j$, then we can further have
 \begin{equation}\label{dxixj}
            (\vZ)|(\partial_{x_i}(\XSlo)_j)(s)|
\lesssim_T 1.
        \end{equation}
        Thus, we also have  
for $i=1,2,3,$ 
\begin{equation}\label{derivatives of tbxbvb}
    \begin{split}
    |\partial_{x_i}\tblo|&\lesssim_T \frac{1}{|(\hat{\mathcal{V}}_\pm^{l+1})_3(t-\tblo)|}, \quad 
|\partial_{v_i}\tblo|\lesssim_T \frac{1}{|(\hat{\mathcal{V}}_\pm^{l+1})_3(t-\tblo)|}\frac{\tblo}{\vZ},\\
         |\partial_{x_i} (\xblo)_\parallel|&\lesssim    |\partial_{x_i} \tblo|  +|    (\partial_{x_i} \XSlo)_\parallel(t - \tblo)|, \quad 
 | \partial_{x_i} (\xblo)_3|\lesssim_T 1 , \\
      |\partial_{v_i} (\xblo)_\parallel|&\lesssim    |\partial_{v_i} \tblo|  +|    (\partial_{v_i} \XSlo)_\parallel(t - \tblo)|, \quad
|\partial_{v_i} (\xblo)_3|\lesssim_T \frac{\tblo}{\vZ},\\
|  \partial_{x_i} \vblo|&\lesssim_T \left(\frac{1}{|(\hat{\mathcal{V}}_\pm^{l+1})_3(t - \tblo)  |}+1\right),\quad
|\partial_{v_i} \vblo|\lesssim_T \left(\frac{1}{|(\hat{\mathcal{V}}_\pm^{l+1})_3(t - \tblo)  |}+1\right)\frac{\tblo}{\vZ}.
    \end{split}
\end{equation}For $i=1,2$, we can further have
\begin{equation}\label{derivatives of tbxbvb better}
    \begin{split}
    |\partial_{x_i}\tblo|&\lesssim_T \frac{1}{|(\hat{\mathcal{V}}_\pm^{l+1})_3(t-\tblo)|}\frac{\tblo}{\vZ},\quad 
    | \partial_{x_i} (\xblo)_3|\lesssim_T \frac{\tblo}{\vZ},\\ 
    |  \partial_{x_i} \vblo|&\lesssim_T\left(\frac{1}{|(\hat{\mathcal{V}}_\pm^{l+1})_3(t - \tblo)  |}+1\right)\frac{\tblo}{\vZ}.
    \end{split}\end{equation}
    \end{lemma}
    \begin{proof}Note that \eqref{derivatives of char} and \eqref{dxixj} are simplified presentations of \eqref{derivatives of char.proto} and \eqref{dxixj.proto}, respectively. The first line of \eqref{derivatives of tbxbvb} is a direct consequence of \eqref{partial xv tb} and \eqref{derivatives of char} with $s=t-\tblo$. The derivatives of the parallel components in the second and third lines of \eqref{derivatives of tbxbvb} are directly by \eqref{Jacobian for xbvb}.  For the derivatives normal component in the second and third lines of \eqref{derivatives of tbxbvb}, we observe that \eqref{Jacobian for xbvb} implies
        \begin{equation}\notag
    \begin{split}
 \partial_{x_i} (\xblo)_3&=  -(\hat{\mathcal{V}}_\pm^{l+1})_3(t - \tblo; t, x, v)    \partial_{x_i} \tblo  +(    \partial_{x_i} \XSlo)_3(t - \tblo; t, x, v), \textup{ and }\\
         \partial_{v_i} (\xblo)_3& = -(\hat{\mathcal{V}}_\pm^{l+1})_3(t - \tblo; t, x, v)    \partial_{v_i} \tblo  +(    \partial_{v_i} \XSlo)_3(t - \tblo; t, x, v).
    \end{split}
\end{equation}Then by using the first line of \eqref{derivatives of tbxbvb} as well as \eqref{derivatives of char} with $s=t-\tblo,$ we obtain the derivatives of the normal component. 
Furthermore, for the fourth line of \eqref{derivatives of tbxbvb}, we recall \eqref{Jacobian for xbvb} which goes by
\begin{equation*}
    \begin{split}
      \partial_{x_i} \vblo&=- (\partial_{x_i}\tblo) \FS(t - \tblo, \xblo,\vblo)+ (\partial_{x_i} \VSlo)(t - \tblo; t, x, v),\\
    \partial_{v_i} \vblo&=-  (\partial_{v_i}\tblo  )\FS(t - \tblo, \xblo,\vblo)+ (\partial_{v_i} \VSlo)(t - \tblo; t, x, v).
    \end{split}
\end{equation*}Then by using \eqref{F bound C2}, \eqref{derivatives of char}, and the first line of \eqref{derivatives of tbxbvb}, we obtain the last line of \eqref{derivatives of tbxbvb}.

Lastly, if $i=1,2$, then we can use \eqref{dxixj} and \eqref{partial xv tb} instead of \eqref{derivatives of char} and obtain \eqref{derivatives of tbxbvb better}.  This completes the proof.
    \end{proof}

\subsection{First-Round Estimate}
This section is devoted to obtaining a global-in-time uniform upper-bound estimate for the derivatives of $\flo$. Note that the representation \eqref{solution f} of $\flo$ consists of the initial-value part and the boundary-value part, and the derivatives on $\flo$ involves the derivatives of the characteristics $\nabla_x \ZS$ and hence $\nabla_x \tblo$. As we have already observed in \eqref{partial xv tb}, the derivatives of $\tblo$ contains possible singularity on $(\hat{\mathcal{V}}_\pm^{l+1})_3(t-\tblo)$ and we have to handle this singularity to obtain a derivative estimate for $f.$ To this end, we define the following kinetic-type weights:\begin{definition}
   Define the kinetic weight \begin{equation}\label{alpha}
    \alpha_\pm(t,x,v)= \sqrt{x_3^2+|(\hat{v}_\pm)_3|^2-2 \left((\FS)_3(t,x_\parallel,0,v)\right)\frac{x_3}{(\vZ)}
}.
\end{equation} This weight is well-defined as long as $-(\FS)_3> 0.$ Note that \begin{equation}\label{alphav3}
\alpha_\pm(t,x_\parallel,0,v)=|(\hat{v}_\pm)_3|.
\end{equation} In addition,  we define a special weight in the form of \begin{equation}
\label{alpha.s.def}\tilde{\alpha}_{\pm}^2(t,x,v)\eqdef \frac{\alpha_\pm^2(t,x,v)}{1+\alpha_\pm^2(t,x,v)}.
  \end{equation}  This special weight $\tilde{\alpha}_{\pm}^2$ is uniformly bounded from above by 1 and is small when $\alpha_\pm$ is small. One crucial property of $\tilde{\alpha}_{\pm}$ is on the fact that it vanishes at the grazing point $(x_3,v_3)=(0,0)$ as \begin{equation}\label{alphav3.s}\tilde{\alpha}_{\pm}(t,x_\parallel,0,v)= \frac{|(\hat{v}_\pm)_3|}{\sqrt{1+|(\hat{v}_\pm)_3|^2}}.\end{equation}
\end{definition}

\begin{remark}
We note that the form of the weight \eqref{alpha.s.def} differs slightly from the classical kinetic weight $\alpha$ introduced in~\cite{MR1354697}, as well as from the variant employed in~\cite{MR4645724}. 
In particular, even for large values of $\alpha$, the weight $\tilde{\alpha}$ remains uniformly bounded by $1$.
\end{remark}
We first study the upper-bound estimates of $\tilde{\alpha}_{\pm}^2$ along the characteristic trajectory. 
We introduce the following velocity lemma. The velocity lemma is originally established in \cite{MR1354697}. \begin{lemma}[Velocity Lemma]
    \label{lemma.trajectory of alpha.s}
        Let $\alpha_\pm$ and $\tilde{\alpha}_{\pm}$ be defined as in \eqref{alpha} and \eqref{alpha.s.def}, respectively. Define $\FS$ as \eqref{eq.forcing}. Suppose $$
        \sup_{0\le t\le T} \bigg(\|\Elbf(t)\|_{L^\infty}+ \|\Blbf(t)\|_{L^\infty} + \|(\partial_t,\nabla_x)\FS(t)\|_{L^\infty} \bigg)<C.
   $$ Suppose that for all $(t,x_\parallel)\in (0,T)\times  
   \mathbb{R}^2
   ,$ $-(\FS)_3(t,x_\parallel,0)>c_0,$  for some $c_0>0.$ Then for any $(t,x,v)\in(0,T)\times \Omega\times \rth,$ with the trajectory $\XSlo(s;t,x,v)$ and $\VSlo(s;t,x,v)$ satisfying \eqref{leading char}, 
   \begin{equation}
       \label{alpha.s.lem}
       e^{-10\frac{C}{c_0}|t-s|}\tilde{\alpha}_{\pm}(t,x,v)\le \tilde{\alpha}_{\pm}(s,\XSlo(s;t,x,v),\VSlo(s;t,x,v))\le e^{10\frac{C}{c_0}|t-s|}\tilde{\alpha}_{\pm}(t,x,v)
   \end{equation}In addition, define the material derivative
    $\frac{D}{Dt}\eqdef \partial_t+(\hat{v}_\pm)\cdot \nabla_x +\FS\cdot \nabla_v.$ Then we have 
    \begin{equation}
        \label{bound of DalphaDt}
    \left|\frac{D}{Dt}\alpha_\pm^2
    \right|\le 20\frac{C}{c_0}\alpha_\pm^2.\end{equation}
\end{lemma}

\begin{proof}
    We first observe that
    $$\frac{D}{Dt}\tilde{\alpha}_{\pm}^2=\frac{1}{1+\alpha_\pm^2}\frac{D}{Dt}\alpha_\pm^2-\frac{\alpha_\pm^2}{(1+\alpha_\pm^2)^2}\frac{D}{Dt}\alpha_\pm^2=\frac{1}{(1+\alpha_\pm^2)^2}\frac{D}{Dt}\alpha_\pm^2.$$Then using the bound \eqref{bound of DalphaDt} of the material derivative $\frac{D}{Dt}\alpha_\pm^2$ we further obtain
    $$\frac{D}{Dt}\tilde{\alpha}_{\pm}^2\le 20\frac{C}{c_0(1+\alpha_\pm^2)}\frac{\alpha_\pm^2}{1+\alpha_\pm^2}\le 20\frac{C}{c_0}\tilde{\alpha}_{\pm}^2.$$ By the Gr\"onwall lemma, we finally obtain $$\tilde{\alpha}_{\pm}^2(s,\XSlo(s),\VSlo(s))\le e^{\frac{20C}{c_0}|t-s|}\tilde{\alpha}_{\pm}^2(t,x,v).$$
    This completes the proof of Lemma \ref{lemma.trajectory of alpha.s}.  Lastly, the proof of \eqref{bound of DalphaDt} follows by \cite[Eq. (4.10)]{MR4645724} with $E_e=B_e=0$ and $C_1=C.$
\end{proof}

Now we will prove that this weight $\tilde{\alpha}_{\pm}$ also satisfies the following crucial property that if $t\le \tblo,$ $$t\lesssim \sup_{t-\tblo<s<t}\sqrt{m_\pm^2+|\VSlo(s)|^2}\tilde{\alpha}_{\pm} (0,\XSlo(0),\VSlo(0)).$$
To prove this, we first need to obtain the following prerequisite lemma.
\begin{lemma}
    \label{v3hat bound}
    For $(t,x,v)\in(0,T)\times \Omega\times \rth,$ let the trajectory $\XSlo(s;t,x,v)$ and $\VSlo(s;t,x,v)$ satisfy \eqref{leading char}. 
    Suppose for all $t,x,v$, \begin{equation}\label{F0 bound}
        -(\FS)_3(t,x_\parallel,0,v)>c_0,
    \end{equation} for some $c_0>0$. Then if $t< \tblo,$ we have
    \begin{equation}\label{v3hat bound by alpha.s}
        |(\hat{v}_\pm)_3|^2\le \frac{2\alpha_\pm^2(t,x,v)}{1+\alpha_\pm^2(t,x,v)}.
    \end{equation}
\end{lemma}
\begin{proof}
    Note that since we have \eqref{F0 bound}, we first observe from the definition of $\alpha_\pm$ in \eqref{alpha} that $\alpha_\pm^2(t,x,v)\ge |(\hat{v}_\pm)_3|^2.$ Since $|(\hat{v}_\pm)_3|\le 1,$ we obtain 
    $$|(\hat{v}_\pm)_3|^2\le \alpha_\pm^2(t,x,v)\le (2-|(\hat{v}_\pm)_3|^2)\alpha_\pm^2(t,x,v).$$ This provides the final conclusion \eqref{v3hat bound by alpha.s}.
\end{proof}
As a corollary of Lemma \ref{lemma.trajectory of alpha.s} and Lemma \ref{v3hat bound}, we can prove the following crucial lemma.
\begin{lemma}
    \label{lemma.t bound by v and alpha.s}
    For $(t,x,v)\in(0,T)\times \Omega\times \rth,$ let the trajectory $\XSlo(s;t,x,v)$ and $\VSlo(s;t,x,v)$ satisfy \eqref{leading char}. 
    Suppose for all $t,x,v$, assume \eqref{F0 bound} for some $c_0>0$, then there exists a constant $C$ depending on $T,$ $g,$ $\|\Elbf|_{W^{1,\infty}((0,T)\times \Omega)},$ and $\|\Blbf\|_{W^{1,\infty}((0,T)\times \Omega)},$ such that if $t< \tblo,$ then
    \begin{equation}\label{t bound by v and alpha.s}
        t<\max\left\{\langle \VSlo(0)\rangle,(\vZ)\right\}\frac{\sqrt{2}}{c_0}\left(1+e^{\frac{10CT}{c_0}}\right)\tilde{\alpha}_{\pm}(0,\XSlo(0),\VSlo(0)).
    \end{equation}
    Furthermore, if $t< \tblo$ and $s\in (0,\min\{\tblo,T\})$ then    \begin{equation}\label{t-s bound by v and alpha.s}
        |t-s|<\max\left\{\langle \VSlo(s)\rangle,(\vZ)\right\}\frac{\sqrt{2}}{c_0}\left(1+e^{\frac{10C|t-s|}{c_0}}\right)\tilde{\alpha}_{\pm}(s,\XSlo(s),\VSlo(s)).
    \end{equation}
\end{lemma}
\begin{proof}
    For $t<\tblo,$ we observe that
    $$\int_0^t c_0 ds <\int_0^t -(\FS)_3(s,\XSlo(s),\VSlo(s)) ds=(\VSlo)_3(0)-v_3.$$ Thus, we have
    \begin{multline*}
        c_0t<|(\VSlo)_3(0)|+|v_3|\le \langle\VSlo(0)\rangle |(\hat{\mathcal{V}}_\pm^{l+1})_3(0)|+(\vZ) |(\hat{v}_\pm)_3|\\
        \le \max\left\{\langle \VSlo(0)\rangle,(\vZ)\right\}(|(\hat{\mathcal{V}}_\pm^{l+1})_3(0)|+|(\hat{v}_\pm)_3|).
    \end{multline*} Now we use \eqref{v3hat bound by alpha.s} and further obtain
    $$c_0t<\max\left\{\langle \VSlo(0)\rangle,(\vZ)\right\}\sqrt{2}(\tilde{\alpha}_{\pm}(0,\XSlo(0),\VSlo(0))+\tilde{\alpha}_{\pm}(t,x,v)).$$ Finally, using \eqref{alpha.s.lem}, we obtain$$c_0t<\max\left\{\langle \VSlo(0)\rangle,(\vZ)\right\}\sqrt{2}\left(1+e^{\frac{10Ct}{c_0}}\right)\tilde{\alpha}_{\pm}(0,\XSlo(0),\VSlo(0)).$$ This completes the proof of Lemma \ref{t bound by v and alpha.s}.
\end{proof}

On the other hand, if $t\ge \tblo,$ we introduce the following bound on the singularity $\frac{1}{|(\hat{\mathcal{V}}_\pm^{l+1})_3|}.$
\begin{lemma}[Lemma 10 of \cite{MR4645724}]
\label{lemma.tb over v3}
     For $(t,x,v)\in(0,T)\times \Omega\times \rth,$ let the trajectory $\XSlo(s;t,x,v)$ and $\VSlo(s;t,x,v)$ satisfy \eqref{leading char}. 
    Suppose for all $t,x,v$, $-(\FS)_3(t,x_\parallel,0,v)>c_0,$  then there exists a constant $C$ depending on $T,$ $g,$ $\|\Elbf\|_{W^{1,\infty}((0,T)\times \Omega)},$ and $\|\Blbf|_{W^{1,\infty}((0,T)\times \Omega)},$ such that if $t\ge \tblo,$
    \begin{equation}\label{V3V0 bound max}
        \frac{\tblo(t,x,v)}{(\hat{\mathcal{V}}_\pm^{l+1})_3(t-\tblo)}\le \frac{C}{c_0}\max_{s\in\{t-\tblo,t\}}\sqrt{m_\pm^2+|\VSlo(s)|^2}.
    \end{equation} 
 \end{lemma}
\begin{proof}
    For $t\ge \tblo,$ we observe that
    $$\int_{t-\tblo}^t c_0 ds <\int_{t-\tblo}^t -(\FS)_3(s,\XSlo(s),\VSlo(s)) ds=(\VSlo)_3({t-\tblo})-v_3.$$ Thus, we have
    \begin{multline*}
        c_0\tblo<|(\VSlo)_3({t-\tblo})|+|v_3|\le \langle\VSlo({t-\tblo})\rangle |(\hat{\mathcal{V}}_\pm^{l+1})_3({t-\tblo})|+(\vZ) |(\hat{v}_\pm)_3|\\
        \le \max\left\{\langle \VSlo({t-\tblo})\rangle,(\vZ)\right\}(|(\hat{\mathcal{V}}_\pm^{l+1})_3({t-\tblo})|+\alpha_\pm(t,x,v)),
    \end{multline*} since $|(\hat{v}_\pm)_3|\le \alpha_\pm(t,x,v).$  Now we use \eqref{bound of DalphaDt} and further obtain
    \begin{multline*}
        c_0\tblo\lesssim \max\left\{\langle \VSlo(t-\tblo)\rangle,(\vZ)\right\}(|(\hat{\mathcal{V}}_\pm^{l+1})_3({t-\tblo})|+\alpha_\pm(t-\tblo,\XSlo(t-\tblo),\VSlo(t-\tblo)))\\
        \lesssim \max\left\{\langle \VSlo(t-\tblo)\rangle,(\vZ)\right\}|(\hat{\mathcal{V}}_\pm^{l+1})_3({t-\tblo})|.
    \end{multline*}  This completes the proof of Lemma \ref{lemma.tb over v3}.
\end{proof}

\subsection{Enhanced Estimates of the Momentum Derivatives}
By compensating for some loss of decay from the initial and boundary profiles (cf. Section \ref{sec.deri decay.st} in the stationary case), we can further prove some additional decay-in-$(x,v)$ estimates for the momentum derivatives of $\flo$. This additional decay will be crucial for the uniform estimates on the temporal derivatives of the electromagnetic fields $(\Elobf,\Blobf)$, which will be used for the uniform estimates on $x_3$-derivatives. To this end, we will prove the following decay estimates of the momentum derivatives:
\begin{proposition}\label{Prop.decay.mom.deri}Suppose the same assumptions made in Proposition \ref{prop.derivative}. 
In addition, suppose that the initial profile $F^{\textup{in}}_\pm$ and the incoming boundary profile $G_\pm$ further satisfies the following fast-decay condition on the first-order derivative in the velocity variable:
\begin{equation}
    \label{decay condition for hpm deri}
  \|\mathrm{w}^2_{\pm,\beta}(x,v)\nabla_{x,v}F^{\textup{in}}_\pm(x,v)\|_{L^\infty_{x,v}}+  \|\mathrm{w}^2_{\pm,\beta}(x_\parallel,0,v)\nabla_{x_\parallel,v}G_\pm(x_\parallel,v)\|_{L^\infty_{x_\parallel,v}}<\infty,
\end{equation}where the weight $\mathrm{w}_{\pm,\beta}$ is defined as in \eqref{weights.wholehalf}. 
Given \eqref{decay condition for hpm deri}, we will prove the following estimate on a sequential level; for each $l\in\mathbb{N},$ we have in $\mathbb{R}^2\times\mathbb{R}_+,$ \begin{equation}\label{decay of momentum deri F} \|\mathrm{w}_{\pm,\beta}\nabla_v \flo\|_{L^\infty_{x,v}}
\le C\left(\|\mathrm{w}^2_{\pm,\beta}(x,v)\nabla_{x,v}F^{\textup{in}}_\pm(x,v)\|_{L^\infty_{x,v}}+\|\mathrm{w}^2_{\pm,\beta}(x_\parallel,0,v)\nabla_{x_\parallel,v}G_\pm(x_\parallel,v)\|_{L^\infty_{x_\parallel,v}}\right), \end{equation} 
for some $C>0$.  
\end{proposition}
\begin{proof}
By taking the momentum derivative on $\flo$, we obtain as in \eqref{Derivatives of f}
\begin{align*}
      \nabla_v \flo(t, x, v) &= 1_{t \leq \tblo(t, x, v)}\bigg(
        \nabla_x F^{\textup{in}}_\pm(\XSlo(0; t, x, v), \VSlo(0; t, x, v)) \cdot \nabla_v \XSlo(0; t, x, v)\\
        &+\nabla_v F^{\textup{in}}_\pm(\XSlo(0; t, x, v), \VSlo(0; t, x, v))\cdot \nabla_v \VSlo(0; t, x, v)\bigg) \\
       &+ 1_{t > \tblo(t, x, v)}\bigg((\nabla_{x_\parallel} G_\pm)( (\xblo)_\parallel, \vblo) \cdot \nabla_v (\xblo)_\parallel
       + (\nabla_v G_\pm)( (\xblo)_\parallel, \vblo) \cdot \nabla_v \vblo\bigg).
\end{align*}Then we note that the derivatives of $\XSlo,$ $\VSlo$, $\xblo$ and $\vblo$ satisfy the upper-bounds estimates \eqref{derivatives of char} and \eqref{derivatives of tbxbvb}. Therefore, using \eqref{derivatives of char} and \eqref{derivatives of tbxbvb}, we observe that
  \begin{align*}
    &|\mathrm{w}_{\pm,\beta}\nabla_v \flo(t,x, v)|\\
&\le    \mathrm{w}_{\pm,\beta}(x,v) |(\nabla_x F^{\textup{in}}_\pm)(\XSlo(0; t, x, v), \VSlo(0; t, x, v))||\nabla_v \XSlo(0; t, x, v)|\\
&+  \mathrm{w}_{\pm,\beta}(x,v) |(\nabla_v F^{\textup{in}}_\pm)(\XSlo(0; t, x, v), \VSlo(0; t, x, v))||\nabla_v \VSlo(0; t, x, v)|\\
   & +  \mathrm{w}_{\pm,\beta}(x,v) |(\nabla_{x_\parallel} G_\pm)((\xblo)_\parallel, \vblo)||\nabla_v  (\xblo)_\parallel|
      +\mathrm{w}_{\pm,\beta}(x,v) |(\nabla_v G_\pm)( (\xblo)_\parallel, \vblo)||\nabla_v  \vblo|  \\
&\lesssim
 \mathrm{w}_{\pm,\beta}(x,v)\bigg(|(\nabla_x F^{\textup{in}}_\pm)(\XSlo(0; t, x, v), \VSlo(0; t, x, v))|
 +|(\nabla_v F^{\textup{in}}_\pm)(\XSlo(0; t, x, v), \VSlo(0; t, x, v))|\\
&+       \mathrm{w}_{\pm,\beta}(x,v)\bigg( |(\nabla_{x_\parallel} G_\pm)((\xblo)_\parallel, \vblo)|\left(\frac{\tblo}{|(\hat{V}^{l+1}_\pm)_3(-\tblo)|(\vZ) }+\frac{\tblo}{\vZ}\right)\\
      & +  |(\nabla_v G_\pm)((\xblo)_\parallel, \vblo)|\bigg|\frac{\tblo}{|(\hat{V}^{l+1}_\pm)_3(-\tblo)|(\vZ) }+(\vZ)^{-1}\bigg|\bigg)\\
       &\lesssim   \frac{1}{\mathrm{w}_{\pm,\beta}(x,v)}\left(\frac{\mathrm{w}_{\pm,\beta}(x,v)}{\mathrm{w}_{\pm,\beta}(t-\tblo(t,x, v),\xblo(t,x, v),\vblo(t,x, v))}\right)^2\\*&\times \bigg(|(\mathrm{w}^2_{\pm,\beta}\nabla_x F^{\textup{in}}_\pm)(\XSlo(0; t, x, v), \VSlo(0; t, x, v))|
 +|(\mathrm{w}^2_{\pm,\beta}\nabla_v F^{\textup{in}}_\pm)(\XSlo(0; t, x, v), \VSlo(0; t, x, v))|\\&+|(\mathrm{w}^2_{\pm,\beta}\nabla_{x_\parallel} G_\pm)((\xblo)_\parallel, \vb)|
       +  |(\mathrm{w}^2_{\pm,\beta}\nabla_v G_\pm)((\xblo)_\parallel, \vblo)|\bigg),    \end{align*} by Lemma \ref{lemma.tb over v3}. Then we further use the weight comparison \eqref{w comparison 1.whole} and observe that 
\begin{align*}
   & \frac{1}{\mathrm{w}_{\pm,\beta}(x,v)}\left(\frac{\mathrm{w}_{\pm,\beta}\left( \ZSlo\left(t ; t, x, v\right)\right)}{\mathrm{w}_{\pm,\beta}( \ZSlo(t-\tblo(t,x, v); t, x, v))} \right)^2
    \leq  \frac{1}{\mathrm{w}_{\pm,\beta}(x,v)}e^{\left(\left\|\Elbf\right\|_{L_{ x}^{\infty}} +1\right)\frac{16\beta}{5m_\pm g} (\sqrt{m_\pm^2+|v_\pm|^2}+m_\pm gx_3)}\\
 & \le   e^{(\min\{m_-,m_+\}g)\left(\frac{1}{8}+\frac{1}{32}\right)\frac{16\beta}{5m_\pm g} (\sqrt{m_\pm^2+|v_\pm|^2}+m_\pm gx_3)}e^{-\beta v_\pm^0-m_\pm g\beta x_3-\frac{\beta}{2}|x_{\parallel}|}\\
 &  \le  e^{\frac{\beta}{2} (\sqrt{m_\pm^2+|v_\pm|^2}+m_\pm gx_3)}e^{-\beta v_\pm^0-m_\pm g\beta x_3-\frac{\beta}{2}|x_{\parallel}|}
   \le 1,
\end{align*}given that $\Elbf$ satisfies the upper-bound \eqref{final estimate for flo.st} and \eqref{apriori_EB} and that 
$\min\{m_+,m_-\}g\gg 1.$ This proves the decay estimate \eqref{decay of momentum deri F} for the momentum derivative $\nabla_v \flo.$
This completes the proof.
    
\end{proof}

We close this section by introducing uniform-boundedness estimates on the derivatives.

\begin{proposition}\label{prop.derivative}
    Fix $m>4.$ Define $\FS$ as \eqref{eq.forcing}.     Suppose that the initial-boundary data satisfy \begin{align*}
        \|(\vZ)^m \nabla_{x_\parallel} F^{\textup{in}}_\pm\|_{L^\infty_{t,x,v}}+\|(\vZ)^m \partial _{x_3} \tilde{\alpha}_\pm F^{\textup{in}}_\pm\|_{L^\infty_{t,x,v}}
      +\|(\vZ)^m \nabla_v F^{\textup{in}}_\pm\|_{L^\infty_{t,x,v}}&<\infty,\\*
      \|(\vZ)^m \nabla_{x_\parallel} G_\pm\|_{L^\infty_{x_\parallel,v}}
      +\|(\vZ)^m \nabla_v G_\pm\|_{L^\infty_{x_\parallel,v}}&<\infty.
    \end{align*}
    Consider the corresponding solution sequence $(\fl,\Elbf,\Blbf)_{l\in\mathbb{N}}$ associated to the initial-boundary data $F^{\textup{in}}_\pm$ and $G_\pm$. 
Suppose further that \begin{equation}\sup_{0\le t\le T} \|\nabla_x (\Elbf(t),\Blbf(t))\|_{L^\infty}<C_1\text{ and }\label{Fl C2 bound.s}
    \sup_{0\le t\le T} \|\FS\|_{L^\infty}<C_2,
\end{equation} for some $T>0$, $C_1>0$ and $C_2>0.$ 
Suppose that $-(\FS)_3(t,x_\parallel,0,v)>c_0,$ for $t\in[0,T]$, $x_\parallel\in\mathbb{R}^2$ and $v\in\rth.$ 
    Then we have 
    \begin{multline}
        \label{deri F estimate final}\sup_{0\le t\le T} (\| (v^0_\pm)^m\partial_t\flo(t)\|_{L^\infty_{x,v}}+\|(\vZ)^m \nabla_{x_\parallel} F^{l+1}_\pm\|_{L^\infty_{t,x,v}}+\|(\vZ)^m \partial _{x_3} \tilde{\alpha}_\pm F^{l+1}_\pm\|_{L^\infty_{t,x,v}}
      +\|(\vZ)^m \nabla_v F^{l+1}_\pm\|_{L^\infty_{t,x,v}}) \\\le C_T,
    \end{multline} for some constant $C_T>0$ which depends only on $C_1, C_2, T,$ $F^{\textup{in}}_\pm$ and $G_\pm$.
\end{proposition}
\begin{remark}
    The derivative estimate \eqref{deri F estimate final} is uniform in $l$, and hence the limit $F_\pm^\infty$ also satisfies the same estimate.
\end{remark}

\begin{proof}
  Fix $m>4$. By \eqref{Derivatives of f}, we observe that $ (\vZ) ^m|\nabla_{x_\parallel} \flo|$ is bounded from above by  
    \begin{equation*}
    \begin{split}          (\vZ)^m|\nabla_{x_\parallel} \flo(t, x, v)| &\le  (\vZ)^m\bigg|
       (\nabla_x F^{\textup{in}}_\pm)(\XSlo(0; t, x, v), \VSlo(0; t, x, v)) \cdot \nabla_{x_\parallel} \XSlo(0; t, x, v) \\&+(\nabla_v F^{\textup{in}}_\pm)(\XSlo(0; t, x, v), \VSlo(0; t, x, v))\cdot \nabla_{x_\parallel} \VSlo(0; t, x, v)\bigg|1_{t \leq \tblo(t, x, v)}\\
      & + (\vZ)^m\bigg|  (\nabla_{x_\parallel} G_\pm)( (\xblo)_\parallel, \vblo) \cdot \nabla_{x_\parallel} (\xblo)_\parallel\\&+ (\nabla_v G_\pm)( (\xblo)_\parallel, \vblo) \cdot \nabla_{x_\parallel} \vblo\bigg|1_{t > \tblo(t, x, v)}\\
&\lesssim  1_{t \leq \tblo} (\vZ)^m|(\nabla_{x_\parallel}F^{\textup{in}}_\pm)(\XSlo(0; t, x, v), \VSlo(0; t, x, v))| |\nabla_{x_\parallel} (\XSlo)_\parallel(0; t, x, v)| \\&+ 1_{t \leq \tblo} (\vZ)^m|(\partial_{x_3}F^{\textup{in}}_\pm)(\XSlo(0; t, x, v), \VSlo(0; t, x, v))| |\nabla_{x_\parallel} (\XSlo)_3(0; t, x, v)|\\&+1_{t \leq \tblo} (\vZ)^m|(\nabla_v F^{\textup{in}}_\pm)(\XSlo(0; t, x, v), \VSlo(0; t, x, v))| |\nabla_{x_\parallel} \VSlo(0; t, x, v)|\\
       &+1_{t > \tblo} (\vZ)^m |(\nabla_{x_\parallel} G_\pm)( (\xblo)_\parallel, \vblo)||\nabla_{x_\parallel} (\xblo)_\parallel|\\&+1_{t > \tblo} (\vZ)^m |(\nabla_v G_\pm)( (\xblo)_\parallel, \vblo)||\nabla_{x_\parallel} \vblo|.
       \end{split}
    \end{equation*} 
In general, notice that \begin{multline*}
        \sup_{t-\tblo<s<t}\langle \VSlo(s)\rangle \lesssim \sup_{t-\tblo<s<t}\left( 1+|\VSlo(0)|+\bigg|\int_0^sd\tau\  \FS(\tau,\XSlo(\tau),\VSlo(\tau))\bigg|\right) \\\lesssim\langle \VSlo(0)\rangle +C_2\max\{T,|t-\tblo|\},
    \end{multline*} by \eqref{Fl C2 bound.s}.
    In addition, note that for $0\le s\le T,$
\begin{equation}\label{additional v bound.s}
        (\vZ) \lesssim  \langle\VSlo(s)\rangle+\bigg|\int_s^td\tau\  \FS(\tau,\XSlo(\tau),\VSlo(\tau))\bigg|\lesssim\langle \VSlo(s)\rangle +C_2T\lesssim C_T\langle \VSlo(s)\rangle  ,
    \end{equation} by \eqref{Fl C2 bound.s}. 
    Using \eqref{additional v bound.s}, 
     \eqref{derivatives of char}, and \eqref{dxixj} with $s=0$ and \eqref{t bound by v and alpha.s} 
     for $t\le \tblo$ terms and using \eqref{derivatives of tbxbvb}--\eqref{derivatives of tbxbvb better} for $t>\tblo$ terms, we obtain
    \begin{align*}
 (\vZ)^m|\nabla_{x_\parallel} \flo(t, x, v)|
&\lesssim  C_T\bigg(\| (\vZ)^m\nabla_{x_\parallel}F^{\textup{in}}_\pm\|_{L^\infty_{x,v}} +  \| (\vZ)^m\tilde{\alpha}_{\pm}\partial_{x_3}F^{\textup{in}}_\pm\|_{L^\infty_{x,v}}  +\| (\vZ)^m\nabla_v F^{\textup{in}}_\pm\|_{L^\infty_{x,v}}\bigg)\\*
       &+C_T\bigg(1_{t>\tblo} (\vZ)^m |(\nabla_{x_\parallel} G_\pm)((\xblo)_\parallel, \vblo)|\bigg|\frac{\tblo}{|\hat{\mathcal{V}}_\pm^{l+1}(t-\tblo)|(\vZ) }+1\bigg|\\*
      & +1_{t>\tblo} (\vZ)^m |(\nabla_v G_\pm)( (\xblo)_\parallel, \vblo)|\bigg|\frac{\tblo}{|\hat{\mathcal{V}}_\pm^{l+1}(t-\tblo)|(\vZ) }+1\bigg|\bigg),
    \end{align*}where also used $\tblo\le T$ for the terms with $1_{t>\tblo}.$ 
    For the terms with $1_{t>\tblo},$ by Lemma \ref{lemma.tb over v3}, we further observe that if $t>\tblo$,
    \begin{multline}\label{tb over v3 bound.s}
        \bigg|\frac{\tblo}{|(\hat{\mathcal{V}}_\pm^{l+1})_3(t-\tblo)|(\vZ) }\bigg|\le \frac{C}{c_0}\frac{\max_{s\in\{t-\tblo,t\}}\sqrt{m_\pm^2+|\VSlo(s)|^2}}{(\vZ)}\\\lesssim \frac{C}{c_0}\sup_{t-\tblo<s<t}\left(1+ \frac{1}{(\vZ)}\bigg|\int_s^t \FS(\tau,\XSlo(\tau),\VSlo(\tau))d\tau\bigg|\right)\lesssim \frac{C}{c_0}\left(1+\frac{C_2\tblo}{(\vZ)}\right)\lesssim C_T.
    \end{multline}
    Also for $1_{t>\tblo} $ terms, we use \eqref{additional v bound.s} with $s=t-\tblo$ and that $\tblo<t\le T$ to conclude that
    \begin{multline*}
\| (\vZ)^m\nabla_{x_\parallel} \flo(t, \cdot,\cdot)\|_{L^\infty_{x,v}}
\lesssim  C_T\bigg(\| (\vZ)^m\nabla_{x_\parallel}F^{\textup{in}}_\pm\|_{L^\infty_{x,v}} +  \| (\vZ)^m\tilde{\alpha}_{\pm}\partial_{x_3}F^{\textup{in}}_\pm\|_{L^\infty_{x,v}} +\| (\vZ)^m\nabla_v F^{\textup{in}}_\pm\|_{L^\infty_{x,v}}\bigg)\\*
       +C_T \bigg(\| (\vZ)^m \nabla_{x_\parallel} G_\pm\|_{L^\infty_{x_\parallel,v}}
      +\| (\vZ)^m \nabla_v G_\pm\|_{L^\infty_{x_\parallel,v}}\bigg).
    \end{multline*}

Regarding the derivative $\partial_{x_3}\flo,$ we observe that \eqref{Derivatives of f} implies
\begin{align*}
      &  (\vZ)^m\tilde{\alpha}_{\pm}(t,x,v)  |\partial_{x_3}\flo(t, x, v)|\\ &\le (\vZ)^m\tilde{\alpha}_{\pm}(t,x,v) \bigg|
       (\nabla_x F^{\textup{in}}_\pm)(\XSlo(0; t, x, v), \VSlo(0; t, x, v)) \cdot \partial_{x_3} \XSlo(0; t, x, v) \\
       &+(\nabla_v F^{\textup{in}}_\pm)(\XSlo(0; t, x, v), \VSlo(0; t, x, v))\cdot \partial_{x_3} \VSlo(0; t, x, v)\bigg|1_{t \leq \tblo(t, x, v)}\\
       &+ (\vZ)^m\tilde{\alpha}_{\pm}(t,x,v)\bigg| (\nabla_{x_\parallel} G_\pm)( (\xblo)_\parallel, \vblo) \cdot \partial_{x_3} (\xblo)_\parallel+ (\nabla_v G_\pm)( (\xblo)_\parallel, \vblo) \cdot \partial_{x_3} \vblo\bigg|1_{t > \tblo(t, x, v)}\\
&\lesssim  1_{t \leq \tblo} (\vZ)^m\tilde{\alpha}_{\pm}(t,x,v)|(\nabla_{x_\parallel}F^{\textup{in}}_\pm)(\XSlo(0; t, x, v), \VSlo(0; t, x, v))| |\partial_{x_3} (\XSlo)_\parallel(0; t, x, v)| \\
&+ 1_{t \leq \tblo} (\vZ)^m\tilde{\alpha}_{\pm}(t,x,v)|(\partial_{x_3}F^{\textup{in}}_\pm)(\XSlo(0; t, x, v), \VSlo(0; t, x, v))| |\partial_{x_3}(\XSlo)_3(0; t, x, v)|\\
&+1_{t \leq \tblo} (\vZ)^m\tilde{\alpha}_{\pm}(t,x,v)|(\nabla_v F^{\textup{in}}_\pm)(\XSlo(0; t, x, v), \VSlo(0; t, x, v))| |\partial_{x_3} \VSlo(0; t, x, v)|\\
       &+1_{t > \tblo} (\vZ)^m\tilde{\alpha}_{\pm}(t,x,v)|(\nabla_{x_\parallel} G_\pm)( (\xblo)_\parallel, \vblo)||\partial_{x_3} (\xblo)_\parallel|\\
       &+1_{t > \tblo}  (\vZ)^m\tilde{\alpha}_{\pm}(t,x,v)|(\nabla_v G_\pm)( (\xblo)_\parallel, \vblo)||\partial_{x_3}\vblo|.
    \end{align*}
For the terms with $1_{t\le \tblo}$, we use \eqref{dxixj} with $i=3$, $j=1,2$ and $s=0$ for $\partial_{x_3} (\XSlo)_\parallel(0)$ term and use \eqref{derivatives of char} with $i=3$ and $s=0$ for $\partial_{x_3} (\XSlo)_3(0)$ and $\partial_{x_3} \VSlo(0)$ terms to further obtain that
\begin{multline*}
     (\vZ)^m\tilde{\alpha}_{\pm}(t,x,v) 1_{t \leq \tblo}|(\nabla_{x_\parallel}F^{\textup{in}}_\pm)(\XSlo(0; t, x, v), \VSlo(0; t, x, v))| |\partial_{x_3} (\XSlo)_\parallel(0; t, x, v)|\\
     \lesssim C_T (\vZ)^{m-1}\tilde{\alpha}_{\pm}(t,x,v)1_{t\le \tblo}|(\nabla_{x_\parallel}F^{\textup{in}}_\pm)(\XSlo(0; t, x, v), \VSlo(0; t, x, v))|,
\end{multline*}
\begin{multline*}
    (\vZ)^m\tilde{\alpha}_{\pm}(t,x,v) 1_{t \leq \tblo}|(\partial_{x_3}F^{\textup{in}}_\pm)(\XSlo(0; t, x, v), \VSlo(0; t, x, v))| |\partial_{x_3}(\XSlo)_3(0; t, x, v)|\\*
     \lesssim C_T (\vZ)^m\tilde{\alpha}_{\pm}(t,x,v)1_{t\le \tblo}|(\partial_{x_3}F^{\textup{in}}_\pm)(\XSlo(0; t, x, v), \VSlo(0; t, x, v))|,
\end{multline*}
\begin{multline*}
     (\vZ)^m\tilde{\alpha}_{\pm}(t,x,v) 1_{t \leq \tblo}|(\nabla_v F^{\textup{in}}_\pm)(\XSlo(0; t, x, v), \VSlo(0; t, x, v))| |\partial_{x_3} \VSlo(0; t, x, v)|\\
     \lesssim C_T (\vZ)^m\tilde{\alpha}_{\pm}(t,x,v)1_{t\le \tblo}|(\nabla_v F^{\textup{in}}_\pm)(\XSlo(0; t, x, v), \VSlo(0; t, x, v))|,
\end{multline*}since $t\le T.$
Therefore, by the fact that $\tilde{\alpha}_{\pm}\le 1$ and that $\tilde{\alpha}_{\pm}$ also satisfies the additional bound \eqref{alpha.s.lem} with $s=0,$ we have
\begin{multline*}
     (\vZ)^m\tilde{\alpha}_{\pm}(t,x,v) 1_{t \leq \tblo}|(\nabla_{x_\parallel}F^{\textup{in}}_\pm)(\XSlo(0; t, x, v), \VSlo(0; t, x, v))| |\partial_{x_3} (\XSlo)_\parallel(0; t, x, v)|\\
     \lesssim C_T \langle \VSlo(0)\rangle^{m-1} 1_{t\le \tblo}|(\nabla_{x_\parallel}F^{\textup{in}}_\pm)(\XSlo(0; t, x, v), \VSlo(0; t, x, v))|\lesssim C_T \| (\vZ)^{m-1}\nabla_{x_\parallel}F^{\textup{in}}_\pm\|_{L^\infty_{x,v}},
\end{multline*}
\begin{multline*}
    (\vZ)^m\tilde{\alpha}_{\pm}(t,x,v) 1_{t \leq \tblo}|(\partial_{x_3}F^{\textup{in}}_\pm)(\XSlo(0; t, x, v), \VSlo(0; t, x, v))| |\partial_{x_3}(\XSlo)_3(0; t, x, v)|\\
     \lesssim C_T \langle \VSlo(0)\rangle^m \tilde{\alpha}_{\pm}(0,\XSlo(0),\VSlo(0))1_{t\le \tblo}|(\partial_{x_3}F^{\textup{in}}_\pm)(\XSlo(0; t, x, v), \VSlo(0; t, x, v))|\\
     \lesssim C_T \|(\vZ)^m\tilde{\alpha}_{\pm} \partial_{x_3}F^{\textup{in}}_\pm\|_{L^\infty_{x,v}} ,
\end{multline*}
\begin{multline*}
  (\vZ)^m\tilde{\alpha}_{\pm}(t,x,v) 1_{t \leq \tblo}|(\nabla_vF^{\textup{in}}_\pm)(\XSlo(0; t, x, v), \VSlo(0; t, x, v))| |\partial_{x_3}(\XSlo)_3(0; t, x, v)|\\
     \lesssim C_T \langle \VSlo(0)\rangle^m  1_{t\le \tblo}|(\nabla_vF^{\textup{in}}_\pm)(\XSlo(0; t, x, v), \VSlo(0; t, x, v))|
     \lesssim C_T \|(\vZ)^m   \nabla_v F^{\textup{in}}_\pm\|_{L^\infty_{x,v}}.
\end{multline*}
    On the other hand, if $t>\tblo,$ by \eqref{derivatives of tbxbvb},
\begin{align*}
 & (\vZ)^m\tilde{\alpha}_{\pm}(t,x,v) 1_{t > \tblo}|(\nabla_{x_\parallel} G_\pm)( (\xblo)_\parallel, \vblo)||\partial_{x_3} (\xblo)_\parallel|\\
       &+(\vZ)^m\tilde{\alpha}_{\pm}(t,x,v) 1_{t > \tblo} |(\nabla_v G_\pm)( (\xblo)_\parallel, \vblo)||\partial_{x_3}\vblo|\\
       &\lesssim C_T\bigg(1_{t>\tblo}(\vZ)^m |(\nabla_{x_\parallel} G_\pm)((\xblo)_\parallel, \vblo)|\tilde{\alpha}_{\pm}(t,x,v)\left|\frac{1}{|(\hat{\mathcal{V}}_\pm^{l+1})_3(t-\tblo)|} +\frac{1}{\langle v \rangle}\right|\\
      & +1_{t>\tblo}(\vZ)^m |(\nabla_v G_\pm)( (\xblo)_\parallel, \vblo)|\tilde{\alpha}_{\pm}(t,x,v)\left|\frac{1}{|(\hat{\mathcal{V}}_\pm^{l+1})_3(t-\tblo)|} +1\right|\bigg),
    \end{align*}where also used $\tblo\le T$ for the terms with $1_{t>\tblo}.$      By using \eqref{alphav3.s}, \eqref{alpha.s.lem} with $\tblo\le T$, and \eqref{additional v bound.s} with $s=t-\tblo$  for the terms with $1_{t>\tblo},$ we conclude that
    \begin{align*}
&\|(\vZ)^m\tilde{\alpha}_{\pm} \partial_{x_3}  \flo(t, \cdot,\cdot)\|_{L^\infty_{x,v}}\\
&\lesssim  C_T\bigg(\| (\vZ)^{m-1}\nabla_{x_\parallel}F^{\textup{in}}_\pm\|_{L^\infty_{x,v}} +  \|(\vZ)^m\tilde{\alpha}_{\pm}\partial_{x_3}F^{\textup{in}}_\pm\|_{L^\infty_{x,v}} +\| (\vZ)^m\nabla_v F^{\textup{in}}_\pm\|_{L^\infty_{x,v}}\bigg)\\
       &+C_T\bigg( \| (\vZ)^m \nabla_{x_\parallel} G_\pm\|_{L^\infty_{x_\parallel,v}}
      +\| (\vZ)^m\nabla_v G_\pm\|_{L^\infty_{x_\parallel,v}}\bigg).   \end{align*}

Finally, we consider a weighted upper-bound estimate for the momentum derivative $|\nabla_v\flo|.$ By \eqref{Derivatives of f}, we observe that $ (\vZ)^m|\nabla_v\flo|$ is bounded from above by
    \begin{align*}
         & (\vZ)^m|\nabla_v \flo(t, x, v)| \le  (\vZ)^m\bigg|
       (\nabla_x F^{\textup{in}}_\pm)(\XSlo(0; t, x, v), \VSlo(0; t, x, v)) \cdot \nabla_v  \XSlo(0; t, x, v) \\*
       &+(\nabla_v F^{\textup{in}}_\pm)(\XSlo(0; t, x, v), \VSlo(0; t, x, v))\cdot \nabla_v  \VSlo(0; t, x, v)\bigg|1_{t \leq \tblo(t, x, v)}\\*
       &+ (\vZ)^m\bigg|(\nabla_{x_\parallel} G_\pm)( (\xblo)_\parallel, \vblo) \cdot \nabla_v  (\xblo)_\parallel+ (\nabla_v G_\pm)( (\xblo)_\parallel, \vblo) \cdot \nabla_v  \vblo\bigg|1_{t > \tblo(t, x, v)}\\
&\lesssim  1_{t \leq \tblo} (\vZ)^m|(\nabla_{x_\parallel}F^{\textup{in}}_\pm)(\XSlo(0; t, x, v), \VSlo(0; t, x, v))| |\nabla_v  (\XSlo)_\parallel(0; t, x, v)| \\*
&+ 1_{t \leq \tblo} (\vZ)^m|(\partial_{x_3}F^{\textup{in}}_\pm)(\XSlo(0; t, x, v), \VSlo(0; t, x, v))| |\nabla_v  (\XSlo)_3(0; t, x, v)|\\*
&+1_{t \leq \tblo} (\vZ)^m|(\nabla_v F^{\textup{in}}_\pm)(\XSlo(0; t, x, v), \VSlo(0; t, x, v))| |\nabla_v  \VSlo(0; t, x, v)|\\*
       &+1_{t > \tblo} (\vZ)^m |(\nabla_{x_\parallel} G_\pm)( (\xblo)_\parallel, \vblo)||\nabla_v  (\xblo)_\parallel|+1_{t > \tblo} (\vZ)^m |(\nabla_v G_\pm)( (\xblo)_\parallel, \vblo)||\nabla_v  \vblo|.
    \end{align*} 
Using \eqref{additional v bound.s} and \eqref{derivatives of char} with $s=0$ and \eqref{t bound by v and alpha.s} for $t\le \tblo$ terms and using \eqref{derivatives of tbxbvb} for $t>\tblo$ terms, we obtain
    \begin{align*}
 &(\vZ)^m|\nabla_v  \flo(t, x, v)|
\lesssim  C_T\bigg(\| (\vZ)^{m-1}\nabla_{x_\parallel}F^{\textup{in}}_\pm\|_{L^\infty_{x,v}} +  \| (\vZ)^m\tilde{\alpha}_{\pm}\partial_{x_3}F^{\textup{in}}_\pm\|_{L^\infty_{x,v}}  +\| (\vZ)^m\nabla_v F^{\textup{in}}_\pm\|_{L^\infty_{x,v}}\bigg)\\
       &+C_T(\vZ)^{m-1}\bigg(1_{t>\tblo}  |(\nabla_{x_\parallel} G_\pm)((\xblo)_\parallel, \vblo)|\frac{\tblo}{|\hat{\mathcal{V}}_\pm^{l+1}(t-\tblo)| }\\
       &+1_{t>\tblo}  |(\nabla_v G_\pm)( (\xblo)_\parallel, \vblo)|\bigg|\frac{\tblo}{|\hat{\mathcal{V}}_\pm^{l+1}(t-\tblo)|(\vZ) }+1\bigg|\bigg),
    \end{align*}where also used $\tblo\le T$ for the terms with $1_{t>\tblo}.$ 
    By using \eqref{tb over v3 bound.s} and \eqref{additional v bound.s} with $s=t-\tblo$ for the terms with $1_{t>\tblo},$ we conclude that
    \begin{align*}
&\| (\vZ)^m\nabla_v  \flo(t, \cdot,\cdot)\|_{L^\infty_{x,v}}
\lesssim  C_T\bigg(\| (\vZ)^{m-1}\nabla_{x_\parallel}F^{\textup{in}}_\pm\|_{L^\infty_{x,v}} +  \| (\vZ)^m\tilde{\alpha}_{\pm}\partial_{x_3}F^{\textup{in}}_\pm\|_{L^\infty_{x,v}} +\| (\vZ)^m\nabla_v F^{\textup{in}}_\pm\|_{L^\infty_{x,v}}\bigg)\\*
       &+C_T\bigg(  \| (\vZ)^m \nabla_{x_\parallel} G_\pm\|_{L^\infty_{x_\parallel,v}}
      +\| (\vZ)^{m-1}\nabla_v G_\pm\|_{L^\infty_{x_\parallel,v}}\bigg).
    \end{align*}

Lastly, we consider the temporal derivative $\partial_t \flo.$ Using the Vlasov equation \eqref{2speciesVM}$_1$, we have
$$|(\vZ)^m \partial_t \flo|\le |(\vZ)^m \nabla_{x_\parallel} \flo|+|(\vZ)^m (\hat{v}_\pm)_3\partial_{x_3} \flo|+|(\vZ)^m \nabla_{v} \flo||\FS|. $$ Note that $$(\hat{v}_\pm)_3=\alpha_\pm(t,x_\parallel,0,v)\le 2\sqrt{\frac{\alpha_\pm^2(t,x_\parallel,0,v)}{1+\alpha_\pm^2(t,x_\parallel,0,v)}}\lesssim_T \tilde{\alpha}_{\pm}(s,\XSlo(s),\VSlo(s)), $$ for any $s\in(0,T).$ Therefore, we conclude
\begin{multline*}
    \|(\vZ)^m \partial_t \flo(t,\cdot,\cdot)\|_{L^\infty_{x,v}}\\
    \lesssim_T \| (\vZ)^m\nabla_{x_\parallel}  \flo(t, \cdot,\cdot)\|_{L^\infty_{x,v}}+\| (\vZ)^m\tilde{\alpha}_{\pm}\partial_{x_3}  \flo(t, \cdot,\cdot)\|_{L^\infty_{x,v}}+\| (\vZ)^m\nabla_v  \flo(t, \cdot,\cdot)\|_{L^\infty_{x,v}}.
\end{multline*}

This completes the proof of Proposition \ref{prop.derivative}.
 \end{proof}

\section{Regularity Estimates for the Electromagnetic Fields}\label{sec.deri.EB}
In this section, we provide consider derivative estimates of the self-consistent electromagnetic fields $(\Elobf,\Blobf)$ whose representations are given via \eqref{Eparallel homo solution}, \eqref{Ei5}, \eqref{E3}, \eqref{B_half_final}, \eqref{Bpar_half_final}, \Black{and \eqref{BparS_half_final}}. In the following three subsections, we consider the derivatives of $(\Elobf,\Blobf)$ in tangential, normal, and temporal directions and eventually prove the following proposition:   \begin{proposition}  \label{prop.main.deri.field}Suppose that $\Elobf$ and $\Blobf$ are defined through \eqref{Ei5}, \eqref{E3},  \eqref{B_half_final}, \eqref{Bpar_half_final}, \Black{and \eqref{BparS_half_final}}. Suppose that $-(\FS)_3(t,x_\parallel,0,v)>c_0,$ for some $c_0>0.$  Let $g \geq 1$ and $\beta > 1$ be chosen sufficiently large so that
$$
\min\{m_+^2,m_-^2\}g^2\beta\gg 1 \quad \textup{and} \quad \min\{m_+^2,m_-^2\}\beta^4\gg 1.
$$Also, suppose that the temporal derivatives of the initial profiles, understood through the system of equations, satisfy the assumptions \eqref{initial E0i}--\eqref{f initial condition}.  Then for any given $T>0$ and some $m>4$, we have\begin{multline}\label{EloBlo estimate final}
       \|(\Elobf,\Blobf)\|_{W^{1,\infty}_{t,x}([0,T]\times \Omega)}
       \lesssim_T (1+ \|(\E^{\textup{in}},\B^{\textup{in}})\|_{C^2_x(\Omega)})(1+\|(\vZ)^m \flo\|^2_{L^\infty_{t,x,v}([0,T]\times \Omega\times \rth)})\\*+\|(\vZ)^m\nabla_{x_\parallel}\flo\|_{L^\infty_{t,x,v}([0,T]\times \Omega\times \rth)}
       +\|(\vZ)^m\tilde{\alpha}_{\pm}(t,x,v)\partial_{x_3}\flo\|_{L^\infty_{t,x,v}([0,T]\times \Omega\times \rth)}.
   \end{multline} 
   \end{proposition}
\begin{proof}
The proposition follows directly from Lemmas~\ref{lem.normal deri},\  ~\ref{lem.temporal deri}, and ~\ref{lem.tangential deri}, which are established in the subsequent sections.
\end{proof}

\subsection{Normal Derivatives}We first introduce the estimate for normal derivatives.  
We emphasize that these derivatives are controlled by tangential and temporal derivatives in conjunction with the governing Maxwell equations.  
This represents a fundamentally different methodology from the traditional approach (cf. \cite{MR4645724}).

\begin{lemma}\label{lem.normal deri}
  Suppose that $\Elobf$ and $\Blobf$ are defined through \eqref{Ei5}, \eqref{E3},  \eqref{B_half_final}, \eqref{Bpar_half_final}, \Black{and \eqref{BparS_half_final}}. Suppose that $-(\FS)_3(t,x_\parallel,0,v)>c_0,$ for some $c_0>0.$ Then for any given $T>0$ and some $m>4$, we have \begin{multline*}
    \hspace{-3mm}\sup_{t\in[0,T]}\|\partial_{x_3}(\Elobf,\Blobf)\|_{L^\infty_{x,v}}
        \lesssim_T \sup_{t\in[0,T]}\left(\|\nabla_{x_\parallel}(\Elobf, \Blobf)\|_{L^\infty_{x,v}}+\|\partial_t(\Elobf_\parallel, \Blobf_\parallel)\|_{L^\infty_{x,v}}  + \|(\vZ)^m \flo\|_{L^\infty_{x,v} }\right).
  \end{multline*} 
\end{lemma}\begin{proof}
  Using \eqref{2speciesVM}$_2$-\eqref{2speciesVM}$_5$, we obtain 
    \begin{equation}\notag
\begin{split}
\partial_{x_3}\Eth^{l+1}&=-\nabla_{x_\parallel}\cdot \Elobf_{\parallel}+4\pi\rho^{l+1},\ 
\partial_{x_3}\Bth^{l+1}=-\nabla_{x_\parallel}\cdot \Blobf_{\parallel},\\
\partial_{x_3}\Eone^{l+1}&=\partial_{x_1}\Eth^{l+1}- \partial_t \Btwo^{l+1},\ 
\partial_{x_3}\Etwo^{l+1}=\partial_{x_2}\Eth^{l+1}+ \partial_t \Bone^{l+1}, \\
\partial_{x_3}\Bone^{l+1}&=\partial_{x_1}\Bth^{l+1}+ \partial_t \Etwo^{l+1}+4\pi J_2^{l+1},\ 
\partial_{x_3}\Btwo^{l+1}=\partial_{x_2}\Bth^{l+1}- \partial_t \Eone^{l+1}-4\pi J_1^{l+1}.
\end{split}
    \end{equation}
    Therefore, we obtain \begin{multline*}
        \|\partial_{x_3}(\Elobf,\Blobf)\|_{L^\infty}\lesssim \|\nabla_{x_\parallel}(\Elobf, \Blobf)\|_{L^\infty}+\|\partial_t(\Elobf_\parallel, \Blobf_\parallel)\|_{L^\infty} +\|\rho^{l+1}\|_{L^\infty } + \|J^{l+1}\|_{L^\infty } \\
        \lesssim \|\nabla_{x_\parallel}(\Elobf, \Blobf)\|_{L^\infty}+\|\partial_t(\Elobf_\parallel, \Blobf_\parallel)\|_{L^\infty}  + \|(\vZ)^m \flo\|_{L^\infty_{x,v} },
    \end{multline*} for $m>4$, since \begin{equation}
    \label{j bound}|J^{l+1}(t,x)|\le \rho^{l+1}(t,x)= \int_\rth \flo(t,x,v) dv \le \|(\vZ)^m \flo\|_{L^\infty_{x,v} } \int_\rth(\vZ)^{-m}dv\lesssim \|(\vZ)^m \flo\|_{L^\infty_{x,v} }. 
\end{equation} Thus, we obtain the lemma by using the uniform estimates \eqref{EB tangential deri claim} and \eqref{final estimate for EloBlo tempo} on $\|\nabla_{x_\parallel}(\Elobf, \Blobf)\|_{L^\infty}$ and $\|\partial_t(\Elobf_\parallel, \Blobf_\parallel)\|_{L^\infty}$, respectively. 
\end{proof}

\subsection{Temporal Derivatives}
This section is devoted to deriving uniform estimates on the temporal derivatives of the fields $(\Elobf, \Blobf)$. Special care is required due to the presence of temporal-physical boundaries at $t=0$ and $x_3=0$. We study the system satisfied by $\partial_t \Elobf$, $\partial_t \Blobf$, and $\partial_t \flo$ by formally differentiating in time the Vlasov--Maxwell system \eqref{iterated Vlasov.M}--\eqref{iterated Maxwell.M}, the continuity equation \eqref{continuity eq}, and the boundary conditions \eqref{perfect cond. boundary} and \eqref{perfect.conductor.neumann} at the sequential level $(l+1)$. 
This represents a fundamentally different methodology from the traditional approach (cf. \cite{MR4645724}).

Formally applying $\partial_t$ yields the following system for $\partial_t \Elbf$, $\partial_t \Blbf$, and $\partial_t \fl$, which must be understood in the sense of distributions:
\begin{equation}
    \label{eq.tempo flo}
    (\partial_t +\hat{v}_\pm \cdot \nabla_x + (\pm\Elbf\pm\hat{v}_\pm\times \Blbf- m_\pm g \hat{e}_3)\cdot\nabla_v)(\partial_t \flo) = \mp(\partial_t \Elbf+\hat{v}_\pm\times \partial_t \Blbf)\cdot\nabla_v\flo,
\end{equation}
\begin{equation}\label{eq.tempo Maxwell}
    \begin{split}
         \partial_t (\partial_t \Elbf) - \nabla_x\times (\partial_t \Blbf) &= - 4\pi \partial_t J^{l}, \
         \partial_t (\partial_t \Blbf) + \nabla_x\times (\partial_t \Elbf) = 0, \\
         \nabla_x \cdot (\partial_t \Elbf) &= 4\pi \partial_t \rho^{l}, \
         \nabla_x \cdot (\partial_t \Blbf) = 0,
    \end{split}
\end{equation}
and the differentiated continuity equation:
\begin{equation}
    \label{continuity eq tempo}
    \partial_t (\partial_t \rho^{l+1}) + \nabla_x \cdot (\partial_t J^{l+1}) = 0.
\end{equation}

In addition, formally differentiating the boundary conditions yields, again in the sense of distributions:
\begin{align}
(\partial_t\Eone^{l+1})\big|_{\partial\Omega} = 0 = (\partial_t\Etwo^{l+1})\big|_{\partial\Omega}, \quad (\partial_t\Bth^{l+1})\big|_{\partial\Omega} = 0, \label{perfect cond. boundary tempo}
\end{align}
and the Neumann-type boundary conditions:
\begin{equation}
\label{perfect.conductor.neumann.tempo}
\partial_{x_3}(\partial_t\Eth^{l+1}) = 4\pi (\partial_t\rho^{l+1}),\quad 
\partial_{x_3}(\partial_t\Etwo^{l+1}) = -4\pi (\partial_t J^{l+1}_1),\quad 
\partial_{x_3}(\partial_t\Bone^{l+1}) = 4\pi (\partial_t J^{l+1}_2).
\end{equation}
These boundary conditions are to be interpreted in the weak sense, meaning they hold through integration against test functions rather than pointwise evaluation. Accordingly, we ensure that $\partial_t \rho^{l+1}$ and $\partial_t J^{l+1}$ are controlled in $L^\infty$ where these boundary relations make sense.

Finally, we prescribe the initial data for the temporal derivatives of the fields for each $i=1,2,3$:
\begin{equation}\label{initial tempo eb}
    (\partial_t \Ei)(0,x) = \mathbf{E}^1_{0i}(x), \
    (\partial_t^2 \Ei)(0,x) = \mathbf{E}^2_{0i}(x), \
    (\partial_t \Bi)(0,x) = \mathbf{B}^1_{0i}(x), \
    (\partial_t^2 \Bi)(0,x) = \mathbf{B}^2_{0i}(x).
\end{equation}
The initial temporal derivatives $\partial_t \Elbf(0,x)$ and $\partial_t \Blbf(0,x)$ are determined from the initial data\newline $(\Elbf(0,x), \Blbf(0,x), \rho(0,x), J(0,x))$ via the Maxwell equations evaluated at $t=0$. Similarly, $\partial_t^2 \Elbf(0,x)$ and $\partial_t^2 \Blbf(0,x)$ are obtained by differentiating the system once in time. All initial values are understood in the distributional sense.

Given the decay estimates for the momentum derivatives \eqref{decay of momentum deri F} for $\nabla_v \flo$, we provide decay estimates for the temporal derivative $\partial_t \flo$ of the distribution $\flo$ and uniform estimates on $\partial_t\Elobf$ and $\partial_t \Blobf$ via a bootstrap argument. For a bootstrap argument, we make the following bootstrap assumptions on $\partial_t \fl$, $\partial_t\Elbf$ and $\partial_t \Blbf$. In the case when $\Omega=\mathbb{R}^2\times \mathbb{R}_+$, let $\partial_t\fl,\partial_t\Elbf$, and $\partial_t\Blbf$
satisfy
\begin{multline}\label{Ansatz for ptfl} 
        \sup_{t\ge 0}\left\|e^{\frac{\beta}{2} |x_{\parallel}|}e^{\frac{\beta}{4}(\vZ+m_\pm gx_3)}\partial_t\fl(t,\cdot,\cdot)\right\|_{L^\infty} 
        \le  \bigg(\|\mathrm{w}_{\pm,\beta} \partial_tF_\pm(0,\cdot,\cdot)\|_{L^\infty_{x,v}}\\
        +\frac{64CD_0}{5\beta e}\left(\|\mathrm{w}^2_{\pm,\beta}(x,v)\nabla_{x,v}F^{\textup{in}}_\pm(x,v)\|_{L^\infty_{x,v}}+\|\mathrm{w}^2_{\pm,\beta}(x_\parallel,0,v)\nabla_{x_\parallel,v}G_\pm(x_\parallel,v)\|_{L^\infty_{x_\parallel,v}}\right)\bigg),
    \end{multline}and   \begin{equation}\label{ansatz for ptEBl}
        \sup_{t\ge 0}\|(\partial_t\Elbf,\partial_t\Blbf)\|_{L^\infty}\le D_0 \min\{m_-,m_+\}g,
    \end{equation} for some uniform constant $D_0>0$ where $C>0$ is the same constant as that of \eqref{decay of momentum deri F} and the weight $\mathrm{w}_{\pm,\beta}$ is defined as \eqref{weights.wholehalf}. Note that this constant $D_0$ can be sufficiently large.  

    In the following subsections, we will prove that the bootstrap ansatz \eqref{Ansatz for ptfl} and \eqref{ansatz for ptEBl} hold also on the sequential level of $(l+1)$ given the momentum derivative estimate \eqref{decay of momentum deri F}.
    \subsubsection{Estimates for $\partial_t\flo$ for $\Omega=\mathbb{R}^2\times\mathbb{R}_+$}We first prove that \eqref{Ansatz for ptfl} holds for $\flo$. Since $\partial_t\flo$ satisfies \eqref{eq.tempo flo}, we can write $\partial_t\flo$ in the mild form as 
\begin{multline}\label{solution ptflo}
    \partial_t\flo(t,x,v)= 1_{t\le \tblo(t,x,v)}\partial_tF_\pm(0,\XSlo(0;t,x,v),\VSlo(0;t,x,v))\\\mp \int^t_{\max\{0,t-\tblo\}} \left(\partial_t \Elbf(s,\XSlo(s))+\hat{\mathcal{V}}^{l+1}_\pm(s)\times \partial_t \Blbf(s,\XSlo(s))\right)\cdot \nabla_v \flo(s,\XSlo(s),\VSlo(s))ds.
\end{multline}
 Given that $\|\mathrm{w}_{\pm,\beta}\nabla_v \flo\|_{L^\infty}$ is bounded (see \eqref{decay of momentum deri F}), by \eqref{ansatz for ptEBl}, \eqref{solution ptflo} and that $|\hat{\mathcal{V}}^{l+1}_\pm|\le 1$, we obtain that
\begin{align*}\notag
    &|\partial_t\flo(t,x,v)|\\&\le  1_{t\le \tblo(t,x,v)}|\partial_tF_\pm(0,\XSlo(0;t,x,v),\VSlo(0;t,x,v))|\\
    &+1_{t\le \tblo(t,x,v)}\int^t_0\ D_0\min\{m_-,m_+\}g |\nabla_v \flo(s,\XSlo(s),\VSlo(s))|ds\\
    &+1_{t> \tblo(t,x,v)}\int^t_{t-\tblo} D_0\min\{m_-,m_+\}g |\nabla_v \flo(s,\XSlo(s),\VSlo(s))|ds\\
   &\le  \frac{1_{t\le \tblo(t,x,v)}}{\mathrm{w}_{\pm,\beta}( \ZSlo(0 ; t, x, v))}\|\mathrm{w}_{\pm,\beta} \partial_tF_\pm(0,\cdot,\cdot)\|_{L^\infty_{x,v}}\\
  & +1_{t\le \tblo(t,x,v)}D_0\min\{m_-,m_+\}g \|\mathrm{w}_{\pm,\beta}\nabla_v \flo\|_{L^\infty_{t,x,v}}\int^t_0 \frac{1}{\mathrm{w}_{\pm,\beta}( \ZSlo(s ; t, x, v))}ds\\
   &+1_{t> \tblo(t,x,v)}D_0\min\{m_-,m_+\}g \|\mathrm{w}_{\pm,\beta}\nabla_v \flo\|_{L^\infty_{t,x,v}}\int^t_{t-\tblo} \frac{1}{\mathrm{w}_{\pm,\beta}( \ZSlo(s ; t, x, v))}ds.
\end{align*}
Using \eqref{exit time bound.whole} and \eqref{w comparison 3.whole}, we further have 
\begin{align*}\notag
   & |\partial_t\flo(t,x,v)|\\
   &\le  1_{t\le \tblo(t,x,v)}e^{-\frac{1}{2}\beta v_\pm^0-\frac{1}{2}m_\pm g\beta x_3-\frac{\beta}{2}|x_{\parallel}|}\|\mathrm{w}_{\pm,\beta} \partial_tF_\pm(0,\cdot,\cdot)\|_{L^\infty_{x,v}}\\
  & +1_{t\le \tblo(t,x,v)}D_0\min\{m_-,m_+\}g \|\mathrm{w}_{\pm,\beta}\nabla_v \flo\|_{L^\infty_{t,x,v}}e^{-\frac{1}{2}\beta v_\pm^0-\frac{1}{2}m_\pm g\beta x_3-\frac{\beta}{2}|x_{\parallel}|} t\\
  & +1_{t> \tblo(t,x,v)}D_0\min\{m_-,m_+\}g \|\mathrm{w}_{\pm,\beta}\nabla_v \flo\|_{L^\infty_{t,x,v}}e^{-\frac{1}{2}\beta v_\pm^0-\frac{1}{2}m_\pm g\beta x_3-\frac{\beta}{2}|x_{\parallel}|} \tblo\\
   & \le  e^{-\frac{1}{2}\beta v_\pm^0-\frac{1}{2}m_\pm g\beta x_3-\frac{\beta}{2}|x_{\parallel}|}\|\mathrm{w}_{\pm,\beta} \partial_tF_\pm(0,\cdot,\cdot)\|_{L^\infty_{x,v}}\\*
   &+D_0\min\{m_-,m_+\}g \|\mathrm{w}_{\pm,\beta}\nabla_v \flo\|_{L^\infty_{t,x,v}}e^{-\frac{1}{2}\beta v_\pm^0-\frac{1}{2}m_\pm g\beta x_3-\frac{\beta}{2}|x_{\parallel}|} \tblo\\
   &\le  e^{-\frac{1}{2}\beta v_\pm^0-\frac{1}{2}m_\pm g\beta x_3-\frac{\beta}{2}|x_{\parallel}|}\|\mathrm{w}_{\pm,\beta} \partial_tF_\pm(0,\cdot,\cdot)\|_{L^\infty_{x,v}}\\
   &+\frac{16D_0}{5} (v_\pm^0+m_\pm g\beta x_3)\|\mathrm{w}_{\pm,\beta}\nabla_v \flo\|_{L^\infty_{t,x,v}}e^{-\frac{1}{2}\beta v_\pm^0-\frac{1}{2}m_\pm g\beta x_3-\frac{\beta}{2}|x_{\parallel}|} \\
   &\le  e^{-\frac{1}{2}\beta v_\pm^0-\frac{1}{2}m_\pm g\beta x_3-\frac{\beta}{2}|x_{\parallel}|}\|\mathrm{w}_{\pm,\beta} \partial_tF_\pm(0,\cdot,\cdot)\|_{L^\infty_{x,v}}
   \\& +\frac{64D_0}{5\beta e}\|\mathrm{w}_{\pm,\beta}\nabla_v \flo\|_{L^\infty_{t,x,v}}e^{-\frac{1}{4}\beta v_\pm^0-\frac{1}{4}m_\pm g\beta x_3-\frac{\beta}{2}|x_{\parallel}|} ,
\end{align*}where the last inequality is by the inequality that $xe^{-\frac{\beta}{2} x}\le \frac{4}{\beta e}e^{-\frac{\beta}{4}x}.$ 
Therefore, by \eqref{decay of momentum deri F}, we conclude
    \begin{multline}\label{decay of temporal deri F} 
        \sup_{t\ge 0}\left\|e^{\frac{\beta}{2} |x_{\parallel}|}e^{\frac{\beta}{4}(\vZ+m_\pm gx_3)}\partial_t\flo(t,\cdot,\cdot)\right\|_{L^\infty} 
        \le  \bigg(\|\mathrm{w}_{\pm,\beta} \partial_tF_\pm(0,\cdot,\cdot)\|_{L^\infty_{x,v}}\\
        +\frac{64CD_0}{5\beta e}\left(\|\mathrm{w}^2_{\pm,\beta}(x,v)\nabla_{x,v}F^{\textup{in}}_\pm(x,v)\|_{L^\infty_{x,v}}+\|\mathrm{w}^2_{\pm,\beta}(x_\parallel,0,v)\nabla_{x_\parallel,v}G_\pm(x_\parallel,v)\|_{L^\infty_{x_\parallel,v}}\right)\bigg).
    \end{multline}

    This completes the decay estimate for the temporal derivative $\partial_t \flo.$

    \subsubsection{Estimates for $\partial_t\Elobf$ and $\partial_t \Blobf$ }\label{subsubsec.ptEB whole}
    Given \eqref{ansatz for ptEBl} and \eqref{decay of momentum deri F}, we now prove that \eqref{ansatz for ptEBl} also holds for $\partial_t\Elobf$ and $\partial_t \Blobf$.
Notice that the system \eqref{eq.tempo flo}--\eqref{eq.tempo Maxwell} has the same structure with the system of \eqref{iterated Vlasov.M}--\eqref{iterated Maxwell.M} if we translate the notations in \eqref{eq.tempo flo}--\eqref{eq.tempo Maxwell} as follows:
\begin{equation}\label{translation}
\begin{split}
        \partial_t \flo &\mapsto \fplo,\ 
        \partial_t \Elbf\mapsto \Epl,\ 
        \partial_t \Blbf\mapsto \Bpl,\ 
        \nabla_v \flo \mapsto \nabla_v F_{\pm,\textup{st}},\ 
        \partial_t \rho^l\mapsto \rho^l_{\textup{pert}},\   
        \partial_t J^l\mapsto J^l_{\textup{pert}}.
    \end{split}
\end{equation}Note that the structure of the continuity equation and the initial-boundary conditions \eqref{continuity eq tempo}--\eqref{initial tempo eb} are also the same under the translation of the notations. Therefore, we do not need to repeat the uniform estimates for $\partial_t \Elobf$, and $\partial_t \Blobf$ given that 
\begin{enumerate}
    \item $\partial_t \flo$ satisfies the same upper-bound estimate for $\fplo$ in \eqref{fplo decay est},
    \item $\partial_t \Elbf$ and $\partial_t \Blbf$ satisfy the same bootstrap ansatz for $\Epl$ and $\Bpl$ in \eqref{apriori_EB},
    \item $\nabla_v\flo$ satisfies the same upper-bound estimate for $\nabla_v F_{\pm,\textup{st}}$ in \eqref{decay of Fst}.
\end{enumerate}
Indeed, all of the necessary conditions for the temporal derivative estimates are already satisfied by the decay estimates \eqref{decay of momentum deri F} and \eqref{decay of temporal deri F}, together with the bootstrap ansatz \eqref{ansatz for ptEBl}. The only difference compared to the bootstrap ansatz \eqref{apriori_EB} for $\Epl$ and $\Bpl$ lies in the constant coefficient: in \eqref{ansatz for ptEBl}, the constant is $D_0$ instead of $\frac{1}{16}$ for the previous estimate on $\Elbf$ and $\Blbf$ via \eqref{apriori_EB} and \eqref{bootstrap.assump.st}.

This difference, however, does not create any additional difficulty. Throughout the analysis, we continue to follow the same characteristic trajectory $(\XSlo, \VSlo)$, which is based on the fields $(\Elbf, \Blbf)$ and not on their temporal derivatives $(\partial_t \Elbf, \partial_t \Blbf)$. Since we already have the uniform bound \eqref{final estimate for flo.st} and \eqref{apriori_EB} for $(\Elbf, \Blbf)$, the characteristic trajectories $(\XSlo, \VSlo)$ remain well-controlled. In particular, the weight comparison argument \eqref{w comparison 3.whole} used in the proof of \eqref{decay of temporal deri F} remains valid.

In the uniform estimate for $\partial_t \Elbf$ and $\partial_t \Blbf$, the main new feature is the nonlinear terms $\partial_t \Elbf_S$ and $\partial_t \Blbf_S$, which now involve the larger constant $D_0$ rather than $\frac{1}{8}$. However, thanks to the additional factor of $\frac{1}{\beta^4}$ in the coefficient $c_{\pm,\beta}$ appearing in \eqref{est.acc.iS} and \eqref{est.st.iS}, we can absorb this difference by choosing $\beta$ sufficiently large. Specifically, the final estimates remain sufficiently small to close the bootstrap for \eqref{ansatz for ptEBl}.
Therefore, by following the same proof strategy as in Section~\ref{sec.boot.decay}, but adapted with the new notations introduced in \eqref{translation}, we consequently obtain the following lemma:
\begin{lemma}\label{lem.temporal deri}
    Let $g \geq 1$ and $\beta > 1$ be chosen sufficiently large so that
$$
\min\{m_+^2,m_-^2\}g^2\beta\gg 1 \quad \textup{and} \quad \min\{m_+^2,m_-^2\}\beta^4\gg 1.
$$Also, suppose that the temporal derivatives of the initial profiles, understood through the system of equations, satisfy the assumptions \eqref{initial E0i}--\eqref{f initial condition}. 
Then the uniform upper bound
\begin{equation}\label{final estimate for EloBlo tempo}
    \sup_{t\geq 0} \|(\partial_t\Elobf,\partial_t\Blobf)\|_{L^\infty} \leq D_0\min\{m_+,m_-\} g
\end{equation}
holds for the temporal derivatives.
\end{lemma}


Lastly, we introduce the following lemma on the tangential derivatives. 
\begin{lemma}[Tangential derivatives]\label{lem.tangential deri}Suppose that $\Elobf$ and $\Blobf$ are defined through \eqref{Ei5}, \eqref{E3}, \eqref{B_half_final}, \eqref{Bpar_half_final}, \Black{and \eqref{BparS_half_final}}. Suppose that $-(\FS)_3(t,x_\parallel,0,v)>c_0,$ for some $c_0>0.$ For some $T>0$ and $m>4,$ the following estimates hold:
    \label{lemma.deri.EB}
\begin{multline}\label{EB tangential deri claim}\|\nabla_{x_\parallel}\Elobf(t)\|_{L^\infty_x}+\|\nabla_{x_\parallel}\Blobf(t)\|_{L^\infty_x}
\lesssim_{T,m_\pm,m,g} (1+\|(\E^{\textup{in}},\B^{\textup{in}})\|_{C^2_x(\Omega)})(1+\|(\vZ)^m \flo\|^2_{L^\infty_{t,x,v}([0,T]\times \Omega\times \rth)})\\+\|(\vZ)^m\nabla_{x_\parallel}\flo\|_{L^\infty_{t,x,v}([0,T]\times \Omega\times \rth)}.
\end{multline}
\end{lemma}

\begin{proof}
Given the derivative estimates on the trajectories and the velocity distribution obtained in Section \ref{sec.deri.dist} above, Lemma \ref{lemma.deri.EB} on the tangential derivatives is proved in the same manner of \cite[Lemma 7, Eq. (3.2)]{MR4645724}. We omit the proof for the sake of simplicity.
\end{proof}

\begin{remark}
    All the derivative estimates made in this section are uniform in $l$ by the additional estimate \eqref{deri F estimate final} on the derivatives of $F_\pm$. Thus, the limit $(\E^\infty,\B^\infty)$ also satisfies the same estimate.
\end{remark}

\section{Global Existence}\label{sec.wellposed}
In this section, we finally provide the proof of the existence and uniqueness of solutions to the dynamical Vlasov--Maxwell systems.
\subsection{Global Existence and Regularity}
We now prove the global-in-time existence of solutions for the dynamical problems on the Vlasov--Maxwell system \eqref{2speciesVM}. 
In both cases, we consider the iterated sequences of perturbations $(f^l_\pm,\mathcal{E}^l,\mathcal{B}^l)$ to the linear systems \eqref{iterated Vlasov.M}-\eqref{iterated Maxwell.M}. Both linear systems admit solutions $(F^l_\pm,\E^l,\B^l)$ and $(f^l_\pm,\mathcal{E}^l,\mathcal{B}^l)$ for each $l\ge 0$ due to the hyperbolicity of the operators.  
In order to pass to the limit as $l\to\infty$ and to prove that these limits actually solve the original nonlinear Vlasov--Maxwell system \eqref{2speciesVM} in the weak sense, we have to pass to the limit of all the linear and nonlinear terms appearing in the iterated system \eqref{iterated Vlasov.M}-\eqref{iterated Maxwell.M}. To this end, we will additionally prove here that $F_\pm^l$ and $f_\pm^l$ are indeed Cauchy so that $F_\pm^l\to F^\infty_\pm$ and $f_\pm^l\to f^\infty_\pm$ strongly as $l\to \infty$. We introduce the following propositions on the Cauchy property of the sequences.
\begin{proposition}\label{prop.cauchy}
For each fixed $(t,x,v)\in (0,T)\times (\bar{\Omega}\times \rth\setminus\gamma_0)$,    $(F_\pm^l(t,x,v))_{l\in\mathbb{N}}$ and $(f_\pm^l(t,x,v))_{l\in\mathbb{N}}$ are Cauchy. 
\end{proposition}
\begin{remark}
    The decay estimate for the momentum derivatives $\nabla_v F_\pm$ (Proposition~\ref{Prop.decay.mom.deri}) plays a crucial role in this proof below.
\end{remark}
\begin{proof}Since the perturbation $f_\pm^l$ can also be written as $F_\pm^l-F_{\pm,\textup{st}}$ for the steady-state  $F_{\pm,\textup{st}}$ with J\"uttner-Maxwell upper bound solving \eqref{2speciesVM-steady}, it suffices to prove the Cauchy property for $(F_\pm^l(t,x,v))_{l\in\mathbb{N}}$. Fix $N_0\in \mathbb{N}.$ Then for any $k,n\ge N_0$ integers with $k\ge n$, we have \begin{equation}\label{zero initial boundary cauchy}\begin{split}(
&\fm-\fn)(0,x,v)=0,\ (\fm-\fn) (t,x_\parallel,0,v)|_{\gamma_-}=0,
        \end{split}
\end{equation}
and 
   \begin{multline*}
    \partial_t (\fm-\fn)+(\hat{v}_\pm) \cdot \nabla_x (\fm-\fn) +\left(\pm\Emmbf\pm(\hat{v}_\pm)\times \Bmmbf-m_\pm g\hat{e}_3\right)\cdot \nabla_v (\fm-\fn)\\
    = -\left(\pm(\Emmbf-\Enmbf)\pm(\hat{v}_\pm)\times (\Bmmbf-\Bnmbf)\right)\cdot \nabla_v \fn.
   \end{multline*} By \eqref{zero initial boundary cauchy}, we have
\begin{multline}\notag 
    (\fm-\fn)(t,x,v)
    = \mp \int^t_{\max\{0,t-\tbk\}} \left((\Emmbf-\Enmbf)(s,\XSk(s))+\hat{\mathcal{V}}^{k}_\pm(s)\times (\Bmmbf-\Bnmbf)(s,\XSk(s))\right)\\\cdot \nabla_v \fn(s,\XSk(s),\VSk(s))ds,
\end{multline}using the iterated characteristic trajectories \eqref{iterated char}.  Here, note that $(\Emmbf,\Bmmbf)$ and $(\Enmbf,\Bnmbf)$ solve the iterated Maxwell equations under the same initial data, we have zero initial conditions for the difference $(\Emmbf-\Enmbf,\Bmmbf-\Bnmbf)$. Therefore, using the energy comparison that
\begin{equation}\label{additional v bound.s.2}
        (\vZ) \lesssim  \langle\VSlo(s)\rangle+\bigg|\int_s^td\tau\  \FS(\tau,\XSlo(\tau),\VSlo(\tau))\bigg|\lesssim\langle \VSlo(s)\rangle +C_2T\lesssim C_T\langle \VSlo(s)\rangle  ,
    \end{equation} given by \eqref{Fl C2 bound.s}, we obtain for some positive $\delta\in(0,1) $, \begin{equation}\begin{split}\label{fm fn diff}
  & |((v_\pm^0)^{4+\delta} (\fm-\fn))(t,x,v)|\\
    &\le C_t \sup_{s\in [0,t]} \|((v_\pm^0)^{4+\delta} \nabla_v \fn)(t,\cdot,\cdot)\|_{L^\infty_{x,v}} \int_0^t\left(\| (\Emmbf-\Enmbf)(s,\cdot)\|_{L^\infty_x}+\| (\Bmmbf-\Bnmbf)(s,\cdot)\|_{L^\infty_x}\right)ds\\
   & \le C'_t \int_0^t\left(\| (\Emmbf-\Enmbf)(s,\cdot)\|_{L^\infty_x}+\| (\Bmmbf-\Bnmbf)(s,\cdot)\|_{L^\infty_x}\right)ds,
\end{split}\end{equation}via the derivative upper bound estimate \eqref{deri F estimate final} for every sequence element $\fn$. Now we make estimates on each decomposed piece of 
 $(\Emmbf-\Enmbf,\Bmmbf-\Bnmbf)$ using the representations \eqref{Ei5},  \eqref{E3}, \eqref{B_half_final}, \eqref{Bpar_half_final}, \Black{and \eqref{BparS_half_final}}.

First of all, note that the differences $(\Emmbf-\Enmbf,\Bmmbf-\Bnmbf)$ have zero homogeneous terms in their representations since $\E^{m-1}_{\textup{hom}}=\E^{n-1}_{\textup{hom}}$ and $\B^{m-1}_{\textup{hom}}=\B^{n-1}_{\textup{hom}}$. 

Regarding the $b2$ boundary terms, we observe that for $i=1,2$ and $3$,
\begin{multline}\label{eq.b2.cauchy}
      |(\Emm-\Enm)^{(1)}_{\pm,ib2}(t,x)|+|(\Emm-\Enm)^{(2)}_{\pm,ib2}(t,x)|\\\le 2\int_{B(x;t)\cap \{y_3=0\}} \frac{dy_\parallel}{|y-x|}\int_{v_3\le 0} dv\  \left(1+|(\hat{v}_\pm)_3|\frac{\sqrt{m_\pm^2+|v|^2}}{m_\pm}\right)|(\fmm-\fnm)(t-|x-y|,y_\parallel,0,v)|,
\end{multline}by \eqref{3.13}, since $(\fmm-\fnm) (t,x_\parallel,0,v)|_{\gamma_-}=0.$ Then by further making the changes of variables $y_\parallel \mapsto z\eqdef y_\parallel-x_\parallel$ and then $z\mapsto (r,\theta)$ with $|z|=r,$ we have
\begin{multline*}
      |(\Emm-\Enm)^{(1)}_{\pm,ib2}(t,x)|+|(\Emm-\Enm)^{(2)}_{\pm,ib2}(t,x)|\\\le \frac{8\pi}{m_\pm} \int_0^{\sqrt{t^2-x_3^2}} dr \ \frac{r}{\sqrt{r^2+x_3^2}}\left\|((v^0_\pm)^{4+\delta}(\fmm-\fnm))\left(t-\sqrt{r^2+x_3^2},\cdot,\cdot\right)\right\|_{L^\infty_{x,v}}\int_{v_3\le 0} dv\ (v^0_\pm)^{-3-\delta} ,
\end{multline*}for any fixed $\delta>0$. By further making the change of variables $r\mapsto \tau \eqdef t-\sqrt{r^2+x_3^2}, $ we have 
\begin{multline*}
     \hspace{-2mm} |(\Emm-\Enm)^{(1)}_{\pm,ib2}(t,x)|+|(\Emm-\Enm)^{(2)}_{\pm,ib2}(t,x)|\lesssim \frac{1}{m_\pm} \int_0^{t-x_3} d\tau \ \left\|((v^0_\pm)^{4+\delta}(\fmm-\fnm))\left(\tau,\cdot,\cdot\right)\right\|_{L^\infty_{x,v}} .
\end{multline*} 
\Black{We also recall that no $b2$ boundary terms appear in the magnetic differences
$\Bmmbf-\Bnmbf$: the magnetic field representation \eqref{B_half_final} and
\eqref{Bpar_half_final} contains no $b2$ terms, owing to the cancellation of the
boundary contributions under the curl relation established in
Section~\ref{sec.B.halfspace}.}

Note that the $b1$ initial-value parts in the Glassey-Strauss representations depends only on the initial difference  $(\fmm-\fnm)(0,x,v)$ which is zero. Therefore, all the $b1$ terms in the representations of $(\Emmbf-\Enmbf,\Bmmbf-\Bnmbf)$ are zero.

Regarding the $T$ terms, we observe that by the representation \eqref{Ei5}-\eqref{E3}, and the kernel estimate \eqref{ET kernel estimate}, we have\begin{multline}\notag
    |(\Emmbf-\Enmbf)^{(1)}_{\pm,T}(t,x)|\lesssim \int_{B^+(x;t)} \frac{dy}{|y-x|^2}\int_\rth dv\ \frac{v^0_\pm}{m_\pm} (\fmm-\fnm)(t-|x-y|,y,v)\\
    +2\int_{B(x;t)\cap \{y_3=0\}} \int_\rth \frac{(\fmm-\fnm)(t-|y-x|,y_\parallel,0,v)}{|y-x|}dvdy_\parallel.
\end{multline} 
Note that 
\begin{align*}\notag
    &\int_{B^+(x;t)} \frac{dy}{|y-x|^2}\int_\rth dv\ \frac{v^0_\pm}{m_\pm} (\fmm-\fnm)(t-|x-y|,y,v)\\
    &\le \int_{B^+(x;t)} \frac{dy}{|y-x|^2}\left\|((v^0_\pm)^{4+\delta}(\fmm-\fnm))\left(t-|x-y|,y,\cdot\right)\right\|_{L^\infty_{v}}\int_\rth dv\ \frac{(v^0_\pm)^{-3-\delta}}{m_\pm} \\
   & \le 4\pi C_{m_\pm,\delta} \int_0^tdr \   \left\|((v^0_\pm)^{4+\delta}(\fmm-\fnm))\left(t-r,\cdot,\cdot\right)\right\|_{L^\infty_{x,v}}\\
   &=4\pi C_{m_\pm,\delta} \int_0^td\tau \   \left\|((v^0_\pm)^{4+\delta}(\fmm-\fnm))\left(\tau,\cdot,\cdot\right)\right\|_{L^\infty_{x,v}},
\end{align*}where we made the changes of variables $y\mapsto y-x=r\omega$ with $r\eqdef |y-x|$ and $\omega \in \mathbb{S}^2$, and then $r\mapsto \tau\eqdef t-r.$ On the other hand, we also note that
\begin{equation}\begin{split}\label{additional term estimates}
    &2\int_{B(x;t)\cap \{y_3=0\}} \int_\rth \frac{(\fmm-\fnm)(t-|y-x|,y_\parallel,0,v)}{|y-x|}dvdy_\parallel\\
    &\le 2\int_{B(x;t)\cap \{y_3=0\}} \frac{dy_\parallel}{|y-x|} \left\|((v^0_\pm)^{4+\delta}(\fmm-\fnm))\left(t-|x-y|,y,\cdot\right)\right\|_{L^\infty_{v}}\int_\rth dv\ \frac{(v^0_\pm)^{-3-\delta}}{m_\pm}\\
   & \le 4\pi C_{m_\pm,\delta }  
    \int_0^{t-x_3}d\tau\ \left\|((v^0_\pm)^{4+\delta}(\fmm-\fnm))\left(\tau,\cdot,\cdot\right)\right\|_{L^\infty_{x,v}},
\end{split}\end{equation}by further making the changes of variables $y_\parallel \mapsto z\eqdef y_\parallel-x_\parallel$, then $z\mapsto (r,\theta)$ with $|z|=r$ and $\theta\in[0,2\pi],$ and finally $r\mapsto \tau\eqdef t-\sqrt{r^2+x_3^2}.$ Altogether, we have  \begin{equation*}
     |(\Emmbf-\Enmbf)^{(1)}_{\pm,T}(t,x)|\lesssim C_{m_\pm,\delta }  
    \int_0^{t}d\tau\ \left\|((v^0_\pm)^{4+\delta}(\fmm-\fnm))\left(\tau,\cdot,\cdot\right)\right\|_{L^\infty_{x,v}}.
\end{equation*}Similarly, we have the same upper bound for $ |(\Emmbf-\Enmbf)^{(2)}_{\pm,T}(t,x)|$. Regarding the difference in magnetic fields $(\Bmmbf-\Bnmbf)_{\pm,T}$ we instead use the kernel estimate \eqref{B35 kernel final}, and obtain the same upper bound. Thus we conclude that
\begin{multline}\notag
     |(\Emmbf-\Enmbf)_{\pm,T}(t,x)|+|(\Bmmbf-\Bnmbf)_{\pm,T}(t,x)|\\
     \lesssim C_{m_\pm,\delta }  
    \int_0^{t}d\tau\ \left\|((v^0_\pm)^{4+\delta}(\fmm-\fnm))\left(\tau,\cdot,\cdot\right)\right\|_{L^\infty_{x,v}}.
\end{multline}

Regarding the nonlinear $S$ terms, the integrands in the differences $(\Emmbf-\Enmbf)_{\pm,S}$ and $(\Bmmbf-\Bnmbf)_{\pm,S}$ involve the following difference by the representations in \eqref{Ei5}, \eqref{E3}, \eqref{B_half_final}, and \Black{ \eqref{BparS_half_final}}:
$$(\pm \Emtbf\pm \hat{v}_\pm\times \Bmtbf-m_\pm g\hat{e}_3) \fmm-(\pm \Entbf\pm \hat{v}_\pm\times 
\Bntbf-m_\pm g\hat{e}_3) \fnm.$$ We further write this as 
\begin{equation*}
    (\pm \Emtbf\pm \hat{v}_\pm\times \Bmtbf-m_\pm g\hat{e}_3) (\fmm-\fnm) \pm ((\Emtbf-\Entbf)+ \hat{v}_\pm\times 
(\Bmtbf-\Bntbf)) \fnm.
\end{equation*}
We use these two split terms in each of the differences $(\Emmbf-\Enmbf)_{\pm,S}$. 
We will use the kernel estimates \eqref{a2} and \eqref{a1} for the difference
$(\Emmbf-\Enmbf)_{\pm,S}$\Black{, and the kernel estimate \eqref{aB.final} for the
magnetic difference $(\Bmmbf-\Bnmbf)_{\pm,S}$, whose kernel $a^{\B}_{\pm,i}$ of
\eqref{aBi} obeys the same upper bound $\lesssim \frac{v^0_\pm}{m_\pm^2}$ as
\eqref{a2}--\eqref{a1}}.  Then we obtain 
\begin{equation*}\begin{split}
  & |(\Emmbf-\Enmbf)^{(1)}_{\pm,S}(t,x)|\\
  &  \lesssim \int_{B^+(x;t)}\frac{dy}{|x-y|} \int_{\rth} dv\ \frac{v^0_\pm}{m_\pm^2} \bigg|(\pm \Emtbf\pm \hat{v}_\pm\times \Bmtbf-m_\pm g\hat{e}_3) (\fmm-\fnm) (t-|x-y|,y,v) \\
  &  \pm ((\Emtbf-\Entbf)+ \hat{v}_\pm\times 
(\Bmtbf-\Bntbf)) \fnm (t-|x-y|,y,v)\bigg| .\end{split}
\end{equation*} Here, we again make the changes of variables $y\mapsto y-x=r\omega$ with $r\eqdef |y-x|$ and $\omega \in \mathbb{S}^2$, and then $r\mapsto \tau\eqdef t-r$ to obtain that
\begin{equation*}\begin{split}
   & |(\Emmbf-\Enmbf)^{(1)}_{\pm,S}(t,x)| \\
    &\lesssim \frac{C_{m_\pm,\delta}}{m_\pm}\int_0^t d\tau\ (t-\tau) \bigg[(\|  (\Emtbf(\tau,\cdot)\|_{L^\infty_x}+\|\Bmtbf(\tau,\cdot)\|_{L^\infty_x}+m_\pm g) \| ((v^0_\pm)^{4+\delta}(\fmm-\fnm))(\tau,\cdot,\cdot)\|_{L^\infty_{x,v}} \\
  & + (\|(\Emtbf-\Entbf)(\tau,\cdot)\|_{L^\infty_x}+ \|(\Bmtbf-\Bntbf)(\tau,\cdot)\|_{L^\infty_x})\| \fnm (\tau,\cdot,\cdot)\|_{L^\infty_{x,v}}\bigg] \\
  & \lesssim C_{m_\pm,\delta}gt\int_0^t d\tau\   \| ((v^0_\pm)^{4+\delta}(\fmm-\fnm))(\tau,\cdot,\cdot)\|_{L^\infty_{x,v}} \\
  & + \frac{C_{m_\pm,\delta}t}{m_\pm}\int_0^t d\tau\ (\|(\Emtbf-\Entbf)(\tau,\cdot)\|_{L^\infty_x}+ \|(\Bmtbf-\Bntbf)(\tau,\cdot)\|_{L^\infty_x}),
\end{split}\end{equation*}using the $L^\infty$ estimates \eqref{apriori_f}, \eqref{final estimate for flo.st} and \eqref{apriori_EB}, \eqref{steady state L infty} for $\fnm$, $\Emtbf$, $\Bmtbf$ obtained via the bootstrap arguments. Estimates for $(\Emmbf-\Enmbf)^{(2)}_{\pm,S}$ \Black{and for
$(\Bmmbf-\Bnmbf)^{(1)}_{\pm,S}$ and $(\Bmmbf-\Bnmbf)^{(2)}_{\pm,S}$} also give the
same upper bound\Black{, via the kernel estimate \eqref{aB.final} for the magnetic
$S$-kernel}. Therefore, we conclude that 
\begin{equation*}\begin{split}
    &|(\Emmbf-\Enmbf)_{\pm,S}(t,x)|\Black{+|(\Bmmbf-\Bnmbf)_{\pm,S}(t,x)|}\\
   &\lesssim C_{m_\pm,\delta}gt\int_0^t d\tau\   \| ((v^0_\pm)^{4+\delta}(\fmm-\fnm))(\tau,\cdot,\cdot)\|_{L^\infty_{x,v}} \\
  & + \frac{C_{m_\pm,\delta}t}{m_\pm}\int_0^t d\tau\ (\|(\Emtbf-\Entbf)(\tau,\cdot)\|_{L^\infty_x}+ \|(\Bmtbf-\Bntbf)(\tau,\cdot)\|_{L^\infty_x}).
  \end{split}
\end{equation*}

Lastly, for the differences in the field components $\Emm_3-\Enm_3$, $\Bmm_1-\Bnm_1$, and $\Bmm_2-\Bnm_2$ which satisfy the Neumann-type boundary conditions for wave equations, the following additional terms appear in the differences:
$$I_1=2\int_{B(x;t)\cap \{y_3=0\}} \int_\rth \frac{(\fmm-\fnm)(t-|y-x|,y_\parallel,0,v)}{|y-x|}dvdS_y.$$Note that the term $I_1$ is bounded from above as
\begin{equation}\label{eq.addi.cauchy}
    I_1\lesssim 4\pi C_{m_\pm,\delta }  
    \int_0^{t-x_3}d\tau\ \left\|((v^0_\pm)^{4+\delta}(\fmm-\fnm))\left(\tau,\cdot,\cdot\right)\right\|_{L^\infty_{x,v}},
\end{equation} by the estimate \eqref{additional term estimates}. 

    Collecting all the estimates for the components of $\Emmbf-\Enmbf$ and $\Bmmbf-\Bnmbf$, we conclude that 
    \begin{equation}\begin{split}\label{Emm Enm prefinal}
       & |(\Emmbf-\Enmbf)(t,x)|+|(\Bmmbf-\Bnmbf)(t,x)|\\
    & \le C\bigg(  
   (1+t) \int_0^{t}d\tau\ \sum_\pm \left\|((v^0_\pm)^{4+\delta}(\fmm-\fnm))\left(\tau,\cdot,\cdot\right)\right\|_{L^\infty_{x,v}}\\
   &+ t\int_0^t d\tau\ (\|(\Emtbf-\Entbf)(\tau,\cdot)\|_{L^\infty_x}+ \|(\Bmtbf-\Bntbf)(\tau,\cdot)\|_{L^\infty_x})\bigg),
    \end{split}\end{equation}where the constant $C$ depends only on $m_\pm,$ $g$, and $\delta$. 
By using \eqref{fm fn diff} in \eqref{Emm Enm prefinal} we further obtain that 
\begin{equation}\begin{split}
    \label{Emm Enm final}
   &  |(\Emmbf-\Enmbf)(t,x)|+|(\Bmmbf-\Bnmbf)(t,x)|\\
    & \le C\bigg(  
   (1+t) \int_0^{t}d\tau\ C_\tau' \int_0^\tau ds(\|(\Emtbf-\Entbf)(s,\cdot)\|_{L^\infty_x}+ \|(\Bmtbf-\Bntbf)(s,\cdot)\|_{L^\infty_x})\\
  & + t\int_0^t d\tau\ (\|(\Emtbf-\Entbf)(\tau,\cdot)\|_{L^\infty_x}+ \|(\Bmtbf-\Bntbf)(\tau,\cdot)\|_{L^\infty_x})\bigg)\\
  & \le C  
  ( (1+t)t C_t'+t) \int_0^t ds(\|(\Emtbf-\Entbf)(s,\cdot)\|_{L^\infty_x}+ \|(\Bmtbf-\Bntbf)(s,\cdot)\|_{L^\infty_x}),
\end{split}\end{equation}noting that the coefficient $C'_\tau$ in \eqref{fm fn diff} has its maximum at $\tau=t$ for $\tau\in[0,t].$ Now, define $C''_t\eqdef C  
  ( (1+t)t C_t'+t).$ By iterating \eqref{Emm Enm final}, we finally obtain  \begin{equation}\begin{split}\label{Emm Enm final2}
        &|(\Emmbf-\Enmbf)(t,x)|+|(\Bmmbf-\Bnmbf)(t,x)|\\
    & \le C''_t \int_0^t ds(\|(\Emtbf-\Entbf)(s,\cdot)\|_{L^\infty_x}+ \|(\Bmtbf-\Bntbf)(s,\cdot)\|_{L^\infty_x})\\
     &\le 
    (C''_t)^{n-1}\int_0^t ds\ \left(\prod_{j=1}^{n-2}\int_0^{\tau_{j-1}}d\tau_j\right)\ (\|(\E^{k-n}-\E^0)(\tau_{n-2},\cdot)\|_{L^\infty_x}+ \|(\B^{k-n}-\B^0)(\tau_{n-2},\cdot)\|_{L^\infty_x})\\
    &\le \frac{1}{8}(C''_t)^{n-1} \max\{m_+,m_-\}g \frac{t^{n-1}}{(n-1)!},
    \end{split}\end{equation} given that $\E^0=\B^0=0$ and the uniform estimate \eqref{final estimate for flo.st}, \eqref{apriori_EB}, and \eqref{steady state L infty} for $\E^{k-n}$ and $\B^{k-n}$ obtained via the bootstrap argument. We also used the notation $\tau_0\eqdef s$. Lastly, plugging \eqref{Emm Enm final2} into \eqref{fm fn diff}, we obtain for $t\in[0,T]$
    $$|((v_\pm^0)^{4+\delta} (\fm-\fn))(t,x,v)|
    \le \frac{1}{8} (C''_t)^{n}\max\{m_+,m_-\}g   \frac{t^{n}}{n!},$$
   which can be made sufficiently small as $n$ gets sufficiently large. This is via the Stirling approximation that 
   $$ n!\approx \sqrt{2\pi n} \left(\frac{n}{e}\right)^n,$$ and that $C''_t\le C''_T.$ This completes the proof of Proposition \ref{prop.cauchy} that states $(\fl(t,x,v))_{l\in\mathbb{N}}$ is Cauchy for each fixed $(t,x,v)$.
\end{proof}

Given the Cauchy property of the sequences, we are now ready to pass to the limit. For the linear terms, we directly pass to the limit via the subsequence $l_{k_i}$ as $k_i\to \infty$ for each $i=1,2,...,6$ by testing with any given $C^\infty_c$ test function which is also a $L^1$ function. 

Also, for the nonlinear terms appearing in \eqref{iterated Vlasov.M} in the case of steady states with J\"uttner-Maxwell upper bound in $\mathbb{R}^3_+,$ we observe that
\begin{equation}\begin{split}\label{nonlinear.passingtothelimit.M}
   & \left|\iiint \phi\left((\E^{l_{k_6}}+(\hat{v}_\pm)\times \B^{l_{k_6}})\cdot \nabla_v f^{l_{k_6}+1}_\pm-(\E^\infty+(\hat{v}_\pm)\times \B^\infty)\cdot \nabla_v f^{l_{k_6}+1}_\pm\right) \right|\\
   & =\left|-\iiint \nabla_v \phi\cdot \left((\E^{l_{k_6}}+(\hat{v}_\pm)\times \B^{l_{k_6}}) f^{l_{k_6}+1}_\pm -(\E^\infty+(\hat{v}_\pm)\times \B^\infty) f^\infty_\pm\right)\right|\\
 & \le  \left|\iiint \nabla_v \phi\cdot (\E^{l_{k_6}}+(\hat{v}_\pm)\times \B^{l_{k_6}}) (f^{l_{k_6}+1}_\pm-f^\infty_\pm)\right|
  +\left|\iiint \nabla_v  \phi\cdot (\E^{l_{k_6}}-\E^\infty+(\hat{v}_\pm)\times \B^{l_{k_6}}-(\hat{v}_\pm)\times \B^\infty) f^\infty_\pm\right|\\
 & \le (\|\E^{l_{k_6}}\|_{L^\infty}+\|\B^{l_{k_6}}\|_{L^\infty} ) \iiint |\nabla_v \phi||f^{l_{k_6}+1}_\pm-f^\infty_\pm|+\|f^\infty_\pm\|_{L^\infty}\iiint |\nabla_v  \phi|\left(|\E^{l_{k_6}}-\E^\infty|+|\B^{l_{k_6}}- \B^\infty| \right)\to 0, 
\end{split}\end{equation}
as $k_6\to \infty$ for any $C_c^\infty$ test function $\phi,$ since $(\mathcal{E}^{l_{k_6}},\mathcal{B}^{l_{k_6}})$ converges strongly as $k_6\to \infty$ and $f_\pm^l$ converges strongly as $l\to \infty$ so that we can use the dominated convergence theorem and the $L^\infty$ upper-bounds of $\E^{l_{k_6}}$ and $\B^{l_{k_6}}$.  Thus, we conclude that $(f^\infty_\pm,\mathcal{E}^\infty,\mathcal{B}^\infty)$ also solves the original Vlasov--Maxwell system \eqref{2speciesVM} as the perturbations from the steady states with J\"uttner-Maxwell upper bound $(F_{\pm,\textup{st}}, \E_{\textup{st}},\B_{\textup{st}})$ in the weak sense.


\subsection{Uniqueness and Non-Negativity}
We now prove the uniqueness of solutions to the dynamical Vlasov--Maxwell system~\eqref{2speciesVM}.  
The decay estimate for the momentum derivatives $\nabla_v F_\pm$ (Proposition~\ref{Prop.decay.mom.deri}) plays a crucial role in this proof.

Suppose that there are two global-in-time solutions $(\fone, \Eoned, \Boned)$ and $(\ftwo, \Etwod, \Btwod)$ for the system \eqref{2speciesVM} in the time interval $[0,T]$ with \eqref{2speciesIni}, \eqref{inflow boundary}, and \eqref{perfect cond. boundary}.
Then note that we have \begin{equation}\label{zero initial boundary difference}\begin{split}(
&\fone-\ftwo)(0,x,v)=0,\ (\fone-\ftwo) (t,x_\parallel,0,v)|_{\gamma_-}=0,
        \end{split}
\end{equation}
and the difference $\fone-\ftwo$ solves the following Vlasov equation:
   \begin{multline}\label{eq.uni1}
    \partial_t (\fone-\ftwo)+(\hat{v}_\pm) \cdot \nabla_x (\fone-\ftwo) +\left(\pm\Eoned\pm(\hat{v}_\pm)\times \Boned-m_\pm g\hat{e}_3\right)\cdot \nabla_v (\fone-\ftwo)\\
    = -\left(\pm(\Eoned-\Etwod)\pm(\hat{v}_\pm)\times (\Boned-\Btwod)\right)\cdot \nabla_v \ftwo.
   \end{multline} Note that the characteristic trajectory follows the one generated by the fields $\Eoned$ and $ \Boned$. Then, by integrating \eqref{eq.uni1} along the characteristics $\ZS(s)=(\XS(s),\VS(s))$ (associated with $\Eoned$ and $ \Boned$) for $s\in [\max\{0,t-\tb\},t]$ defined in the sense of \eqref{leading char}, we obtain\begin{multline}\notag 
    (\fone-\ftwo)(t,x,v)
    = \mp \int^t_{\max\{0,t-\tb\}} \left((\Eoned-\Etwod)(s,\XS(s))+\hat{\mathcal{V}}_\pm(s)\times (\Boned-\Btwod)(s,\XS(s))\right)\\\cdot \nabla_v \ftwo(s,\XS(s),\VS(s))ds.
\end{multline}Therefore, we obtain 
\begin{multline}\label{eq.F.diff}
   | (v^0_\pm)^{4+\delta}(\fone-\ftwo))(t,x,v)|\\
    \le C_T\sup_{s\in [0,t]} \|( (v^0_\pm)^{4+\delta}\nabla_v \ftwo)(s,\cdot,\cdot)\|_{L^\infty_{x,v}} \int_0^t\left(\| (\Eoned-\Etwod)(s,\cdot)\|_{L^\infty_x}+\| (\Boned-\Btwod)(s,\cdot)\|_{L^\infty_x}\right)ds,
\end{multline} by the energy comparison \eqref{additional v bound.s.2}. Now we make upper bound estimates on the $\Eoned-\Etwod$ and $\Boned-\Btwod$ differences using the representations \eqref{Ei5}, \eqref{E3},  \eqref{B_half_final}, \eqref{Bpar_half_final}, \Black{and \eqref{BparS_half_final}}. Note that $(\Eoned,\Boned)$ and $(\Etwod,\Btwod)$ satisfy the same initial-boundary data, and hence their difference have zero homogeneous terms in their representations. For the rest of the terms including $b2, b1, T, S$ terms, we follow the exactly same argument from \eqref{eq.b2.cauchy} to \eqref{eq.addi.cauchy} with $\fone=\fmm$, $\ftwo=\fnm$, $\Eoned=\Emmbf=\Emtbf$,
$\Etwod=\Enmbf=\Entbf$,
$\Boned=\Bmmbf=\Bmtbf$, and
$\Btwod=\Bnmbf=\Bntbf$, we obtain  
\begin{multline*}|(\Eoned-\Btwod)(t,x)|+|(\Boned-\Btwod)(t,x)|
     \le C\bigg(  
   (1+t) \int_0^{t}d\tau\ \sum_{\iota=\pm} \left\|((\vZio)^{4+\delta}(\foneio-\ftwoio))\left(\tau,\cdot,\cdot\right)\right\|_{L^\infty_{x,v}}\\
   + t\int_0^t d\tau\ (\|(\Eoned-\Etwod)(\tau,\cdot)\|_{L^\infty_x}+ \|(\Boned-\Btwod)(\tau,\cdot)\|_{L^\infty_x})\bigg),\end{multline*}
by \eqref{Emm Enm prefinal}. Then by the Gr\"onwall lemma, we obtain for $t\in[0,T]$,
\begin{multline}\label{diff.gronwall}\|(\Eoned-\Etwod)(t,\cdot)\|_{L^\infty_x}+ \|(\Boned-\Btwod)(t,\cdot)\|_{L^\infty_x}\\
     \le C\bigg(  
   (1+t) \int_0^{t}d\tau\ \sum_{\iota=\pm} \left\|((\vZio)^{4+\delta}(\foneio-\ftwoio))\left(\tau,\cdot,\cdot\right)\right\|_{L^\infty_{x,v}}\bigg)e^{\frac{t^2}{2}}.\end{multline}
Plugging \eqref{diff.gronwall} into \eqref{eq.F.diff}, we obtain
\begin{equation}\begin{split}\notag
   | (v^0_\pm)^{4+\delta}(\fone-\ftwo))(t,x,v)|
    &\le C C_T\sup_{s\in [0,t]} \|( (v^0_\pm)^{4+\delta}\nabla_v \ftwo)(s,\cdot,\cdot)\|_{L^\infty_{x,v}}\\
  &  \times \int_0^t d\tau\left(   
   (1+\tau)e^{\frac{\tau^2}{2}} \int_0^{\tau}d\tau'\ \sum_{\iota=\pm} \left\|((\vZio)^{4+\delta}(\foneio-\ftwoio))\left(\tau',\cdot,\cdot\right)\right\|_{L^\infty_{x,v}}\right)\\
  & \le C C_TT (1+T)e^{\frac{T^2}{2}}\sup_{s\in [0,T]} \|( (v^0_\pm)^{4+\delta}\nabla_v \ftwo)(s,\cdot,\cdot)\|_{L^\infty_{x,v}}\\
   & \times \int_0^t d\tau  \left(
    \sum_{\iota=\pm}\sup_{0\le \tau'\le \tau} \left\|((\vZio)^{4+\delta}(\foneio-\ftwoio)) \left(\tau',\cdot,\cdot\right)\right\|_{L^\infty_{x,v}}\right).
\end{split}\end{equation}By defining \begin{equation}
    \notag\begin{split}
        D_T&\eqdef C C_TT (1+T)e^{\frac{T^2}{2}}\sup_{s\in [0,T]} \|( (v^0_\pm)^{4+\delta}\nabla_v \ftwo)(s,\cdot,\cdot)\|_{L^\infty_{x,v}}\\
   \tilde{U}(\tau)&\eqdef   
     \sum_{\iota=\pm}\sup_{0\le \tau'\le \tau} \left\|((\vZio)^{4+\delta}(\foneio-\ftwoio)) \left(\tau',\cdot,\cdot\right)\right\|_{L^\infty_{x,v}},
    \end{split}
\end{equation} we obtain the Volterra inequality $$\tilde{U}(t)\le 2 D_T \int_0^t\tilde{U}(\tau)d\tau. $$ Further define $U(t)= e^{-2D_Tt}\tilde{U}(t).$ Then we observe that 
$$\frac{d}{dt}U(t)= e^{-2D_Tt}\frac{d}{dt}\tilde{U}(t) -2D_T \tilde{U}(t)\le 0.$$ Therefore, $U(t)$ is non-decreasing. Since $U(0)=0$ by $(\fone-\ftwo)(0,\cdot,\cdot)=0,$ we have that $U(t)=0$ for any $t\ge 0$ since $U(t)$ is non-negative.
Therefore,  we conclude that
$$\notag
   \sum_{\iota=\pm}\sup_{0\le \tau\le t} \left\|((\vZio)^{4+\delta}(\foneio-\ftwoio)) \left(\tau,\cdot,\cdot\right)\right\|_{L^\infty_{x,v}}
    =0.$$ Then this also implies that $\Eoned=\Etwod$ and $\Boned=\Btwod$ almost everywhere by \eqref{diff.gronwall}. This completes the proof of the uniqueness.

Lastly, for the proof of non-negativity, assume that the initial distributions $ F_\pm^{\textup{in}} $ and the inflow boundary profile $G_\pm$ are non-negative. Since $ F_\pm $ remains constant along the characteristics described by \eqref{leading char}, it follows that $ F_\pm $ is also non-negative.  

Consequently, Proposition \ref{prop.dyna.boot}, Proposition \ref{Prop.decay.mom.deri}, Proposition \ref{prop.derivative}, Proposition \ref{prop.main.deri.field} and Proposition \ref{prop.cauchy} together with the uniqueness and the non-negativity completes the proof of our main well-posedness theorem (Theorem \ref{thm.asymp.rth}) of the paper. In the next section, we lastly provide a generalized setting for astrophysical applications.

\section{Discussion on Astrophysical Applications}\label{sec.astro}

Many astrophysical environments, such as the regions surrounding stars, can be modeled using the Vlasov--Maxwell system, which describes the interaction of charged particles with electromagnetic fields.  
Consider, for example, a star and near the star, intense gravitational and electromagnetic forces dominate, making the Vlasov--Maxwell system under a constant gravitational field a relevant model (see \cite{2310.09865}). Near a small patch of a stellar surface the curvature is negligible on the dynamical scales of interest, so a geodesic normal chart identifies the atmosphere locally with the half-space $\R^3_+$ (or the periodic-in-$x_\parallel$ variant $\mathbb{T}^2\times\R_+$). In this chart the vertical direction coincides with the inward normal and the constant-$g$ approximation models the leading stratifying effect; dispersion in the tangential directions is preserved, while boundary interactions are localized to $x_3=0$. This is the geometric rationale for our use of the half-space domain in Sections~\ref{sec.main.comparison.weight}–\ref{sec.wellposed}. 
In this generalized model, the Lorentz force term in \eqref{2speciesVM} becomes
\[
\pm \left( \mathbf{E} + \mathbf{E}_{\textup{ext}} + \hat{v}_\pm \times (\mathbf{B} + \mathbf{B}_{\textup{ext}}) \mp m_\pm g \hat{e}_3 \right),
\]
where \( (\mathbf{E}_{\textup{ext}}, \mathbf{B}_{\textup{ext}}) \) are fixed, time-independent background fields. 
To preserve the gravitational confinement mechanism, we assume the physically natural smallness conditions
\begin{equation}
    \label{smallness.ext force}
|\mathbf{E}_{\textup{ext},3}| \ll\min\{m_-,m_+\} g, \quad |\mathbf{B}_{\textup{ext},1}|, \, |\mathbf{B}_{\textup{ext},2}| \ll \min\{m_-,m_+\}g.
\end{equation}
The condition on \( \mathbf{E}_{\textup{ext},3} \) is well-known in plasma physics, as it ensures that the net vertical force remains directed inward toward the boundary, preserving particle confinement. The assumptions on the transverse magnetic components \( \mathbf{B}_{\textup{ext},1}, \mathbf{B}_{\textup{ext},2} \) are imposed to control the additional drift effects introduced by the magnetic force \( \hat{v} \times \mathbf{B}_{\textup{ext}} \), which otherwise may dominate the stabilizing gravitational force.
\begin{remark}[Physical relevance of gravity]
 Most commonly observed in the atmospheres of stars, a classical stellar–atmosphere force balance for a hydrogen plasma (protons/electrons) shows that, with $e_+=-e_-$ and $m_+\gg m_-$, local quasi–neutrality together with vertical equilibrium yields
\[
\frac{\mathbf{E}_3}{g}\approx\frac{m_+-m_-}{\,e_+-e_-\,}\ \Longrightarrow\
m_+g\approx 2\,e_+\mathbf{E}_3.
\] See~\cite{MR4645724} and the references therein for related studies in physics. 
Thus the ion gravitational pull typically dominates the electric effect and fixes
the sign of the ambipolar field ($\mathbf{E}_3>0$ upward), favoring inward vertical
confinement. Hence the constant-$g$ regime and the smallness conditions on $(\mathbf{E}_{\mathrm{ext},3},\mathbf{B}_{\mathrm{ext},\parallel})$ in~\eqref{smallness.ext force} reflect the (most common) ubiquitous near-stellar situation for plasmas rather than a laboratory corner case.  This remark
is only a subsidiary justification—the analysis itself is purely
mathematical—but it clarifies why gravity can be treated as the leading vertical
effect in our setting.
\end{remark}
\begin{remark}[Hydrogenic composition]
Most stellar plasmas are effectively hydrogenic: protons and electrons dominate number density, while heavier ions occur at lower abundances. The two-species choice ($e_+=-e_-$) therefore captures the typical charge separation and vertical force balance. Our arguments only use the charge magnitudes and masses through sign and size, so the method extends verbatim to finitely many ion species with comparable bounds; the exit-time control and weight comparisons are unchanged.
\end{remark}

Given these assumptions, the only part of the nonlinear analysis that requires modification is the estimate on the backward exit time \( t_{\mathbf{b}}(t,x,v) \). The original estimate \eqref{exit time bound.whole} in the gravitational-only setting must be adapted to account for the influence of the external fields. In particular, under the smallness conditions above, we can prove that the modified backward exit time still satisfies a comparable upper bound:
\begin{lemma}
\label{lem.gen.exit.time.est}Suppose \eqref{smallness.ext force} holds. Then the backward exit time $\tb$ (Definition \ref{def.exit.times}) satisfies
\begin{equation}\label{gen.exit.time.est}
\tb(t,x,v) \leq \frac{C}{m_\pm g} (\vZ + m_\pm  g x_3),
\end{equation}
with a constant \( C > 0 \) that depends on the relative magnitudes of \( \mathbf{E}_{\textup{ext},3} \) and \( \mathbf{B}_{\textup{ext},\parallel} \).\end{lemma} The remainder of the proof structure, including all the decay estimates and the nonlinear bootstrap, remains unchanged.

In the non-relativistic setting, this estimate can be verified more explicitly by Taylor-expanding the vertical trajectory under the total force and observing that the dominant term is still governed by gravity when \( |\mathbf{B}_{\textup{ext},\parallel}|\ll \min\{m_-,m_+\}g  \) or the associated Larmor frequency. For the relativistic case, a fully explicit formula for the exit time may not be available. However, as shown in \cite{2310.09865}, a similar conclusion holds under analogous smallness assumptions on the external field components.

We therefore conclude that our results naturally extend to the more general setting with fixed ambient fields \( (\mathbf{E}_{\textup{ext}}, \mathbf{B}_{\textup{ext}}) \), provided the vertical component of the external electric field and the horizontal components of the magnetic field are small in comparison to the gravitational force. The initial theorems stated in Sections~\ref{sec.intro} and~\ref{sec.main thms diff and strat} can accordingly be reformulated for the full system. All later sections (Sections~\ref{sec.main.comparison.weight} through~\ref{sec.wellposed}) remain valid as written, except for the single modification to the exit time estimate. 

We close this section by introducing a generalized weight comparison argument to derive the upper bound estimates \eqref{gen.exit.time.est} on the backward exit time $\tb$.

\begin{proof}[Proof of Lemma \ref{lem.gen.exit.time.est}]
Suppose that the self-consistent electromagnetic fields $(\E,\B)$ satisfies the following assumption:
\begin{equation}\label{EB apriori bound.gen}
       \sup_t \|(\E,\B)\|_{L^\infty }\le \min\{m_+,m_-\}\frac{g}{16}. 
    \end{equation} Also, in \eqref{smallness.ext force}, we suppose the following assumption on the external background fields $(\mathbf{E}_{\textup{ext}},\mathbf{B}_{\textup{ext}})$:\begin{equation}
    \label{smallness.ext force.2}
|\mathbf{E}_{\textup{ext},3}| \le\min\{m_+,m_-\}\frac{g}{16}, \quad |\mathbf{B}_{\textup{ext},1}|, \, |\mathbf{B}_{\textup{ext},2}| \le\min\{m_+,m_-\}\frac{g}{16}.
\end{equation}
Define the characteristic trajectory $(\XS,\VS)$ such that now we have $$\frac{d\VS}{ds}=\pm(\E+\E_{\text{ext}}\pm\hat{\VS}\times (\B+\B_{\text{ext}}))- m_\pm g\hat{e}_3.$$Then we have
\begin{multline}\label{ODE for particle energy.gen}
    \frac{d}{ds}\left(\sqrt{m_\pm^2+|\VS(s)|^2}+m_\pm g(\XS)_3(s)\right)=\hat{\VS}(s)\cdot \frac{d\VS}{ds}+m_\pm g(\hat{\VS})_3(s)\\
    =\pm \hat{\VS}(s)\cdot (\E(s,\XS(s))+\E_{\text{ext}}(s,\XS(s))).
\end{multline}
Then we observe that by \eqref{EB apriori bound.gen} and \eqref{smallness.ext force.2}, we have 
   \begin{multline*}
       \frac{d(\VS)_3}{ds}(s)=-(\E+\mathbf{E}_{\textup{ext}}+\hat{\VS}\times (\B+\mathbf{B}_{\textup{ext}}))_3-m_\pm g\\
       =\mathbf{E}_3+\mathbf{E}_{\textup{ext},3}+(\hat{\VS})_1(\mathbf{B}_2+\mathbf{B}_{\textup{ext},2})-(\hat{\VS})_2 (\mathbf{B}_1+\mathbf{B}_{\textup{ext},1})-m_\pm g\le -\frac{3}{4}m_\pm g,
   \end{multline*} since $|\hat{\VS}|\le 1.$ 
Now if we define a trajectory variable $s^*=s^*(t,x,v)\in[t-\tb,t+\tf]$ such that $(\VS)_3(s^*;t,x,v)=0$, then we have
   \begin{align*}(\VS)_3(t+\tf )-(\VS)_3(s^*)&=\int^{t+\tf }_{s^*}\frac{d(\VS)_3}{ds}(\tau)d\tau\le -\frac{3}{4} m_\pm g(t+\tf -s^*),\text{ and }\\ 
  (\VS)_3(s^*)-(\VS)_3(t-\tb )&=\int_{t-\tb }^{s^*}\frac{d(\VS)_3}{ds}(\tau)d\tau\le -\frac{3}{4}  m_\pm g(s^*-(t-\tb) ).\end{align*}Therefore, we have
   \begin{equation}
       \label{tbf bound.gen}
       \tb +\tf \le -\frac{4}{3m_\pm g}((\VS)_3(t+\tf )-(\VS)_3(t-\tb )).
   \end{equation}On the other hand, using \eqref{EB apriori bound.gen}-\eqref{ODE for particle energy.gen}, we have
\begin{multline*}
    \sqrt{m_\pm^2+|\VS(t-\tb )|^2}=\left(\vZ+ m_\pm gx_{3}\right)
\pm \int_t^{t-\tb }  \hat{\VS}(s)\cdot (\E(s,\XS(s))+\E_{\text{ext}}(s,\XS(s)))ds\\
    \le \left(\vZ+ m_\pm gx_{3}\right)+\frac{ m_\pm g}{8}\tb,
\end{multline*} 
 and   
\begin{multline*}
    \sqrt{m_\pm^2+|\VS(t+\tf )|^2}=\left(\vZ+ m_\pm gx_{3}\right)
    \pm \int_t^{t+\tf }  \hat{\VS}(s)\cdot (\E(s,\XS(s))+\E_{\text{ext}}(s,\XS(s)))ds\\*
    \le \left(\vZ+ m_\pm gx_{3}\right)+\frac{ m_\pm g}{8}\tf.
\end{multline*} Thus, together with \eqref{tbf bound.gen}, we have$$
       \tb+\tf\le \frac{4}{3  m_\pm g}\left(2(\vZ+  m_\pm gx_{3})+\frac{ m_\pm  g}{8} (\tb+\tf)\right).
 $$     Therefore, we have
    $$\tb+\tf\le \frac{16}{5 m_\pm g}(\vZ+m_\pm gx_3),$$
    and this completes the proof of \eqref{gen.exit.time.est}.
\end{proof}

\section{Asymptotic Stability of Vacuum in \texorpdfstring{$\mathbb{T}^2\times \mathbb{R}_+$}{}}\label{sec.bootstrap.torus}
In this section, we consider the three-dimensional relativistic Vlasov--Maxwell system in a periodic channel domain $\Omega\eqdef \mathbb{T}^2\times \mathbb{R}_+= (0,1)^2 \times (0,\infty)$ with the background ions in the vacuum steady-state. The main goal is to establish the asymptotic stability of the vacuum state in the horizontally periodic geometry $\mathbb{T}^2 \times \mathbb{R}_+$. We prove that the velocity distribution functions $F_\pm$ exhibit exponential decay in a weighted $L^\infty_{x,v}$ norm. The self-consistent electromagnetic fields satisfy Maxwell’s equations, which, combined with the continuity equation, reduce to inhomogeneous wave equations in $\mathbb{T}^2 \times \mathbb{R}_+$. A central difficulty arises from the lack of dispersion in the periodic directions, which precludes time decay of the fields.
Nevertheless, we establish asymptotic stability in the sense that the velocity distribution $F_\pm$ pointwisely decays exponentially fast and the $L^\infty_x$ norms of the fields remain uniformly bounded at the scale of the gravitational strength. The proof relies on a careful analysis of the characteristic flows and weighted energy estimates that exploit decay in the $x_3$-direction.

For the problem in $\mathbb{T}^2\times (0,\infty),$ we assume not just the initial-boundary conditions \eqref{initial E0i} and \eqref{f initial condition} with the inflow boundary condition \eqref{inflow boundary} and the perfect conductor condition \eqref{perfect cond. boundary} but also the following periodic boundary conditions in $x_\parallel\eqdef (x_1,x_2)$ in the form of
\begin{equation}\label{2speciesBoundary}
\begin{split}
    & \partial_\alpha^\beta F_\pm (t,0,x_2,x_3,v)=\partial_\alpha^\beta  F_\pm (t,1,x_2,x_3,v), \ \partial_\alpha^\beta F_\pm (t,x_1,0,x_3,v)=\partial_\alpha^\beta  F_\pm (t,x_1,1,x_3,v),\\
     & \partial_\alpha \E (t,0,x_2,x_3)= \partial_\alpha \E (t,1,x_2,x_3), \ \partial_\alpha \E (t,x_1,0,x_3)=\partial_\alpha \E (t,x_1,1,x_3),\\
       &\partial_\alpha \B (t,0,x_2,x_3)=\partial_\alpha \B (t,1,x_2,x_3), \ \partial_\alpha \B (t,x_1,0,x_3)= \partial_\alpha \B(t,x_1,1,x_3),
\end{split}
\end{equation}where $\alpha_\pm $ and $\beta$ are multi-indices that stand for the spatial and momentum derivatives, respectively. We consider $|\alpha_\pm|,|\beta|=0,1$. We also impose that the inflow boundary profile $G_\pm$ satisfies the following decay assumption:
\begin{equation}
    \label{h assumption}
    \notag
    \|e^{\frac{m_\pm g\beta}{16}t}e^{\beta \vZ}G_\pm(t,x_\parallel,v)\|_{L^\infty_{x,v}(\gamma_-)}\le C,
\end{equation}
We remark that the compact support condition can be replaced by the following condition in the proof of  vacuum stability. Suppose that the initial fields $ \E^{\textup{in}}=(\mathbf{E}_{01},\mathbf{E}_{02},\mathbf{E}_{03})^\top$ and $\B^{\textup{in}}=(\mathbf{B}_{01},\mathbf{B}_{02},\mathbf{B}_{03})^\top$ satisfies the following condition \eqref{initial E0i rep};
 for $P_0=\mathbf{E}_{0i}$ or $=\mathbf{B}_{0i}$ for $i=1,2,3$, we assume \begin{equation}
    \label{initial E0i rep} \frac{1}{t}\sup_{(x_1,x_2)\in\mathbb{T}^2}\int_{x_3}^{x_3+t} dz' \   {}\left| P_0(\cdot,\cdot,z') \right|\le c_0 \min\{m_-,m_+\}g,
\end{equation} for a sufficiently small $c_0>0.$ 

We first state our main theorem on the asymptotic stability of vacuum in $\mathbb{T}^2\times \mathbb{R}_+$ where we lack sufficient dispersion for decay:\begin{theorem}[Asymptotic Stability of the Vacuum Steady State in $\mathbb{T}^2 \times \mathbb{R}_+$] \label{thm.main.asymp.vacuum}
Let $\Omega = \mathbb{T}^2 \times \mathbb{R}_+$. Fix $g > 0$, $\gamma > 1$, and choose $\beta > 1$ such that
$$
\min\{m_-^2, m_+^2\} g^2 \beta \gg 1 \quad \text{and} \quad \min\{m_-^2, m_+^2\} g \beta^3 \gg 1.
$$
Suppose the initial and boundary data satisfy \eqref{initial E0i}, \eqref{f initial condition}, and \eqref{2speciesBoundary} with the inflow boundary condition \eqref{inflow boundary} and the perfect conductor condition \eqref{perfect cond. boundary}. Then there exists a unique weak solution $(F_\pm, \E, \B)$ to the two-species Vlasov--Maxwell system in $\Omega$. 
Moreover, the solution satisfies the exponential decay estimate:
$$
\sup_{t \ge 0} e^{\frac{m_\pm g \beta}{16} t} \left\| e^{\frac{\beta}{4}(\vZ + m_\pm g x_3)} F_\pm(t) \right\|_{L^\infty_{x,v}} \le C,
$$
and the electromagnetic fields remain uniformly bounded:
$$
\sup_{t \ge 0} \|(\E, \B)(t)\|_{L^\infty_x} \le \min\{m_+, m_-\} \frac{g}{8}.
$$
In addition, for any $T > 0$, the derivatives of the solution satisfy
\begin{align*}
\sup_{0 \le t \le T} \left( \| (v^0_\pm)^m \partial_t F_\pm(t) \|_{L^\infty_{x,v}} + |||F_\pm(t) ||| \right) &\le C_T, \
\| (\E, \B) \|_{W^{1,\infty}_{t,x}([0,T] \times \Omega)} \le C_T,
\end{align*}
for some $m > 4$.
\end{theorem}
In the horizontally periodic geometry $ \mathbb{T}^2\times \mathbb{R}_+ $, the wave equations satisfied by the electromagnetic fields lose two directions of dispersion compared to the whole half-space case scenario in $ \mathbb{R}^3_+$. Since the Green function of the 3D wave equation in $ \mathbb{T}^2\times \mathbb{R}_+ $ does not yield decay in the $ x_\parallel = (x_1,x_2) $ directions, any decay in the self-consistent fields must come solely from the $ x_3 $-direction. This limited dispersion precludes classical $ t $-decay for the full field components. Nevertheless, we obtain exponential decay of $ F_\pm $ in a pointwise sense by carefully analyzing the characteristics and applying weighted estimates.

 \subsection{On the Stability of Vacuum for the Vlasov--Maxwell System}\label{sec.history.asymp}
The study of the asymptotic stability of vacuum in the context of kinetic theory and the Vlasov--Maxwell system has evolved significantly since its inception. The first major result was established by Glassey and Strauss \cite{MR0919231} in 1987, who proved the decay behavior of the system where the density $\rho(t,x)$ satisfies $\rho(t,x) \lesssim \langle t\rangle^{-3}$ and the electromagnetic fields $|\E|(t,x) + |\B|(t,x) \lesssim \langle t - |x| \rangle^{-1} \langle t + |x| \rangle^{-1}$. 
Then subsequently the requirement for compact support in phase space and the smallness condition in the initial fields have been removed by the works \cite{schaeffer2004small,bigorgne2020sharp,wang2022propagation}.  

The results cited above are set in the whole space and establish decay of the local density \(\rho(t,x)\) together with decay of the electromagnetic fields \((\E,\B)\). To the best of the authors' knowledge, pointwise-in-\((x,v)\) decay of the distribution functions \(F_\pm\) themselves has not been known even in the whole space. In addition, we work on the slab with periodic tangential directions,
\[
\Omega:=\mathbb{T}^2\times(0,\infty),\qquad \partial\Omega={x_3=0},
\]
where the two periodic directions preclude the use of full three-dimensional dispersion and coercivity estimates via volume rescaling. Nevertheless, the first part of this paper establishes \emph{pointwise} asymptotic stability for the Vlasov--Maxwell system in this domain: under suitable initial data and inflow source on \(\gamma_-\),
\begin{multline*}
\hspace{-0.4cm}\sup_{t\ge 0}e^{\frac{m_\pm g\beta}{16}t}\big\|e^{\frac{\beta}{4}(\vZ+m_\pm g x_3)}F_\pm(t,\cdot,\cdot)\big\|_{L^\infty_{x,v}}\
\le \left\|e^{\beta\left(\vZ+m_\pm g x_{3}\right)} F^{\textup{in}}_\pm\right\|_{L^\infty_{x,v}}
+\sup_{\tau\ge 0}\big\|e^{\frac{m_\pm g\beta}{16}\tau}e^{\beta\vZ}G_\pm(\tau,x_\parallel,v)\big\|_{L^\infty_{x,v}(\gamma_-)} ,
\end{multline*}
and, in addition,
\[
\sup_{t\ge 0}\|(\E,\B)(t)\|_{L^\infty}\le \min\{m_+,m_-\},\frac{g}{8}.
\]
Here \(\beta>1\) plays the role of the inverse local temperature \(1/T\). If we choose \(g\approx \beta^{-1}\), then the weighted \(L^\infty\) norm \(\|w_{\pm,\beta} F^{\textup{in}}_\pm\|_{L^\infty_{x,v}}\) of the initial profile \(F^{\textup{in}}_\pm\) can be taken arbitrarily large.

For the velocity distribution functions $F_\pm(t,x,v)$ that solve the system on $\Omega\times \rth,$ we will take the periodic extension of $\Omega$ in $x_1$ and $x_2$ directions to obtain the half-space $\mathbb{R}^2\times \mathbb{R}_+$ and will equivalently consider the $(x_1,x_2)$-periodic solution $\bar{F}_\pm(t,x,v)$ on $\mathbb{R}^2\times \mathbb{R}_+$ which is the $(x_1,x_2)$-periodic solution of $F_\pm(t,x,v)$ on $(x,v)\in \Omega\times \rth.$ We will abuse the notation and write the periodic extension $\bar{F}_\pm$ also as $F_\pm$ from now on. 

We will also consider the characteristic position and momentum variables $\ZS(s)\eqdef (\XS(s),\VS(s))$ along the characteristic trajectory underlying the partial differential equation \eqref{2speciesVM} for 2-species dynamics.   Throughout the rest of the section, we understand $\XS$ in the whole half-space $\mathbb{R}^3_+$ which is the periodic extension of the actual domain $\Omega=\mathbb{T}^2\times \mathbb{R}_+$ in tangential directions whenever the derivative of $\XS$ is being considered. Then $\ZS(s)=(\XS(s),\VS(s))$ satisfies \eqref{leading char}. The solution $\allowbreak (\XS(s), \VS(s))$ to \eqref{leading char} is well-defined under the a priori assumption that $\E$ and $\B$ are in $W^{1,\infty}$ and hence are locally Lipshitz continuous in the spatial variables uniformly in the temporal variable.

We first establish crucial comparison estimates of weight functions along the characteristics.  For a given $\beta>1$, we define a weight function for a 2-species problem in the periodic-in-$x_\parallel$ space $\mathbb{T}^2\times \mathbb{R}_+$
\begin{equation}\label{weights.torus}
w_\pm(x, v)  =w_{\pm,\beta}(x, v)=e^{\beta\left(\vZ+m_\pm g x_{3}\right)} =e^{\beta\left(\sqrt{m_\pm^2+|v|^2}+m_\pm g x_{3}\right)}.\end{equation}
\subsection{Weight Comparison in $\mathbb{T}^2\times \mathbb{R}_+$}To begin with, we introduce the following ODE for particle energy. Consider the characteristic trajectory variables and note that we have
\begin{equation}\label{ODE for particle energy}
    \frac{d}{ds}\left(\sqrt{m_\pm^2+|\VS(s)|^2}+m_\pm g(\XS)_3(s)\right)=\hat{\VS}(s)\cdot \frac{d\VS}{ds}+m_\pm g(\hat{\VS})_3(s)
    =\pm \hat{\VS}(s)\cdot \E(s,\XS(s)),
\end{equation}because $$\frac{d\VS}{ds}=\pm\E\pm\hat{\VS}\times \B-m_\pm g\hat{e}_3.$$
 Using \eqref{ODE for particle energy} we have the following estimate between weight functions along the characteristics as in the lemma below. 
 \begin{lemma}Let $w_{\pm,\beta}$, $\ZS$, $\tb,$ and $\tf$ be defined as in \eqref{weights.torus}, \eqref{leading char}, and \eqref{Back Forw exit time}, respectively. Assume that the self-consistent electro-magnetic fields $(\E,\B)$ satisfies the following bound:
\begin{equation}\label{EB apriori bound}
       \sup_t \|(\E,\B)\|_{L^\infty }\le \min\{m_+,m_-\}\frac{g}{8}. 
    \end{equation}Then for $s, s^{\prime} \in\left[t-\tb(t, x, v), t+\tf(t, x, v)\right]$,
\begin{align}\label{w comparison 1}
\frac{w_{\pm,\beta}\left( \ZS\left(s^{\prime} ; t, x, v\right)\right)}{w_{\pm,\beta}( \ZS(s ; t, x, v))} &\leq e^{\frac{16 \beta}{5 m_\pm g}\left\|\E\right\|_{L_{t, x}^{\infty}} (\vZ+m_\pm gx_3)},\text{ and }\\
\label{w comparison 2}
    \frac{1}{w_{\pm,\beta}( \ZS(s ; t, x, v))} &\leq  e^{-\frac{1}{2} \beta \vZ} e^{-\frac{1}{2}m_\pm g\beta x_{3}}.
\end{align}
\end{lemma}\begin{proof} The proof of 
    \eqref{w comparison 1}--\eqref{w comparison 2} depends on the bound of $\tb+\tf$. We claim that we obtain the bound 
    \begin{equation}\label{exit time bound}
        \tb(t,x,v) + \tf (t,x,v) \leq \frac{16}{5 m_\pm g}(\vZ+m_\pm gx_3).
    \end{equation}  Write $\VS=((\VS)_1,(\VS)_2,(\VS)_3)^\top.$ Note that the relativistic vertical velocity $(\hat{\VS})_3(s)$ depends not just on the vertical momentum $(\VS)_3(s)$ but also on the horizontal momenta  $(\VS)_1(s)$ and $(\VS)_2(s)$, and hence we need to consider an additional control on those momenta as well.  The bound \eqref{exit time bound} depends crucially on the bound \eqref{EB apriori bound} of the field $\E$ and $\B$ due to the characteristic ODEs \eqref{leading char}.  
    Upon the assumption that we have \eqref{exit time bound}, we obtain \eqref{w comparison 1}   from \eqref{ODE for particle energy} as follows; note that for $s,s'\in[t-\tb,t+\tf],
    $
    we have    \begin{equation*}\begin{split}
        \frac{w_{\pm,\beta}\left( \ZS\left(s^{\prime} ; t, x, v\right)\right)}{w_{\pm,\beta}( \ZS(s ; t, x, v))} &=e^{\beta g ((\XS)_3(s')-(\XS)_3(s)) + \beta (\VS^0(s')-\VS^0(s))}\\
        &=e^{\beta\left(\int_s^{s'}\frac{d}{d\tau}(\VS^0(\tau)+m_\pm g(\XS)_3(\tau))d\tau\right)}
        \le e^{\beta |s'-s|\sup_\tau|\hat{\VS}(\tau)||\E(\tau,\XS(\tau))|}\\
        &\le e^{\beta (\tb+\tf)\sup_\tau|\hat{\VS}(\tau)||\E(\tau,\XS(\tau))|}
        \leq e^{\frac{16 \beta}{5 m_\pm g}\left\|\E\right\|_{L_{t, x}^{\infty}} (\vZ+m_\pm gx_3)},
    \end{split}\end{equation*}
    by \eqref{ODE for particle energy} and \eqref{exit time bound}. 
  Now, in order to obtain \eqref{w comparison 2}, we put $s'=t$ in \eqref{w comparison 1} such that $$w_{\pm,\beta}(\ZS(t;t,x,v))=w_{\pm,\beta}(x,v)=e^{\beta \vZ+m_\pm g\beta x_3},$$ and hence we have
  \begin{multline*}
       \frac{1}{w_{\pm,\beta}( \ZS(s ; t, x, v))} \le  e^{\frac{16 \beta}{5m_\pm g}\left\|\E\right\|_{L_{t, x}^{\infty}} (\vZ+m_\pm gx_3)}e^{-\beta \vZ-m_\pm g\beta x_3}
       \le  e^{\frac{2 \beta}{5}(\vZ+m_\pm gx_3)}e^{-\beta \vZ-m_\pm g\beta x_3}\\
       \le e^{-\frac{1}{2}\beta \vZ-\frac{1}{2}m_\pm g\beta x_3},
  \end{multline*} by \eqref{EB apriori bound}. This concludes the proof for \eqref{w comparison 2}.

   Now it suffices to prove \eqref{exit time bound} under \eqref{EB apriori bound}. Suppose the assumption \eqref{EB apriori bound} holds. Then we observe that by \eqref{EB apriori bound} we have $$\left|\left(\frac{d(\VS)_1}{ds}(s),\frac{d(\VS)_2}{ds}(s)\right)\right|\le \left|\E+\hat{\VS}\times \B\right| \le m_\pm \frac{g}{4},$$ and
   $$\frac{d(\VS)_3}{ds}(s)=-(\E+\hat{\VS}\times \B)_3-m_\pm g \le -\frac{3}{4}m_\pm g,$$ since $|\hat{\VS}|\le 1.$ 
Now if we define a trajectory variable $s^*=s^*(t,x,v)\in[t-\tb,t+\tf]$ such that $(\VS)_3(s^*;t,x,v)=0$, then we have
   $$(\VS)_3(t+\tf )-(\VS)_3(s^*)=\int^{t+\tf }_{s^*}\frac{d(\VS)_3}{ds}(\tau)d\tau\le -\frac{3}{4} m_\pm g(t+\tf -s^*),$$ 
   and $$(\VS)_3(s^*)-(\VS)_3(t-\tb )=\int_{t-\tb }^{s^*}\frac{d(\VS)_3}{ds}(\tau)d\tau\le -\frac{3}{4}  m_\pm g(s^*-(t-\tb) ).$$Therefore, we have
   \begin{equation}
       \label{tbf bound}
       \tb +\tf \le -\frac{4}{3m_\pm g}((\VS)_3(t+\tf )-(\VS)_3(t-\tb )).
   \end{equation}On the other hand, using \eqref{ODE for particle energy} and \eqref{EB apriori bound}, we have
$$ \sqrt{m_\pm^2+|\VS(t-\tb )|^2}=\left(\vZ+ m_\pm gx_{3}\right)
\pm \int_t^{t-\tb }  \hat{\VS}(s)\cdot \E(s,\XS(s))ds
    \le \left(\vZ+ m_\pm gx_{3}\right)+\frac{ m_\pm g}{8}\tb,$$
 and   
$$ \sqrt{m_\pm^2+|\VS(t+\tf )|^2}=\left(\vZ+ m_\pm gx_{3}\right)
    \pm \int_t^{t+\tf }  \hat{\VS}(s)\cdot \E(s,\XS(s))ds
    \le \left(\vZ+ m_\pm gx_{3}\right)+\frac{ m_\pm g}{8}\tf.$$Thus, together with \eqref{tbf bound}, we have$$
       \tb+\tf\le \frac{4}{3  m_\pm g}\left(2(\vZ+  m_\pm gx_{3})+\frac{ m_\pm  g}{8} (\tb+\tf)\right).
 $$     Therefore, we have
    $$\tb+\tf\le \frac{16}{5 m_\pm g}(\vZ+m_\pm gx_3),$$
    and this completes the proof of \eqref{exit time bound}.
\end{proof}

In the remainder of this section, we present the proof of the main theorem on the asymptotic stability of vacuum (Theorem~\ref{thm.main.asymp.vacuum}). The essential task is to construct the solution iteration scheme and carry out the bootstrap argument establishing the $L^\infty$ estimates in Theorem~\ref{thm.main.asymp.vacuum}. The regularity estimates and the Cauchy property ensuring existence and uniqueness follow directly from the arguments already developed in Sections~\ref{sec.deri.dist}, \ref{sec.deri.EB}, and \ref{sec.wellposed}, once the domain $\Omega = \mathbb{T}^2 \times (0,\infty)$ is periodically extended to the half-space $\mathbb{R}^3_+$. Recall that we work with the characteristic trajectory variables in the extended half-space, where both the domain and the trajectories have already been periodically extended.
 To avoid redundancy, the details of these steps are omitted.

\subsection{Estimates for the Distributions and the Fields}
We now consider the following solution iteration. For any $l\in \mathbb{N}\cup \{0\},$ let $\flo$ solve the following equation:
\begin{equation}\label{iterated Vlasov}
    \begin{split}
        &\partial_t \flo+(\hat{v}_\pm) \cdot \nabla_x \flo +\left(\pm\Elbf\pm(\hat{v}_\pm)\times \Blbf-m_\pm g\hat{e}_3\right)\cdot \nabla_v \flo = 0,\\
        & \flo(0,x,v)= F^{\textup{in}}_\pm(x,v),\\
        &\flo (t,x_\parallel,0,v)|_{\gamma_-}=G_\pm(t,x_\parallel,v)|_{\gamma_-},
    \end{split}
\end{equation} where $G_\pm:\gamma_-\to [0,\infty)$ is a given incoming profile (also denoting its periodic extension by $G_\pm$) and we assume that  $F_\pm^0,\E^0,\B^0\eqdef 0$.  Recall that the profile $G_\pm$ satisfies the decay assumption in time as in \eqref{h assumption}
\begin{equation}
    \notag
    \|e^{\frac{m_\pm g\beta}{16}t}w_{\pm,\beta} (x_\parallel,0,v)G_\pm(t,x_\parallel,v)\|_{L^\infty_{x,v}(\gamma_-)}\le C,
\end{equation}where $\beta>1$ and the weight $w=w_{\pm,\beta}$ is defined as \eqref{weights.torus}. Note that we will choose $g$ and $\beta$ such that $g^2\beta^3\gg 1.$ 

The electromagnetic fields $(\Elbf,\Blbf)$ for $l\in\mathbb{N}$ solve the Maxwell equations:
\begin{equation}\label{iterated Maxwell}
    \begin{split}
        &\partial _t \Elbf-\nabla_x\times \Blbf= - 4\pi J^l =-4\pi \int_\rth (\hat{v}_+F^l_+-\hat{v}_-F^l_-) dv,\\
        &\partial_t \Blbf+\nabla_x\times \Elbf=0,\\
        &\nabla_x \cdot \Elbf=4\pi \rho^l=4\pi \int_\rth (F^l_+-F^l_-) dv,\\
        &\nabla_x \cdot \Blbf=0.
        \end{split}
        \end{equation}
     By \eqref{leading char}, we can also define the characteristic trajectory $\ZSlo=(\XSlo,\VSlo)$ which solves
\begin{equation}\notag
 \begin{split}
   \frac{d\XSlo(s)}{ds}&=\hVSlo(s)=\frac{\VSlo(s)}{\sqrt{m_\pm^2+|\VSlo(s)|^2}},\\
        \frac{d\VSlo(s)}{ds}&=\pm\Elbf(s,\XSlo(s))\pm\hVSlo(s)\times \Blbf(s,\XSlo(s))-m_\pm g\hat{e}_3,
 \end{split} 
\end{equation}   where $\XSlo(s)=\XSlo(s;t,x,v)$, $\VSlo(s)=\VSlo(s;t,x,v)$, $\hat{e}_3\eqdef (0,0,1)^\top$, and $(\hat{v}_\pm)\eqdef \frac{v}{\sqrt{m_\pm^2+|v|^2}}$. 
Following \eqref{Back Forw exit time} we define the backward and the forward exit times for the iterated characteristic trajectory:
\begin{equation}\label{iterated Back Forw exit time}\begin{split}
   & \tflo (t,x,v)= \sup\{s\in [0,\infty): (\XSlo)_3(t+\tau;t,x,v)>0\ \textup{ for all } \tau\in(0,s)\}\ge 0,\\
   & \tblo (t,x,v)= \sup\{s\in [0,\infty): (\XSlo)_3(t-\tau;t,x,v)>0\ \textup{ for all } \tau\in(0,s)\}\ge 0.
   \end{split}
\end{equation}
Then following \eqref{solution f} we can write our solution $\flo$ as
\begin{multline}\label{solution flo}
    \flo(t,x,v)= 1_{t\le \tblo(t,x,v)}F^{\textup{in}}_\pm(\XSlo(0;t,x,v),\VSlo(0;t,x,v))\\+1_{t>\tblo(t,x,v)}G_\pm(t-\tblo,\XSlo(t-\tblo;t,x,v),\VSlo(t-\tblo;t,x,v)).
\end{multline}In the rest of the section, we prove the following main proposition:
\begin{proposition}\label{prop.dyna.boot.vacuum}
For any $l\in\mathbb{N}$, we have
\begin{align}\label{Ansatz for fl} 
        \sup_{t\ge 0}e^{\frac{m_\pm g\beta}{16}t}\|e^{\frac{\beta}{4}(\vZ+m_\pm gx_3)}\fl(t,\cdot,\cdot)\|_{L^\infty} &\le C+\|w_{\pm,\beta} F^{\textup{in}}_\pm\|_{L^\infty_{x,v}},\text{ and }\\\label{ansatz for EBl}
        \sup_{t\ge 0}\|(\Elbf,\Blbf)\|_{L^\infty}&\le \min\{m_+,m_-\}\frac{g}{8},
    \end{align}
for a sufficiently large $\beta>1$ such that $\min\{m_-^2,m_+^2\}g^2\beta\gg1$ and $\min\{m_-^2,m_+^2\}g \beta^3\gg 1$ where the weight $w=w_{\pm,\beta}$ is defined as \eqref{weights.torus}
\end{proposition} In the following sections, we fix $ l \in \mathbb{N} $ and assume that \eqref{Ansatz for fl}--\eqref{ansatz for EBl} hold at the iteration level $ (l) $. We then show that these same estimates remain valid at the next level $ (l+1) $, thereby closing the bootstrap argument.

\begin{proof}[Proof of Proposition \ref{prop.dyna.boot.vacuum}]
  Proposition \ref{prop.dyna.boot.vacuum} follows from Lemma \ref{lem.dyna.fplo.vacuum} and Lemma \ref{lem.dyna.EB.vacuum}, which will be established in the subsequent sections.
\end{proof}Physically, the ansatz \eqref{ansatz for EBl} corresponds to the situation where the self-consistent fields are bounded from above by the one-eighth of the product of the gravity $(\beta g)$ and the local temperature $(\beta^{-1}=T)$.

\subsubsection{Estimates for the Velocity Distribution Function}In this section, we prove the estimate \eqref{Ansatz for fl} at the iteration level $(l+1)$.
\begin{lemma}\label{lem.dyna.fplo.vacuum}
   Fix $l\in\mathbb{N}$ and suppose \eqref{apriori_EB} hold for $(\Epl,\Bpl).$ Then $\fplo$ satisfies  \begin{multline}
    \label{final estimate for flo}
    e^{\frac{\beta}{4} \vZ} e^{\frac{1}{4}m_\pm g\beta x_{3}}e^{\frac{m_\pm g\beta}{16}t}|\flo(t,x,v)|\\\le   \left(\|w_{\pm,\beta} F^{\textup{in}}_\pm\|_{L^\infty_{x,v}}+\sup_{0\le\tau\le t}\|e^{\frac{m_\pm g\beta}{16}\tau}w_{\pm,\beta} (x_\parallel,0,v)G_\pm(\tau,x_\parallel,v)\|_{L^\infty_{x,v}(\gamma_-)}\right).
\end{multline}
\end{lemma}
\begin{proof}
Based on the bootstrap ansatz \eqref{Ansatz for fl}--\eqref{ansatz for EBl}, we provide a decay estimate for $\flo.$ Using \eqref{solution flo}, we obtain that
\begin{align*}
    |\flo(t,x,v)|&\le  1_{t\le \tblo(t,x,v)}|F^{\textup{in}}_\pm(\XSlo(0;t,x,v),\VSlo(0;t,x,v))|\\&+1_{t>\tblo(t,x,v)}|G_\pm(t-\tblo,\XSlo(t-\tblo;t,x,v),\VSlo(t-\tblo;t,x,v))|\\
   &\le  \frac{1_{t\le \tblo(t,x,v)}}{w_{\pm,\beta}( \ZSlo(0 ; t, x, v))}\|w_{\pm,\beta} F^{\textup{in}}_\pm\|_{L^\infty_{x,v}}\\&+\frac{1_{t> \tblo(t,x,v)}e^{-\frac{m_\pm g\beta}{16}(t-\tblo)}}{w_{\pm,\beta}( \ZSlo(t-\tblo ; t, x, v))}\|e^{\frac{m_\pm g\beta}{16}(t-\tblo)}(w_{\pm,\beta} h)(t-\tblo,(\xblo)_\parallel,\vblo)\|_{L^\infty_{x,v}(\gamma_-)},
\end{align*}where $\xblo\eqdef \XSlo(t-\tb)$ and $\vblo\eqdef \VSlo(t-\tb)$. 
Using \eqref{w comparison 2} and \eqref{h assumption}, we further have 
\begin{multline}\notag
    |\flo(t,x,v)|
   \le  1_{t\le  \tblo(t,x,v)}e^{-\frac{1}{2}\beta \vZ} e^{-\frac{1}{2}m_\pm g\beta x_{3}}\|w_{\pm,\beta} F^{\textup{in}}_\pm\|_{L^\infty_{x,v}}\\
   +1_{t> \tblo(t,x,v)} e^{-\frac{1}{2}\beta \vZ} e^{-\frac{1}{2}m_\pm g\beta x_{3}}e^{-\frac{m_\pm g\beta}{16}(t-\tblo)} \sup_{0\le\tau\le t}\|e^{\frac{m_\pm g\beta}{16}\tau}w_{\pm,\beta} (x_\parallel,0,v)G_\pm(\tau,x_\parallel,v)\|_{L^\infty_{x,v}(\gamma_-)}.
\end{multline}
Now we recall the following bound from \eqref{exit time bound}:
\begin{equation}\notag
        \tblo(t,x,v)\leq \frac{16}{5m_\pm g}(\vZ+m_\pm gx_3),
    \end{equation} and note that $$  e^{-\frac{m_\pm g\beta}{16}(t-\tblo(t,x,v))}\\
    \le  e^{-\frac{m_\pm g\beta}{16}t}e^{\frac{\beta}{5}(\vZ+m_\pm gx_3)}.$$
Also, if $t\le \tblo,$ then using \eqref{exit time bound} again, we obtain $$t\le \tblo \le \frac{16}{5m_\pm g}(\vZ+m_\pm gx_3).$$
Thus, we finally have
\begin{align*}\notag
    &|\flo(t,x,v)|\le  1_{t\le  \tblo(t,x,v)}e^{-\frac{1}{4}\beta \vZ} e^{-\frac{1}{4}m_\pm g\beta x_{3}}e^{-\frac{\beta}{4 }( \vZ+m_\pm gx_3)} \|w_{\pm,\beta} F^{\textup{in}}_\pm\|_{L^\infty_{x,v}}
   \\ &+1_{t> \tblo(t,x,v)} e^{-\frac{1}{2}\beta \vZ} e^{-\frac{1}{2}m_\pm g\beta x_{3}}e^{-\frac{m_\pm g\beta}{16}(t-\tblo)}\sup_{0\le\tau\le t}\|e^{\frac{m_\pm g\beta}{16}\tau}w_{\pm,\beta} (x_\parallel,0,v)G_\pm(\tau,x_\parallel,v)\|_{L^\infty_{x,v}(\gamma_-)}\\
  & \le  1_{t\le  \tblo(t,x,v)}e^{-\frac{1}{4}\beta \vZ} e^{-\frac{1}{4}m_\pm g\beta x_{3}}e^{-\frac{5\beta}{64}m_\pm gt} \|w_{\pm,\beta} F^{\textup{in}}_\pm\|_{L^\infty_{x,v}}
   \\ &+1_{t> \tblo(t,x,v)} e^{-\frac{1}{2}\beta \vZ} e^{-\frac{1}{2}m_\pm g\beta x_{3}}e^{-\frac{m_\pm g\beta}{16}t}e^{\frac{\beta}{5}(\vZ+m_\pm gx_3)}\sup_{0\le\tau\le t}\|e^{\frac{m_\pm g\beta}{16}\tau}w_{\pm,\beta} (x_\parallel,0,v)G_\pm(\tau,x_\parallel,v)\|_{L^\infty_{x,v}(\gamma_-)}\\
   &\le   e^{-\frac{\beta}{4} \vZ} e^{-\frac{1}{4}m_\pm g\beta x_{3}}e^{-\frac{m_\pm g\beta}{16}t}\|w_{\pm,\beta} F^{\textup{in}}_\pm\|_{L^\infty_{x,v}}\\&+ e^{-\frac{1}{4}\beta \vZ} e^{-\frac{1}{4}m_\pm g\beta x_{3}}e^{-\frac{m_\pm g\beta}{16}t}\sup_{0\le\tau\le t}\|e^{\frac{m_\pm g\beta}{16}\tau}w_{\pm,\beta} (x_\parallel,0,v)G_\pm(\tau,x_\parallel,v)\|_{L^\infty_{x,v}(\gamma_-)}\\
    &\le  e^{-\frac{\beta}{4} \vZ} e^{-\frac{1}{4}m_\pm g\beta x_{3}}e^{-\frac{m_\pm g\beta}{16}t}(\|w_{\pm,\beta} F^{\textup{in}}_\pm\|_{L^\infty_{x,v}}+\sup_{0\le\tau\le t}\|e^{\frac{m_\pm g\beta}{16}\tau}w_{\pm,\beta} (x_\parallel,0,v)G_\pm(\tau,x_\parallel,v)\|_{L^\infty_{x,v}(\gamma_-)}),
\end{align*} using \eqref{ansatz for EBl}. This completes the proof.
\end{proof}

Now we prove the estimate \eqref{ansatz for EBl} for the iteration level $(l+1)$ in the following subsections. For the field representations in $\Omega=\mathbb{T}^2\times (0,\infty),$ we refer to Sections \ref{sec.5} and \ref{sec.4} for the electric and the magnetic field representations, respectively. Note that in the case of $\Omega=\mathbb{T}^2\times \mathbb{R}_+,$ we additionally define the periodic extension of the Cauchy data for the fields in the tangential $x_\parallel$ directions and use them in the representations, instead.

\subsubsection{Estimates for the Homogeneous Wave Solutions}
\label{sec.Ei.start} 
Recall that in \eqref{B3 homo solution} and \eqref{Eparallel homo solution}, we obtained the representations for the homogeneous terms of the self-consistent fields $\Ei$ for $i=1,2$ and $\Bth$ which solve the homogeneous wave equation under the Dirichlet boundary conditions \eqref{sys.Eihomo}. Also, recall \eqref{Bi homo solution} solves the homogeneous wave equation under the Neumann boundary condition. They are given by
\begin{multline}\notag
\Eihomo(t, x) =\frac{1}{4\pi t^2} \int_{\partial B(x; t)\cap \{y_3>0\}}  \left(t \tilde{\mathbf{E}}^1_{0i}(y) + \tilde{\mathbf{E}}_{0i}(y) + \nabla \tilde{\mathbf{E}}_{0i}(y) \cdot (y - x)\right) dS_y\\
-\frac{1}{4\pi t^2} \int_{\partial B(x; t)\cap \{y_3<0\}}  \left(t \tilde{\mathbf{E}}^1_{0i}(\bar{y}) + \tilde{\mathbf{E}}_{0i}(\bar{y}) + \nabla \tilde{\mathbf{E}}_{0i}(\bar{y}) \cdot (\bar{y} - \bar{x})\right) dS_y
=:\mathbf{E}_{ih1}+\mathbf{E}_{ih2},
\end{multline} 
\begin{multline}\notag
\Bihomo(t, x) =\frac{1}{4\pi t^2} \int_{\partial B(x; t)\cap \{y_3>0\}}  \left(t \tilde{\mathbf{B}}^1_{0i}(y) + \tilde{\mathbf{B}}_{0i}(y) + \nabla \tilde{\mathbf{B}}_{0i}(y) \cdot (y - x)\right) dS_y\\
+\frac{1}{4\pi t^2} \int_{\partial B(x; t)\cap \{y_3<0\}}  \left(t \tilde{\mathbf{B}}^1_{0i}(\bar{y}) + \tilde{\mathbf{B}}_{0i}(\bar{y}) + \nabla \tilde{\mathbf{B}}_{0i}(\bar{y}) \cdot (\bar{y} - \bar{x})\right) dS_y
=:\mathbf{B}_{ih1}+\mathbf{B}_{ih2},
\end{multline}and \begin{multline}\notag
\Bthhomo(t, x) =\frac{1}{4\pi t^2} \int_{\partial B(x; t)\cap \{y_3>0\}}  \left(t \tilde{\mathbf{B}}^1_{0i}(y) + \tilde{\mathbf{B}}_{03}(y) + \nabla \tilde{\mathbf{B}}_{03}(y) \cdot (y - x)\right) dS_y\\
-\frac{1}{4\pi t^2} \int_{\partial B(x; t)\cap \{y_3<0\}}  \left(t \tilde{\mathbf{B}}^1_{03}(\bar{y}) + \tilde{\mathbf{B}}_{03}(\bar{y}) + \nabla \tilde{\mathbf{B}}_{03}(\bar{y}) \cdot (\bar{y} - \bar{x})\right) dS_y
=:\mathbf{B}_{3h1}+\mathbf{B}_{3h2},
\end{multline} for $i=1,2,$ where $\bar{y}\eqdef (y_1,y_2,-y_3)^\top$ and we denote the periodic extension of the Cauchy data $\mathbf{E}_{0i}$, $\mathbf{E}^1_{0i}$, $\mathbf{B}_{0i}$, $\mathbf{B}^1_{0i}$, $\mathbf{B}_{03}$, and $\mathbf{B}^1_{03}$ as $ \tilde{\mathbf{E}}_{0i} $, $ \tilde{\mathbf{E}}^1_{0i} $, $
\tilde{\mathbf{B}}_{0i}$, $\tilde{\mathbf{B}}^1_{0i}$, $
\tilde{\mathbf{B}}_{03}$ and $\tilde{\mathbf{B}}^1_{03}$ in the $x_\parallel$ directions, respectively.
Here we note that the homogeneous terms of the fields $\Ethhomo$ and $\Bihomo$ that solve the homogeneous wave equation under the Neumann boundary conditions also have the same representation (e.g., the first two lines in \eqref{E3.T} for $\Ethhomo$) up to a sign change in the second line. Therefore, we only introduce the estimate for $\Eihomo$ here.

For the estimate of $\mathbf{E}_{ih1}$ note that
\begin{equation}\label{5.10}
    \left|t \tilde{\mathbf{E}}^1_{0i}(y) + \tilde{\mathbf{E}}_{0i}(y) + \nabla \tilde{\mathbf{E}}_{0i}(y) \cdot (y-x)\right|\le  t\left(\left| \tilde{\mathbf{E}}^1_{0i}(y)\right|+\left|\nabla \tilde{\mathbf{E}}_{0i}(y) \right|\right) + \left| \tilde{\mathbf{E}}_{0i}(y) \right|,
\end{equation}
since $t=|y-x|$ on $\partial B(x;t).$ We can write $y = x+t\omega $ where $\omega $ is an angular variable on $\mathbb{S}^2.$ By choosing the direction of $y_3$ as the $z$-axis direction ($\theta=0$) for the spherical coordinates $\omega = \omega(\phi,\theta)$, we observe that
   \begin{equation*}
       \frac{1}{4\pi}\int_{\partial B(x;t)} \frac{dS_{y}}{t^2}(tP(y)+G(y))=\frac{1}{4\pi}\int_0^{2\pi}d\phi \int _0 ^{\frac{\pi}{2}}d\theta\  t^2\sin\theta \frac{1}{t^2}(tP(x+t\omega )+G(x+t\omega)),
   \end{equation*} for sufficiently regular non-negative functions $P$ and $G$. 
       By writing $$\omega = (\sin\theta\cos\phi, \sin\theta\sin\phi,\cos\theta),$$ we further obtain that
       \begin{multline}
       \label{5.11}\int_0^{2\pi}d\phi \int _0 ^{\frac{\pi}{2}}d\theta\  \sin\theta \ (tP+G)(x+t(\sin\theta\cos\phi, \sin\theta\sin\phi,\cos\theta))\\
           \le 4\pi \sup_{x_1,x_2}\int_0^1 dz \ (tP+G) (\cdot,\cdot,x_3+tz)=4\pi \sup_{(x_1,x_2)\in\mathbb{T}^2}\int_{x_3}^{x_3+t} dz' \  \left(P+\frac{G}{t}\right)(\cdot,\cdot,z'),
       \end{multline}where we made the changes of variables $\theta \mapsto z\eqdef \cos\theta$ and then $z\mapsto z'=x_3+tz.$ 
For the derivative terms $\tilde{\mathbf{E}}^1_{0i}$ and $\nabla \tilde{\mathbf{E}}_{0i}$ in \eqref{5.10} (which corresponds to $P$ of \eqref{5.11} for each), we observe that \begin{equation}\begin{split}\label{upper remark 5.1}
           \sup_{(x_1,x_2)\in\mathbb{T}^2}\int_{x_3}^{x_3+t} dz' \   {}\left| P(\cdot,\cdot,z') \right|
          & \le \left\| \bar{w}_{\pm,\beta} P \right\|_{L^\infty(\mathbb{T}^2\times [x_3,x_3+t])} \int_{x_3}^{x_3+t} dz' \  e^{-m_\pm g\beta z'}\\
           & = \frac{1}{m_\pm g\beta }\left\| \bar{w}_{\pm,\beta} P \right\|_{L^\infty(\mathbb{T}^2\times [x_3,x_3+t])} \int_{m_\pm g\beta x_3}^{m_\pm g\beta (x_3+t)} dz \  e^{-z}   
            \\
           & = \frac{1}{m_\pm g\beta }\left\| \bar{w}_{\pm,\beta} P \right\|_{L^\infty(\mathbb{T}^2\times [x_3,x_3+t])}  e^{-m_\pm g\beta x_3}(1-e^{-m_\pm g\beta t}).
     \end{split}\end{equation} Therefore, as long as the initial data $P=\tilde{\mathbf{E}}^1_{0i}$ or $=\nabla \tilde{\mathbf{E}}_{0i}$ satisfies that its weighted $L^\infty$ norm  $\left\| \bar{w}_{\pm,\beta} P \right\|_{L^\infty(\mathbb{T}^2\times (0,\infty))}$ is bounded, then by \eqref{upper remark 5.1}, the corresponding parts in $\Eihomo$ can be made small enough such that it satisfies \eqref{ansatz for EBl}. 
     On the other hand, for $G$ term which corresponds to the zeroth-order term $\mathbf{E}_{0i},$ we follow the initial condition \eqref{initial E0i} that
     $$\frac{1}{t}\sup_{(x_1,x_2)\in\mathbb{T}^2}\int_{x_3}^{x_3+t} dz' \   {}\left| \mathbf{E}_{0i}(\cdot,\cdot,z') \right|\le c_0\min\{m_-,m_+\}g,$$ for a sufficiently small $c_0>0.$

       Since the estimate for $\mathbf{E}_{ih2}$ is also similar, altogether we conclude
    \begin{multline*}
       |\Eihomo(t,x)|\\
       \le 2\bigg(\frac{1}{m_\pm g\beta }  e^{-m_\pm g\beta x_3}(1-e^{-m_\pm g\beta t})\bigg(\left\| \bar{w}_{\pm,\beta} \mathbf{E}^1_{0i} \right\|_{L^\infty(\mathbb{T}^2\times [x_3,x_3+t])}
       +\left\| \bar{w}_{\pm,\beta} \nabla \mathbf{E}_{0i} \right\|_{L^\infty(\mathbb{T}^2\times [x_3,x_3+t])}\bigg)\\+c_0\min\{m_-,m_+\}g\bigg),
    \end{multline*}since $\tilde{\mathbf{E}}_{0i}$ and $\tilde{\mathbf{E}}^1_{0i}$ are $x_\parallel$-periodic extensions of $\mathbf{E}_{0i}$ and $\mathbf{E}^1_{0i}$, respectively. Then if $\beta>1$ sufficiently large such that $\min\{m_-^2,m_+^2\}g^2\beta\gg1,$ then we have   $$|\Eihomo(t,x)|\ll \min\{m_-,m_+\}g.$$

    Similarly, by \eqref{initial E0i}, we obtain the same upper-bound for $\Bthhomo$ (and the same form also for $\Bihomo$ and $\Ethhomo$) as 
    \begin{multline*}
       |\Bthhomo(t,x)|\\
       \le 2\bigg(\frac{1}{m_\pm g\beta }  e^{-m_\pm g\beta x_3}(1-e^{-m_\pm g\beta t})\bigg(\left\| \bar{w}_{\pm,\beta} \mathbf{B}^1_{03} \right\|_{L^\infty(\mathbb{T}^2\times [x_3,x_3+t])}
       +\left\| \bar{w}_{\pm,\beta} \nabla \mathbf{B}_{03} \right\|_{L^\infty(\mathbb{T}^2\times [x_3,x_3+t])}\bigg)\\+c_0\min\{m_-,m_+\}g\bigg).
    \end{multline*}

\subsubsection{Estimates for the Boundary-Value Component of the Tangential Electric Fields} 
\label{sec.5.5}

In this subsection, we make upper-bound estimates for the boundary-value parts $(\Elo)^{(1)}_{\pm,ib2}$ and $(\Elo)^{(2)}_{\pm,ib2}$ in the Glassey-Strauss representation of the field $\Elo_i$ in \eqref{Ei5.T}--\eqref{Ei6.T} which satisfies the Dirichlet boundary conditions.  

For the estimates of $(\Elo)^{(1)}_{\pm,ib2}$ and $(\Elo)^{(2)}_{\pm,ib2}$ in \eqref{Ei5.T} and \eqref{Ei6.T}, we have to estimate the following term 
\begin{equation}
    \notag (0,0,1)^\top-\frac{(\omega+\hat{v}_\pm)(\hat{v}_\pm)_3}{1+\hat{v}_\pm\cdot \omega}\textup{ and }(0,0,1)^\top-\frac{(\bar{\omega}+(\hat{v}_\pm))(\hat{v}_\pm)_3}{1+(\hat{v}_\pm)\cdot \bar{\omega}},
\end{equation} where $\bar{\omega}=(\omega_1,\omega_2,-\omega_3)^\top.$ For each, we use \eqref{ortho decomp}--\eqref{bound for v+w term 2} (and the latter one with $\bar{w}$ replacing $w$) and obtain that in both cases we have
\begin{equation}\notag\left|(0,0,1)^\top-\frac{(\omega+\hat{v}_\pm)(\hat{v}_\pm)_3}{1+\hat{v}_\pm\cdot \omega}\right|,\ \left|(0,0,1)^\top-\frac{(\bar{\omega}+(\hat{v}_\pm))(\hat{v}_\pm)_3}{1+(\hat{v}_\pm)\cdot \bar{\omega}}\right| \le 1+|(\hat{v}_\pm)_3|\frac{\sqrt{m_\pm^2+|v|^2}}{m_\pm}.\end{equation} Therefore, using the estimate \eqref{final estimate for flo} on $\flo$ and the assumption \eqref{h assumption} on the boundary profile $G_\pm$, we obtain that
\begin{equation}\begin{split}\label{Eb2 estimate beginning}
    &|(\Elo)^{(1)}_{\pm,ib2}(t,x)|+|(\Elo)^{(2)}_{\pm,ib2}(t,x)|\\&\le 2\int_{B(x;t)\cap \{y'_3=0\}} \frac{dy'_\parallel}{|y'-x|}\int_{v_3>0} dv\  \left(1+|(\hat{v}_\pm)_3|\frac{\sqrt{m_\pm^2+|v|^2}}{m_\pm}\right)|G_\pm(t-|x-y'|,y'_\parallel,v)|\\
    &+ 2\int_{B(x;t)\cap \{y'_3=0\}} \frac{dy'_\parallel}{|y'-x|}\int_{v_3\le 0} dv\  \left(1+|(\hat{v}_\pm)_3|\frac{\sqrt{m_\pm^2+|v|^2}}{m_\pm}\right)|\flo(t-|x-y'|,y'_\parallel,0,v)|\\
   & \le \frac{4}{m_\pm}\int_{B(x;t)\cap \{y'_3=0\}} \frac{dy'_\parallel}{|y'-x|}\int_{v_3>0} dv\  \vZ w_{\pm,\beta}^{-1}(y'_\parallel,0,v)e^{-\frac{m_\pm g\beta}{16}(t-|x-y'|)}\\&\times \sup_{0\le\tau\le t}\|e^{\frac{m_\pm g\beta}{16}\tau}w_{\pm,\beta} (x_\parallel,0,v)G_\pm(\tau,x_\parallel,v)\|_{L^\infty_{x,v}(\gamma_-)}\\
   & + \frac{4}{m_\pm}\int_{B(x;t)\cap \{y'_3=0\}} \frac{dy'_\parallel}{|y'-x|}\int_{v_3\le 0} dv\  \vZ e^{-\frac{\beta}{4} \vZ}e^{-\frac{m_\pm g\beta}{16}(t-|x-y'|)}\\&\times  \left(\|w_{\pm,\beta} F^{\textup{in}}_\pm\|_{L^\infty_{x,v}}+\sup_{0\le\tau\le t}\|e^{\frac{m_\pm g\beta}{16}\tau}w_{\pm,\beta} (x_\parallel,0,v)G_\pm(\tau,x_\parallel,v)\|_{L^\infty_{x,v}(\gamma_-)}\right)\\
    & \le \frac{4}{m_\pm}c_{\pm,\beta} e^{-\frac{m_\pm g\beta}{16}t}  \left(\|w_{\pm,\beta} F^{\textup{in}}_\pm\|_{L^\infty_{x,v}}+\sup_{0\le\tau\le t}\|e^{\frac{m_\pm g\beta}{16}\tau}w_{\pm,\beta} (x_\parallel,0,v)G_\pm(\tau,x_\parallel,v)\|_{L^\infty_{x,v}(\gamma_-)}\right)\\&\times \int_{B(x;t)\cap \{y'_3=0\}} \frac{dy'_\parallel}{|y'-x|}e^{\frac{m_\pm g\beta}{16}|x-y'|} .
\end{split}\end{equation}
Here, $c_{\pm,\beta} \eqdef \int_{\rth} dv\  \vZ e^{-\frac{\beta}{4} \vZ} $ and by rescaling we have $c_{\pm,\beta} \approx \frac{1}{\beta^4}$ (see \eqref{additional beta decay}). 

We now make a change of variables $y'_\parallel\mapsto y'_\parallel-x_\parallel=:z$ in the last line of \eqref{Eb2 estimate beginning}. This $z$ lies in $B(0; \sqrt{t^2-|x_3|^2}),$ since $|y'_\parallel-x_\parallel|^2+|x_3|^2<t^2.$ Then we make an additional polar coordinate change $(z_1,z_2)\mapsto (r,\theta) $. Then we observe that 
\begin{equation}\begin{split}\label{Eib2 estimate middle}
   & |(\Elo)^{(1)}_{\pm,ib2}(t,x)|+|(\Elo)^{(2)}_{\pm,ib2}(t,x)|\\
    & \le \frac{4}{m_\pm}c_{\pm,\beta} \left(\|w_{\pm,\beta} F^{\textup{in}}_\pm\|_{L^\infty_{x,v}}+\sup_{0\le\tau\le t}\|e^{\frac{m_\pm g\beta}{16}\tau}w_{\pm,\beta} (x_\parallel,0,v)G_\pm(\tau,x_\parallel,v)\|_{L^\infty_{x,v}(\gamma_-)}\right)e^{-\frac{m_\pm g\beta}{16}t}\\&\times \int_{B(0;\sqrt{t^2-|x_3|^2})} \frac{dz}{\sqrt{|z|^2+|x_3|^2}}e^{\frac{m_\pm g\beta}{16}\sqrt{|z|^2+|x_3|^2}}
   \\ 
 &=\frac{8}{m_\pm}\pi c_{\pm,\beta} \left(\|w_{\pm,\beta} F^{\textup{in}}_\pm\|_{L^\infty_{x,v}}+\sup_{0\le\tau\le t}\|e^{\frac{m_\pm g\beta}{16}\tau}w_{\pm,\beta} (x_\parallel,0,v)G_\pm(\tau,x_\parallel,v)\|_{L^\infty_{x,v}(\gamma_-)}\right)e^{-\frac{m_\pm g\beta}{16}t}\\&\times \int_0^{\sqrt{t^2-|x_3|^2}} \frac{rdr}{\sqrt{r^2+|x_3|^2}}e^{\frac{m_\pm g\beta}{16}\sqrt{r^2+|x_3|^2}}.
\end{split}\end{equation}Finally, we make a change of variable $r\mapsto r'\eqdef \sqrt{r^2+|x_3|^2}$ with $dr' = \frac{rdr}{\sqrt{r^2+|x_3|^2}}$ and obtain that \begin{equation}\begin{split}\label{Final estimate for b2 terms}
   & |(\Elo)^{(1)}_{\pm,ib2}(t,x)|+|(\Elo)^{(2)}_{\pm,ib2}(t,x)|\\
   & \le \frac{8}{m_\pm}\pi c_{\pm,\beta} \left(\|w_{\pm,\beta} F^{\textup{in}}_\pm\|_{L^\infty_{x,v}}+\sup_{0\le\tau\le t}\|e^{\frac{m_\pm g\beta}{16}\tau}w_{\pm,\beta} (x_\parallel,0,v)G_\pm(\tau,x_\parallel,v)\|_{L^\infty_{x,v}(\gamma_-)}\right)e^{-\frac{m_\pm g\beta}{16}t}\\&\times \int_{|x_3|}^{t} dr'\ e^{\frac{m_\pm g\beta}{16}r'}\\
    &=\frac{128\pi c_{\pm,\beta} }{m_\pm^2 g\beta}\left(\|w_{\pm,\beta} F^{\textup{in}}_\pm\|_{L^\infty_{x,v}}+\sup_{0\le\tau\le t}\|e^{\frac{m_\pm g\beta}{16}\tau}w_{\pm,\beta} (x_\parallel,0,v)G_\pm(\tau,x_\parallel,v)\|_{L^\infty_{x,v}(\gamma_-)}\right)\\&\times (1- e^{-\frac{m_\pm g\beta}{16}(t-|x_3|)})\\
    &\le \frac{128\pi c_{\pm,\beta} }{m_\pm^2 g\beta}\left(\|w_{\pm,\beta} F^{\textup{in}}_\pm\|_{L^\infty_{x,v}}+\sup_{0\le\tau\le t}\|e^{\frac{m_\pm g\beta}{16}\tau}w_{\pm,\beta} (x_\parallel,0,v)G_\pm(\tau,x_\parallel,v)\|_{L^\infty_{x,v}(\gamma_-)}\right),
\end{split}\end{equation}where the last inequality is from the inequality, $|x_3|< t$, which is the outcome of the fact that on $B(x;t)$, we have $|x-y'|<t$, and that the further restriction $\{y'_3=0\} $ gives $|x_3|<t.$  Note that the upper-bound of $\Elo_{b2}$ terms in \eqref{Final estimate for b2 terms} can be made sufficiently small for a sufficiently large $\beta$ since $c_{\pm,\beta} \lesssim \beta^{-4}$ by \eqref{additional beta decay}.
 
 Then in \eqref{Final estimate for b2 terms}, if $\beta>1$ is chosen sufficiently large such that $\min\{m_-^2,m_+^2\}g^2\beta^5\gg 1$, then we have $$|(\Elo)^{(1)}_{\pm,ib2}(t,x)|+|(\Elo)^{(2)}_{\pm,ib2}(t,x)|\ll \min\{m_-,m_+\}g. $$ This completes the estimates for $|(\Elo)^{(1)}_{\pm,ib2}(t,x)|$ and $|(\Elo)^{(2)}_{\pm,ib2}(t,x)|$ boundary terms.

    \subsubsection{Estimates for the Initial-Value Components of the Tangential Electric Fields} 
    \label{sec.5.7} In this subsection, we make upper-bound estimates for the initial-value parts $(\Elo)^{(1)}_{\pm,ib1}$ and $(\Elo)^{(2)}_{\pm,ib1}$ in the Glassey-Strauss representation of the field $\Elo_i$ in \eqref{Ei5.T}--\eqref{Ei6.T}.  Recall that $(\Elo)^{(1)}_{\pm,ib1}$ is given by 
\begin{multline*}
    (\Elo)^{(1)}_{\pm,ib1}(t,x)
    =\pm\int_{\partial B(x;t)\cap \{y'_3>0\}} \frac{dS_{y'}}{|y'-x|}\int_\rth dv\ \left((\delta_{ij})^\top_{i=1,2,3}-\frac{(\omega+\hat{v}_\pm)(\hat{v}_\pm)_j}{1+\hat{v}_\pm\cdot \omega}\right)\omega^j\flo(0,y',v),
\end{multline*} with the standard Einstein summation convention. $(\Elo)^{(2)}_{\pm,ib1}$ also has a similar structure. By using \eqref{ortho decomp}--\eqref{bound for v+w term 2} again, we have \begin{equation}
    \notag\left|\left((\delta_{ij})^\top_{i=1,2,3}-\frac{(\omega+\hat{v}_\pm)(\hat{v}_\pm)_j}{1+\hat{v}_\pm\cdot \omega}\right)\omega^j\right|\le 1+\frac{\sqrt{m_\pm^2+|v|^2}}{m_\pm}\le 2\frac{\sqrt{m_\pm^2+|v|^2}}{m_\pm}.
\end{equation} Using the notation of the energy density $\energy(t,x)$ defined in \eqref{energy density}, we have
\begin{equation}
    \label{Eb1 bound by ke}|(\Elo)^{(1)}_{\pm,ib1}(t,x)|\le \frac{1}{2\pi m_\pm}\int_{\partial B(x;t)\cap \{y'_3>0\}} \frac{dS_{y'}}{|y'-x|}\energy(0,y').
\end{equation}
    Note that on $\partial B(x;t)$ we have $|x-y'|=t.$ So we can write $y' = x+t\omega $ where $\omega $ is an angular variable on $\mathbb{S}^2.$ By choosing the direction of $y'_3$ as the $z$-axis direction ($\theta=0$) for the spherical coordinates $\omega = \omega(\phi,\theta)$, we observe that
   $$\int_{\partial B(x;t)\cap \{y'_3>0\}} \frac{dS_{y'}}{|y'-x|}\energy(0,y')=\int_0^{2\pi}d\phi \int _0 ^{\pi}d\theta\  t^2\sin\theta \frac{1}{t}\energy(0,x+t\omega)1_{x_3+t\cos\theta>0},$$
       by writing $\omega = (\sin\theta\cos\phi, \sin\theta\sin\phi,\cos\theta).$ We further obtain that
       \begin{multline}
           \label{5.12}\left|\frac{1}{2\pi m_\pm}\int_0^{2\pi}d\phi \int _0 ^{\frac{\pi}{2}}d\theta\  t\sin\theta \energy(0,x+t(\sin\theta\cos\phi, \sin\theta\sin\phi,\cos\theta))\right|\\
           \le  \frac{1}{ m_\pm}\sup_{x_1,x_2}\int_0^1 dz \ t\ \energy(0,\cdot,\cdot,x_3+tz)= \frac{1}{ m_\pm}\sup_{(x_1,x_2)\in\mathbb{T}^2}\int_{x_3}^{x_3+t} dz' \  \energy(0,\cdot,\cdot,z').
       \end{multline}where we made the changes of variables $\theta \mapsto z\eqdef \cos\theta$ and then $z\mapsto z'=x_3+tz.$ 
       Since the estimate for $(\Elo)^{(2)}_{\pm,ib1}$ is also similar, we have
    \begin{equation}\label{Eb1 final upper bound}
        \sup_{t,x}|\Elo_{b1}(t,x)|\le \frac{2}{ m_\pm}\|  \energy^{\textup{in}}\|_{L^\infty_{x_1,x_2}(\mathbb{T}^2)L^1_{x_3}((0,\infty))},
    \end{equation}where $\Elo_{b1}\eqdef (\Elo)^{(1)}_{\pm,ib1}+(\Elo)^{(2)}_{\pm,ib1}.$
    \begin{remark}\label{remark.more x3 decay}
        If the initial energy density \textnormal{$\energy(0,x)$} has a rapid decay in $x_3$ direction (i.e., if we further assume that for any $x_3\ge 0, $ \begin{equation}
            \label{ex.ke decay condition}
    \|\textnormal{$\energy^{\textup{in}}$}(\cdot,\cdot,x_3)\|_{L^\infty_{x_1,x_2}(\mathbb{T}^2)}\le M e^{-\beta^\gamma m_\pm g x_3}
        \end{equation} for some $M>0$ and $\gamma >1$ where $g$ and $\beta$ are the same constants from \eqref{weights.torus}), then the right-hand side of \eqref{5.12} can be estimated further and we obtain that the corresponding part of $\E$ also decays fastly in $x_3$ direction as
    \begin{equation}\label{Eb1 estimate ex}
\sup_{(x_1,x_2)\in\mathbb{T}^2}|\textnormal{$\Elo_{b1}$}(t,x)|\lesssim \sup_{(x_1,x_2)\in\mathbb{T}^2}\int_{x_3}^{x_3+t} dz' \  \textnormal{$\energy$}(0,\cdot,\cdot,z')\lesssim \frac{M}{\beta^\gamma m_\pm g} e^{-\beta^\gamma m_\pm g x_3}(1-e^{-\beta^\gamma m_\pm gt}),
    \end{equation} for any $x_3\ge 0$ and $t\ge 0.$ 
        Then by choosing $\beta> 1$ such that $\min\{m_-^2,m_+^2\}g^2\beta^\gamma \gg 1,$ we can make the upper bound \eqref{Eb1 estimate ex} of the field \textnormal{$|\Elo_{b1}|$} be small enough such that the bootstrap ansatz \eqref{ansatz for EBl} holds. 
    \end{remark}
    \begin{remark}
        The usual appearance of the growth with respect to $t$ in $t \times \textnormal{$\energy$} (0,x+t\omega)$ term should require that the initial data \textnormal{$\energy(0,x)$ } decays in $|x|$. But since our extended domain $\mathbb{R}^3_+$ is the periodic extension of $\mathbb{T}^2\times (0,\infty)$, we cannot assume the decay in $x_1$ and $x_2$ directions for the initial data. Since the decay in $x_3$ can be achieved, we resolve this problem by taking the initial profile in $L^1$ with respect to $x_3$ and in $L^\infty$ with respect to $x_1$ and $x_2$.
    \end{remark}

\subsubsection{Estimates for the Initial-Value Components of the Magnetic Field} From \eqref{Bpar_half_final} note that $(\Blo)^{(1)}_{\pm,ib1}$ is defined as
$$(\Blo)^{(1)}_{\pm,ib1}(t,x)=\int_{\partial B(x;t)\cap \{y'_3>0\}} \frac{dS_{y'}}{t}\int_\rth dv\ \left(1-\frac{\hat{v}_\pm\cdot \omega}{1+\hat{v}_\pm\cdot \omega}\right)(\omega\times (\hat{v}_\pm))_i\flo(0,y',v).$$ $(\Blo)^{(2)}_{\pm,ib1}$ is defined similarly on $\{y'<0\}$ with $\bar{y}'$ and $\bar{w}$ replacing $y'$ and $w$, respectively. Using \eqref{wvcross kernel estimate}, note that 
\begin{equation}
    \label{eq. ib1 kernel.mag}\left|\left(1-\frac{\hat{v}_\pm\cdot \omega}{1+\hat{v}_\pm\cdot \omega}\right)(\omega\times (\hat{v}_\pm))_i\right|=\left|\frac{(\omega\times (\hat{v}_\pm))_i}{1+\hat{v}_\pm\cdot \omega}\right|\le \frac{\sqrt{m_\pm^2+|v|^2}}{m_\pm}.
\end{equation} Therefore, we have\begin{equation}
    \label{Bb1 bound by ke}|(\Blo)^{(1)}_{\pm,ib1}(t,x)|\le \frac{1}{2\pi m_\pm }\int_{\partial B(x;t)\cap \{y'_3>0\}} \frac{dS_{y'}}{t}\energy(0,y'),
\end{equation} where the initial energy density $\energy(0,y')$ is defined in \eqref{energy density}. Note that this esimate \eqref{Bb1 bound by ke} is the same as the one for $(\Elo)^{(1)}_{\pm,ib1}$ term \eqref{Eb1 bound by ke}, since on $\partial B(x;t)$ we have $|x-y'|=t.$ Therefore, we conclude that
 \begin{equation}\label{Bb1 final upper bound}
        \sup_{t,x}|\Blo_{b1}(t,x)|\le \frac{2}{ m_\pm} \|  \energy^{\textup{in}}\|_{L^\infty_{x_1,x_2}(\mathbb{T}^2)L^1_{x_3}((0,\infty))},
    \end{equation}where $\Blo_{b1}\eqdef (\Blo)^{(1)}_{\pm,ib1}+(\Blo)^{(2)}_{\pm,ib1},$
as the upper-bound estimate for $(\Blo)^{(2)}_{\pm,ib1}$ is also similar (up to some sign changes).

\subsubsection{Estimates for the Transverse and Nonlinear Source Components of the Tangential Electric Fields}
\label{sec.field est}
We now use the decaying bound \eqref{final estimate for flo} for $\flo$ for a sufficiently large choice of $\beta$ and get the upper-bound estimates for $(\Elo)^{(1)}_{\pm,iT}$ and $(\Elo)^{(1)}_{\pm,iS}$ for $i=1,2$. Note that the other parts $(\Elo)^{(2)}_{\pm,iT}$ and $(\Elo)^{(2)}_{\pm,iS}$ are similar. Denote that $\E_T\eqdef \E^{(1)}_{+,T}+\E^{(1)}_{-,T}+\E^{(2)}_{+,T}+\E^{(2)}_{-,T}$ where 
$\E^{(j)}_{\pm,T}\eqdef (\mathbf{E}^{(j)}_{\pm,1T},\mathbf{E}^{(j)}_{\pm,2T},\mathbf{E}^{(j)}_{\pm,3T})^\top$. Similarly, denote $\E_S\eqdef \E^{(1)}_{+,S}+\E^{(1)}_{-,S}+\E^{(2)}_{+,S}+\E^{(2)}_{-,S}$ where $\E^{(j)}_{\pm,S}\eqdef (\mathbf{E}^{(j)}_{\pm,1S},\mathbf{E}^{(j)}_{\pm,2S},\mathbf{E}^{(j)}_{\pm,3S})^\top$ for $j=1,2$. See \eqref{Ei5.T}-\eqref{Ei6.T} for the further definition on each term.

\subsubsection{Estimates for $\E_T$ Terms}
To begin with, we observe that $\E_T$ terms, for example $\E^{(1)}_{\pm,T}$, are written as 
\begin{multline}\label{ET1 representation}
    \E^{(1)}_{\pm,T}(t,x)=\int_{B^+(x;t)} \frac{dy'}{|y'-x|^2}\int_\rth dv\ \frac{(|(\hat{v}_\pm)|^2-1)(\hat{v}_\pm+\omega)}{(1+\hat{v}_\pm\cdot \omega)^2} F_\pm(t-|x-y'|,y',v)\\
    - 2(0,0,1)^{\top}\int_{B(x;t)\cap \{y_3=0\}} \int_\rth \frac{F_\pm(t-|y-x|,y_\parallel,0,v)}{|y-x|}dvdy_\parallel.
\end{multline} Thus, for the estimates of the first integral component of $\E^{(1)}_{\pm,T}$ with $i=1,2$ above, we use the kernel estimate \eqref{ET kernel estimate} on the kernel $\left|\frac{(|(\hat{v}_\pm)|^2-1)(\hat{v}_\pm+\omega)}{(1+\hat{v}_\pm\cdot \omega)^2}\right|$ and utilize the estimate \eqref{final estimate for flo} that we have obtained for $\flo$.  Then the kernel estimate \eqref{ET kernel estimate} and the estimate \eqref{final estimate for flo} imply for $i=1,2$,
\begin{equation}\begin{split}\label{E5T estimate}|(\Elo)^{(1)}_{\pm,iT}(t,x)|&\lesssim \int_{B^+(x;t)} \frac{dy'}{|y'-x|^2}\int_\rth dv\ \frac{\vZ}{m_\pm}\flo(t-|x-y'|,y',v)\\
&\lesssim \int_{B^+(x;t)} \frac{dy'}{|y'-x|^2}\int_\rth dv\ \frac{\vZ}{m_\pm}e^{-\frac{m_\pm g\beta}{16} (t-|x-y'|)}e^{-\frac{\beta}{4}(\vZ+m_\pm gy'_3)}\\&\times \left(\|w_{\pm,\beta} F^{\textup{in}}_\pm\|_{L^\infty_{x,v}}+\sup_{0\le\tau\le t}\|e^{\frac{m_\pm g\beta}{16}\tau}w_{\pm,\beta} (x_\parallel,0,v)G_\pm(\tau,x_\parallel,v)\|_{L^\infty_{x,v}(\gamma_-)}\right)\\
&\approx \frac{c_{\pm,\beta}}{m_\pm}M_{\textup{data}}  \int_{B^+(x;t)} \frac{dy'}{|y'-x|^2}e^{-\frac{m_\pm g\beta}{16} (t-|x-y'|)}e^{-\frac{m_\pm g\beta}{4}y'_3}\\
&\approx \frac{c_{\pm,\beta}}{m_\pm}M_{\textup{data}} \int_{B(x;t)\cap \{z_3+x_3>0\}} \frac{dz}{|z|^2} e^{-\frac{m_\pm g\beta}{16} (t-|z|)-\frac{m_\pm g\beta}{4} (z_3+x_3)}\\
&\approx \frac{c_{\pm,\beta}}{m_\pm}M_{\textup{data}} \int_0^tdr\int_{\mathbb{S}^2}d\omega\  e^{-\frac{m_\pm g\beta}{16} (t-r)}1_{\{(r\omega)_3+x_3>0\}} e^{-\frac{m_\pm g\beta}{4} ((r\omega)_3+x_3)}\\
&\approx 2\pi \frac{c_{\pm,\beta}}{m_\pm}M_{\textup{data}}  \int_0^tdr \int_0^\pi d\phi \  \sin\phi \ e^{-\frac{m_\pm g\beta}{16} (t-r)-\frac{m_\pm g\beta}{4} x_3}1_{\{r\cos\phi+x_3>0\}} e^{-\frac{m_\pm g\beta}{4} r\cos\phi}\\&
\approx 2\pi \frac{c_{\pm,\beta}}{m_\pm}M_{\textup{data}} \int_0^tdr \int_{-1}^1 d(\cos\phi) \  e^{-\frac{m_\pm g\beta}{16} (t-r+4m_\pm x_3)}1_{\{r\cos\phi+x_3>0\}} e^{-\frac{m_\pm g\beta}{4} r\cos\phi}\\
&\approx 2\pi \frac{c_{\pm,\beta}}{m_\pm}M_{\textup{data}}  \int_0^tdr \frac{4e^{-\frac{m_\pm g\beta}{16} (t-r+4m_\pm x_3)}}{m_\pm g\beta r} \int_{\max\{-\frac{m_\pm g\beta x_3}{4},-\frac{m_\pm g\beta r}{4}\}}^{\frac{m_\pm g\beta}{4}r} dk \   e^{-k}\\
&\approx8\pi \frac{c_{\pm,\beta}}{m_\pm}M_{\textup{data}} \int_0^tdr \frac{e^{-\frac{m_\pm g\beta}{16}(t-r+4m_\pm x_3)}}{m_\pm g\beta r}(e^{\min\{\frac{m_\pm g\beta x_3}{4},\frac{m_\pm g\beta r}{4}\}}-e^{-\frac{m_\pm g\beta r}{4}}),
\end{split}\end{equation}where we define $M_{\textup{data}}$ and $c_{\pm,\beta}$ as
$$M_{\textup{data}}\eqdef \left(\|w_{\pm,\beta} F^{\textup{in}}_\pm\|_{L^\infty_{x,v}}+\sup_{0\le\tau\le t}\|e^{\frac{m_\pm g\beta}{16}\tau}w_{\pm,\beta} (x_\parallel,0,v)G_\pm(\tau,x_\parallel,v)\|_{L^\infty_{x,v}(\gamma_-)}\right) ,$$ and
$$c_{\pm,\beta} \eqdef \int_\rth dv \ \vZ e^{-\frac{\beta}{4}\vZ}. $$ 
If $t<\frac{16}{m_\pm g\beta},$  then we directly follow from the fifth line in \eqref{E5T estimate} that for $i=1,2$,
\begin{multline}\label{E1T less than 1}
    |(\Elo)^{(1)}_{\pm,iT}(t,x)|
    \lesssim \frac{c_{\pm,\beta}}{m_\pm}M_{\textup{data}} \int_0^tdr\int_{\mathbb{S}^2}d\omega\  e^{-\frac{m_\pm g\beta}{16} (t-r)}1_{\{(r\omega)_3+x_3>0\}} e^{-\frac{m_\pm g\beta}{4} ((r\omega)_3+x_3)}\\
    \lesssim \frac{c_{\pm,\beta}}{m_\pm}M_{\textup{data}} \int_0^tdr\  e^{-\frac{m_\pm g\beta}{16} (t-r)}
    \lesssim \frac{c_{\pm,\beta}}{m_\pm^2 g\beta}M_{\textup{data}} (1-e^{-\frac{m_\pm g\beta t}{16}})\lesssim \frac{c_{\pm,\beta}}{m_\pm^2 g\beta}M_{\textup{data}}.
\end{multline} If $t\ge \frac{16}{m_\pm g\beta},$ then we further split $\int_0^t$ into $\int_0^{\min\{x_3,t\}}+\int_{\min\{x_3,t\}}^t,$ and observe that the last integral in the last line of \eqref{E5T estimate} is bounded from above as
\begin{align*}
   & \int_0^tdr \frac{e^{-\frac{m_\pm g\beta}{16}(t-r+4m_\pm x_3)}}{m_\pm g\beta r}(e^{\min\{\frac{m_\pm g\beta x_3}{4},\frac{m_\pm g\beta r}{4}\}}-e^{-\frac{m_\pm g\beta r}{4}})\\
   & \le \int_0^{\min\{x_3,t\}}dr \frac{e^{-\frac{m_\pm g\beta}{16}(t-r+4m_\pm x_3)}}{m_\pm g\beta r}(e^{\frac{m_\pm g\beta r}{4}}-e^{-\frac{m_\pm g\beta r}{4}})\\
  &  +\int_{\min\{x_3,t\}}^tdr \frac{e^{-\frac{m_\pm g\beta}{16}(t-r+4m_\pm x_3)}}{m_\pm g\beta r}(e^{\frac{m_\pm g\beta x_3}{4}}-e^{-\frac{m_\pm g\beta r}{4}})\\
   &=  \frac{1}{2} \int_0^{\min\{x_3,t\}}dr \ e^{-\frac{m_\pm g\beta}{16}(t-r+4m_\pm x_3)}\sinh\left(\frac{m_\pm g\beta r}{4}\right)+\int_{\min\{x_3,t\}}^tdr \frac{e^{-\frac{m_\pm g\beta}{16}(t-r)}}{m_\pm g\beta r}.
\end{align*}
Using the following formula that 
$$ \int_0^a e^x\sinh(4cx)dx= \frac{1}{2} \left( \frac{e^{(4c+1)a} - 1}{4c+1} - \frac{e^{(1-4c)a} - 1}{1-4c} \right),$$ for $a>0$ with $c=1$ and $x=\frac{m_\pm g\beta r}{16}$ in our case, we further obtain that
\begin{align*}
 &  \frac{1}{2} \int_0^{\min\{x_3,t\}}dr \ e^{-\frac{m_\pm g\beta}{16}(t-r+4m_\pm x_3)}\sinh\left(\frac{m_\pm g\beta r}{4}\right)+\int_{\min\{x_3,t\}}^tdr \frac{e^{-\frac{m_\pm g\beta}{16}(t-r)}}{m_\pm g\beta r}\\
      & \lesssim  \frac{1}{m_\pm g\beta}e^{-\frac{m_\pm g\beta}{16}(t+4m_\pm x_3)} \left( \frac{e^{5\frac{m_\pm g\beta\min\{x_3,t\}}{16}} - 1}{5} +\frac{e^{-3\frac{m_\pm g\beta\min\{x_3,t\}}{16}} - 1}{3} \right)+\int_{\min\{x_3,t\}}^tdr \frac{e^{-\frac{m_\pm g\beta}{16}(t-r)}}{m_\pm g\beta r}\\
      &  \lesssim  \frac{1}{m_\pm g\beta }+\int_{\min\{x_3,t\}}^tdr \frac{e^{-\frac{m_\pm g\beta}{16}(t-r)}}{m_\pm g\beta r}\\
      &\lesssim \frac{1}{m_\pm g\beta }+e^{-\frac{m_\pm g\beta t}{16}}\frac{1}{m_\pm g\beta}\left(\textup{Ei}\left(\frac{m_\pm g\beta t}{16}\right)-\textup{Ei}\left(\frac{m_\pm g\beta}{16}\min\{x_3,t\}\right)\right),
\end{align*}where Ei$(x)$ is the exponential integral function which is defined as 
$$\textup{Ei}(x)=-\int_{-x}^\infty \frac{e^{-t}}{t}dt.$$Using the property that Ei$(x)\le 2\frac{e^x}{x}$ for $x\ge 1,$ we conclude that for $t\ge \frac{16}{m_\pm g\beta}$
\begin{multline*}
    \int_0^tdr \frac{e^{-\frac{m_\pm g\beta}{16}(t-r+4m_\pm x_3)}}{g\beta r}(e^{\min\{\frac{m_\pm g\beta x_3}{4},\frac{m_\pm g\beta r}{4}\}}-e^{-\frac{m_\pm g\beta r}{4}})
      \lesssim   \frac{1}{m_\pm g\beta }+\frac{1}{m_\pm g\beta}\left(\frac{1}{m_\pm g\beta t}\right)\\
      \lesssim \frac{1}{m_\pm g\beta}\left(1+\frac{1}{m_\pm g\beta t}\right)\lesssim \frac{1}{m_\pm g\beta} .
\end{multline*}Plugging this bound for $t\ge \frac{16}{m_\pm g\beta}$ into \eqref{E5T estimate} and collecting \eqref{E1T less than 1} for the other case that $t<\frac{16}{m_\pm g\beta},$ we conclude that for $i=1,2$,
\begin{equation}
    \label{E1T final}|(\Elo)^{(1)}_{\pm,iT}(t,x)|\lesssim \frac{1}{m_\pm^2 g\beta^5}\left(\|w_{\pm,\beta} F^{\textup{in}}_\pm\|_{L^\infty_{x,v}}+\sup_{0\le\tau\le t}\|e^{\frac{m_\pm g\beta}{16}\tau}w_{\pm,\beta} (x_\parallel,0,v)G_\pm(\tau,x_\parallel,v)\|_{L^\infty_{x,v}(\gamma_-)}\right),
\end{equation}since we have the estimate \eqref{additional beta decay} for the coefficient $c_{\pm,\beta}$. Thus, if $\beta>1$ sufficiently large such that $\min\{m_-^3,m_+^3\} g^2\beta^5\gg 1,$ then we obtain $$|(\Elo)^{(1)}_{\pm,iT}(t,x)|\ll \min\{m_-,m_+\} g.$$

Similarly, we obtain the same upper bound for $(\Elo)^{(2)}_{\pm,iT}$ for $i=1,2$.

\subsubsection{Estimates for $\E_S$ Terms}
On the other hand, regarding the nonlinear $S$ terms, $(\Elo)^{(1)}_{\pm,iS}$ and $(\Elo)^{(2)}_{\pm,iS}$ with $i=1,2$ for example, we recall \eqref{Ei5.T}--\eqref{Ei6.T} again. Here we need a kernel estimate on  $$a^{\E}_{\pm,i}(v,\omega)=\frac{(\partial_{v_i}v-(\hat{v}_\pm)_i(\hat{v}_\pm))}{(\vZ) (1+\hat{v}_\pm\cdot \omega)}-\frac{(\omega_i+(\hat{v}_\pm)_i)(\omega-(\omega\cdot (\hat{v}_\pm))(\hat{v}_\pm))}{(\vZ) (1+\hat{v}_\pm\cdot \omega)^2}=:a^{(1)}_{\pm,i}+a^{(2)}_{\pm,i},$$and we use the kernel estimates \eqref{a2}--\eqref{a1} for $|a_i^E|.$ Then again using the estimate \eqref{final estimate for flo} on $\flo$, we finally obtain for $i=1,2,$
\begin{equation}\begin{split}|(\Elo)^{(1)}_{\pm,iS}(t,x)|
&\lesssim (m_\pm g+\|(\Elbf,\Blbf)\|_{L^\infty})\int_{B^+(x;t)} \frac{dy'}{|y'-x|} \int_\rth dv\ \frac{\vZ}{m_\pm^2}e^{-\frac{m_\pm g\beta}{16} (t-|x-y'|)}e^{-\frac{\beta}{4}(\vZ+m_\pm gy'_3)}\\&\times\left(\|w_{\pm,\beta} F^{\textup{in}}_\pm\|_{L^\infty_{x,v}}+\sup_{0\le\tau\le t}\|e^{\frac{m_\pm g\beta}{16}\tau}w_{\pm,\beta} (x_\parallel,0,v)G_\pm(\tau,x_\parallel,v)\|_{L^\infty_{x,v}(\gamma_-)}\right)\\
&\approx c_{\pm,\beta} \frac{1}{m_\pm} gM_{\textup{data}}\int_{B^+(x;t)} \frac{dy'}{|y'-x|} e^{-\frac{m_\pm g\beta}{16}(t-|x-y'|)-\frac{m_\pm g\beta}{4}y'_3}\\
&\approx c_{\pm,\beta}\frac{1}{m_\pm} gM_{\textup{data}}\int_{B(x;t)\cap \{z_3+x_3>0\}} \frac{dz}{|z|} e^{-\frac{m_\pm g\beta}{16}(t-|z|)-\frac{m_\pm g\beta}{4}(z_3+x_3)}\\
&\approx c_{\pm,\beta}\frac{1}{m_\pm} gM_{\textup{data}}\int_0^tdr\int_{\mathbb{S}^2}d\omega\  re^{-\frac{m_\pm g\beta}{16}(t-r)}1_{\{(r\omega)_3+x_3>0\}} e^{-\frac{m_\pm g\beta}{4}((r\omega)_3+x_3)}\\
&\approx \frac{2\pi c_{\pm,\beta} }{m_\pm} gM_{\textup{data}}\int_0^tdr \int_0^\pi d\phi \  \sin\phi \ re^{-\frac{m_\pm g\beta}{16}(t-r+4m_\pm x_3)}1_{\{r\cos\phi+x_3>0\}} e^{-\frac{m_\pm g\beta}{4}r\cos\phi}\\
&\approx  \frac{c_{\pm,\beta}}{m_\pm} gM_{\textup{data}}\int_0^tdr \int_{-1}^1 d(\cos\phi) \  re^{-\frac{m_\pm g\beta}{16}(t-r+4m_\pm x_3)}1_{\{r\cos\phi+x_3>0\}} e^{-\frac{m_\pm g\beta}{4}r\cos\phi}\\
&\approx c_{\pm,\beta} \frac{1}{m_\pm} gM_{\textup{data}}\int_0^tdr \frac{e^{-\frac{m_\pm g\beta}{16}(t-r+4m_\pm x_3)}}{m_\pm g\beta} \int_{-\frac{m_\pm g\beta}{4}\min\{x_3,r\}}^{\frac{m_\pm g\beta}{4} r} dk \   e^{-k}\\&\approx c_{\pm,\beta} \frac{1}{m_\pm} gM_{\textup{data}}\int_0^tdr \frac{e^{-\frac{m_\pm g\beta}{16}(t-r+4m_\pm x_3)}}{m_\pm g\beta}(e^{\frac{m_\pm g\beta}{4}\min\{x_3,r\}}-e^{-\frac{m_\pm g\beta}{4}r})\\
&\lesssim \frac{1}{m_\pm^2\beta ^5}M_{\textup{data}}\int_0^tdr \ e^{-\frac{m_\pm g\beta}{16}(t-r)}
\approx \frac{1}{m_\pm^3g\beta^6}M_{\textup{data}}(1-e^{-\frac{m_\pm g\beta}{16}t}),
\label{E5S estimate}
\end{split}\end{equation}
where we define again $$M_{\textup{data}}\eqdef \left(\|w_{\pm,\beta} F^{\textup{in}}_\pm\|_{L^\infty_{x,v}}+\sup_{0\le\tau\le t}\|e^{\frac{m_\pm g\beta}{16}\tau}w_{\pm,\beta} (x_\parallel,0,v)G_\pm(\tau,x_\parallel,v)\|_{L^\infty_{x,v}(\gamma_-)}\right),$$ and  used the a priori ansatz (cf. \eqref{ansatz for EBl})
$$\sup_{t\ge 0}\|(\Elbf,\Blbf)\|_{L^\infty}\le \min\{m_+,m_-\}\frac{g}{8}.$$ Thus, if $\beta>1$ is sufficiently large such that $\min\{m_-^4,m_+^4\}g^2\beta^6\gg 1,$ then we have $$|(\Elo)^{(1)}_{\pm,iS}(t,x)|\ll \min\{m_-,m_+\}g.$$

Similarly, we obtain the same upper bounds for $(\Elo)^{(2)}_{\pm,iS}$ terms for $i=1,2,$ as those of \eqref{E5S estimate}.
In conclusion, we have for $i=1,2,$
\begin{multline}\label{EiTEiS final i12}|(\Elo)^{(1)}_{\pm,iT}(t,x)|+|(\Elo)^{(2)}_{\pm,iT}(t,x)|+|(\Elo)^{(1)}_{\pm,iS}(t,x)|+|(\Elo)^{(2)}_{\pm,iS}(t,x)|\ll \min\{m_-,m_+\}g.\end{multline} 
\subsubsection{Estimates for the Transverse \Black{and Nonlinear Source} Components of the Magnetic Field }Note that the representations for $(\Blo)^{(1)}_{\pm,iT}$ and $(\Blo)^{(2)}_{\pm,iT}$  in \eqref{Bpar_half_final}\Black{, as well as those for $(\Blo)^{(1)}_{\pm,iS}$ and $(\Blo)^{(2)}_{\pm,iS}$ in \eqref{BparS_half_final} (with the Cauchy data and the distributions extended periodically in the $x_\parallel$ directions),} differ from the electric field representation\Black{s} $(\Elo)^{(1)}_{\pm,iT}$ and $(\Elo)^{(2)}_{\pm,iT}$ \Black{and $(\Elo)^{(1)}_{\pm,iS}$ and $(\Elo)^{(2)}_{\pm,iS}$} in \eqref{Ei5.T}--\eqref{Ei6.T}, respectively, by the sign and the kernel\Black{s}: $\frac{(1-|(\hat{v}_\pm)|^2)(\omega\times (\hat{v}_\pm))_i}{(1+\hat{v}_\pm\cdot \omega)^2}$ \Black{for the $T$-terms, and $a^{\B}_{\pm,i}(v,\omega)=\nabla_v\big(\frac{(\omega\times \hat{v}_\pm)_i}{1+\hat{v}_\pm\cdot \omega}\big)$ of \eqref{aBi} for the $S$-terms}. Therefore, in the rest of this section, we make upper-bound estimates for \Black{these kernels} and use them to obtain the estimates for $\Blo_{\pm,iT}$ \Black{and $\Blo_{\pm,iS}$}.
Regarding the $T$-terms, we use the kernel estimate \eqref{B35 kernel final} which provides
\begin{equation}\notag
    \left|\frac{(1-|(\hat{v}_\pm)|^2)(\omega\times (\hat{v}_\pm))_i}{(1+\hat{v}_\pm\cdot \omega)^2}\right|\le 2\frac{\sqrt{m_\pm^2+|v|^2}}{m_\pm}.
\end{equation}
This upper-bound for the kernel of $\mathbf{B}^{(1)}_{\pm,iT}$ and $\mathbf{B}^{(2)}_{\pm,iT}$ is the same as those for $\mathbf{E}^{(1)}_{\pm,iT}$ and $\mathbf{E}^{(2)}_{\pm,iT}$ for $i=1,2$ in \eqref{ET kernel estimate}, and hence the upper-bounds for $\mathbf{B}^{(1)}_{\pm,iT}$ and $\mathbf{B}^{(2)}_{\pm,iT}$ are the same as those for $\mathbf{E}^{(1)}_{\pm,iT}$ and $\mathbf{E}^{(2)}_{\pm,iT}$ for $i=1,2$ of \eqref{EiTEiS final i12}.
{\color{black}
Regarding the nonlinear $S$-terms, we recall that the magnetic $S$-kernel obeys the upper bound
\begin{equation}\notag
    \big|a^{\B}_{\pm,i}(v,\omega)\big|\lesssim \frac{\vZ}{m_\pm^2},\qquad i=1,2,3,
\end{equation}
by \eqref{aB.final}, which is the same upper bound as those for the electric $S$-kernels $a^{\E}_{\pm,i}$ in \eqref{a2}--\eqref{a1}; the bound is uniform in $\omega\in\mathbb{S}^2$ and hence holds with $\bar{\omega}$ as well. Since the integrands of $(\Blo)^{(1)}_{\pm,iS}$ and $(\Blo)^{(2)}_{\pm,iS}$ differ from those of $(\Elo)^{(1)}_{\pm,iS}$ and $(\Elo)^{(2)}_{\pm,iS}$ only through this kernel, the estimate \eqref{E5S estimate} holds verbatim for $(\Blo)^{(1)}_{\pm,iS}$ and $(\Blo)^{(2)}_{\pm,iS}$:
$$
|(\Blo)^{(1)}_{\pm,iS}(t,x)|+|(\Blo)^{(2)}_{\pm,iS}(t,x)|\lesssim \frac{1}{m_\pm^3g\beta^6}M_{\textup{data}}(1-e^{-\frac{m_\pm g\beta}{16}t}),\qquad i=1,2,3.
$$
Thus, if $\beta>1$ is sufficiently large such that $\min\{m_-^4,m_+^4\}g^2\beta^6\gg 1$, as for the electric $S$-terms, then we have, for $i=1,2,3$,
$$|(\Blo)^{(1)}_{\pm,iS}(t,x)|+|(\Blo)^{(2)}_{\pm,iS}(t,x)|\ll \min\{m_-,m_+\}g.$$
Hence the upper-bounds for $\mathbf{B}^{(1)}_{\pm,iS}$ and $\mathbf{B}^{(2)}_{\pm,iS}$ are the same as those for $\mathbf{E}^{(1)}_{\pm,iS}$ and $\mathbf{E}^{(2)}_{\pm,iS}$ of \eqref{EiTEiS final i12}.
}

\subsubsection{Estimates for Additional Field-Components in the Representations }
\label{sec.E3.end}
In this subsection, we make upper-bound estimates for additional field-components in the representations of $\Elo_3$, \textnormal{$\Blo_1$}, and \textnormal{$\Blo_2$} which arise from the Neumann boundary conditions \eqref{perfect.conductor.neumann}. Observing the representations for these fields in \eqref{E3.T}, we notice that these fields contain additional terms \eqref{additional term E3} compared to those for  $\Elo_1$ and $\Elo_2$ in \eqref{Ei5.T} and \eqref{Ei6.T} which solve the wave equations under Dirichlet boundary conditions \eqref{perfect cond. boundary}. Similarly, we have additional terms \eqref{additional term B1}  appearing in the representations of \textnormal{$\Blo_1$}, and \textnormal{$\Blo_2$}. These common terms share the same upper-bound estimates, since the kernels differ only by some signs. It suffices to make additional estimates on those extra terms \eqref{additional term E3}, since $|\hat{v}_i|\le 1$. 

Note that the integral in \eqref{additional term E3} can be estimated via the upper-bound estimate for $\flo$ \eqref{final estimate for flo}. Using \eqref{final estimate for flo}, we obtain
\begin{equation}\begin{split}\label{additional estimates E3}
   & 2\int_{B(x;t)\cap \{y_3=0\}} \int_\rth \frac{\flo(t-|y-x|,y_\parallel,0,v)}{|y-x|}dvdS_y\\&\le 2\int_{B(x;t)\cap \{y'_3=0\}} \frac{dy'_\parallel}{|y'-x|}\int_{v_3>0} dv\  |G_\pm(t-|x-y'|,y'_\parallel,v)|\\
   & + 2\int_{B(x;t)\cap \{y'_3=0\}} \frac{dy'_\parallel}{|y'-x|}\int_{v_3\le 0} dv\  |\flo(t-|x-y'|,y'_\parallel,0,v)|\\
  &  \le 2\int_{B(x;t)\cap \{y'_3=0\}} \frac{dy'_\parallel}{|y'-x|}\int_{v_3>0} dv\  w_{\pm,\beta}^{-1}(y'_\parallel,0,v)e^{-\frac{m_\pm g\beta}{16}(t-|x-y'|)}\\&\times \sup_{0\le\tau\le t}\|e^{\frac{m_\pm g\beta}{16}\tau}w_{\pm,\beta} (x_\parallel,0,v)G_\pm(\tau,x_\parallel,v)\|_{L^\infty_{x,v}(\gamma_-)}\\
   & + 2\int_{B(x;t)\cap \{y'_3=0\}} \frac{dy'_\parallel}{|y'-x|}\int_{v_3\le 0} dv\   e^{-\frac{\beta}{4} \vZ}e^{-\frac{m_\pm g\beta}{16}(t-|x-y'|)}  \\&\times \left(\|w_{\pm,\beta} F^{\textup{in}}_\pm\|_{L^\infty_{x,v}}+\sup_{0\le\tau\le t}\|e^{\frac{m_\pm g\beta}{16}\tau}w_{\pm,\beta} (x_\parallel,0,v)G_\pm(\tau,x_\parallel,v)\|_{L^\infty_{x,v}(\gamma_-)}\right)\\
    & \le 2e^{-\frac{m_\pm g\beta}{16}t}\int_{B(x;t)\cap \{y'_3=0\}} \frac{dy'_\parallel}{|y'-x|}e^{\frac{m_\pm g\beta}{16}|x-y'|}\int_{v_3>0} dv\   e^{-\frac{\beta}{4}\sqrt{m_\pm^2+|v|^2}}\\&\times \sup_{0\le\tau\le t}\|e^{\frac{m_\pm g\beta}{16}\tau}w_{\pm,\beta} (x_\parallel,0,v)G_\pm(\tau,x_\parallel,v)\|_{L^\infty_{x,v}(\gamma_-)}\\
   & + 2 e^{-\frac{m_\pm g\beta}{16}t}  \left(\|w_{\pm,\beta} F^{\textup{in}}_\pm\|_{L^\infty_{x,v}}+\sup_{0\le\tau\le t}\|e^{\frac{m_\pm g\beta}{16}\tau}w_{\pm,\beta} (x_\parallel,0,v)G_\pm(\tau,x_\parallel,v)\|_{L^\infty_{x,v}(\gamma_-)}\right)\\&\times \int_{B(x;t)\cap \{y'_3=0\}} \frac{dy'_\parallel}{|y'-x|}e^{\frac{m_\pm g\beta}{16}|x-y'|}\int_{v_3\le 0} dv\   e^{-\frac{\beta}{4}\sqrt{m_\pm^2+|v|^2}}\\
   & \le  2c'_{\pm,\beta} e^{-\frac{m_\pm g\beta}{16}t}  \left(\|w_{\pm,\beta} F^{\textup{in}}_\pm\|_{L^\infty_{x,v}}+\sup_{0\le\tau\le t}\|e^{\frac{m_\pm g\beta}{16}\tau}w_{\pm,\beta} (x_\parallel,0,v)G_\pm(\tau,x_\parallel,v)\|_{L^\infty_{x,v}(\gamma_-)}\right)\\&\times \int_{B(x;t)\cap \{y'_3=0\}} \frac{dy'_\parallel}{|y'-x|}e^{\frac{m_\pm g\beta}{16}|x-y'|}.
\end{split}\end{equation}Here, $c'_{\pm,\beta} \eqdef \int_{\rth} dv\   e^{-\frac{\beta}{4} \vZ} $ and by rescaling we have $c'_{\pm,\beta} \approx \frac{1}{\beta^3}$ by \eqref{additional beta decay.st}. The last line in \eqref{additional estimates E3} is the same as the second line of \eqref{Eib2 estimate middle} except for the additional coefficient $\frac{1}{m_\pm}$ in \eqref{Eib2 estimate middle}, and hence by \eqref{Final estimate for b2 terms}, we obtain that
 the additional term \eqref{additional term E3} 
 can be bounded from above as
 \begin{multline}\label{additional term final estimate E3}
      2\int_{B(x;t)\cap \{y_3=0\}} \int_\rth \frac{\flo(t-|y-x|,y_\parallel,0,v)}{|y-x|}dvdS_y \\\lesssim  \frac{1}{m_\pm g\beta^4}\left(\|w_{\pm,\beta} F^{\textup{in}}_\pm\|_{L^\infty_{x,v}}+\sup_{0\le\tau\le t}\|e^{\frac{m_\pm g\beta}{16}\tau}w_{\pm,\beta} (x_\parallel,0,v)G_\pm(\tau,x_\parallel,v)\|_{L^\infty_{x,v}(\gamma_-)}\right),
 \end{multline} by \eqref{additional beta decay.st}.
Thus, if $\beta>1$ is chosen sufficiently large such that $\min\{m_-^2,m_+^2\}g^2\beta^4\gg 1,$ then we observe that these additional terms are bounded from above by $c\min\{m_-,m_+\}g,$ for a sufficiently small constant $c>0.$ This completes the upper-bound estimates for the field-components $\Elo_3$, \textnormal{$\Blo_1$}, and \textnormal{$\Blo_2$} which solve the wave equations under the Neumann boundary conditions \eqref{perfect.conductor.neumann}.

\subsubsection{Final Estimates for the Electromagnetic Fields} Altoghther, we collect all the estimates made for $\Elobf$ and $\Blobf$ in Section~\ref{sec.Ei.start}--Section~\ref{sec.E3.end} above and finally establish the following lemma:
Combining the previous estimates, we obtain the following lemma on the linear-in-time decay upper bound for $\Eplo$ and $\Bplo$:
\begin{lemma}\label{lem.dyna.EB.vacuum}Fix $l\in\mathbb{N}$ and suppose that $\beta>1 $ is sufficiently large such that $\min\{m_-^2,m_+^2\}g^2\beta\gg1$ and $\min\{m_-^2,m_+^2\}g \beta^3\gg 1$. Suppose \eqref{Ansatz for fl}-\eqref{ansatz for EBl}  hold for $(F^l_\pm,\Elbf,\Blbf).$  Then $(\Elobf,\Blobf)$ satisfies 
\begin{equation}\label{final esimate for EloBlo}
    \sup_{t \geq 0} \  \|(\Elobf, \Blobf)\|_{L^\infty} \leq \min\{m_+, m_-\} \frac{g}{8}.
\end{equation}
\end{lemma}
This bound guarantees the validity of \eqref{ansatz for EBl} at the $(l+1)$-th iteration level, provided that the parameter $\beta > 1$ is chosen sufficiently large. Consequently, the estimates \eqref{Ansatz for fl}-\eqref{ansatz for EBl} are verified uniformly for all $l \in \mathbb{N}$, and thus remain valid in the limit as $l \to \infty$.

\begin{remark}
    One can think of a special property of $1$-dimensional wave equation on a whole line (or a half-line). Our situation in $\mathbb{T}^2\times \mathbb{R}_+$ also corresponds to this case in the sense of dispersion through the whole line. One can think of the possibility that there is a non-zero lower bound for the fields $\E$ or $\B$ in the $1$-dimensional case.  Then one can also consider that $\E$ or $\B$ converge to some non-zero state and try to prove the stability near a non-zero state.
\end{remark}

\section*{Acknowledgment}
J.W.J. is supported by the National Research Foundation of Korea (NRF) under grants RS-2023-00210484, RS-2023-00219980, and 2021R1A6A1A10042944. J.W.J. gratefully acknowledges the hospitality of the Department of Mathematics at the University of Wisconsin-Madison during his visits. C.K. is partially supported by NSF grants CAREER-2047681. This material is partly based upon work supported by the National Science Foundation under Grant No. DMS-2424139, while one of the  authors (C.K.) was in residence at the Simons Laufer Mathematical Sciences Institute in Berkeley, California, during the Fall 2025 semester.

\appendix

\section{Electric Field Representation}
For the representation of the self-consistent electric field $\E$ in the half space $\mathbb{R}^3_+$, we follow the half-space Glassey-Strauss formula derived in \cite[eq. (35), (37)-(41), (47)-(50)]{MR4414612} as follows. We write $\E = \E_{\textup{hom}}+\E_{\textup{par}}$ where the tangential components $\E_{\textup{par},\parallel}$ of the particular solution ($i=1,2$) are given by 
\begin{equation}\label{Ei5}
\begin{split}
   \mathbf{E}_{\textup{par},i}(t,x)= &\sum_{\iota=\pm} 
    (-\iota)\int_{B^+(x;t)} dy \int_\rth dv\ a^{\mathbf{E}}_{\iota,i}(v,\omega)\cdot (\iota\E+\iota(\hat{v}_\iota)\times \B-m_\iota g\hat{e}_3)\frac{F_\iota(t-|x-y|,y,v)}{|x-y|}\\
   &+\sum_{\iota=\pm}\iota\int_{\partial B(x;t)\cap \{y_3>0\}} \frac{dS_{y}}{|y-x|}\int_\rth dv\ \left(\delta_{ij}-\frac{(\omega_i+(\hat{v}_\iota)_i)(\hat{v}_\iota)_j}{1+\hat{v}_\iota\cdot \omega}\right)\omega^jF_\iota(0,y,v)\\
    &+\sum_{\iota=\pm}(-\iota)\int_{B(x;t)\cap \{y_3=0\}} \frac{dy_\parallel}{|y-x|}\int_\rth dv  \left(\delta_{i3}-\frac{(\omega_i+(\hat{v}_\iota)_i)(\hat{v}_\iota)_3}{1+\hat{v}_\iota\cdot \omega}\right)F_\iota(t-|x-y|,y_\parallel,0,v)\\
   &+\sum_{\iota=\pm}(-\iota)\int_{B^+(x;t)} \frac{dy}{|y-x|^2}\int_\rth dv\ \frac{(|(\hat{v}_\iota)|^2-1)((\hat{v}_\iota)_i+\omega_i)}{(1+\hat{v}_\iota\cdot \omega)^2} F_\iota(t-|x-y|,y,v)\\
 &+\sum_{\iota=\pm}\iota\int_{B^-(x;t)} dy \int_\rth dv\ a^{\E}_{\iota,i}(v,\bar{\omega})\cdot (\iota\E+\iota(\hat{v}_\iota)\times \B-m_\iota g\hat{e}_3)\frac{F_\iota(t-|x-y|,\bar{y},v)}{|x-y|}\\
 &+\sum_{\iota=\pm}(-\iota)\int_{\partial B(x;t)\cap \{y_3<0\}} \frac{dS_{y}}{|y-x|}\int_\rth dv\ (\delta_{ij}-\frac{(\bar{\omega}_i+(\hat{v}_\iota)_i)(\hat{v}_\iota)_j}{1+(\hat{v}_\iota)\cdot \bar{\omega}})\bar{\omega}^jF_\iota(0,\bar{y},v)\\
  &+\sum_{\iota=\pm}\iota\int_{B(x;t)\cap \{y_3=0\}} \frac{dy_\parallel}{|y-x|}\int_\rth dv\  (\delta_{i3}-\frac{(\bar{\omega}_i+(\hat{v}_\iota)_i)(\hat{v}_\iota)_3}{1+(\hat{v}_\iota)\cdot \bar{\omega}})F_\iota(t-|x-y|,y_\parallel,0,v)\\
   &+\sum_{\iota=\pm}\iota\int_{B^-(x;t)} \frac{dy}{|y-x|^2}\int_\rth dv\ \frac{(|(\hat{v}_\iota)|^2-1)((\hat{v}_\iota)_i+\bar{\omega}_i)}{(1+(\hat{v}_\iota)\cdot \bar{\omega})^2} F_\iota(t-|x-y|,\bar{y},v)\\ \eqdef &\sum_{\iota=\pm}\iota(\mathbf{E}^{(1)}_{\iota,iS}+\mathbf{E}^{(1)}_{\iota,ib1}+\mathbf{E}^{(1)}_{\iota,ib2}+\mathbf{E}^{(1)}_{\iota,iT}-\mathbf{E}^{(2)}_{\iota,iS}-\mathbf{E}^{(2)}_{\iota,ib1}-\mathbf{E}^{(2)}_{\iota,ib2}-\mathbf{E}^{(2)}_{\iota,iT}),
\end{split}\end{equation}where $\bar{z}\eqdef(z_1,z_2, -z_3)^\top$, $B^\pm(x;t)$ are the upper- and the lower open half balls, respectively, defined as $B^+(x;t)=B(x;t)\cap \{y_3>0\}$ and $B^-(x;t)=B(x;t)\cap \{y_3<0\}$,  $dy_\parallel$ is the 2-dimensional Lebesgue measure on $B(x;t)\cap \{y_3=0\}$, and
\begin{equation}
    \label{aEi}
a^{\E}_{\iota,i}(v,\omega)\eqdef\nabla_v \left(\frac{\omega_i+(\hat{v}_\iota)_i}{1+\hat{v}_\iota\cdot \omega}\right)=\frac{(\partial_{v_i}v-(\hat{v}_\iota)_i(\hat{v}_\iota))(1+\hat{v}_\iota\cdot \omega)-(\omega_i+(\hat{v}_\iota)_i)(\omega-(\omega\cdot (\hat{v}_\iota))(\hat{v}_\iota))}{(\vZio) (1+\hat{v}_\iota\cdot \omega)^2}.\end{equation}
For the tangential component $\E_{\textup{hom},\parallel}$ of the homogeneous solution, we have 
\begin{multline}\label{Eparallel homo solution}
\Eihomo(t, x) =\frac{1}{4\pi t^2} \int_{\partial B(x; t)\cap \{y_3>0\}}  \left(t \mathbf{E}^1_{0i}(y) + \mathbf{E}_{0i}(y) + \nabla \mathbf{E}_{0i}(y) \cdot (y - x)\right) dS_y\\
-\frac{1}{4\pi t^2} \int_{\partial B(x; t)\cap \{y_3<0\}}  \left(t \mathbf{E}^1_{0i}(\bar{y}) + \mathbf{E}_{0i}(\bar{y}) + \nabla \mathbf{E}_{0i}(\bar{y}) \cdot (\bar{y} - \bar{x})\right) dS_y,
\end{multline} where $\bar{z}\eqdef (z_1,z_2,-z_3)^\top$.

On the other hand, for the normal component $\Eth$, for each $(t,x)\in [0,\infty)\times \mathbb{R}^2\times (0,\infty),$ we have the Glassey-Strauss representation as
\begin{equation}\label{E3}\begin{split}
    \Eth(t,x)=&\frac{1}{4\pi t^2} \int_{\partial B(x; t)\cap \{y_3>0\}}  \left(t \mathbf{E}^1_{03}(y) + \mathbf{E}_{03}(y) + \nabla \mathbf{E}_{03}(y) \cdot (y - x)\right) dS_y\\
&+\frac{1}{4\pi t^2} \int_{\partial B(x; t)\cap \{y_3<0\}}  \left(t \mathbf{E}^1_{03}(\bar{y}) + \mathbf{E}_{03}(\bar{y}) + \nabla \mathbf{E}_{03}(\bar{y}) \cdot (\bar{y} - \bar{x})\right) dS_y\\
  &+\sum_{\iota=\pm}\iota  \int_{B^+(x;t)} dy \int_\rth dv\ a^{\E}_{\iota,3}(v,\omega)\cdot (\iota\E+\iota(\hat{v}_\iota)\times \B-m_\iota g\hat{e}_3)\frac{F_\iota(t-|x-y|,y,v)}{|x-y|}\\
    &+\sum_{\iota=\pm}(-\iota)\int_{\partial B(x;t)\cap \{y_3>0\}} \frac{dS_{y}}{|y-x|}\int_\rth dv\ \left(\delta_{3j}-\frac{(\omega_3+(\hat{v}_\iota)_3)(\hat{v}_\iota)_j}{1+\hat{v}_\iota\cdot \omega}\right)\omega^jF_\iota(0,y,v)\\
    &+\sum_{\iota=\pm}\iota\int_{B(x;t)\cap \{y_3=0\}} \frac{dy_\parallel}{|y-x|}\int_\rth dv\  \left(1-\frac{(\omega_3+(\hat{v}_\iota)_3)(\hat{v}_\iota)_3}{1+\hat{v}_\iota\cdot \omega}\right)F_\iota(t-|x-y|,y_\parallel,0,v)\\
     &+\sum_{\iota=\pm}\iota\int_{B^+(x;t)} \frac{dy}{|y-x|^2}\int_\rth dv\ \frac{(|(\hat{v}_\iota)|^2-1)((\hat{v}_\iota)_3+\omega_3)}{(1+\hat{v}_\iota\cdot \omega)^2} F_\iota(t-|x-y|,y,v)\\
    &+\sum_{\iota=\pm}\iota\int_{B^-(x;t)} dy \int_\rth dv\ a^{\E}_{\iota,3}(v,\bar{\omega})\cdot (\iota\E+\iota(\hat{v}_\iota)\times \B-m_\iota g\hat{e}_3)\frac{F_\iota(t-|x-y|,\bar{y},v)}{|x-y|}\\
     &+\sum_{\iota=\pm}(-\iota)\int_{\partial B(x;t)\cap \{y_3<0\}} \frac{dS_{y}}{|y-x|}\int_\rth dv\ (\delta_{3j}-\frac{(\bar{w}_3+(\hat{v}_\iota)_3)(\hat{v}_\iota)_j}{1+(\hat{v}_\iota)\cdot \bar{\omega}})\bar{\omega}^jF_\iota(0,\bar{y},v)\\
     &+\sum_{\iota=\pm}\iota\int_{B(x;t)\cap \{y_3=0\}} \frac{dy_\parallel}{|y-x|}\int_\rth dv\  (1-\frac{(\bar{w}_3+(\hat{v}_\iota)_3)(\hat{v}_\iota)_3}{1+(\hat{v}_\iota)\cdot \bar{\omega}})F_\iota(t-|x-y|,y_\parallel,0,v)\\
     &+\sum_{\iota=\pm}\iota\int_{B^-(x;t)} \frac{dy}{|y-x|^2}\int_\rth dv\ \frac{(|(\hat{v}_\iota)|^2-1)((\hat{v}_\iota)_3+\bar{w}_3)}{(1+(\hat{v}_\iota)\cdot \bar{\omega})^2} F_\iota(t-|x-y|,\bar{y},v)\\
     &+\sum_{\iota=\pm}(-\iota) 2\int_{B(x;t)\cap \{y_3=0\}} \int_\rth \frac{F_\iota(t-|y-x|,y_\parallel,0,v)}{|y-x|}dvdy_\parallel,
\end{split}\end{equation}where $a^{\E}_{\iota,3}$ is defined in \eqref{aEi}, $\bar{z}\eqdef(z_1,z_2, -z_3)^\top$, and 
$dy_\parallel$ is the 2-dimensional Lebesgue measure on $B(x;t)\cap \{y_3=0\}$.

\section{Green Functions and the Glassey-Strauss Representations in \texorpdfstring{$\mathbb{T}^2\times \mathbb{R}_+$}{}}\label{sec.5}
In this section, we derive the Glassey-Strauss representations of the self-consistent electric field $\E(t,x)$ via the Green functions of the wave equations corresponding to the Maxwell equations \eqref{2speciesVM}.

\subsection{Wave Structure of the Maxwell Equations}
 Using the vector calculus identity $$\nabla\times (\nabla\times \E)= \nabla(\nabla\cdot \E)-\nabla^2 \E,$$ we observe that the Maxwell equations in \eqref{2speciesVM} imply$$
      4\pi\nabla \rho - \nabla^2 \E=\nabla(\nabla\cdot \E)-\nabla^2 \E = - \frac{\partial}{\partial t}(\nabla\times \B)= - \frac{\partial^2}{\partial t^2}\E -4\pi\frac{\partial}{\partial t} j.$$
      Thus, we have
      \begin{equation}
          \label{wave eq for E}\partial_t ^2 \E-\Delta_x \E =-4\pi (\nabla_x \rho + \partial_t J).
      \end{equation}
      Similarly, we have
      \begin{equation}
          \label{wave eq for B}\partial_t ^2 \B-\Delta_x \B=4\pi  \nabla\times J.
      \end{equation}
      Here
$\rho(t,x)$ and $ j(t,x)$ are the total charge density and the total current density, respectively. 
It is known that the 
wave equation \eqref{wave eq for E} 
recovers the original four Maxwell equations in \eqref{2speciesVM} if the continuity equation \eqref{continuity eq} is also satisfied. 

In the steady case that we will discuss in Section \ref{sec.bootstrap.st} and Section \ref{sec.exist.steady}, note that the wave equation \eqref{wave eq for E} 
reduces to the following Poisson equation:
  \begin{equation}
          \label{Poisson eq for E}-\Delta_x \E =-4\pi \nabla_x \rho.
      \end{equation}
      We also write the stationary counterpart of the continuity equation \eqref{continuity eq}:
\begin{equation}
    \label{continuity eq.st}
    \nabla_x\cdot J_{\textup{st}}=0.
\end{equation}
      The Green functions for the Poisson equations with Dirichlet boundary conditions and Neumann compatibility conditions in the half-space $\mathbb{R}^3_+$ will be introduced in Section \ref{sec.exist.steady}.

      \subsection{Green Function to the Wave Equation in $\mathbb{T}^2\times \mathbb{R}_+$}
We extend $\mathbb{R}_+$ to $\mathbb{R}$ and derive the representation again by performing the time-variable reduction in the Green function. 
Note that we have the Green function of the wave equation for $(x_\parallel,x_3)\in \mathbb{T}^2\times \mathbb{R},$
\begin{equation}
    G_{t-\tau}(x_\parallel,x_3)=\frac{1}{2\pi}\sum_{n\in \mathbb{Z}^2}\delta((t-\tau)^2-x_3^2-|x_\parallel -n|^2 )1_{x_3^2+|x_\parallel-n|^2\le (t-\tau)^2}.
\end{equation}

Then, in $\mathbb{T}^2\times (0,\infty)$, we have the Dirichlet-Green function for $x_3\ge 0,$ 
\begin{multline}\label{Green in the half channel}
    \bar{G}(t,\tau,x_\parallel,y_\parallel,x_3,y_3)=G_{t-\tau}(x_\parallel-y_\parallel,x_3-y_3)-G_{t-\tau}(x_\parallel-y_\parallel,x_3+y_3)\\=\frac{1}{2\pi}\sum_{n\in \mathbb{Z}^2}\bigg\{\delta((t-\tau)^2-(x_3-y_3)^2-|x_\parallel -y_\parallel-n|^2 )1_{(x_3-y_3)^2+|x_\parallel-y_\parallel-n|^2\le (t-\tau)^2}\\-\delta((t-\tau)^2-(x_3+y_3)^2-|x_\parallel -y_\parallel-n|^2 )1_{(x_3+y_3)^2+|x_\parallel-y_\parallel-n|^2\le (t-\tau)^2}\bigg\} ,
\end{multline}where we denote the set of integers as $\mathbb{Z}$ and we define the characteristic function $1_M(x) $ as 1 if $x\in M$ and $=0$, otherwise.
 This Green function $\bar{G}$ will be used to obtain the particular solutions to the self-consistent fields $\Eone,$ and $\Etwo,$ 
 which satisfy the Dirichlet boundary conditions on $x_3=0$. For these, we will consider the odd extension beyond the boundary. 

On the other hand, the self-consistent field $\Eth,$ 
will satisfy the Neumann boundary conditions on $x_3=0$, and for these, we will consider the even extension beyond the boundary and the alternative Neumann-Green function $\tilde{G}$ which is defined as
\begin{multline}\label{Green in the half channel.Neumann}
    \tilde{G}(t,\tau,x_\parallel,y_\parallel,x_3,y_3)=G_{t-\tau}(x_\parallel-y_\parallel,x_3-y_3)+G_{t-\tau}(x_\parallel-y_\parallel,x_3+y_3)\\=\frac{1}{2\pi}\sum_{n\in \mathbb{Z}^2}\bigg\{\delta((t-\tau)^2-(x_3-y_3)^2-|x_\parallel -y_\parallel-n|^2 )1_{(x_3-y_3)^2+|x_\parallel-y_\parallel-n|^2\le (t-\tau)^2}\\+\delta((t-\tau)^2-(x_3+y_3)^2-|x_\parallel -y_\parallel-n|^2 )1_{(x_3+y_3)^2+|x_\parallel-y_\parallel-n|^2\le (t-\tau)^2}\bigg\}.
\end{multline}
We now begin with deriving the representations for the fields $\Eone$ and $\Etwo$ with Dirichlet boundary conditions.  The representations for $\Eth$ 
 will be derived later.

\subsection{Representations for \texorpdfstring{\textnormal{$\Eone$ and $\Etwo$}}{}}
In this section, we derive the Glassey-Strauss representations for the self-consistent fields $\Eone$ and $\Etwo$ 
that satisfy the Dirichlet boundary conditions via the perfect conductor boundary conditions \eqref{perfect cond. boundary}. 
\subsubsection{Representations for the Tangential Electric Fields $\Eone$ and $\Etwo$}By the wave equation \eqref{wave eq for E} for the vector $\E$ and the perfect conductor boundary condition \eqref{perfect cond. boundary}, the parallel components $\Ei$ for $i=1,2$ satisfies the following wave equations for $x=(x_\parallel,x_3)\in \mathbb{T}^2\times (0,\infty)$:
\begin{equation}
    \begin{split}
        (\partial^2_t - \Delta_x) \Ei&=-4\pi (\partial _{x_i}\rho +\partial_t J_i),\\
    \Ei(0,x)&=\mathbf{E}_{0i}(x),\\
         \partial_t \Ei(0,x)&=\mathbf{E}^1_{0i}(x),\\ \Ei(t,x_\parallel,0)&=0.
    \end{split}
\end{equation}Here, we assume that $\E_0^1$ further satisfies $$\E_0^1=\nabla\times \B^{\textup{in}}-4\pi J_0,$$ for the compatibility with \eqref{2speciesVM}$_2$. 
To begin with, let us denote the homogeneous solution as $\Eihomo$ for $i=1,2$ with Cauchy data which solves the following system of the wave equations for $x\in\rth$:\begin{equation}\label{sys.Eihomo}
    \begin{split}
        (\partial^2_t - \Delta_x) \Eihomo&=0,\\
        \Eihomo(0,x)&=\tilde{\mathbf{E}}_{0i}^{\textup{odd}}(x),\\
         \partial_t \Eihomo(0,x)&=\tilde{\mathbf{E}}_{0i}^{1,\textup{odd}}(x),\\ \Eihomo(t,x_\parallel,0)&=0,
    \end{split}
\end{equation}where we first extended the data $\mathbf{E}_{0i}$ and $\mathbf{E}^1_{0i}$ periodically in $x_\parallel$ (which we denote their periodic extension as $ \tilde{\mathbf{E}}^1_{0i} $ and $ \tilde{\mathbf{E}}^1_{0i} $) and applied the odd extension with respect to $x_3$ across $x_3 = 0$ and the periodic extension with respect to  $x_\parallel$ of the Cauchy data to the whole space $x\in\rth$  as
\begin{equation}
\label{odd1}
\tilde{\mathbf{E}}_{0i}^{\textup{odd}}(x_\parallel, x_3) = 
\begin{cases} 
\tilde{\mathbf{E}}_{0i}(x_\parallel, x_3) & \textup{if } x_3 \geq 0, \\
-\tilde{\mathbf{E}}_{0i}(x_\parallel, -x_3) & \textup{if } x_3 < 0,
\end{cases}
\end{equation}
and similarly for the time derivative of the initial condition
\begin{equation}
    \label{odd2}
\tilde{\mathbf{E}}_{0i}^{1,\textup{odd}}(x_\parallel, x_3) = 
\begin{cases} 
\tilde{\mathbf{E}}^1_{0i}(x_\parallel, x_3) & \textup{if } x_3 \geq 0, \\
-\tilde{\mathbf{E}}^1_{0i}(x_\parallel, -x_3) & \textup{if } x_3 < 0.
\end{cases}
\end{equation}
Then, we can write the Kirchhoff formula of the solution to the homogeneous wave equation  
as 
\begin{equation*}
\Eihomo(t, x) = \frac{1}{4\pi t^2} \int_{\partial B(x; t)}  \left(t \tilde{\mathbf{E}}^{1,\textup{odd}}_{0i}(y) + \tilde{\mathbf{E}}^{\textup{odd}}_{0i}(y) + \nabla \tilde{\mathbf{E}}^{\textup{odd}}_{0i}(y) \cdot (y - x)\right) dS_y,
\end{equation*} where 
we denote the open ball of radius $t$ around $x$ in the extended whole space $\rth$ as $$B(x;t)=\{y'\in \mathbb{R}^3:|x-y'|<t\}$$ such that $$\partial B(x;t)=\{y'\in \mathbb{R}^3:|x-y'|=t\}.$$ Retrieving the original data representation via \eqref{odd1}--\eqref{odd2}, we have
\begin{multline}\notag
\Eihomo(t, x) =\frac{1}{4\pi t^2} \int_{\partial B(x; t)\cap \{y_3>0\}}  \left(t \tilde{\mathbf{E}}^1_{0i}(y) + \tilde{\mathbf{E}}_{0i}(y) + \nabla \tilde{\mathbf{E}}_{0i}(y) \cdot (y - x)\right) dS_y\\
-\frac{1}{4\pi t^2} \int_{\partial B(x; t)\cap \{y_3<0\}}  \left(t \tilde{\mathbf{E}}^1_{0i}(\bar{y}) + \tilde{\mathbf{E}}_{0i}(\bar{y}) + \nabla \tilde{\mathbf{E}}_{0i}(\bar{y}) \cdot (\bar{y} - \bar{x})\right) dS_y,
\end{multline} where $\bar{y}\eqdef (y_1,y_2,-y_3)^\top$ and we denote the periodic extension of the Cauchy data as $ \tilde{\mathbf{E}}_{0i} $ and $ \tilde{\mathbf{E}}^1_{0i} $  in $x_\parallel$ directions. 

On the other hand, the particular solution $\Eipar$ for $i=1,2$ (as well as $i=3$) solve the following system of wave equation for $x=(x_\parallel,x_3)\in \mathbb{T}^2\times (0,\infty)$ by \eqref{wave eq for E}:
\begin{equation}\label{sys.parEi}
    \begin{split}
        (\partial^2_t - \Delta_x) \Eipar&=-4\pi (\partial_i \rho +\partial_t J_i),\\
        \Eipar(0,x)&=0,\\
         \partial_t \Eipar(0,x)&=0,\\ \Eipar(t,x_\parallel,0)&=0.
    \end{split}
\end{equation}
The particular solution $ \Eipar$ for $i=1,2$ can be written
\begin{multline}\label{Epar Green}
    \Eipar(t,x)=-4\pi \int_0^t d\tau\int_{\mathbb{T}^2}dy_\parallel \int_0^\infty dy_3\ \bar{G}(t,\tau,x,y)(\partial_i \rho +\partial_t J_i)(\tau,y)\\
    =-2 \int_0^t d\tau\int_{\mathbb{T}^2}dy_\parallel \int_0^\infty dy_3 \sum_{n\in \mathbb{Z}^2}(\partial_i \rho +\partial_t J_i)(\tau,y)\\\times \bigg\{\delta((t-\tau)^2-(x_3-y_3)^2-|x_\parallel -y_\parallel-n|^2 )1_{(x_3-y_3)^2+|x_\parallel-y_\parallel-n|^2\le (t-\tau)^2}\\-\delta((t-\tau)^2-(x_3+y_3)^2-|x_\parallel -y_\parallel-n|^2 )1_{(x_3+y_3)^2+|x_\parallel-y_\parallel-n|^2\le (t-\tau)^2}\bigg\} 
    =: -2(I_1-I_2),
\end{multline}
using the Dirichlet-Green function $\bar{G}$ \eqref{Green in the half channel}. For $i=3$, $\Ethpar$ has a similar form of $\eqref{Epar Green}$ but with $\tilde{G}$ replacing $\bar{G}$. We first consider \eqref{Epar Green} of $i=1,2$.

For $I_1$ and $I_2$, we observe that 
\begin{equation*}
\begin{split}
   & \delta((t-\tau)^2-(x_3\pm y_3)^2-|x_\parallel -y_\parallel-n|^2 )\\
    &=\frac{1}{2\sqrt{(x_3\pm y_3)^2+|x_\parallel -y_\parallel-n|^2 }}\bigg(\delta\bigg((t-\tau)-\sqrt{(x_3\pm y_3)^2+|x_\parallel -y_\parallel-n|^2} )\\&+\delta((t-\tau)+\sqrt{(x_3\pm y_3)^2+|x_\parallel -y_\parallel-n|^2}\bigg)\bigg)
    =\frac{\delta\left((t-\tau)-\sqrt{(x_3\pm y_3)^2+|x_\parallel -y_\parallel-n|^2} \right)}{2\sqrt{(x_3\pm y_3)^2+|x_\parallel -y_\parallel-n|^2 }},
    \end{split}
\end{equation*}
since $t-\tau\ge 0.$
Integrating this delta function with respect to $\tau$, we obtain 
\begin{multline*}
    -2(I_1-I_2)=-\sum_{n\in \mathbb{Z}^2}\int_{\mathbb{T}^2}dy_\parallel \int_0^\infty dy_3 \bigg(\frac{(\partial_i \rho +\partial_t J_i)\left(t-\sqrt{(x_3- y_3)^2+|x_\parallel -y_\parallel-n|^2},y\right)}{2\sqrt{(x_3-y_3)^2+|x_\parallel -y_\parallel-n|^2 }}\\
    \times 1_{(x_3-y_3)^2+|x_\parallel-y_\parallel-n|^2\le t^2}\\
    -\frac{(\partial_i \rho +\partial_t J_i)(t-\sqrt{(x_3+y_3)^2+|x_\parallel -y_\parallel-n|^2},y)}{2\sqrt{(x_3+ y_3)^2+|x_\parallel -y_\parallel-n|^2 }}
    1_{(x_3+y_3)^2+|x_\parallel-y_\parallel-n|^2\le t^2}\bigg)
    =:-2(I_3-I_4).
\end{multline*}
For both $I_3$ and $I_4$, we apply the change of variables $y_\parallel \in \mathbb{T}^2\mapsto y_\parallel + n - x_\parallel = \tilde{y}_\parallel \in \mathbb{T}^2+\{n-x_\parallel\}$ to get
\begin{multline*}
     \sum_{n\in \mathbb{Z}^2}\int_{\mathbb{T}^2+\{n-x_\parallel\}}d\tilde{y}_\parallel \int_0^\infty dy_3 \frac{(\partial_i \rho +\partial_t J_i)(t-\sqrt{(x_3\pm y_3)^2+|\tilde{y}_\parallel |^2},\tilde{y}_\parallel -n+x_\parallel,y_3)}{2\sqrt{(x_3\pm y_3)^2+|\tilde{y}_\parallel |^2 }}\\
    \times 1_{(x_3\pm y_3)^2+|\tilde{y}_\parallel |^2\le t^2}.
\end{multline*}
Using that $\rho$ and $J$ are periodic in $x_\parallel$ variable, we further have that for $I_3$ and $I_4$,
\begin{multline*}
   \sum_{n\in \mathbb{Z}^2}\int_{\mathbb{T}^2+\{n-x_\parallel\}}d\tilde{y}_\parallel \int_0^\infty dy_3 \frac{(\partial_i \rho +\partial_t J_i)(t-\sqrt{(x_3\pm y_3)^2+|\tilde{y}_\parallel |^2},\tilde{y}_\parallel +x_\parallel,y_3)}{2\sqrt{(x_3\pm y_3)^2+|\tilde{y}_\parallel |^2 }}1_{(x_3\pm y_3)^2+|\tilde{y}_\parallel |^2\le t^2}\\
    =\int_{\mathbb{R}^2}d\tilde{y}_\parallel \int_0^\infty dy_3 \frac{(\partial_i \rho +\partial_t J_i)(t-\sqrt{(x_3\pm y_3)^2+|\tilde{y}_\parallel |^2},\tilde{y}_\parallel +x_\parallel,y_3)}{2\sqrt{(x_3\pm y_3)^2+|\tilde{y}_\parallel |^2 }} 1_{(x_3\pm y_3)^2+|\tilde{y}_\parallel |^2\le t^2}\\
    =\int_{\mathbb{R}^2}dy'_\parallel \int_0^\infty dy'_3 \frac{(\partial_i \rho +\partial_t J_i)(t-\sqrt{(x_3\pm y'_3)^2+|y'_\parallel-x_\parallel|^2},y')}{2\sqrt{(x_3\pm y'_3)^2+|y'_\parallel-x_\parallel |^2 }}1_{(x_3\pm y'_3)^2+|y'_\parallel-x_\parallel |^2\le t^2},
\end{multline*}
by redefining $y'_\parallel = \tilde{y}_\parallel + x_\parallel$, $y'_3 = y_3,$ and $y'=(y'_\parallel,y'_3)^\top$. Thus, we finally have for $i=1,2$,
\begin{multline}\label{2.11}
    \Eipar(t,x)= -\int_{\mathbb{R}^2}dy'_\parallel \int_0^\infty dy'_3 \frac{(\partial_i \rho +\partial_t J_i)(t-\sqrt{(x_3- y'_3)^2+|y'_\parallel-x_\parallel|^2},y')}{\sqrt{(x_3-y'_3)^2+|y'_\parallel-x_\parallel |^2 }}\\
    \times 1_{(x_3-y'_3)^2+|y'_\parallel-x_\parallel |^2\le t^2}\\
+\int_{\mathbb{R}^2}dy'_\parallel \int_0^\infty dy'_3 \frac{(\partial_i \rho +\partial_t J_i)(t-\sqrt{(x_3+ y'_3)^2+|y'_\parallel-x_\parallel|^2},y')}{\sqrt{(x_3+ y'_3)^2+|y'_\parallel-x_\parallel |^2 }}\\
    \times 1_{(x_3+ y'_3)^2+|y'_\parallel-x_\parallel |^2\le t^2}=:-\Eipar^{(1)}+\Eipar^{(2)}.
\end{multline}
Denote the open ball of radius $t$ around $x$ in the extended whole space $\rth$ as $$B^\pm(x;t)=\{y'\in \mathbb{R}^3:|x-y'|<t\textup{ and }\pm y_3'>0\}.$$
Then, for $\Eipar^{(1)} $ and $\Eipar^{(2)},$ we have  \begin{multline}
    \label{2.12}\Eipar^{(1)}=\int_{B^+(x;t)} dy' \frac{(\partial_i \rho +\partial_t J_i)(t-|x-y'|,y'_\parallel,  y'_3)}{|x-y'|}\\
   =\int_{B^+(x;t)} dy' \frac{\int_\rth dv\ (\partial_i F_++(\hat{v}_+)\partial_t F_+)(t-|x-y'|,y'_\parallel,  y'_3,v)}{|x-y'|}\\
    -\int_{B^+(x;t)} dy' \frac{\int_\rth dv\ (\partial_i F_-+(\hat{v}_-)\partial_t F_-)(t-|x-y'|,y'_\parallel,  y'_3,v)}{|x-y'|},
\end{multline}and
\begin{multline}
    \label{2.13}\Eipar^{(2)}=\int_{B^-(x;t)} dy' \frac{(\partial_i \rho +\partial_t J_i)(t-|x-y'|,y'_\parallel,  -y'_3)}{|x-y'|}\\
   =\int_{B^-(x;t)} dy' \frac{\int_\rth dv\ (\partial_i F_++(\hat{v}_+)\partial_t F_+)(t-|x-y'|,y'_\parallel,  -y'_3,v)}{|x-y'|}\\
    -\int_{B^-(x;t)} dy' \frac{\int_\rth dv\ (\partial_i F_-+(\hat{v}_-)\partial_t F_-)(t-|x-y'|,y'_\parallel,  -y'_3,v)}{|x-y'|}.
\end{multline}

Now we obtain the Glassey-Strauss representation in a half space $\mathbb{T}^2\times (0,\infty)$ as follows. Define $S_\pm=\partial_t + (\hat{v}_\pm)\cdot \nabla_x $ and $T_i=\partial_i -\omega_i\partial_t $ such that 
$$T_jF_\pm(t-|y'-x|,y',v)=\partial_{y_j}(F_\pm(t-|x-y'|,y',v)),$$  and $$\partial_t=\frac{S_\pm-(\hat{v}_\pm)\cdot T}{1+\hat{v}_\pm\cdot \omega},\ \partial_i = \frac{\omega_i S_\pm}{1+\hat{v}_\pm\cdot \omega}+(\delta_{ij}-\frac{\omega_i(\hat{v}_\pm)_j}{1+\hat{v}_\pm\cdot \omega})T^j,$$ for $\omega=\frac{y'-x}{|y'-x|}$. 
Now, we use the half-space formula obtained in \cite[eq. (35), (37)-(41)]{MR4414612} for the integrals in $F_\pm$ by considering generic masses $m_\pm$. Then we obtain for $j=1,2$ $$\Eipar^{(j)}=\mathbf{E}^{(j)}_{+,\textup{par},i}-\mathbf{E}^{(j)}_{-,\textup{par},i},$$ where
\begin{multline}\label{Ei5.T}
    \Eione=  
    \mp\int_{B^+(x;t)} dy' \int_\rth dv\ a^{\mathbf{E}}_{\pm,i}(v,\omega)\cdot (\pm\E\pm(\hat{v}_\pm)\times \B-m_\pm g\hat{e}_3)\frac{F_\pm(t-|x-y'|,y',v)}{|x-y'|}\\
    \pm\int_{\partial B(x;t)\cap \{y'_3>0\}} \frac{dS_{y'}}{|y'-x|}\int_\rth dv\ \left(\delta_{ij}-\frac{(\omega_i+(\hat{v}_\pm)_i)(\hat{v}_\pm)_j}{1+\hat{v}_\pm\cdot \omega}\right)\omega^jF_\pm(0,y',v)\\
    \mp\int_{B(x;t)\cap \{y'_3=0\}} \frac{dy'_\parallel}{|y'-x|}\int_\rth dv\  \left(\delta_{i3}-\frac{(\omega_i+(\hat{v}_\pm)_i)(\hat{v}_\pm)_3}{1+\hat{v}_\pm\cdot \omega}\right)F_\pm(t-|x-y'|,y'_\parallel,0,v)\\
    \mp\int_{B^+(x;t)} \frac{dy'}{|y'-x|^2}\int_\rth dv\ \frac{(|(\hat{v}_\pm)|^2-1)((\hat{v}_\pm)_i+\omega_i)}{(1+\hat{v}_\pm\cdot \omega)^2} F_\pm(t-|x-y'|,y',v)\\
    =: \mathbf{E}^{(1)}_{\pm,iS}+\mathbf{E}^{(1)}_{\pm,ib1}+\mathbf{E}^{(1)}_{\pm,ib2}+\mathbf{E}^{(1)}_{\pm,iT},
\end{multline}where $dy'_\parallel$ is the 2-dimensional Lebesgue measure on $B(x;t)\cap \{y'_3=0\}$, and
\begin{equation}
    \label{aEi.T}
a^{\E}_{\pm,i}(v,\omega)\eqdef\nabla_v \left(\frac{\omega_i+(\hat{v}_\pm)_i}{1+\hat{v}_\pm\cdot \omega}\right)=\frac{(\partial_{v_i}v-(\hat{v}_\pm)_i(\hat{v}_\pm))(1+\hat{v}_\pm\cdot \omega)-(\omega_i+(\hat{v}_\pm)_i)(\omega-(\omega\cdot (\hat{v}_\pm))(\hat{v}_\pm))}{(\vZ) (1+\hat{v}_\pm\cdot \omega)^2}.\end{equation} Note that we here follow the standard Einstein summation convention. Here we note that the gravity $-m_\pm g\hat{e}_3$ only appears in the $\mathbf{E}^{(1)}_{\pm,iS}$ term.

On the other hand, by \cite[eq. (36), (47)-(50)]{MR4414612}, we also obtain 
\begin{multline}\label{Ei6.T}
    \Eitwo=  \mp\int_{B^-(x;t)} dy' \int_\rth dv\ a^{\E}_{\pm,i}(v,\bar{\omega})\cdot (\pm\E\pm(\hat{v}_\pm)\times \B-m_\pm g\hat{e}_3)\frac{F_\pm(t-|x-y'|,\bar{y}',v)}{|x-y'|}\\
    \pm\int_{\partial B(x;t)\cap \{y'_3<0\}} \frac{dS_{y'}}{|y'-x|}\int_\rth dv\ (\delta_{ij}-\frac{(\bar{\omega}_i+(\hat{v}_\pm)_i)(\hat{v}_\pm)_j}{1+(\hat{v}_\pm)\cdot \bar{\omega}})\bar{\omega}^jF_\pm(0,\bar{y}',v)\\
    \mp\int_{B(x;t)\cap \{y'_3=0\}} \frac{dy'_\parallel}{|y'-x|}\int_\rth dv\  (\delta_{i3}-\frac{(\bar{\omega}_i+(\hat{v}_\pm)_i)(\hat{v}_\pm)_3}{1+(\hat{v}_\pm)\cdot \bar{\omega}})F_\pm(t-|x-y'|,y'_\parallel,0,v)\\
    \mp\int_{B^-(x;t)} \frac{dy'}{|y'-x|^2}\int_\rth dv\ \frac{(|(\hat{v}_\pm)|^2-1)((\hat{v}_\pm)_i+\bar{\omega}_i)}{(1+(\hat{v}_\pm)\cdot \bar{\omega})^2} F_\pm(t-|x-y'|,\bar{y}',v)\\
    =:\mathbf{E}^{(2)}_{\pm,iS}+\mathbf{E}^{(2)}_{\pm,ib1}+\mathbf{E}^{(2)}_{\pm,ib2}+\mathbf{E}^{(2)}_{\pm,iT},
\end{multline}
where  $\bar{y}'\eqdef(y'_\parallel, -y'_3)^\top$, $\bar{\omega}\eqdef(\omega_1,\omega_2, -\omega_3)^\top$, and 
$dy'_\parallel$ is the 2-dimensional Lebesgue measure on $B(x;t)\cap \{y'_3=0\}$.

The representation $\mathbf{E}_3$ for the case for $i=3$ will be given later in \eqref{E3}.

\begin{remark}
    Note that the Glassey-Strauss representations \eqref{Ei5.T}-\eqref{Ei6.T} (as well as all other representations in the following section) for the fields are derived in the periodically extended whole half-space $\mathbb{R}^3_+$. These representations indeed apply to both cases $\Omega = \mathbb{T}^2 \times \mathbb{R}_+$ and $\Omega = \mathbb{R}^3_+$. In the former case, we additionally require that $F_\pm$, $\E$, and $\B$ remain periodic in the $x_\parallel$ directions under the same representations.
\end{remark}

\subsection{Representation for \texorpdfstring{\textnormal{$\Eth$}}{}}

We now derive the representation formulas for $\Eth$, 
 which is associated with the Neumann-type boundary conditions given in \eqref{perfect.conductor.neumann}. Unlike the Dirichlet boundary conditions imposed on $\Eone$ and $\Etwo$, 
  there are no explicit boundary conditions prescribed for $\Eth$, 
   in the original formulation.

However, by combining the Maxwell equations, specifically \eqref{2speciesVM}$_2$ and \eqref{2speciesVM}$_4$, with the perfect conductor boundary conditions \eqref{perfect cond. boundary}, we formally derive the Neumann-type boundary conditions \eqref{perfect.conductor.neumann} for $\Eth$.

Since the normal derivative $\partial_{x_3} \mathbf{E}_3$ 
is only available in the sense of distributions, not as continuous functions, the Neumann boundary conditions are understood weakly. Consequently, when constructing the Green function representations for $\Eth$, 
the boundary integrals involving the Neumann data are also interpreted in the sense of distributions. In particular, these boundary terms are realized through their action on smooth test functions, consistent with the weak formulation of the problem.

\begin{remark}
Since the Neumann boundary data are only available in a distributional sense, the Green function representation for $\Eth$ 
is interpreted accordingly. In particular, the boundary integrals are understood as dual pairings with test functions rather than pointwise evaluations.
\end{remark}

\subsubsection{Representation for the Normal Electric Field $\Eth$}By the wave equation \eqref{wave eq for E} for the vector $\E$ and the Neumann boundary condition 
\eqref{perfect.conductor.neumann} for the normal component $\Eth$, we note that the normal components $\Eth$ satisfies the following wave equations for $x=(x_\parallel,x_3)\in \mathbb{T}^2\times (0,\infty)$:
\begin{equation}
    \begin{split}
        (\partial^2_t - \Delta_x) \Eth&=-4\pi (\partial _{x_3}\rho +\partial_t J_3),\\
        \Eth(0,x)&=\mathbf{E}_{03}(x),\\
         \partial_t \Eth(0,x)&=\mathbf{E}^1_{03},\\ \partial_{x_3}\Eth(t,x_\parallel)&=4\pi\rho.
    \end{split}
\end{equation}Here, we assume that $\E_0^1$ further satisfies $$\E_0^1=\nabla\times \B^{\textup{in}}-4\pi J_0,$$ for the compatibility with \eqref{2speciesVM}$_2$.  It is convenient to decompose the solution $\Eth$ into two parts: the one $\mathbf{E}_{\textup{neu},3}$ which satisfies \begin{equation}\label{eq.E3neu}
    \begin{split}
        (\partial^2_t - \Delta_x) \mathbf{E}_{\textup{neu},3}&=0,\\
        \mathbf{E}_{\textup{neu},3}(0,x)&=0,\\
         \partial_t \mathbf{E}_{\textup{neu},3}(0,x)&=0,\\ \partial_{x_3}\mathbf{E}_{\textup{neu},3}(t,x_\parallel,0)&=4\pi\rho,
    \end{split}
\end{equation} and the other $\mathbf{E}_{\textup{rest},3}$ that satisfies
 \begin{equation}\label{eq.for E3rest}
    \begin{split}
        (\partial^2_t - \Delta_x) \mathbf{E}_{\textup{rest},3}&=-4\pi (\partial _{x_3}\rho +\partial_t J_3)=:G_3(t,x),\\
        \mathbf{E}_{\textup{rest},3}(0,x)&=\mathbf{E}_{03}(x),\\
         \partial_t \mathbf{E}_{\textup{rest},3}(0,x)&=\mathbf{E}^1_{03},\\ \partial_{x_3}\mathbf{E}_{\textup{rest},3}(t,x_\parallel,0)&=0.
    \end{split}
\end{equation} 
To write the solution $\mathbf{E}_{\textup{rest},3}$ to \eqref{eq.for E3rest},
we consider the even extension with respect to $x_3$ and the periodic extension with respect to  $x_\parallel$ of the Cauchy data to the whole space. Namely we define  the even extension across $x_3 = 0$ as
\begin{equation}
\label{even0}
G_3^{\textup{even}}(x_\parallel, x_3) = 
\begin{cases} 
G_3(x_\parallel, x_3) & \textup{if } x_3 \geq 0, \\
G_3(x_\parallel, -x_3) & \textup{if } x_3 < 0,
\end{cases}
\end{equation}
\begin{equation}
\label{even1}
\mathbf{E}_{03}^{\textup{even}}(x_\parallel, x_3) = 
\begin{cases} 
\mathbf{E}_{03}(x_\parallel, x_3) & \textup{if } x_3 \geq 0, \\
\mathbf{E}_{03}(x_\parallel, -x_3) & \textup{if } x_3 < 0,
\end{cases}
\end{equation}
and similarly for the time derivative of the initial condition
\begin{equation}
    \label{even2}
\mathbf{E}_{03}^{1,\textup{even}}(x_\parallel, x_3) = 
\begin{cases} 
\mathbf{E}^1_{03}(x_\parallel, x_3) & \textup{if } x_3 \geq 0, \\
\mathbf{E}^1_{03}(x_\parallel, -x_3) & \textup{if } x_3 < 0.
\end{cases}
\end{equation}
By extending these data periodically in $x_\parallel$ (which we denote their periodic extension as $ \tilde{G}^{\textup{even}}_3, \tilde{\mathbf{E}}^{\textup{even}}_{03} $ and $ \tilde{\mathbf{E}}^{1,\textup{even}}_{03} $), we can write the Kirchhoff formula of the solution to the homogeneous wave equation as 
\begin{multline*}
\mathbf{E}_{\textup{rest},3}(t, x) = \frac{1}{4\pi t^2} \int_{\partial B(x; t)}  \left(t \tilde{\mathbf{E}}^{1,\textup{even}}_{03}(y) + \tilde{\mathbf{E}}^{\textup{even}}_{03}(y) + \nabla \tilde{\mathbf{E}}^{\textup{even}}_{03}(y) \cdot (y - x)\right) dS_y
+\Ethpar(t,x),
\end{multline*} where $\Ethpar$ is the particular solution which satisfies the wave equation \eqref{eq.for E3rest} with zero initial data. The derivation of the representation for the particular solution $\Ethpar$ follows exactly the same as the derivation \eqref{sys.parEi}--\eqref{Ei6.T} with the Green function $\tilde{G}$ replacing $\bar{G}.$ This replacement results in the sign change on $\mathbf{E}^{(2)}_{\textup{par},i}$ term at \eqref{2.11} and we define $$\Ethpar\eqdef -\mathbf{E}^{(1)}_{\textup{par},3}-\mathbf{E}^{(2)}_{\textup{par},3},$$ where as before we further split them as $\mathbf{E}^{(1)}_{\textup{par},3}= \mathbf{E}^{(1)}_{+,\textup{par},3}-\mathbf{E}^{(1)}_{-,\textup{par},3}$ and $\mathbf{E}^{(2)}_{\textup{par},3}= \mathbf{E}^{(2)}_{\pm,\textup{par},3}- \mathbf{E}^{(2)}_{-,\textup{par},3}$ with $ \Ethone$ and $\Ethtwo$ defined in \eqref{Ei5.T} and \eqref{Ei6.T}.  Retrieving the original data representation via \eqref{even1}--\eqref{even2}, we have

\begin{multline}\label{E3 rest solution}
\mathbf{E}_{\textup{rest},3}(t, x) =\frac{1}{4\pi t^2} \int_{\partial B(x; t)\cap \{y_3>0\}}  \left(t \tilde{\mathbf{E}}^1_{03}(y) + \tilde{\mathbf{E}}_{03}(y) + \nabla \tilde{\mathbf{E}}_{03}(y) \cdot (y - x)\right) dS_y\\
+\frac{1}{4\pi t^2} \int_{\partial B(x; t)\cap \{y_3<0\}}  \left(t \tilde{\mathbf{E}}^1_{03}(\bar{y}) + \tilde{\mathbf{E}}_{03}(\bar{y}) + \nabla \tilde{\mathbf{E}}_{03}(\bar{y}) \cdot (\bar{y} - \bar{x})\right) dS_y\\
-\sum_{\pm} \left(\pm \Ethone(t,x)\pm\Ethtwo(t,x)\right),
\end{multline} where $\bar{y}\eqdef (y_1,y_2,-y_3)^\top$ and we denote the periodic extension of the Cauchy data as $ \tilde{\mathbf{E}}_{03} $ and $ \tilde{\mathbf{E}}^1_{03} $  in $x_\parallel$ directions. 

On the other hand, for the derivation of the representation of $\mathbf{E}_{\textup{neu},3},$ we again take the periodic extension in the parallel direction $x_\parallel$ and denote the extension as $\tilde{\mathbf{E}}_{\textup{neu},3}.$ This field component $\tilde{\mathbf{E}}_{\textup{neu},3}$ solves the same system in $\mathbb{R}^3_+$ as $w$ in \cite[Eq. (55)]{MR4414612}, and hence the representation follows by \cite[Eq. (73)]{MR4414612} as
\begin{multline*}
    \tilde{\mathbf{E}}_{\textup{neu},3}(t,x)=2\int_{B(x;t)\cap \{y_3=0\}}\frac{\rho(t-|x-y|,y_\parallel,0)}{|x-y|}dy_\parallel\\
    = 2\sum_{\pm}\pm\int_{B(x;t)\cap \{y_3=0\}} \int_\rth \frac{F_\pm(t-|y-x|,y_\parallel,0,v)}{|y-x|}dvdy_\parallel.
\end{multline*}

Therefore, for each $(t,x)\in [0,\infty)\times \mathbb{T}^2\times (0,\infty),$ we finally obtain the Glassey-Strauss representation for the field $\Eth$ as
\begin{multline}
    \label{E3.T}
    \Eth(t,x)=\frac{1}{4\pi t^2} \int_{\partial B(x; t)\cap \{y_3>0\}}  \left(t \tilde{\mathbf{E}}^1_{03}(y) + \tilde{\mathbf{E}}_{03}(y) + \nabla \tilde{\mathbf{E}}_{03}(y) \cdot (y - x)\right) dS_y\\
+\frac{1}{4\pi t^2} \int_{\partial B(x; t)\cap \{y_3<0\}}  \left(t \tilde{\mathbf{E}}^1_{03}(\bar{y}) + \tilde{\mathbf{E}}_{03}(\bar{y}) + \nabla \tilde{\mathbf{E}}_{03}(\bar{y}) \cdot (\bar{y} - \bar{x})\right) dS_y\\
  \sum_\pm \pm  \int_{B^+(x;t)} dy' \int_\rth dv\ a^{\E}_{\pm,3}(v,\omega)\cdot (\pm\E\pm(\hat{v}_\pm)\times \B-m_\pm g\hat{e}_3)\frac{F_\pm(t-|x-y'|,y',v)}{|x-y'|}\\
    \sum_\mp \mp\int_{\partial B(x;t)\cap \{y'_3>0\}} \frac{dS_{y'}}{|y'-x|}\int_\rth dv\ \left(\delta_{3j}-\frac{(\omega_3+(\hat{v}_\pm)_3)(\hat{v}_\pm)_j}{1+\hat{v}_\pm\cdot \omega}\right)\omega^jF_\pm(0,y',v)\\
    \sum_\pm \pm\int_{B(x;t)\cap \{y'_3=0\}} \frac{dy'_\parallel}{|y'-x|}\int_\rth dv\  \left(1-\frac{(\omega_3+(\hat{v}_\pm)_3)(\hat{v}_\pm)_3}{1+\hat{v}_\pm\cdot \omega}\right)F_\pm(t-|x-y'|,y'_\parallel,0,v)\\
     \sum_\pm \pm\int_{B^+(x;t)} \frac{dy'}{|y'-x|^2}\int_\rth dv\ \frac{(|(\hat{v}_\pm)|^2-1)((\hat{v}_\pm)_3+\omega_3)}{(1+\hat{v}_\pm\cdot \omega)^2} F_\pm(t-|x-y'|,y',v)\\
    \sum_\pm \pm\int_{B^-(x;t)} dy' \int_\rth dv\ a^{\E}_{\pm,3}(v,\bar{\omega})\cdot (\pm\E\pm(\hat{v}_\pm)\times \B-m_\pm g\hat{e}_3)\frac{F_\pm(t-|x-y'|,\bar{y}',v)}{|x-y'|}\\
     \sum_\mp \mp\int_{\partial B(x;t)\cap \{y'_3<0\}} \frac{dS_{y'}}{|y'-x|}\int_\rth dv\ (\delta_{3j}-\frac{(\bar{w}_3+(\hat{v}_\pm)_3)(\hat{v}_\pm)_j}{1+(\hat{v}_\pm)\cdot \bar{\omega}})\bar{\omega}^jF_\pm(0,\bar{y}',v)\\
     \sum_\pm \pm\int_{B(x;t)\cap \{y'_3=0\}} \frac{dy'_\parallel}{|y'-x|}\int_\rth dv\  (1-\frac{(\bar{w}_3+(\hat{v}_\pm)_3)(\hat{v}_\pm)_3}{1+(\hat{v}_\pm)\cdot \bar{\omega}})F_\pm(t-|x-y'|,y'_\parallel,0,v)\\
     \sum_\pm \pm\int_{B^-(x;t)} \frac{dy'}{|y'-x|^2}\int_\rth dv\ \frac{(|(\hat{v}_\pm)|^2-1)((\hat{v}_\pm)_3+\bar{w}_3)}{(1+(\hat{v}_\pm)\cdot \bar{\omega})^2} F_\pm(t-|x-y'|,\bar{y}',v)\\
     \sum_\mp \mp 2\int_{B(x;t)\cap \{y_3=0\}} \int_\rth \frac{F_\pm(t-|y-x|,y_\parallel,0,v)}{|y-x|}dvdy_\parallel,
\end{multline}where $a^{\E}_{\pm,3}$ is defined in \eqref{aEi.T}, $\bar{y}'\eqdef(y'_\parallel, -y'_3)^\top$, $\bar{\omega}\eqdef(\omega_1,\omega_2, -\omega_3)^\top$, and 
$dy'_\parallel$ is the 2-dimensional Lebesgue measure on $B(x;t)\cap \{y'_3=0\}$.

Note that there is an additional term appearing in the Glassey-Strauss representations for $\Eth$, compared to the field-component $\Ei$ with $i=1,2$ that solves the Dirichlet boundary condition:
\begin{equation}
    \label{additional term E3}
    \text{E}_{\textup{add},3}(t,x)\eqdef \sum_{\mp}\mp 2\int_{B(x;t)\cap \{y_3=0\}} \int_\rth \frac{F_\pm(t-|y-x|,y_\parallel,0,v)}{|y-x|}dvdy_\parallel.
\end{equation}

\section{Weak Formulations}
\subsection{Weak Formulation to the Maxwell Equations}
\label{sec.weak.deri}In this section, we derive the weak formulation of the Maxwell equations given in the dynamical equations \eqref{2speciesVM}. 
We first write the weak formulation of the Maxwell equations using the fields $\E$ and $\B$, and the inhomogeneities $\rho$ and $J$. The test function will be denoted as $\varphi \in C^1_c([0,T] \times \Omega)$ or its vector-valued counterpart $\boldsymbol{\varphi}\in C^1_c([0,T] \times \Omega)$, depending on the context. We aim to derive the weak formulations for the equations defined in \eqref{2speciesVM}, over the time interval $[0,T]$ and spatial domain $\Omega$, assuming perfect conductor boundary conditions.

Regarding Ampère's Law \eqref{2speciesVM}$_2$, we have
$$
\partial_t \E - \nabla_x \times \B = -4\pi J, \quad j = \int_{\mathbb{R}^3} (\hat{v}_+ F_+ - \hat{v}_- F_-) \, dv.
$$
Multiply by a test function $\boldsymbol{\varphi} \in C^1_c([0,T] \times \Omega)$ and integrate over $[0,T] \times \Omega$:
$$
\int_0^T \int_\Omega \boldsymbol{\varphi} \cdot \left( \partial_t \E - \nabla_x \times \B + 4\pi J \right) \, dx dt = 0.
$$
Separate the terms:
$$
\int_0^T \int_\Omega \boldsymbol{\varphi} \cdot \partial_t \E \, dx dt - \int_0^T \int_\Omega \boldsymbol{\varphi} \cdot (\nabla_x \times \B) \, dx dt + 4\pi \int_0^T \int_\Omega \boldsymbol{\varphi} \cdot J \, dx dt = 0.
$$
Using integration by parts for the curl term and assuming $\B_t = 0$ on $\partial \Omega$ (perfect conductor boundary conditions):
$$
\int_\Omega \boldsymbol{\varphi} \cdot (\nabla_x \times \B) \, dx = \int_\Omega (\nabla_x \times \boldsymbol{\varphi}) \cdot \B \, dx.
$$
For the time derivative, integrate by parts in time:
$$
\int_0^T \int_\Omega \boldsymbol{\varphi} \cdot \partial_t \E \, dx dt = -\int_0^T \int_\Omega \partial_t \boldsymbol{\varphi} \cdot \E \, dx dt + \int_\Omega \boldsymbol{\varphi}(T, x) \cdot \E(T, x) \, dx - \int_\Omega \boldsymbol{\varphi}(0, x) \cdot \E(0, x) \, dx.
$$
The weak form becomes:
\begin{multline*}
-\int_0^T \int_\Omega \partial_t \boldsymbol{\varphi} \cdot \E \, dx dt + \int_0^T \int_\Omega (\nabla_x \times \boldsymbol{\varphi}) \cdot \B \, dx dt + 4\pi \int_0^T \int_\Omega \boldsymbol{\varphi} \cdot J \, dx dt + \int_\Omega \boldsymbol{\varphi}(T, x) \cdot \E(T, x) \, dx \\
- \int_\Omega \boldsymbol{\varphi}(0, x) \cdot \E(0, x) \, dx = 0.
\end{multline*}

For Faraday's Law \eqref{2speciesVM}$_3$, we have
$$
\partial_t \B + \nabla_x \times \E = 0.
$$
Multiply by a test function $\boldsymbol{\varphi} \in C^1_c([0,T] \times \Omega)$ and integrate:
$$
\int_0^T \int_\Omega \boldsymbol{\varphi} \cdot \left( \partial_t \B + \nabla_x \times \E \right) \, dx dt = 0.
$$
Split the terms:
$$
\int_0^T \int_\Omega \boldsymbol{\varphi} \cdot \partial_t \B \, dx dt + \int_0^T \int_\Omega \boldsymbol{\varphi} \cdot (\nabla_x \times \E) \, dx dt = 0.
$$
Using the integration by parts for the curl term and time derivative:
$$
\int_0^T \int_\Omega \boldsymbol{\varphi} \cdot (\nabla_x \times \E) \, dx = \int_0^T \int_\Omega (\nabla_x \times \boldsymbol{\varphi}) \cdot \E \, dx.
$$
$$
\int_0^T \int_\Omega \boldsymbol{\varphi} \cdot \partial_t \B \, dx dt = -\int_0^T \int_\Omega \partial_t \boldsymbol{\varphi} \cdot \B \, dx dt + \int_\Omega \boldsymbol{\varphi}(T, x) \cdot \B(T, x) \, dx - \int_\Omega \boldsymbol{\varphi}(0, x) \cdot \B(0, x) \, dx.
$$
The weak form becomes:
$$
-\int_0^T \int_\Omega \partial_t \boldsymbol{\varphi} \cdot \B \, dx dt + \int_0^T \int_\Omega (\nabla_x \times \boldsymbol{\varphi}) \cdot \E \, dx dt + \int_\Omega \boldsymbol{\varphi}(T, x) \cdot \B(T, x) \, dx - \int_\Omega \boldsymbol{\varphi}(0, x) \cdot \B(0, x) \, dx = 0.
$$

Regarding Gauss's Law for Electricity \eqref{2speciesVM}$_4$, we have
$$
\nabla_x \cdot \E = 4\pi \rho, \quad \rho = \int_{\mathbb{R}^3} (F_+ - F_-) \, dv.
$$
Multiply by $\varphi \in C^1_c([0,T] \times \Omega)$:
$$
\int_0^T \int_\Omega \varphi (\nabla_x \cdot \E - 4\pi \rho) \, dx dt = 0.
$$
Using the divergence theorem:
$$
\int_0^T \int_\Omega (\nabla_x \varphi) \cdot \E \, dx dt = 4\pi \int_0^T \int_\Omega \varphi \rho \, dx dt.
$$

 Lastly, for Gauss's Law for Magnetism \eqref{2speciesVM}$_5$, we have
$$
\nabla_x \cdot \B = 0.
$$
Multiply by $\varphi \in C^1_c([0,T] \times \Omega)$:
$$
\int_0^T \int_\Omega \varphi (\nabla_x \cdot \B) \, dx dt = 0.
$$
Using the divergence theorem:
$$
\int_0^T \int_\Omega (\nabla_x \varphi) \cdot \B \, dx dt = 0.
$$
These weak forms, combined with the perfect conductor boundary conditions, form the complete system.

\subsection{Weak Formulation of the Stationary and the Dynamical Vlasov Systems}
In this followings, we write the weak formulation of the stationary system \eqref{2speciesVM-steady} and the dynamical system \eqref{2speciesVM}:
\begin{definition}[Weak solutions to the stationary system]\label{def.weak.sol.st}
We call $$(F_{\pm,\textup{st}},\E_{\textup{st}},\B_{\textup{st}})\in (L^1\cap W^{1,\infty}_m)( \Omega\times \rth)\times W^{1,\infty}( \Omega)\times W^{1,\infty}( \Omega)$$ for some $m>4$  is a weak solution to the stationary problem \eqref{2speciesVM-steady}  if for any compactly supported smooth test functions $\phi\in C^1_c( \bar{\Omega}\times \rth)$, $\varphi \in C^1_c( \bar{\Omega})$ and its vector-valued counterpart  $\boldsymbol{\varphi}\in C^1_c( \bar{\Omega})$, we have 
\begin{equation}\label{VM-weak.st}
    \begin{split}
        &- \int_{\mathbb{R}^2\textup{ or }\mathbb{T}^2}dx_\parallel \int_{\rth \cap \{v_3>0\}} dv\ (\hat{v}_{+,3}F_{+,\textup{st}}-\hat{v}_{-,3}F_{-,\textup{st}})(x_\parallel,0,v)\phi(x_\parallel,0,v)\\&\qquad\qquad+\iint_{\Omega\times \rth}dxdv\  F_{\pm,\textup{st}}\left(\hat{v}_\pm\cdot\nabla_x \pm \left(\textup{e}\E_{\textup{st}}+\textup{e}\frac{\hat{v}_\pm}{c}\times \B_{\textup{st}}\mp m_\pm g\hat{e}_3\right)\cdot \nabla_v\right)\phi=0, \\
        &-\int_{\mathbb{R}^2\textup{ or }\mathbb{T}^2}(\mathbf{B}_{1,\textup{st}}\boldsymbol{\varphi}_2-\mathbf{B}_{2,\textup{st}}\boldsymbol{\varphi}_1)(x_\parallel,0)dx_\parallel-\int_\Omega (\nabla_x \times \boldsymbol{\varphi}) \cdot \B_{\textup{st}} \, dx = \frac{4\pi}{c}  \int_\Omega \boldsymbol{\varphi} \cdot J_{\textup{st}} \, dx,  \\
&\int_\Omega (\nabla_x \times \boldsymbol{\varphi}) \cdot \E_{\textup{st}} \, dx = 0,\\
&
-\int_{\mathbb{R}^2\textup{ or }\mathbb{T}^2} (\varphi \mathbf{E}_{\textup{st},3})(x_\parallel,0)dx_\parallel -\int_\Omega (\nabla_x \varphi) \cdot \E_{\textup{st}} \, dx = 4\pi \int_\Omega \varphi \rho_{\textup{st}}\, dx ,\\
&
 \int_\Omega (\nabla_x \varphi) \cdot \B_{\textup{st}} \, dx  = 0,\\
 & \int_\Omega (\nabla_x\varphi\cdot J_{\textup{st}})dx= -\int_{\mathbb{R}^2\textup{ or }\mathbb{T}^2} (\varphi J_{\textup{st},3})(x_\parallel,0)dx_\parallel ,
    \end{split}
\end{equation}
where $\rho_{\textup{st}} =\textup{e} \int_\rth (F_{+,\textup{st}}- F_{-,\textup{st}})dv$ and $J_{\textup{st}}=\textup{e}\int_\rth (\hat{v}_+F_{+,\textup{st}}- \hat{v}_+F_{-,\textup{st}})dv$. Note that the solution $F_{\pm,\textup{st}}\in L^1\cap W^{1,\infty}_m$ for $m>4$  makes the traces for $J_{3,st}$ be well-defined via the Ukai-type trace theorem and the $\tilde{\alpha}_{\pm,\textup{st}}$ weight (cf. \eqref{trace derivative}). Also, $\E_\textup{st},\B_\textup{st}\in W^{1,\infty}$ makes the trace of $\E_\textup{st},\B_\textup{st}$  be well-defined.
\end{definition}
 
\begin{definition}[Weak solutions to the dynamical system]\label{def.weak.sol.dyna}We call $$(F_\pm,\E,\B)\in (L^1\cap W^{1,\infty}_m)([0,\infty)\times \Omega\times \rth)\times W^{1,\infty}([0,\infty)\times \Omega)\times W^{1,\infty}([0,\infty)\times \Omega)$$ for some $m>4$ is a weak solution to the dynamical problem \eqref{2speciesVM}-\eqref{perfect cond. boundary} if for any compactly supported smooth test functions $\phi\in C^1_c([0,\infty)\times \bar{\Omega}\times \rth)$, $\varphi \in C^1_c([0,\infty)\times \bar{\Omega})$ and its vector-valued counterpart  $\boldsymbol{\varphi}\in C^1_c([0,\infty)\times \bar{\Omega})$, we have for $T>0$,
\begin{equation}\label{VM-weak}
    \begin{split}
        &\iint_{\Omega\times\rth}dxdv\ (\phi(T,x,v)F_\pm(T,x,v)-\phi(0,x,v)F_\pm(0,x,v))\\
        &\qquad\qquad-\int_0^Tdt \int_{\mathbb{R}^2\textup{ or }\mathbb{T}^2}dx_\parallel \int_{\rth \cap \{v_3>0\}} dv\ (\hat{v}_{+,3}F_+-\hat{v}_{-,3}F_-)(t,x_\parallel,0,v)\phi(t,x_\parallel,0,v)\\&\qquad\qquad=\int_0^Tdt\iint_{\Omega\times \rth}dxdv\  F_\pm\left(\partial_t+\hat{v}_\pm\cdot\nabla_x +\left(\pm\textup{e}\E\pm\textup{e}\frac{\hat{v}_\pm}{c}\times \B-  m_\pm g\hat{e}_3\right)\cdot \nabla_v\right)\phi, \\
        &
-\frac{1}{c}\int_0^T \int_\Omega \partial_t \boldsymbol{\varphi} \cdot \E \, dx dt + \int_0^T \int_\Omega (\nabla_x \times \boldsymbol{\varphi}) \cdot \B \, dx dt + \frac{4\pi}{c} \int_0^T \int_\Omega \boldsymbol{\varphi} \cdot J \, dx dt \\
&\qquad\qquad+\int_0^T\int_{\mathbb{R}^2\textup{ or }\mathbb{T}^2}(\Bone\boldsymbol{\varphi}_2-\Btwo\boldsymbol{\varphi}_1)(t,x_\parallel,0)dx_\parallel dt\\
&\qquad\qquad+ \frac{1}{c}\int_\Omega \boldsymbol{\varphi}(T, x) \cdot \E(T, x) \, dx -\frac{1}{c} \int_\Omega \boldsymbol{\varphi}(0, x) \cdot \E(0, x) \, dx = 0,\\
&
-\frac{1}{c}\int_0^T \int_\Omega \partial_t \boldsymbol{\varphi} \cdot \B \, dx dt + \int_0^T \int_\Omega (\nabla_x \times \boldsymbol{\varphi}) \cdot \E \, dx dt + \frac{1}{c}\int_\Omega \boldsymbol{\varphi}(T, x) \cdot \B(T, x) \, dx \\&\qquad\qquad- \frac{1}{c}\int_\Omega \boldsymbol{\varphi}(0, x) \cdot \B(0, x) \, dx = 0,\\
&
-\int_0^T\int_{\mathbb{R}^2\textup{ or }\mathbb{T}^2} (\varphi \Eth)(t,x_\parallel,0)dx_\parallel dt-\int_0^T \int_\Omega (\nabla_x \varphi) \cdot \E \, dx dt = 4\pi \int_0^T \int_\Omega \varphi \rho\, dx dt,\\
&
\int_0^T \int_\Omega (\nabla_x \varphi) \cdot \B \, dx dt = 0,\\
&\int_0^T \int_\Omega ((\partial_t\varphi)\rho +\nabla_x\varphi\cdot J)dxdt= \int_\Omega ((\varphi\rho)(T,x)-(\varphi\rho)(0,x)) dx -\int_0^T \int_{\mathbb{R}^2\textup{ or }\mathbb{T}^2} (\varphi J_3)(t,x_\parallel,0)dx_\parallel dt,
    \end{split}
\end{equation}
where $\rho = \textup{e}\int_\rth (F_+-F_-)dv$ and $j= \textup{e}\int_\rth (\hat{v}_+F_+-\hat{v}_-F_-)dv$. Note that the solution $F_\pm\in L^1\cap W^{1,\infty}_m$ for $m>4$ makes the trace for $J_3$ be well-defined via the Ukai trace theorem for the derivatives as in  \eqref{trace derivative} via the $\tilde{\alpha}_{\pm}$ weight. Also, $\E,\B\in W^{1,\infty}$ makes the trace of $\E,\B$ be well-defined.
\end{definition}

\section{Conservation Laws}\label{sec.B}
The perfect conductor boundary condition \eqref{perfect cond. boundary} can guarantee the uniqueness in a broad solution space. For instance, regarding a simple toy model of the homogeneous Maxwell equations, we observe
$$
\partial_t \E - \nabla \times \B =0, \ \ \partial_t \B + \nabla \times \E = 0,
$$which 
should yields that 
$$
\frac{1}{2}\frac{d}{dt}  \left( \|\E\|_2^2 + \| \B\|_2^2\right)
= \int_\Omega dx\ \left[( \nabla \times \B )\cdot  \E  - (\nabla \times \E )\cdot  \B\right]
= \int_\Omega dx\ \nabla \cdot (\B \times \E)  .
$$
The Gauss theorem yields that the last term should reduced to the boundary integral on $x_3=0$ such as 
$$
\int_{\partial\Omega} n \cdot (\B \times \E) dx , 
$$ for $n=(0,0,-1)^\top.$ 
While the boundary condition \eqref{perfect cond. boundary} forces we observe that this boundary integral vanishes. Therefore, as long as $\B \times \E$ has a well-defined trace at the boundary, then 
$$
\frac{1}{2}\frac{d}{dt}  \left( \|\E\|_2^2 + \| \B\|_2^2\right)=0,
$$
which should yield a uniqueness of the homogeneous Maxwell equations.

Now, we consider the full system of equations \eqref{2speciesVM}. Define the (rescaled) total energy density $e(t,x)$ as follows:
\begin{equation}
    \label{total energy density}
    e(t,x)=\frac{1}{2}(|\E(t,x)|^2+\B(t,x)|^2) + 4\pi \sum_\pm\int_\rth \vZ F_\pm(t,x,v)dv.
\end{equation}
Then by taking a temporal derivative on $e$ formally and using \eqref{2speciesVM} we observe that
\begin{multline}\label{energy density conserv}
\frac{d}{dt}e(t,x)= \E\cdot \partial_t \E + \B\cdot \partial_t \B + 4\pi \sum_\pm \int_\rth \vZ  \partial_t F_\pm(t,x,v)dv\\
=\E\cdot (\nabla_x\times \B-4\pi J) + \B\cdot (-\nabla_x\times \E) + 4\pi\sum_\pm \int_\rth \vZ  (-\nabla_x \cdot (\hat{v}_\pm F_\pm)-(\pm\E\pm\hat{v}_\pm\times \B-m_\pm g\hat{e}_3)\cdot \nabla_v F_\pm )dv\\
=\nabla_x \cdot (\B\times \E) - 4\pi \E\cdot J -4\pi \sum_\pm \nabla_x \cdot\int_\rth vF_\pm  
dv -4\pi \sum_\pm \int_\rth \nabla_v \cdot (\vZ (\pm\E\pm\hat{v}_\pm\times \B-m_\pm g\hat{e}_3)F_\pm)dv\\
+4\pi \sum_\pm \int_\rth \nabla_v(\vZ)  \cdot ((\pm\E\pm\hat{v}_\pm\times \B-m_\pm g\hat{e}_3)F_\pm)dv\\
=\nabla_x \cdot (\B\times \E) - 4\pi \E\cdot J -4\pi \sum_\pm \nabla_x \cdot\int_\rth vF_\pm dv +4\pi\sum_\pm \int_\rth (\pm\E\pm\hat{v}_\pm\times \B-m_\pm g\hat{e}_3)\cdot \hat{v}_\pm F_\pm)dv\\
=\nabla_x \cdot (\B\times \E)-4\pi\sum_\pm \nabla_x \cdot\int_\rth vF_\pm dv -4\pi \sum_\pm \int_\rth m_\pm g \hat{v}_{\pm,3}F_\pm(t,x,v)dv,
\end{multline}
where we used the integration by parts, the fact that $\nabla_v\cdot (\hat{v}_\pm\times \B)=0,$ and that $\nabla_v \vZ=\nabla_v \sqrt{m_\pm^2+|v|^2}= \frac{v}{\sqrt{m_\pm^2+|v|^2}}=\hat{v}_\pm$.
Therefore, by integrating \eqref{energy density conserv} with respect to $x$ on $\mathbb{T}^2\times (0,\infty)$, we have
\begin{multline}
    \label{energy conserv}
    \frac{d}{dt}\mathcal{H}(t) \eqdef \frac{d}{dt}\int_{\mathbb{T}^2\times (0,\infty)}dx\ e(t,x) \\
    =\int_{\mathbb{T}^2\times (0,\infty)}dx\ \left[\nabla_x \cdot (\B\times \E)-4\pi\sum_\pm \nabla_x \cdot\int_\rth vF_\pm dv -4\pi \sum_\pm \int_\rth m_\pm g \hat{v}_{\pm,3}F_\pm(t,x,v)dv\right]\\
=\int_{\mathbb{T}^2\times \{x_3=0\}}dS_x (-\hat{e}_3)\cdot \ \left[ (\B\times \E)-4\pi \sum_\pm \int_\rth vF_\pm dv\right] -4\pi \sum_\pm \int_\rth m_\pm g \hat{v}_{\pm,3}F_\pm(t,x,v)dv\\
=\int_{\mathbb{T}^2\times \{x_3=0\}}dS_x \ \left[ (-\Bone\Etwo+\Btwo\Eone)+4\pi\sum_\pm \int_\rth v_3F_\pm dv\right] -4\pi \sum_\pm \int_\rth m_\pm g \hat{v}_{\pm,3}F_\pm(t,x,v)dv\\
=4\pi \sum_\pm \int_{\mathbb{T}^2}dx_\parallel  \int_\rth v_3F_\pm (t,x_\parallel,0,v)dv -4\pi \sum_\pm \int_\rth m_\pm g \hat{v}_{\pm,3}F_\pm(t,x,v)dv\\
=4\pi \sum_\pm \int_{\mathbb{T}^2}dx_\parallel  \int_\rth (\vZ -m_\pm g)\hat{v}_{\pm,3}F_\pm (t,x_\parallel,0,v)dv ,
\end{multline}since we assume the perfect conductor boundary condition \eqref{perfect cond. boundary} for $\E$ and $\B$ which gives $(\E\times \B)_3 =0$ on $\{x_3=0\}.$ Here $S_x$ is the Lebesgue measure on the boundary $\{x_3=0\}.$ 

On the other hand, by integrating \eqref{continuity eq} with respect to $x\in \mathbb{T}^2\times (0,\infty)$ and using the divergence law, we further have
$$\frac{d}{dt}\int_{\mathbb{T}^2\times (0,\infty)}\rho dx\ =\int_{\mathbb{T}^2\times \{x_3=0\}}dS_x\ \hat{e}_3\cdot J=\int_{\mathbb{T}^2} J_3 (t,x_\parallel,0)dx_\parallel.$$

\section{Trace Theorems}
\label{sec.trace}
In order for rigorously defining and working with boundary conditions and boundary behaviors of solutions, we introduce several trace theorems in this section. We will provide $L^\infty$ trace theorems for the solution $\flo$ and its derivatives with special weights that we introduced as in the norm $|||\cdot|||$ defined in \eqref{norm X}.

We start with the following $L^\infty$ trace theorem which is a variation of the Ukai-type trace theorem. 
\begin{lemma}
    Suppose $\flo\in L^\infty((0,T)\times \Omega\times\rth)$ solves $$\partial_t \flo + (\hat{v}_\pm)\cdot \nabla_x \flo +\FS\cdot\nabla_v \flo =0,\ \textup{ for } (t,x,v)\in (0,T)\times \Omega\times\rth,$$ where $\FS$ is locally Lipshitz continuous in $(0,T)\times \Omega.$   
    Then we have
    $$\|\flo\|_{L^\infty ((0,T)\times (\Gamma_-\cup\Gamma_+))} \le \|\flo\|_{L^\infty ((0,T)\times \Omega\times \rth)},$$ where the incoming and outgoing boundaries $\Gamma_\pm$ are defined as \begin{equation}\label{incoming boundary}\Gamma_\pm=\{(x,v)\in \partial\mathbb{T}^2\times \rth: n_x\cdot v\gtrless 0\} \cup \gamma_\pm.\end{equation}
\end{lemma}
\begin{proof}
   Note that for any $s\in (0,T)$ sufficiently close to $t$, we have
\begin{equation}\label{f and trajectory}
    \flo(t,x,v)= \flo(s,\XSlo(s;t,x,v),\VSlo(s;t,x,v)),
\end{equation} where $\XSlo$ and $\VSlo$ solve \eqref{leading char}. Since $\FS$ is locally Lipshitz continuous in $(0,T)\times \Omega,$  $\XSlo$ and $\VSlo$ are well-defined. 
    Now we notice that for any $(t,x,v)\in (0,T)\times \Gamma_- $ there exists a sufficiently small $\epsilon=\epsilon(t,x,v)>0$ such that 
    $(t+\epsilon, \XSlo(t+\epsilon;t,x,v),\VSlo(t+\epsilon;t,x,v))\in (0,T)\times \Omega\times \rth$. Similarly, for the outgoing variables, for any $(t,x,v)\in (0,T)\times \Gamma_+ $ there exists a sufficiently small $\epsilon=\epsilon(t,x,v)>0$ such that 
    $(t-\epsilon, \XSlo(t-\epsilon;t,x,v),\VSlo(t-\epsilon;t,x,v))\in (0,T)\times \Omega\times \rth$. Therefore, for any $t\in (0,T)$, we have
  $$  \|\flo(t,\cdot,\cdot)\|_{L^\infty ((0,T)\times \Gamma_\pm)}
     \le \|\flo\|_{L^\infty((0,T)\times\Omega\times \rth)}.$$
         This completes the proof.
\end{proof}
Similarly, we can also provide trace theorems for the derivatives of $\flo.$
To this end, we first take derivatives on \eqref{f and trajectory} with respect to $x_\parallel.$ Suppose that $(t,x,v)\in (0,T)\times\Gamma_\pm$. Then for $s\in(0,T)$ sufficiently close to $t$ such that $(\XSlo(s),\VSlo(s))\in \Omega\times \rth,$ we have
\begin{multline*}
    \nabla_{x_\parallel}\flo (t,x,v)\\
    = (\nabla_x\flo)(s,\XSlo(s),\VSlo(s))\cdot \nabla_{x_\parallel}\XSlo(s)+(\nabla_v\flo)(s,\XSlo(s),\VSlo(s))\cdot \nabla_{x_\parallel}\VSlo(s),
\end{multline*}\begin{multline*}
    \partial_{x_3}\flo (t,x,v)\\= (\nabla_x\flo)(s,\XSlo(s),\VSlo(s))\cdot \partial_{x_3}\XSlo(s)+(\nabla_v\flo)(s,\XSlo(s),\VSlo(s))\cdot \partial_{x_3}\VSlo(s),
\end{multline*}
and \begin{multline*}
    \nabla_{v}\flo (t,x,v)\\= (\nabla_x\flo)(s,\XSlo(s),\VSlo(s))\cdot \nabla_{v}\XSlo(s)+(\nabla_v\flo)(s,\XSlo(s),\VSlo(s))\cdot \nabla_{v}\VSlo(s).
\end{multline*}
Under the same assumptions of Lemma \ref{cor.tbxbvb deri}, we use \eqref{derivatives of char} and \eqref{dxixj} to obtain that
\begin{multline*}
    |(\vZ)^m\nabla_{x_\parallel}\flo (t,x,v)|\le  |(\vZ)^m(\nabla_{x_\parallel}\flo)(s,\XSlo(s),\VSlo(s))||\nabla_{x_\parallel}(\XSlo)_\parallel(s)|
    \\+|(\vZ)^{m-1}(\partial_{x_3}\flo)(s,\XSlo(s),\VSlo(s))||(\vZ)\nabla_{x_\parallel}(\XSlo)_3(s)|\\
    +|(\vZ)^m(\nabla_v\flo)(s,\XSlo(s),\VSlo(s))|| \nabla_{x_\parallel}\VSlo(s)|\\
   \lesssim C_T |\langle \VSlo(s)\rangle^m(\nabla_{x_\parallel}\flo)(s,\XSlo(s),\VSlo(s))|
   \\ +C_T|\langle \VSlo(s)\rangle^{m-1}(\partial_{x_3}\flo)(s,\XSlo(s),\VSlo(s))||s-t|\\
    +C_T|\langle \VSlo(s)\rangle^m(\nabla_v\flo)(s,\XSlo(s),\VSlo(s))||s-t|\\
     \lesssim C_T \|(\vZ)^m\nabla_{x_\parallel}\flo\|_{L^\infty}
 \\   +C_T|\langle \VSlo(s)\rangle^{m-1}(\partial_{x_3}\flo)(s,\XSlo(s),\VSlo(s))||\tilde{\alpha}_{\pm}(s,\XSlo(s),\VSlo(s))\langle \VSlo(s)\rangle|\\+C_T\|(\vZ)^m\nabla_v\flo\|_{L^\infty}
    \\\lesssim_T \|(\vZ)^m\nabla_{x_\parallel}\flo\|_{L^\infty}
    +\|\langle v \rangle^m\tilde{\alpha}_{\pm}\partial_{x_3}\flo\|_{L^\infty}+\|(\vZ)^m\nabla_v\flo\|_{L^\infty},
\end{multline*}where we used \eqref{t-s bound by v and alpha.s} and \eqref{additional v bound.s}. Similarly, we can also obtain
\begin{multline*}
    |(\vZ)^m\nabla_{v}\flo (t,x,v)|\le  |(\vZ)^{m-1}(\nabla_{x_\parallel}\flo)(s,\XSlo(s),\VSlo(s))||(\vZ)\nabla_{v}(\XSlo)_\parallel(s)|
    \\+|(\vZ)^{m-1}(\partial_{x_3}\flo)(s,\XSlo(s),\VSlo(s))||(\vZ)\nabla_{v}(\XSlo)_3(s)|\\
    +|(\vZ)^m(\nabla_v\flo)(s,\XSlo(s),\VSlo(s))|| \nabla_{v}\VSlo(s)|\\
   \lesssim C_T |\langle \VSlo(s)\rangle^{m-1}(\nabla_{x_\parallel}\flo)(s,\XSlo(s),\VSlo(s))|
    \\+C_T|\langle \VSlo(s)\rangle^{m-1}(\partial_{x_3}\flo)(s,\XSlo(s),\VSlo(s))||s-t|\\
    +C_T|\langle \VSlo(s)\rangle^m(\nabla_v\flo)(s,\XSlo(s),\VSlo(s))||s-t|\\
     \lesssim C_T \|(\vZ)^{m-1}\nabla_{x_\parallel}\flo\|_{L^\infty}
  \\  +C_T|(\vZ)^{m-1}(\partial_{x_3}\flo)(s,\XSlo(s),\VSlo(s))||\tilde{\alpha}_{\pm}(s,\XSlo(s),\VSlo(s))\langle \VSlo(s)\rangle|\\
    +C_T\|(\vZ)^m\nabla_v\flo\|_{L^\infty}
  \\  \lesssim_T \|(\vZ)^{m-1}\nabla_{x_\parallel}\flo\|_{L^\infty}
    +\|\langle v \rangle^m\tilde{\alpha}_{\pm}\partial_{x_3}\flo\|_{L^\infty}+\|(\vZ)^m\nabla_v\flo\|_{L^\infty}.
\end{multline*}
Lastly, regarding the weighted derivative $(\vZ) \tilde{\alpha}_{\pm}\partial_{x_3}\flo,$ we observe
\begin{multline*}
    |(\vZ)^m\tilde{\alpha}_{\pm}(t,x,v)\partial_{x_3}\flo (t,x,v)|\\
    \lesssim_T  |\langle \VSlo(s)\rangle^{m-1}(\nabla_{x_\parallel}\flo)(s,\XSlo(s),\VSlo(s))||(\vZ)\partial_{x_3}(\XSlo)_\parallel(s)|
    \\+|\langle \VSlo(s)\rangle^m\tilde{\alpha}_{\pm}(s,\XSlo(s),\VSlo(s))(\partial_{x_3}\flo)(s,\XSlo(s),\VSlo(s))||\partial_{x_3}(\XSlo)_3(s)|\\
    +|\langle \VSlo(s)\rangle^m(\nabla_v\flo)(s,\XSlo(s),\VSlo(s))|| \partial_{x_3}\VSlo(s)|\\
    \lesssim_T \|(\vZ)^{m-1}\nabla_{x_\parallel}\flo\|_{L^\infty}
    +\|\langle v \rangle^m\tilde{\alpha}_{\pm}\partial_{x_3}\flo\|_{L^\infty}+\|(\vZ)^m\nabla_v\flo\|_{L^\infty},
\end{multline*}where we used $\tilde{\alpha}_{\pm}\le 1,$ \eqref{derivatives of char}, \eqref{dxixj}, and Lemma \ref{lemma.trajectory of alpha.s}. Therefore, we obtain the following trace theorem for the derivatives.
\begin{lemma}\label{lemma.deri f}
Suppose $\Elbf,\Blbf\in L^\infty((0,T);W^{1,\infty}(\Omega))$ and $\FS\in W^{1,\infty}((0,T)\times \Omega).$         Suppose $\flo\in W^{1,\infty}((0,T)\times \Omega\times\rth)$ solves $$\partial_t \flo + (\hat{v}_\pm)\cdot \nabla_x \flo +\FS\cdot\nabla_v \flo =0,\ \textup{ for } (t,x,v)\in (0,T)\times \Omega\times\rth.$$Suppose further that for all $t,x_\parallel,$ $$-(\FS)_3(t,x_\parallel,0)>c_0,$$ for some $c_0>0.$ 
    Then we have
    \begin{multline}\label{trace derivative}
        \|\langle v \rangle^m \nabla_{x_\parallel}\flo\|_{L^\infty ((0,T)\times (\Gamma_-\cup\Gamma_+))}+\|\langle v \rangle^m \tilde{\alpha}_{\pm}\partial_{x_3}\flo\|_{L^\infty ((0,T)\times (\Gamma_-\cup\Gamma_+))}\\
        +\|\langle v \rangle^m \nabla_{v}\flo\|_{L^\infty ((0,T)\times (\Gamma_-\cup\Gamma_+))} \\
        \lesssim  \|\langle v \rangle^m\nabla_{x_\parallel}\flo\|_{L^\infty((0,T)\times \Omega\times \rth)}
    +\|\langle v \rangle^m\tilde{\alpha}_{\pm}\partial_{x_3}\flo\|_{L^\infty((0,T)\times \Omega\times \rth)}\\+\|(\vZ)^m\nabla_v\flo\|_{L^\infty((0,T)\times \Omega\times \rth)},
    \end{multline} where the incoming and outgoing boundaries $\Gamma_\pm$ are defined in \eqref{incoming boundary}.
\end{lemma}
\begin{remark}
    Note that we can also obtain the same upper-bound for the time derivative $\partial_t\flo$ in Lemma \ref{lemma.deri f} by using the Vlasov equation and the fact that $$(\hat{v}_\pm)_3=\alpha_\pm(t,x_\parallel,0,v)\le 2 \tilde{\alpha}_{\pm}(t,x_\parallel,0,v)\lesssim_T\tilde{\alpha}_{\pm}(s). $$ 
\end{remark}

\bibliographystyle{amsplaindoi}
\bibliography{bibliography2.bib}{}

\end{document}